\documentclass[10pt]{preprint} 

\usepackage[a4paper,innermargin=1.2in,outermargin=1.2in,
bottom=1.5in,marginparwidth=1in,marginparsep
=3mm]{geometry}

\usepackage[T1]{fontenc} 
\usepackage{fourier} 
\usepackage[english]{babel} 
\usepackage{amsmath,amsfonts,amsthm,  amssymb } 
\usepackage{tikz-cd}

\usepackage[cal=boondoxo]{mathalfa}
\usepackage{shuffle}
\usepackage{fancyhdr} 
\usepackage{stmaryrd}
\usepackage{cprotect}
\usepackage{hyperref}
\usepackage{mhequ}
\usepackage{microtype}
\usepackage{esint}
\usepackage{faktor}
\usepackage{mathtools}
\usepackage{longtable} 
\usepackage[normalem]{ulem} 

\usepackage{xstring}
\usepackage{orcidlink}
\usepackage{RhysAlphabets} 

\def\CT{\mathcal{T}}

\newcommand{\fraks}{\mathfrak{s}}

\newcommand{\eps}{\varepsilon}
\newcommand{\eqdef}{\stackrel{\mbox{\rm\tiny def}}{=}}
\def\SetL{\mathrm{Set}_\mfL}
\def\proj{\hat{\otimes}_\pi}
\def\Ob{\mathrm{Ob}}
\def\hom{\mathrm{Hom}}

\def\sttau{{\langle \boldsymbol{\tau} \rangle}}
\def\sttaun#1{{\langle \boldsymbol{\tau}_#1 \rangle}}
\def\sssigma{{\langle \boldsymbol{\sigma} \rangle}}
\def\sssigman#1{{\langle \boldsymbol{\sigma}_{#1} \rangle}}
\def\SSet{\mathbf{SSet}}
\def\btau{{\boldsymbol{\tau}}}
\def\bsigma{{\boldsymbol{\sigma}}}
\def\ng{\not \ge}
\def\tC{\tilde{C}}
\def\dC{C^\dagger}
\newcommand{\eq}{\mathrm{eq}}
\newcommand{\Ban}{\mathrm{Ban}}

\newcommand{\vertiii}[1]{{\left\vert\kern-0.25ex\left\vert\kern-0.25ex\left\vert #1 
		\right\vert\kern-0.25ex\right\vert\kern-0.25ex\right\vert}}

\newtheorem{theorem}{Theorem}[section]

\newtheorem{corollary}[theorem]{Corollary}

\newtheorem{lemma}[theorem]{Lemma}
\newtheorem{definition}[theorem]{Definition}
\newtheorem{prop}[theorem]{Proposition}
\newtheorem{assumption}[theorem]{Assumption}
\newtheorem{example}[theorem]{Example}
\newtheorem{remark}[theorem]{Remark}

\newcommand{\id}{\mathrm{id}}
\numberwithin{equation}{section} 
\numberwithin{figure}{section} 
\numberwithin{table}{section}

\usetikzlibrary{calc}

\tikzset{
	eps/.style={circle,fill=white,draw=symbols,inner sep=0pt,minimum size=0.8mm},
	}
\makeatletter
\def\DeclareSymbol#1#2#3{\expandafter\gdef\csname MH@symb@#1\endcsname{\tikz[baseline=#2,scale=0.15,draw=symbols]{#3}}}
\def\<#1>{\csname MH@symb@#1\endcsname}
\makeatother

\makeatletter

\DeclareRobustCommand{\TitleEquation}[2]{\texorpdfstring{\StrLeft{\f@series}{1}[\@firstchar]$\if%
		b\@firstchar\boldsymbol{#1}\else#1\fi$}{#2}}

\makeatother

\newcommand{\horrule}[1]{\rule{\linewidth}{#1}} 
\title{	
	\vspace{-5em}
\horrule{0.5pt} \\[0.4cm] 
\large Renormalised Models for Variable Coefficient Singular SPDEs
 \\ 
\horrule{0.5pt} \\[0.5cm] 
}
\author{Lucas Broux$^1$\orcidlink{0009-0007-0716-8059},Harprit Singh$^2$\orcidlink{0000-0002-9991-8393} and Rhys Steele$^3$\orcidlink{0000-0001-6546-5800}}
\institute{
Max Planck Institute for Mathematics in the Sciences, Germany, \email{broux@mis.mpg.de}\and
University of Vienna, AT, \email{singhh92@univie.ac.at}\and Max Planck Institute for Mathematics in the Sciences, Germany, \email{steele@mis.mpg.de}}

\begin{document}

\maketitle 
\begin{abstract}
In this work we prove convergence of renormalised models in the framework of regularity structures \cite{Hai14} for a wide class of variable coefficient singular SPDEs in their full subcritical regimes. In particular, we provide for the first time an extension of the main results of \cite{CH16, HS23, BH23} beyond the translation invariant setting. In the non-translation invariant setting, it is necessary to introduce renormalisation functions rather than renormalisation constants. We show
that under a very general assumption, which we prove covers the case of second order parabolic operators, these renormalisation functions can be chosen to be local in the sense that their space-time dependence enters only through a finite order jet of the coefficient field of the differential operator at the given space-time point. Furthermore we show that the models we construct depend continuously on the coefficient field.
\end{abstract}
\tableofcontents

\section{Introduction}
	
Initiated by the breakthrough works \cite{Hai14, GIP15}, the past decade has seen remarkable progress on the well-posedness for (systems of)  singular SPDEs of the form
\begin{align}\label{eq:a48}
	L u = F((\partial^n u)_{|n| \le N}) + G((\partial^n u)_{|n| \le N}) \xi
\end{align}
where $L$ is a uniformly parabolic operator, $F,G$ are (appropriate) non-linearities and $\xi$ is a random and distribution-valued noise of 
regularity such that the equation is subcritical but singular.

In the case where $L$ has constant coefficients, the initial work of Hairer on regularity structures \cite{Hai14} has been developed over the papers \cite{CH16, BHZ19, BCCH20} (see also \cite{Bru18, BB21, HS23, BH23}) into a systematic machinery that provides a local in time well-posedness theory for a wide class of equations\footnote{posed on the quotient of $\mathbb{R}^d$ by a crystallographic group (which includes flat compact manifolds, c.f.\  \cite{charlap2012bieberbach}).} of the type of \eqref{eq:a48} in their entire subcritical regimes.  Furthermore, the past few years have seen the development of several alternative approaches capable of providing solution theories for equations in their entire subcritical regimes including the flow equation approach of Duch  \cite{Duc21,Duc22,DGR23,CF24a, CF24b} (see also the related works of Kupianen \cite{Kup16})
 and the multi-index based alternative to Hairer's tree based regularity structures developed by Otto and collaborators \cite{OSSW, LOT, LOTT, BL24, Tem24, BOS25}. 

Nonetheless all of these systematic approaches include an assumption that the operator $L$ is a constant coefficient operator. At first sight, and especially from a PDE perspective, this seems like a purely technical assumption. However, the passage from the constant to the variable coefficient setting leads to a meaningful qualitative change in the equation \eqref{eq:a48}. Indeed, due to the singular nature of the equation, the correct interpretation of a solution to \eqref{eq:a48} is as the limit of a sequence of solutions to the equations
\begin{align}\label{eq:a48b}
	Lu^\varepsilon = F((\partial^n u^\varepsilon)_{|n| \le N}) +  G((\partial^n u^\varepsilon)_{|n| \le N}) \xi^\varepsilon + \sum_\tau c_\varepsilon^\tau H_{\tau} ((\partial^n u^\varepsilon)_{|n| \le N})
\end{align} 
where $\xi^\varepsilon$ is obtained from $\xi$ by regularisation at a length-scale $\varepsilon$ and for a suitable choice of `counterterms' $c_\varepsilon^{\tau}H_{\tau} ((\partial^n u^\varepsilon)_{|n| \le N})$ .
 The fact that $L$ is a constant coefficient operator, in combination with an assumed translation invariance of the noise, leads to a formal translation invariance for the equation \eqref{eq:a48}. In combination with formal scaling symmetries, the translation invariance allows one to heuristically argue that in the renormalised equation \eqref{eq:a48b} we should have that $c_\varepsilon^{\tau}$ is a deterministic constant. In the setting where $L$ is a variable coefficient operator, translation invariance  
is replaced with translation equivariance so that the same heuristic argument only leads to the conclusion that $c_\varepsilon^\tau$ should be a local function of the space-time point in the sense that it depends on the space-time point $z$ only through $(\partial^n a(z))_{|n| \le M}$ for some value of $M$.  

The qualitative change in the form of the expected counterterm then has knock on effects in any of the systematic approaches to obtaining a solution theory to \eqref{eq:a48}. Indeed, large parts of the literature providing the algebraic machinery underpinning regularity structures \cite{BHZ19, BL24} rely on the assumption that renormalisation will involve only renormalisation constants. Furthermore, systematic approaches to obtaining the required stochastic estimates \cite{CH16, LOTT, HS23, BH23} all use the fact that the counterterms are constant through either the algebraic underpinnings or through the assumption that the stochastic objects that need to be estimated are translation invariant in law. We remark that a similar assumption is also present in the flow equation approach of Duch \cite{Duc21}.

As a consequence, the state of the art in terms of local well-posedness for equations of the form \eqref{eq:a48} in the situation where the operator has variable coefficients is limited to treating individual equations by hand and far from their critical threshold, see e.g.\ \cite{Sin23, HZZZ24}. Such results fall quite a way short of the results of the systematic framework available in the constant coefficient setting. Nonetheless, some ingredients required for such a systematic result in the framework of regularity structures are available in the variable coefficient setting in the literature.

We recall that roughly speaking, the approach of regularity structures is subdivided into four main parts.
In the first `algebraic' part, one constructs a regularity structure, which typically includes trees describing the successive Picard iterates of the mild formulation of \eqref{eq:a48}. In addition, for a smooth driving noise (e.g.\ obtained via an ultraviolet cut-off of a distributional noise), one has to provide a suitable algebraic recipe for the construction of renormalised models (at this fixed ultraviolet cut-off). In the constant coefficient setting, this task was originally tackled in the work \cite{BHZ19}. Whilst the construction of regularity structures given in \cite{BHZ19} applies in the same way to the variable coefficient setting, the construction of renormalised models does not (since the framework encodes the fact that renormalisation will be by renormalisation constants). However the more recursive approach to algebraic renormalisation provided in \cite{Bru18} does adapt very cleanly to the variable coefficient settings as was demonstrated in \cite{BB21}.

In the `analytic' part of the machinery, which is purely deterministic and pathwise, one formulates and solves an `abstract' fixed-point formulation of \eqref{eq:a48} in a space of modelled distributions assuming the model is provided. This part of the machinery is provided in \cite{Hai14} (see also \cite[Sections 2 and 5]{BCCH20}). We remark that the results of \cite{Hai14} also apply with minimal adaptations in the variable coefficient setting.

With these first two steps in hand, in a further `algebraic' part, it is necessary to identify how renormalisation of the model affects the corresponding equation. In the constant coefficient setting, this was first done in \cite{BCCH20}. In the variable coefficient setting, the analogue of this result is again provided in \cite{BB21}.

To complete the picture, in a `probabilistic' part of the machinery,
it remains to show convergence of suitably renormalised models as the ultraviolet cut-off is removed. The first systematic result in this direction was the work \cite{CH16} which is based on Feynman diagram techniques adapted from the QFT literature. Unfortunately, whilst the result of this paper is extremely flexible, the techniques of the proof are very involved and have not been simplified in the years following the paper's appearance. As such, it seems to be a very challenging technical task to adapt those techniques to more complicated settings such as the one considered here.

In very recent years, an alternative approach to obtaining stochastic estimates for models based on the spectral gap inequality has appeared.  This line of research was initiated by \cite{LOTT} which provides a recursive approach to the stochastic estimates for the model for a particular quasilinear equation in its full subcritical regime and in the multi-index based setting. The idea of obtaining stochastic estimates recursively via the spectral gap inequality was subsequently adopted in the tree based setting to obtain a result of similar practical scope to that of \cite{CH16} in both of the works \cite{HS23, BH23}. Unfortunately, all of these works use translation invariance at essentially the same point in their respective arguments in what seems to be a crucial way. Indeed, when applying the spectral gap inequality, in addition to estimating the Fr\'echet  derivative term, it is necessary to show that the expectation of the model is suitably controlled. In all of \cite{LOTT, HS23, BH23} this is done decomposing this quantity into an integral or a sum over contributions at various scales. The resulting boundary term can be easily controlled as the result of a qualitative renormalisation condition whilst the bulk term must be controlled only in terms of objects that appeared previously in terms of an inductive argument. This is achieved by appealing to translation invariance in law of the model and the fact that the bulk term involves integrating against a test function that kills constants. It is unclear how one would proceed in dealing with this bulk term without the assumption of translation invariance.

Therefore, the missing ingredient for a systematic small-time well-posedness result for \eqref{eq:a48} in the setting where the operator $L$ is a variable coefficient operator is the existence of a choice of renormalisation procedure such that the resulting models converge as the ultraviolet cut-off is removed. Historically, this part of the programme seems to form the bottleneck in the machinery. We mention here that the remaining parts of the machinery have been extended to a number of other settings including to the setting of discrete approximations \cite{EH16} and to the setting of SPDEs on manifolds \cite{HS23m}. However in each of these contexts the corresponding general stochastic estimates are missing.

The main goal of this paper is then to bridge this gap in the setting of variable coefficient operators thus for the first time providing a significant generalisation in scope to the original result of \cite{CH16}. 
 We refer the reader to Theorem~\ref{theo:main_result_1} for a statement of our main result which provides stochastic estimates for suitably renormalised models corresponding to variable coefficient differential operators. In addition to the problem of providing stochastic estimates, we also tackle algebraic problems pertaining to the form of the counterterm. As remarked upon earlier, it is important that rather than just exhibiting an arbitrary convergent sequence of models that one exhibits a convergent sequence of models that correspond to local counterterms at the level of the resulting renormalised SPDE. This issue is somewhat subtle since the naive BPHZ choice of renormalisation which is standard in the literature  will involve iterated integrals in which the coefficient field is evaluated at a space-time point appearing as a variable of integration. Therefore this choice of renormalisation will not lead to local counterterms in general. We refer the reader to Corollary~\ref{theo:main_result_2} for a statement of the locality result that we obtain in this work.

At a very high level, our approach in this work is akin to a `freezing of coefficients' argument at the level of the model. Therefore, we obtain a result that is very flexible with regards to the probabilistic assumptions required; essentially generalising all existing BPHZ theorems in regularity structures to the variable coefficient setting simultaneously. For a more detailed description of our strategy of proof and its advantages, we refer the reader to Section~\ref{ss:prf_strat} below. 

We remark also that beyond the immediate context of obtaining local in time well-posedness for variable coefficient singular SPDEs, the construction of renormalised models given here has other implications in the literature.

We mention that in the context of quasilinear SPDEs, the work \cite{BHK24} constructs solutions to a quasilinear version of the generalised KPZ equation viewed as a perturbation of a suitable variable coefficient analogue of the corresponding semilinear equation. As a result of this viewpoint, that paper needs to assume the input of suitably renormalised models corresponding to a variable coefficient operator. We remark that this assumption is not quite an immediate corollary of our main result due to the presence of a distinct infinite dimensional component in that problem coming from the fact that the perturbation to the semilinear problem arising from the quasilinearity involves an operator whose inverse is formally $0$-regularising. Nonetheless, we expect that the main ideas needed to provide the required input for that paper at a high level of generality are already present in this work.

The present work also has implications for singular SPDEs in geometric settings. 
It relates to singular SPDEs on Riemannian manifolds (see \cite{BB16,DDD19,BB19, Mou22,BDFT23a,BDFT23b,EMR24, CO25}) since (scalar valued) SPDEs  
on manifolds whose universal cover is $\mathbb{R}^d$
can be lifted as a non-constant coefficient equations on $\mathbb{R}^d$. Therefore, we see this work as a
significant step towards a geometric analogue of the black box machinery of \cite{Hai14,BHZ19,BCCH20,CH16}, since the only remaining ingredient in the setting of \cite{HS23m} is a geometric BPHZ-type theorem. 
We mention also that there has been recent work on geometric singular SPDE beyond the Riemannian setting
\cite{BOTW23, MS23, BAUD25}.

Finally, we mention that this work also relates to an ongoing program of homogenisation theory for singular SPDEs 
\cite{CX23,CFX23,HS23per}. Indeed, a robust understanding of variable coefficient problems is an obvious prerequisite to study homogenisation problems for singular SPDEs in full generality.

\subsection*{Acknowledgements}
HS would like to thank the Max Planck Institute for Mathematics in the Sciences for productive and pleasant visits during the course of the preparation of this paper. HS and RS would like to thank the Bernoulli Center at EPFL and Martin Hairer respectively for their hospitality and financial support for concurrent research visits during which part of this work was carried out. In addition, all authors would like to thank the Oberwolfach Research Fellows (OWRF) program for supporting a research stay during which parts of this work were finalised. RS would like to thank Francesco Pedull\`a for valuable discussions on minimal sets of trees for the construction of renormalised models. HS gratefully acknowledges financial support from the Swiss National Science
Foundation (SNSF), grant number 225606, and previously from Ilya Chevyrev’s New Investigator Award EP/X015688/1.

\subsection{Notations and Conventions} \label{ss:notations}

We record here some notation that will be used in the paper that is more standard in the literature. 
We refer the reader to
Appendix~\ref{s:index_notations} 
for an index of the less standardised notation introduced in this work.

We will endow $\mathbb{R}^d$ with a given scaling $\mfs = ( \mfs_1, \cdots, \mfs_d )$, and consider the induced distance $| y - x |_{\mfs} = \sum_{i=1}^d | x_i - y_i |^{1/\mfs_i}$
and the induced degree $|k|_{\mfs} = \sum_i \mfs_i k_i$ of a multi-index $k \in \mathbb{N}^d$. For convenience, we will assume that the pairwise ratios between the $\mfs_i$ are rational.\footnote{This assumption is implicitly used to allow access to wavelet techniques in technical proofs in the paper. We expect that this assumption is non-essential.}
We will denote by $\mfK$ a generic compact set in $\mathbb{R}^d$ and by $\bar{\mfK}$ its 1-fattening.
 
Some of our statements will be uniform in test  functions belonging to
$\mcB^{r} = \lbrace \phi \in C_c^{\infty} : \mathrm{supp} ( \phi ) \subset B_{\mfs} ( 0, 1 ), \| \phi \|_{C^r} \leq 1 \rbrace$, where $B_{\mfs} ( 0,1 )$ denotes the unit ball with respect to $| \cdot |_{\mfs}$
and where by $C^r$ we mean the space of functions having continuous partial derivatives of all orders $k$ with $| k |_{\mfs} \leq r$.
In most cases, statements involving $\mcB^r$ will involve an ambient regularity structure which will be clear from context.
In these cases (and unless explicitly stated) the value of $r$ is prescribed to be the smallest  non-negative integer such that $r + \min \mcA > 0$, where $\mcA$ is the grading of that regularity structure.

Given a set $\mcS$ we will denote by $\mcP(S)$ the power set of $\mcS$ and by $\langle \mcS \rangle$ the span of $\mcS$ (namely the free $\mathbb{R}$-vector space generated by $\mcS$). When $\mcS$ is a subset of an ambient vector space $V$ we will identify $\langle \mcS \rangle$ as a subspace of $V$ in the obvious way.
We will also use the notation $L ( E, F)$ to denote the space of bounded linear maps between given normed spaces $E$ and $F$.

\section{Setting and Main Results} \label{sec: main_results}

To keep the exposition here relatively short, we assume that the reader has some familiarity with the contents of \cite{Hai14, BHZ19}. For an overview of the output of \cite{BHZ19} that is suitable for our purposes we refer to \cite[Sections 2.1 and 2.2]{HS23}. 

As a prototypical example of the context for this paper, we consider a system of uniformly parabolic
semilinear subcritical SPDEs of the form
\begin{align} \label{eq:SPDEsystem} 
L_{\mft} u_{\mft} =  F^0_{\mft}(\mathbf{u}) + \sum_{\mathfrak{l} \in \mathfrak{L}_-} F^{\mathfrak{l}}_{\mft}(\mathbf{u}) \xi_{\mathfrak{l}} \quad \mft \in \mathfrak{L}_+
\end{align}
 where $\mfL = \mathfrak{L}_+ \sqcup \mfL_-$ is a finite set of types indexing components of the system of equations and the set of driving noises, $L_\mathfrak{t}$ are uniformly parabolic differential operators,  $F_{\mathfrak{t}}^\mathfrak{l}$ are smooth functions and $\mathbf{u}$ denotes the vector generated by the solution $u$ and the lower than leading order derivatives of $u$. 
 
\label{def:deg_assignment}
 More precisely, we will assume to have fixed a type set $\mathfrak{L}$ equipped with a degree assignment $|\cdot| _\mfs : \mfL \to \mathbb{R}$ such that $|\mfL_\pm| \subset \mathbb{R}_\pm$ alongside a complete, subcritical rule $R$ in the sense of \cite[Definitions 5.7, 5.14, 5.22]{BHZ19} which encodes the necessary combinatorial information appearing in the form of the non-linearity of \eqref{eq:SPDEsystem}.
 
 In addition to the definitions of \cite{BHZ19}, we will make the following simplifying technical assumption which is satisfied by all equations of interest that we are aware of. 
\begin{assumption}\label{ass:reg}
	We assume that the rule $R$ is such that no tree that conforms to $R$ contains a kernel type edge $e$ such that $|\mfl(e)|_\mfs - |\mfe(e)|_\mfs < 0$. 
\end{assumption}
This assumption says that we never encounter a situation where one convolves with $D^k K$ for a regularising kernel $K$ with a $k$ so large that $D^k K$ is no longer a regularising kernel. This assumption is also often already implicitly used in the literature. For example, in the approach to renormalisation of models via preparation maps of \cite{Bru18} (on which we build in Sections~\ref{s:Alg_Models} and \ref{s:Kolmogorov}), the assumption is used when dealing with the integration case of \cite[Proposition 3.7]{Bru18}. 
 
 \begin{definition}\label{def:BHZreduced}
 	We write $(\mathcal{T}^\eq, \mcG^\eq)$ for the reduced regularity structure corresponding to $R$ (as defined in\footnote{We emphasise that in this paper the extended decoration $\mfo$ appearing in \cite{BHZ19} will play no role} \cite[Definition 6.23]{BHZ19}). We write $\mcT_-^\eq$ for the subspace of $\mcT^\eq$ spanned by trees of negative degree.
 \end{definition}
  The main goal of this paper is to construct suitably renormalised models on $\mcT^\mathrm{eq}$.

\begin{definition}\label{def: model}
	A model over $\mathbb{R}^d$ on a regularity structure $(\mcT, \mcG)$ consists of a pair of maps $(\Pi, \Gamma)$ where
	\begin{enumerate}
		\item $\Gamma: \mathbb{R}^d \times \mathbb{R}^d \to \mcG$ 
			is such that $\Gamma_{x y} \Gamma_{y z} = \Gamma_{x z}$.
		\item $\Pi: \mathbb{R}^d \times \mcT \to \mcD^\prime(\mathbb{R}^d)$ is such that $\Pi_x \Gamma_{xy} = \Pi_y$.
	\end{enumerate}
	We additionally assume that the following analytic bounds hold for all $\gamma > 0$ and compact sets $\mfK \subseteq \mathbb{R}^d$:
	\begin{equs}\label{eq:mod_bounds_1}
		\|\Pi\|_{\gamma; \mfK} & \eqdef \sup_{\varphi \in \mcB^r} \sup_{\alpha < \gamma} \sup_{x \in \mfK} \sup_{\lambda \in (0,1]} \sup_{\tau \in \CT_\alpha} \lambda^{-\alpha} \|\tau\|^{-1} |\Pi_x \tau (\varphi_x^\lambda) | < \infty,
		\\
		\label{eq:mod_bounds_2}
		\|\Gamma\|_{\gamma; \mfK} & \eqdef \sup_{\beta < \alpha < \gamma} \sup_{x,y \in \mfK} \sup_{\tau \in \CT_\alpha} |x-y|_\mfs^{\beta-\alpha} \|\tau\|^{-1} \|\Gamma_{xy} \tau \|_\beta < \infty.
	\end{equs}
	We say that a model is smooth\footnote{We use the word smooth instead of continuous since smooth models as defined here will play the same role in this paper as smooth models do elsewhere in the literature} if $\Pi_x$ takes values in $C(\mathbb{R}^d)$ for each $x \in \mathbb{R}^d$. As usual, we measure the distance between two models $(\Pi^i, \Gamma^i)$ for $i = 1,2$ by replacing $\Pi, \Gamma$ in \eqref{eq:mod_bounds_1} and \eqref{eq:mod_bounds_2} by $\Pi^1 - \Pi^2$ and $\Gamma^1 - \Gamma^2$ respectively.
\end{definition}

An important role will be played by regularising kernels in the spirit of \cite[Assumption~5.1]{Hai14} (see also \cite[Definition~2.3]{BCZ24}).
These kernels are obtained from the Green's kernels of differential operators in terms of a suitable dyadic decomposition. As in \cite{Sin23}, it is natural in this paper to track continuity in the coefficient field of the differential operator which translates to tracking continuity properties with respect to kernels.

We start by defining  spaces $\mcK_\mco^{\beta}$ of translation invariant $\beta$-regularising kernels. Throughout this paper, given a sequence of Banach spaces $(X_n)_{n \ge 0}$ we write $\bigoplus_{n \in \mathbb{N}} X_n$ for their direct sum with the inductive limit topology and $\hat{\bigoplus}_{n \in \mathbb{N}} X_n$ for the completion of this vector space under the alternative topology induced by the norm $\|x\| = \sup_{n \in \mathbb{N}} \|x\|_{X_n}$. 
\begin{definition} \label{d:rk2n}
Let $\beta > 0$ and $\mco \in \mathbb{N}$.
For $n \in \mathbb{N}$ we define 
$\mcK_{\mco, n}^{\beta}$ to be the Banach space of functions 
$F \in C^{2\mco} ( \mathbb{R}^{d} , \mathbb{R} )$
that are supported in 
$B_{\mfs} ( 0, 2^{-n} )
= \lbrace z : | z |_{\mfs} \leq 2^{-n} \rbrace$, 
equipped with the norm
	\begin{align*}
		\| F \|_{\mcK_{\mco,n}^{\beta}} 
		& = \sup_{| k |_\mfs \leq 2 \mco} \sup_{z \in \mathbb{R}^d} \frac{\left| \partial^k F ( z ) \right| }{2^{(| \mfs | - \beta + | k |_{\mathfrak{s}})n}} . 
	\end{align*}
A translation invariant $\beta$-regularising kernel of order $\mco \in \mathbb{N}$
is an element of 
the Banach space
	\begin{align*}
		\mcK_{\mco}^{\beta}
		= 
		\hat{\bigoplus}_{n \in \mathbb{N}} \mcK_{\mco, n}^{\beta} .
	\end{align*}
\end{definition}

\begin{remark}
The reason for the appearance of $2\mco$ instead of $\mco$ in 
Definition~\ref{d:rk2n}
is for consistency 
with the definitions of  \cite[Assumption 5.1]{Hai14} and  \cite[Definition 2.3]{BCZ24} 
\end{remark}

\begin{remark} \label{rem:dec2ker}
Given $(K_n)_{n \in \mathbb{N}} \in \mcK_{\mco}^\beta$, one can obtain a $C^{2\mco}$-function $K: \mathbb{R}^d \setminus \{0\} \to \mathbb{R}$ via $K(x) = \sum_{n \ge 0} K_n(x)$.
This justifies the slight abuse of notation of identifiying an element of $\mcK_{\mco}^\beta$ with a function $K$ which comes with a distinguished decomposition $(K_n)_{n \in \mathbb{N}}$. 
\end{remark}

\begin{definition} \label{d:rk3}
A $\beta$-regularising kernel of order $\mco \in \mathbb{N}$
is an element of 
the Banach space
	\begin{align*}
		C^{\mco}(\mathcal{K}_{\mco}^\beta) = \hat{\bigoplus}_{n \in \mathbb{N}} C^{\mco}(\mcK_{\mco,n}^\beta) .
	\end{align*}
\end{definition}

We will write
a typical element of $C^{\mco}(\mcK_{\mco}^\beta)$ in the form $\mathbb{R}^d \ni w \mapsto K^w \in \mcK_{\mco}^\beta$,
and think of it as a kernel $G$ via the transformation $G ( z, \bar{z} ) \coloneqq K^{z} ( z - \bar{z} )$. We emphasise that $G$ is a regularising kernel in the sense of \cite[Assumption 5.1]{Hai14} up to the fact that we only assume control on finitely many derivatives. Furthermore, one can check that the map $K \mapsto G$ is continuous so long as one measures $G$ in norms similar to those of  \cite[Def.~1]{Sin23}.
The main difference between our definition and that of Hairer is that we require
uniformity i n the `upper slot' $w$. This condition captures the uniform ellipticity of the coefficient field and finitely many of its derivatives.

For a kernel $K \in C^{\mco} ( \mcK_{\mco}^{\beta} )$, 
we define the convolution $Kf$ 
as the distribution
	\begin{align*}
		Kf ( \phi ) 
		= \sum_{n \in \mathbb{N}} f \left( \int_{\mathbb{R}^d} \mathrm{d} z \, \phi ( z ) K_n^z ( z - \cdot ) \right).
	\end{align*}
This definition is such that under the identification of $K$ with $G$ as above, the convolution at the level of $K$ induces the usual convolution at the level of $G$.

\begin{definition} \label{d:kernel_ass_teq}
    A kernel assignment
	of order $\mcO \in \mathbb{N}$
	on $\mcT^{\mathrm{eq}}$
    is an element of
    the Banach space
    $\mathcal{A}_{+}^{\eq; \mcO} \coloneqq \bigoplus_{\mfl \in \mfL_+} C^{\mcO} ( \mcK_{\mcO}^{| \mfl |} )$, 
    equipped with the norm
		\begin{align*}
		 \| K \| _{\mathcal{A}_+^{\mathrm{eq}; \mcO}} 
		 \coloneqq \sum_{\mfl \in \mfL_+} \, 
			 \| K_{\mfl} \|_{C^{\mcO} ( \mcK_{\mcO}^{| \mfl |} )} . 
		\end{align*}
\end{definition}

\begin{definition} \label{d:noise_ass}
	A noise assignment is an element of the space 
	$\mcA_-^{\eq} = \bigoplus_{\mfl \in \mfL_-} \mcD^\prime(\mathbb{R}^d)$
	whilst a smooth noise assignment is an element of the subspace
	$\mcA_{-}^{\infty, \eq} = \bigoplus_{\mfl \in \mfL_-} C^\infty(\mathbb{R}^d)$. 
\end{definition}

\begin{remark}
	Given a noise assignment $\xi$ and a mollifier $\rho$, one naturally obtains a sequence of smooth noise assignments $\xi^\varepsilon$ that converges to $\xi$ in the space of noise assignments via $\xi_\mfl^\varepsilon = \varrho^\varepsilon \ast \xi_\mfl$.
\end{remark}

\begin{definition} \label{def:admissible_model_Teq}
    We say that a model $(\Pi, \Gamma)$ on a sector $W$ of $\mcT^\eq$ is admissible 
    (with respect to a kernel assignment $K$ and noise assignment $\xi$) 
    if it satisfies the relations
    \begin{align*}
        \Pi_x \Xi_\mfl &= \xi_\mfl, \qquad \qquad
        \Pi_x (X^k v) = (\cdot - x)^k \Pi_x v, \qquad \qquad \Pi_x D^k v = D^k \Pi_x v
    \end{align*}
    for all $\mfl \in \mfL_-$ such that $\Xi_\mfl \in W$ and for all $k,j \in \mbN$ and $v \in W$ such that $X^k v, D^j v \in W$. We also require that for $\mfl \in \mfL_+$ and $\tau \in W$ such that $\mcI_\mfl \tau \in W$ we have that
    \begin{align*}
        \Pi_x \mcI_{\mfl} \tau &= 
        K_\mfl \Pi_x \tau - \sum_{|k| \le |\mcI^\mfl \tau|} \frac{(\cdot - x)^k}{k!} D^k ( K_{\mfl} \Pi_x \tau)(x) .
    \end{align*}
We write $\mcM_0(W)$ for the space of all smooth admissible models on $W$ and $\mcM(W)$ for the space of all admissible models on $W$ with the topology arising from the seminorms of Definition~\ref{def: model} on the regularity structure $(W, \mcG)$.
\end{definition}

\subsection{Statement of Main Result}
We recall that \cite{BB21} shows that given a (state-space dependent) preparation map $P: \mathbb{R}^d \times \mcT^\mathrm{eq} \to \mcT^\mathrm{eq}$ (see \cite[Definition 5.1]{BB21}), a smooth noise assignment and a kernel assignment there exists a corresponding renormalised model. We note that our notion of kernel assignment differs to theirs with the meaningful difference being that we only control finitely many derivatives (which corresponds to assuming less regularity of the coefficient field). However, it can be shown that the model constructed in \cite{BB21} still makes sense on suitable sectors even under these weaker assumptions. We perform a construction of this type in a related setting in Section~\ref{ss: alg renorm} and sketch the adaptations to \cite{BB21} required in Section~\ref{ss: mod_eq}.

 \begin{theorem}\label{theo:main_result_1}
 	Let $\mcB$ be a finite set of trees forming the basis of a historic\footnote{This term is defined in Definition~\ref{def:historic} below. The reader should view it as a technical condition that says $\mcB$ contains enough trees to construct the renormalised model recursively} sector in $\mcT^\eq$. Fix a random noise assignment $\xi: \Omega \to \mcA_-^\eq$ and a mollifier $\rho$ and suppose that the pair $(\mcT^\eq, \xi)$ satisfies the assumptions of at least one of \cite{CH16, HS23, BH23}. Suppose also that Assumption~\ref{ass:reg} holds.
 	
 	Then there exists\footnote{The requirements on $\mcO$ depend on $\mcT^\eq$, $\mcB$ and which of \cite{CH16, HS23, BH23} is assumed as input.} an $\mcO > 0$ such that for each $K \in \mcA_+^{\eq; \mcO}$ there exists a sequence of state-space dependent preparation maps $P^\varepsilon$ on the sector $\langle \mcB \rangle$ with the property that if $Z^\varepsilon = (\Pi^\varepsilon, \Gamma^\varepsilon)$ denotes the model on $\langle \mcB \rangle$ associated to the preparation map $P^\varepsilon$, the noise assignment $\xi^\varepsilon$ and the kernel assignment $K$ then there exists a random model $Z = (\Pi, \Gamma)$ such that for each $p \in [1,\infty)$ and for each compact set $\mfK$, we have that
 	\begin{align*}
 		\mathbb{E}[\|Z^\varepsilon; Z\|_{\langle \mcB \rangle; \mfK}^p] \to 0
 	\end{align*}
 	as $\varepsilon \to 0$. In particular, $Z^\varepsilon \to Z$ in probability in $\mcM(\langle \mcB \rangle)$. Furthermore, we have that the map $K \mapsto Z$ is locally Lipschitz continuous from $\mcA_+^{\eq; \mcO}$ into the space of random variables valued in $\mcM(\langle \mcB \rangle)$ equipped with the (extended) pseudometric $d(Z^1, Z^2) = \mathbb{E}^{1/p}[\|Z^1, Z^2\|_{\langle \mcB \rangle; \mfK}^p]$.
 	
 	Furthermore, whilst the choice of $P^\varepsilon$ depends on the mollifier $\rho$, the limiting model $Z$ is independent of this choice. 
 	
 	Finally, there exists an integer $m$ such that $P^\varepsilon(z, \tau)$ has the explicit form $(\ell_z^\varepsilon \otimes \id )\Delta_r^- \tau$ for a family of linear functionals $\ell_z^\varepsilon$ on $\langle \mcB \rangle \cap \mcT_-^\eq$ that depend on $z$ only via $(\partial^k K^z)_{|k| < m}$ where the derivative is with respect to the variable $z$. Here $\Delta_r^-$ is the map defined in \cite[Definition 4.2]{Bru18}.
 	
 \end{theorem}
\begin{remark}
We recall that statements of the type of Theorem~\ref{theo:main_result_1} can be post-processed to obtain convergence of renormalised models associated to approximations $\xi^\varepsilon \to \xi$ that do not come from mollification via a `diagonal argument' using the fact that the map sending a smooth noise to the renormalised model is continuous in a strong enough topology on the space of smooth noises. For brevity, we do not perform this post-processing in this work and instead refer the reader to \cite[Proposition 4.7]{GH21} for a suitable abstract formulation of the diagonal argument and to e.g. \cite[Theorem 1.4]{Tem24} for an implementation of this principle in a related setting. 

In the case where $K \in \mcA_+^{\eq, \infty}$, it is also possible to post-process the convergence statement above to one on all of $\mcT^\eq$ by taking any $\mcB$ as in the statement containing all trees of negative degree and then making use of \cite[Proposition 3.31, Theorem 5.14]{Hai14} in a similar way to the proof of \cite[Theorem 10.7]{Hai14}. 
\end{remark}
 \begin{remark}
 	We note that since all of \cite{CH16, HS23, BH23} assume that the noise $\xi$ is translation invariant and has moments of all orders, our main result also implicitly makes this assumption. Otherwise, the role of these papers is to ensure that we have the stochastic estimates for translation invariant models corresponding to constant coefficient operators in a sufficiently robust manner. The precise estimates we require as input are simply estimates on the right hand side of the inequalities in Lemma~\ref{l:annealed_mb}. We emphasise that as a result this paper does not rely on a particular choice of assumption amongst spectral gap inequalities or moment-cumulant bounds for the noise. We view this flexibility as a significant benefit of our approach.
 \end{remark}
 
The relevance of the final claim in Theorem~\ref{theo:main_result_1} is that in the case where the rule corresponds to the system \eqref{eq:SPDEsystem}, the results of \cite{BB21} show that for fixed $\varepsilon > 0$, the solution of the fixed point problem for modelled distributions corresponding to \eqref{eq:SPDEsystem} has a reconstruction which solves a renormalised PDE with a counterterm that can be written explicitly in terms of $\ell^\varepsilon$. In particular, the final claim of Theorem~\ref{theo:main_result_1} provides a tool to probe the form of the counterterm. This leads us to formulate the following assumption under which the counterterm is local in the sense that it depends on $z$ only through the value of the coefficient field $a$ and finitely many of its derivatives at the point $z$.
\begin{assumption}\label{ass:locality}
	We assume that the kernel assignment $K$ is such that there exists an $N$ such that for each $k$ with $|k| < m$, we can write $\partial^k K_\mfl^z = f_k((\partial^j a(z))_{|j| \le N})$ for some function $f_k$ valued in $\mcK_\mco^{|\mfl|_\mfs}$.
\end{assumption}

We recall that the kernel assignments appearing in this context are typically given by decomposing the Green's kernel for the given operator into a singular part and a sufficiently smooth remainder and including only the singular part of the kernel in the kernel assignment since this is the only part that causes the need for renormalisation. In Section~\ref{s:green_kernel}, we verify that it is possible to exhibit a kernel decomposition for operators of the form $L = \partial_t - \sum_{i,j = 1}^d a_{i j}(t,x) \partial_i \partial_j$ for a suitably `nice' coefficient field $a$ (see Proposition~\ref{prop:ker_dec} for a precise statement) so that the singular part satisfies Assumption~\ref{ass:locality}. We expect that the techniques of this section should generalise to higher order uniformly parabolic operators. Combining Theorem~\ref{theo:main_result_1} and Assumption~\ref{ass:locality} immediately yields the following corollary.
 \begin{corollary}\label{theo:main_result_2}
	  Suppose that the assumptions of Theorem~\ref{theo:main_result_1} and Assumption~\ref{ass:locality} hold. Then the map $z \mapsto \ell_z^\varepsilon$  can be rewritten as a function of $(\partial^j a(z))_{|j| < m}$.
 \end{corollary}

 In particular, if $\varepsilon > 0$ and $U^\varepsilon$ denotes the solution of the lift of equation \eqref{eq:SPDEsystem} to the level of modelled distributions with respect to the model $Z^\varepsilon$ then combining the main results of \cite{BB21} and Corollary~\ref{theo:main_result_2} yields the existence of functions $c_\tau^\varepsilon$ such that $u^\varepsilon = \mcR^\varepsilon U^\varepsilon$ solves
 \begin{align*}
 	L_\mft u_\mft^\varepsilon(z) = F_\mft^0(\boldsymbol{u}^\varepsilon(z)) + \sum_{\mfl \in \mfL_-} F_\mft^\mfl(\boldsymbol{u}(z)) \xi_\mfl^\varepsilon(z) + \sum_{\tau \in \mcT_-(R)} c_\tau^\varepsilon[(\partial^j a(z))_{|j|< N}] \frac{\Upsilon_\mft[\tau, \boldsymbol{u}(z)]}{S(\tau)}
 \end{align*}
 where $\Upsilon$ and $S$ denote the standard `elementary differentials' and `symmetry factors' as were originally defined in \cite{BCCH20}.
 
 In the following subsection, we provide a high level overview of the strategy of proof of Theorem~\ref{theo:main_result_1}. In Subsection~\ref{ss:overview}, we describe in more detail the contributions of the individual sections of the paper. In particular, it is in this subsection that we describe more carefully the challenges addressed in the implementation of the proof strategy and the comparison to the existing literature.
 
\subsection{Strategy of Proof} \label{ss:prf_strat}

Rather than immediately working directly on $\mcT^\eq$, inpired by \cite{GH19}, we instead construct a regularity structure $\mcT^\Ban$ whose components are infinite dimensional Banach spaces which is designed to be large enough to support a family of abstract integration maps realising $K^z$ for each $z \in \mbR^d$. Our strategy is then to show that given a smooth driving noise, it is possible to associate renormalised models to a preparation map on $\mcT^\Ban$ and that there exists a choice of preparation map (which can be thought of as the BPHZ choice) such that the annealed form of the required stochastic estimates for this model can be lifted from the corresponding annealed stochastic estimates for the BPHZ models on a sufficiently large class of suitably constructed regularity structures with one-dimensional components as considered in \cite{CH16, HS23, BH23}. We then prove a Kolmogorov criterion for models on $\mcT^\Ban$ that allows us to obtain quenched estimates for the model on $\mcT^\Ban$.

In order to transfer this model to a model on $\mcT^\eq$, we adapt the framework of pointed modelled distributions introduced in \cite{HS23}. More precisely, we construct an operator on spaces of pointed modelled distributions on $\mcT^\Ban$ of H\"older type that lifts the convolution with the non-translation invariant kernel assignment $K$. In particular, in combination with the calculus of pointed modelled distributions already present in \cite{HS23}, this allows us to construct a family of pointed modelled distributions indexed by $\mcB$ whose reconstructions satisfy the requirements for the first component of an admissible model on $\mcT^\eq$. Furthermore, we show that the map sending a model on $\mcT^\Ban$ to this first component is locally Lipschitz continuous in suitable topologies. 

In order to complete the proof, it then remains to show that the object constructed in this way is in fact the first component of a model on $\mcT^\eq$ associated to a suitable preparation map. In order to do this, we provide an alternative and purely algebraic description of the pointed modelled distributions constructed in the previous step based on a family of characters on a suitable space of trees and a corresponding coproduct-type operation that is reminiscent of the coproduct used in positive renormalisation. This algebraic description is sufficiently strong to determine the corresponding preparation map at the level of $\mcT^\eq$. One notable feature is that the resulting preparation map does not agree with the naive BPHZ choice which seems to be essential in order to obtain the subsequent locality statement.

Let us highlight what we believe are advantages of our approach.
\begin{enumerate}
	\item We do not need to rely on a choice of probabilistic technique for the stochastic estimates, but rather we are able to take the BPHZ estimates of any of \cite{CH16,HS23,BH23} directly \emph{as an input}.
	This widens the range of applicability of our result since those papers are based on different (and not comparable) sets of assumptions on the driving noise.
	
	\item We further develop the function space theory for modelled distributions \cite{Hai14} and pointed modelled distributions \cite{HS23}, which we believe to be of independent interest. In particular, pointed modelled distributions are demonstrated to be a robust tool for constructing one `model-like object' from another.
	
	\item Whilst our strategy of proof involves working with regularity structures with infinite dimensional components, the output is a model on the usual (scalar-valued) regularity structure associated to a given equation. In this way, more technical considerations involving the infinite dimensional nature of $\mcT^\Ban$ can be viewed as hidden in the `back-end of the machinery' and therefore should not arise when working with the solution to the equation itself. In particular, we expect the output of this approach would allow one to more easily extend other results that are known in the constant coefficient setting (such as pathwise approaches to global existence via regularity structures \cite{MW20, CMW23, EW24, CFW24}) 
	 to the variable coefficient setting. 

	\item 
	As mentioned above, in the setting of non-constant coefficient operators it is well-understood that renormalisation functions are required in general and it is desirable that these functions should satisfy constraints such as the locality condition described above (which also ensures translation equivariance of the renormalised equation).	Our construction has the advantage that this property comes built in (assuming appropriate asymptotic expansions\footnote{
	Roughly speaking, we use as input that the non-constant coefficient heat kernel $\Gamma$ close to some point $z_0$ can be well approximated by constant coefficient kernels depending only on the jet of the coefficient field at $z_0$. While such a decomposition is to be expected, we do provide a detailed proof in the case of non-constant coefficent heat operators. 
	} of the Green's kernel) and does not require extended algebraic post-processing.
	\item Finally, we mention that tracking the appearance of the counterterm through our construction yields an exact formula for the counterterm $c_\tau^\varepsilon[(\partial^j a(z))_{|j|< r_0}] $ in terms of renormalisation constants $\bar{c}_\tau$ corresponding to models on a suitable family of regularity structures which are admissible for constant coefficient operators. We expect that this property may be useful for deducing properties of the renormalised equation. 
\end{enumerate}

\begin{remark}
	In order to obtain local-in-time well-posedness for \eqref{eq:SPDEsystem} an alternative approach would be to attempt to freeze coefficients at the level of the equation rather than at the level of the model. This would essentially lead to attempting to lift the SPDE to an appropriate fixed point problem for modelled distributions on $\mcT^\Ban$ with the initial difficulty being the task of providing a setting coming with an appropriate notion of abstract integration operator encoding `variable coefficient convolution' whilst using only the frozen coefficient models.
	
	We note that (up to technical details involving including appropriate weights near the $(t=0)$-hyperplane to allow the treatment of problems with non-trivial initial data) the contents of Section~\ref{sec:mod_dist} provides these analytic tools. In addition, Section~\ref{s:Kolmogorov} provides the renormalised models that would be needed to drive this fixed point problem.
	
	We do not pursue this alternative strategy since its output would interface more poorly with the existing literature due to the fact that the modelled distributions produced as solutions live on $\mcT^\Ban$ rather than the completely standard $\mcT^\eq$. In particular, in order to deduce the form of the renormalised equation, it would be necessary to prove an analogue of the results of \cite{BB21, BCCH20} in this infinite dimensional setting rather than simply taking their result `off the shelf'. Furthermore, we would expect any future attempt to deduce properties of the solution (see the third point above) to be significantly more technically involved when working with the output of this strategy.
\end{remark}

\subsection{Article Structure}\label{ss:overview}
 
 We now provide an overview of the contributions of the individual sections of this paper and a description of how they compare to the literature.
 
 \subsubsection*{Section~\ref{s:Reg_Struc}: Construction of Regularity Structures}
 
 In Section~\ref{s:Reg_Struc}, we construct the required regularity structures supporting abstract integration maps realising $K^z$ for each $z \in \mbR^d$ simultaneously. Here we are heavily inspired by \cite{GH19}. Roughly speaking, the idea adopted in that work is that whereas standard tree-based regularity structures associate to each tree a copy of $\mbR$ and think of each edge as representing a fixed kernel, one now wants to associate to each edge of a given tree a suitable Banach space (say of regularising kernels or distributions). In particular, each tree should now correspond to a (projective) tensor product of potentially infinite dimensional Banach spaces rather than a single copy of $\mbR$. Whilst the constructions of \cite{GH19} are relatively lightweight, they have two deficiencies from the perspective of this work. 
 
 Firstly, the construction of renormalised models in \cite{GH19} proceeds by directly patching together models on more standard regularity structures, rather than by adapting an algebraic construction of renormalised models to the infinite-dimensional setting. This construction, whilst short and certainly sufficient for the purposes of \cite{GH19}, would not carry enough information to identify the resulting model on $\mcT^\eq$ after the remaining constructions of this paper. To this end, it is essential that we develop a stronger algebraic machinery describing models on $\mcT^\Ban$ that is more reminiscent of \cite{BHZ19, Bru18}. 
 
 Secondly, an unfortunate side effect of the lightweight construction of \cite{GH19} is that the tree product necessarily becomes non-commutative due to the fact that the tensor products appearing do not reflect the symmetries of a given tree. This means that product operations that ought to be commutative do not automatically have that property and instead this must be enforced by hand. Whilst doing this would be possible, it would also be quite tedious to do at each of the many steps in the machinery of this paper. As a result, it is convenient to already enforce the appropriate symmetries at the level of the regularity structure itself by constructing the required algebraic operations on appropriately symmetrised (projective) tensor products of infinite dimensional Banach spaces.
 
 Both of these issues were previously addressed in the setting in which the Banach spaces `glued' into edges are finite dimensional in \cite{CCHS22}. This framework was subsequently applied in \cite{HS23m, GHM24}. In \cite{CCHS22}, the authors define a category $\SSet$ of `symmetric sets' with a class of objects that naturally includes the combinatorial trees appearing in regularity structures and a class of morphisms that respect the required symmetries. They then show that operations such as the tree product can be realised at the combinatorial level as morphisms in the category $\SSet$ and furthermore that there exists a functor from $\SSet$ into the category $\mathbf{TVec}$ of topological vector spaces that sends each tree $\btau$ to the corresponding symmetrised tensor product $V^{\otimes \sttau}$. The functoriality of this construction allows them to push symmetric operations such as the coproducts required to define the structure group that are defined at the level of combinatorial trees along the functor to the vector-valued setting.
 
In addition to more minor technical adaptations needed for our infinite dimensional setting, a more significant limitation of this construction is that not all maps that one might want to construct are morphisms that lie in the image of this functor. In particular, when constructing models recursively (as is done in the preparation map  approach of \cite{Bru18, BB21}), it is essential to have a calculus for maps $V^{\otimes \sttau} \to C^{\infty}(\mbR^d)$. Since it does not seem possible to express $C^{\infty}(\mbR^d)$ as a symmetric tensor product corresponding to a symmetric set supporting a rich enough class of morphisms, we would not expect to be able to construct such a calculus only by appealing to functoriality. 

In any of \cite{CCHS22, HS23m, GHM24} this issue is circumvented by writing down a non-recursive definition of the canonical lift and viewing all renormalisation as acting on trees as in \cite{BHZ19}. Whilst it would be possible to adopt this approach here, we would run into significant difficulties when trying to pass the construction back down to $\mcT^\eq$ since in that latter non-translation invariant setting an analogue of \cite{BHZ19} does not exist in the literature. In comparison, the approach of preparation maps given in \cite{Bru18} adapts very cleanly to the non-translation invariant setting, as demonstrated in \cite{BB21}. 

In order to work with compatible algebraic approaches at the level of both $\mcT^\Ban$ and $\mcT^\eq$, we therefore provide an alternative to the category theory based constructions of \cite{CCHS22} that is based on characterising the space $V^{\proj \sttau}$ via a suitable universal property in $\mathbf{TVec}$ directly. This universal property is sufficiently powerful to define all of the maps needed to construct renormalised models in the approach of either of \cite{BHZ19} or \cite{Bru18}. 

In Remark~\ref{rem: univ_colim}, we explain that the fact that $V^{\proj \sttau}$ satisfies our universal property can be interpreted as the fact that the functor of \cite{CCHS22} preserves a certain class of colimits. Whilst this observation would allow us to write down a hybrid of the two approaches, we consider it an advantage that the universal property allows us to circumvent categorical constructions such as the construction of the larger category $\mathbf{TStruc}$ needed in \cite{CCHS22}. As a result, we adopt a more hands on approach using the universal property throughout Sections~\ref{s:Reg_Struc} and \ref{s:Alg_Models}.
 
The structure of this section is as follows. In Section~\ref{s: tensor} we introduce the universal property defining $V^{\proj \sttau}$. Since we will not need properties of $V^{\proj \sttau}$ beyond those contained in the universal property, we defer the construction of the (unique) space satisfying this property to Appendix~\ref{app: tensor}. In Section~\ref{s: tensor_oper}, we then apply the universal property to construct the tree product, gradient and the coproduct realising the structure group in the infinite dimensional setting. We further show that these operations satisfy all of the necessary properties to yield a regularity structure via the usual Hopf algebraic construction.

 \subsubsection*{Section~\ref{s:Alg_Models}: Renormalised Models on $\mcT^\Ban$}
 
 In Section~\ref{s:Alg_Models}, we perform the construction of renormalised models on $\mcT^\Ban$ via the universal property of Definition~\ref{def:tensor_product}. In particular, in Section~\ref{ss:Adm_mod} we introduce an appropriate notion of admissible model for the setting and in Section~\ref{ss: alg renorm} we show that it is possible to construct such models from generic preparation maps. One technical detail tackled in these sections is that since we do not assume the coefficient field of our operator to be smooth, we only have control on finitely many derivatives of the Green's kernels appearing and therefore cannot construct models on the full regularity structure. Since the construction of models followed here is recursive, it is therefore necessary to construct a minimal set of trees needed in that construction in order to codify the correct assumptions on the Green's kernel. Beyond the setting considered here, we would expect this construction to be of independent interest in other situations where more care is needed when constructing the model such as in situations where the Green's kernel has more limited decay at infinity (as is the case for fractional heat operators) or in situations where the regularisation mechanism used in the ultraviolet cut-off only yields limited smoothness of the driving noise. 
 
 \subsubsection*{Section~\ref{s:Kolmogorov}: Stochastic Estimates for the BPHZ Model}
 
 In Section~\ref{ss:annealed}, we specialise to preparation maps built from the coproduct type operation $\Delta_r^-$ introduced at the level of combinatorial trees in \cite{Bru18}. We establish the existence and uniqueness of a `BPHZ model' on $\mcT^\Ban$ in Proposition~\ref{lem: BPHZ exists}. Since the BPHZ preparation map is not defined a priori but is rather constructed during the induction defining the corresponding model, this again requires some care with regards to the minimal set of trees needed to define the BPHZ preparation map. To this end, we introduce a notion of the `history' of a set of trees and show that it is possible to construct simultaneously the BPHZ preparation map and model with respect to a quantity we call the `age of a tree'.
 
 In Lemma~\ref{l:annealed_mb} we show that it is possible to lift stochastic estimates for the BPHZ model in their annealed form from their scalar-valued analogues. This is done by constructing a suitable class of scalar-valued regularity structures and showing in Lemma~\ref{lem:up_down} that  the BPHZ model on $\mcT^\Ban$ evaluated at a fixed `elementary symmetric tensor' coincides with the BPHZ model on an appropriate choice of such structure for a particular choice of kernel assignment.
 
 In Section~\ref{ss:quenched}, we then argue that it is possible to post-process the annealed estimates obtained via Lemma~\ref{l:annealed_mb} to their quenched form. In comparison to the usual scalar-valued setting, the main subtlety here is the presence of an additional supremum over $v \in V^{\proj \sttau}$ which leads to a situation where one would desire a Kolmogorov criterion for stochastic processes indexed by the elements of an infinite dimensional Banach space. Since such results are typically quite subtle, we adopt the idea from \cite{GH19} of obtaining the quenched estimates from the annealed estimates on a larger structure by showing that the BPHZ model on the smaller structure can be viewed as a continuous and deterministic function of a process on the larger structure that is indexed by $\mbR^d$.

 \subsubsection*{Section~\ref{sec:mod_dist}: (Pointed) Modelled Distributions}
 
 With the tools to obtain stochastic estimates on $\mcT^\Ban$ in hand, in Section~\ref{sec:mod_dist}, we turn to the development of the analytic tools needed to transfer the estimates to $\mcT^\eq$. To that end, we adapt the framework of Pointed Modelled Distributions first introduced in \cite{HS23}. In comparison to \cite{HS23}, there are two main differences in this paper. The first is that here we work purely in the setting of modelled distributions of H\"older type. Whilst this is a special case of the setting of \cite{HS23}, due to our stronger assumptions, it is possible to remove several parameter restrictions appearing in that paper. Since removing those parameter restrictions is essential in our application of the framework in Section~\ref{s:ren_models}, we provide the necessary adaptations of the statements including a proof of a new reconstruction theorem for H\"older pointed modelled distributions in Theorem~\ref{theo: pointed recon}. 
 
 The second main difference is that we require an abstract integration operator on modelled distributions on $\mcT^\Ban$ that realises convolution with an element of $C^\mcO(\mcK_\mcO^\beta)$, rather than the analogue considered in \cite{HS23} in which for both the model and the modelled distribution one works with the same translation invariant kernel. In Section~\ref{ss:VML_Schauder} we construct such operators and show that they have the desired mapping properties. Our definition of these operators is in spirit similar to that of \cite[Equation 4.1]{GH19}, though is conceptually different due to the method of freezing in the coefficient. In \cite{GH19}, a non-linear transformation of the equation under consideration is made so as to allow to freeze at the value of the coefficient field and still obtain an appropriate convolution operator. Since there is not an analogous transformation here, an attempt to freeze in the coefficient field and still obtain the right convolution operator would instead require freezing a sufficiently large jet of derivatives of the coefficient field in a way that depends non-trivially on the specific structure of the kernel decomposition under consideration. In order to remain more generic and to track fewer parameters, we instead choose to freeze in the space-time point which leads to changes in both the form of the definition and the spaces of distributions that we are required to work with.
 
 A final more minor difference with \cite{HS23} is that we are again more careful in the assumptions we make on kernel assignments here. In particular, we also carefully track the continuity of various operations in the kernel assignment in order to eventually obtain continuity of the BPHZ model as a function of the coefficient-field.
 
 \subsubsection*{Section~\ref{s:ren_models}: Renormalised Models on $\mcT^\eq$}
 
 In Section~\ref{s:ren_models}, we perform the construction of the preparation map on $\mcT^\eq$ which corresponds to the model for which we will eventually obtain stochastic estimates.
 
 In particular, we first apply the results of Section~\ref{sec:mod_dist} to construct a family of modelled distributions (indexed by trees in $\mcT^\eq$) whose reconstruction will yield the first component $\Pi$ of a model on $\mcT^\eq$. Whilst this construction immediately yields the desired estimates due to continuity properties established in Section~\ref{sec:mod_dist}, it does not provide the more algebraic ingredients such as a suitable recentering map $\Gamma$ or a sufficiently strong algebraic description to identify the form of the counterterm in the renormalised equation that would result from driving a fixed point problem via this model. Therefore, in the remainder of Section~\ref{s:ren_models} our goal is to show that $\Pi$ constructed in this way is in fact the first component of a model associated to a certain preparation map defined in Section~\ref{ss: mod_eq}. 
 
 To achieve this, we first show in Section~\ref{ss:alg_rep} that the recursively constructed family of modelled distributions mentioned above admits an alternative and purely algebraic description in terms of a coproduct type operation $\tilde{\Delta}$. In Section~\ref{ss:mod_ident}, we then show an approximate commutation type relation between $\Delta_r^-$ and $\tilde{\Delta}$ that allows us to identify $\Pi$ as the first component of the model associated to the desired preparation map in Proposition~\ref{prop:f_algebraic_renormalisation}. 
 
\subsubsection*{Section~\ref{s:Proof_Main}: Proof of Theorem~\ref{theo:main_result_1}}

With all of the ingredients assembled, in Section~\ref{s:Proof_Main} we provide a proof of Theorem~\ref{theo:main_result_1}. Since by this point all of the hard work has been done, this section is relatively short.

\subsubsection*{Section~\ref{s:green_kernel}: Decompositions of Green's Kernels for Local Counterterms}

Whilst Theorem~\ref{theo:main_result_1} provides a BPHZ statement for a sequence of models on $\mcT^\eq$, it does not show that the counterterms in a corresponding renormalised equation would be local ones. To achieve this, one also has to assume that Assumption~\ref{ass:locality} holds. In Section~\ref{s:green_kernel} we show that for operators of the form $$ \partial_t - \sum_{i,j = 1}^d a_{i j}(t,x) \partial_i \partial_j - \sum_{i = 1}^d b_{i}(t,x) \partial_i - c (t,x)$$ a decomposition of the corresponding Green's kernel $\Gamma$ with singular part satisfying Assumption~\ref{ass:elliptic_and_bounds} does exist under suitable assumptions on the coefficient fields that are expressed in Assumption~\ref{ass:elliptic_and_bounds}. This means that in combination with \cite{Hai14, BB21}, Theorem~\ref{theo:main_result_1} yields a local in time well-posedness result with local renormalisation functions for a general class of second order parabolic subcritical singular SPDEs with sufficiently nice coefficients.

Our approach comes in two parts. Firstly we briefly provide a construction of $\Gamma$ as the sum of a certain Volterra series $\sum_k Z \ast (-E)^{\ast k}$ for suitably chosen $Z$ and $E$. Whilst such constructions are standard in the literature (see e.g.\ \cite{Fri08, Grieser}), we find it useful to recall some of the details to fix notation and to obtain finer information on the terms in the Volterra series. To this end, we find it convenient to implement the approach taken in \cite{Grieser}. However, in comparison to \cite{Grieser} our setting has several technical differences that we account for. Firstly, rather than working with the Laplace-Beltrami operator on a compact Riemannian manifold, we work on all of $\mbR^{1+d}$ (in particular, allowing the coefficient field to be time dependent which requires some technical adaptations). Secondly, rather than assuming smoothness of the coefficients, we assume only a limited amount of H\"older regularity. Whilst we don't expect the regularity assumption we make to be optimal, it is useful to track this finer information since it allows us to show that the objects constructed in this section depend continuously on $(a,b,c)$ as measured in a suitable H\"older topology. 

In the second part of our approach (which is not present in standard constructions of the heat kernel),  we postprocess the construction of $\Gamma$ as a Volterra series to provide a decomposition satisfying Assumption~\ref{ass:locality}. Since $Z \ast (-E)^{\ast k}$ will typically not be in a form compatible with this assumption, we proceed by a Taylor expansion argument in the `upper slot' in both $Z$ and $E$.  We then show that at the level of $Z \ast (-E)^{\ast k}$ this results in a decomposition where all terms arising purely from the main part of the Taylor expansion of $Z$ and $E$ are of a form that is compatible with Assumption~\ref{ass:locality} whilst all terms that contain at least one instance of a Taylor remainder are much more regularising and thus can be safely absorbed into the remainder part of the kernel decomposition $\Gamma = K + R$ since they will not cause a need for additional renormalisation.
\section{Regularity Structures from Tensor Products of Banach Spaces}\label{s:Reg_Struc}

In this section, we construct regularity structures whose components are partially symmetrised tensor products of infinite dimensional Banach spaces. For a discussion of how our approach in this section relates to the literature, we refer the reader to Section~\ref{ss:overview}.

We begin by recalling some definitions from \cite{CCHS22} that will be used in what follows. We recall that we fixed a set of types $\mathfrak{L}$ in Section~\ref{sec: main_results}.

\begin{definition}
	A typed set $(T, \mfl)$ is a
	(finite) 
	set $T$ equipped with a type map $\mfl: T \to \mfL$. An isomorphism of typed sets is a type-preserving bijection.
\end{definition}
We write $\SetL$ for the category of typed sets with morphisms being the isomorphisms of typed sets. 

\begin{definition}\label{def: SymSet}
	A symmetric set is a connected groupoid in $\SetL$.
\end{definition}

\begin{remark}
This definition is equivalent to 
\cite[Definition 5.3]{CCHS22}, see Remark 5.3 therein. 
Thus a symmetric set is a collection of typed sets $\{a: a \in \Ob(\mcs)\}$ together with a collection of morphisms $(\hom_{\mcs} (a, b) )_{a, b \in \Ob(\mcs)}$
	such that for each $a,b \in \Ob(\mcs)$, $\hom_{\mcs} (a, b)$
	is a non-empty set of isomorphisms from $a$ to $b$ such that 
		\begin{equs}
			\gamma \in \hom_{\mcs} (a, b) & \implies \gamma^{-1} \in \hom_{\mcs} (b, a)
			\\
			\gamma \in \hom_{\mcs} (a,b), \bar{\gamma} \in \hom_{\mcs} (b,c) & \implies \bar{\gamma} \circ \gamma \in \hom_{\mcs} (a,c) .
		\end{equs}
	We will sometimes refer to the maps 
	$\hom_{\mcs} (a, b)$
	as `symmetries' in $\mcs$. 
	
	Given a symmetric set $\mcs$ and $a \in \mathrm{Ob}(\mcs)$, we shall write $a_\mcs$ for the symmetric set with a single object $a$ and morphisms $\hom_\mcs(a,a)$.
\end{remark}

We next recall the main example of symmetric sets that we will be interested in, namely those corresponding to decorated trees. In order to avoid set theoretic issues, we restrict ourselves to trees `drawn' on $\mbN$. Since the precise vertex set does not play an important role in the theory, this has no meaningful impact beyond avoiding technical issues.

\begin{definition}\label{def:decorated_tree}
	A decorated tree is an acyclic graph $\tau$ with vertex set $N(\tau) \subset \mbN$ and edge set $E(\tau) \subset N(\tau)^2$ which is equipped with a distinguished vertex $\rho_\tau$ called its root, a type map $\mft : E(\tau) \to \mfL$, node decoration $\mfn : N(\tau) \to \mbN^d$ and $\mfe : E(\tau) \to \mbN^d$. An isomorphism of decorated trees is a graph isomorphism that preserves the root, type map and decorations.
	
	We write $E_+(\tau)$ for the set of edges $e \in E(\tau)$ such that $\mft(e) \in \mfL_+$ and call such edges kernel edges. Edges that are not kernel edges are called noise edges and we write $E_-(\tau)$ for the subset of $E(\tau)$ consisting of noise edges. 
\end{definition}

Given a decorated tree $\tau$, we define a corresponding symmetric set $\sttau$ as follows.

\begin{definition}\label{def:Tree_SSet}
	We write $\btau$ for the isomorphism class\footnote{Our convention that trees are drawn on $\mbN$ ensures that this isomorphism class is a set.} of the tree $\tau$. We will often refer to elements of $\btau$ as `drawings' of $\btau$. We write $\sttau$ for the symmetric set whose object set is $\{(E(\tau), \mft): \tau = (N(\tau), E(\tau), \mft, \mfn, \mfe) \in \btau\}$ and whose morphisms are simply the isomorphisms between elements of $\btau$. 
\end{definition}
We note that each of the trees $X^k$ consists of a single node and hence correspond to the symmetric set whose only object is the empty set and whose only morphism is the trivial map $\emptyset \to \emptyset$. This definition will ensure that in the constructions that follow the component of our regularity structures which corresponds to $X^k$ will always be a copy of $\mbR$.
\begin{remark}
	As in \cite{CCHS22}, we pay more attention than is typical in the literature to the difference between a tree $\tau$ and its isomorphism class $\btau$. The reason to do this is that the operation informally described as `gluing' vector spaces into edges (that we make rigorous later) requires one to be able to talk concretely about the edge set of the tree and thus to work with the distinguished drawing when defining the tensor products we work with. On the other hand, operations between these tensor products ought to be symmetric and thus depend only on the isomorphism class. This makes both perspectives valuable in the paper. For clarity, we adopt the convention that concrete `drawings' of trees are denoted by symbols such as $\tau, \sigma$ whilst the corresponding isomorphism classes are denoted by the boldface analogues $\btau, \bsigma$.
\end{remark}
	
\subsection{Tensor Products of Banach Spaces Associated to Symmetric Sets}\label{s: tensor}

With a suitable notion of symmetric set in hand, we now turn to the corresponding notion of tensor product. We first fix the Banach spaces we wish to take tensor products of.
\begin{definition}\label{def:Ban_Ass}
	A Banach space assignment is a collection of Banach spaces $V = (V_\mfl)_{\mfl \in \mfL}$.  Given a typed set $a$, we write
	$V^{\times a} = \prod_{x \in a} V_{\mfl(x)}$.
\end{definition}

 We note that given a pair of typed sets $a,b$ and an isomorphism of typed sets $\gamma : a \to b$, there is a corresponding map $\gamma : V^{\times a} \to V^{\times b}$ given by $\gamma((v_x)_{x \in a}) = (v_{\gamma^{-1}(y)})_{y \in b}$.
Given a symmetric set $\mcs$ and a Banach space $Z$, we say that a family of maps $(f^a: V^{\times a} \to Z)_{a \in \Ob(\mcs)}$ is $\mcs$-symmetric if for every $a,b \in \Ob(\mcs)$ and $\gamma \in \hom_\mcs(a,b)$, $f^b \circ \gamma = f^a$. We note that if $(f^a)_{a \in \Ob(\mcs)}$ is $\mcs$-symmetric and multilinear then the operator norm $\|f^a\|$ is independent of $a$. 
			
	\begin{definition}\label{def:tensor_product}
Given a symmetric set $\mcs$ and a Banach space assignment $V$, the $\mcs$-symmetric tensor product is a pair $(V^{\proj \mcs}, (\otimes_\mcs^a)_{a \in \mcs})$ where $V^{\proj \mcs}$ is a Banach space and  $(\otimes_\mcs^a)_{a \in \mcs}$ is an $\mcs$-symmetric family of norm $1$ multilinear maps such that for any complete locally convex topological vector space $Z$ and any $\mcs$-symmetric family $(f^a: V^{\times a} \to Z)_{a \in \Ob(\mcs)}$ of continuous multilinear maps, there exists a unique continuous linear map $f: V^{\proj \mcs} \to Z$ that has the property that $\|f\| = \|f^a\|$ in the case where $Z$ is a Banach space and is such that furthermore for each $a\in \Ob(\mcs))$ the following diagram commutes.
 \begin{center}
			\begin{tikzcd}
				V^{\times a} \arrow[dr, "f^a"'] \arrow[r, "\otimes_\mcs^a"] & V^{\proj \mcs} \arrow[d, "f"] \\
				&  Z
			\end{tikzcd}
		\end{center}
\end{definition}

 In this work, we take the point of view that the only necessary properties of the tensor product associated to a symmetric set $\mcs$ are encoded in the universal property. In particular, this allows us to avoid working by hand with a construction of the symmetric tensor product as a certain subspace of $\prod_{a \in \Ob(\mcs)} V^{\proj a}$ where $V^{\proj a}$ is the usual projective tensor product (which is in particular a Banach space). 

Nonetheless, it is of course necessary to show that the definition is non-vacuous. In the setting where the Banach space assignment is finite dimensional, this can be done by noticing that the tensor product constructed in \cite{CCHS22} satisfies our universal property. With minor adaptations to take care of the need for completions and more careful statements about norm boundedness, this construction provides existence of a suitable tensor product in general. Since the precise presentation of the symmetric tensor products does not matter in the rest of this work (other than through the universal property), we satisfy ourselves here by stating an existence result and deferring the adaptation of the construction of \cite{CCHS22} to Appendix~\ref{app: tensor}.
	
	\begin{prop}\label{prop:Symm_Char} 
		Given a Banach space assignment $V$ and a symmetric set $\mcs$, there exists a unique (up to unique isometric isomorphism)
$\mcs$-symmetric tensor product $(V^{\proj \mcs}, (\otimes_\mcs^a)_{a \in \Ob(\mcs)})$. Furthermore, 
		\begin{align*}
			\|v\|_{V^{\proj \mcs}} = \inf \left \{ \sum_{i = 1}^\infty \prod_{x \in a} \|v_x^i \|_{V_{\mfl(x)}} : v = \sum_{i = 1}^\infty \otimes_\mcs^a (v_x^i)_{x \in a} \right \}.
		\end{align*}
	\end{prop}
	\begin{remark}
	In fact, we will also not need the formula for the norm in this work. However, even though it takes a small amount of work to derive it from the construction, we chose to include it in the proposition above since it takes the same form as the norm for the usual projective tensor product up to replacing the usual tensor map with the tensor $\otimes_\mcs^a$ associated to $\mcs$. We hope that this serves to reassure the reader that the universal property really does capture the right notion of ``partially symmetrised projective tensor product''.
	\end{remark}
		  
		\begin{remark}\label{rem: univ_colim}
			We now compare the construction given in Definition~\ref{def:tensor_product} to the one in \cite{CCHS22}. This remark is intended for the interested reader and is not essential for the remainder of the paper.
			
			For this remark, we assume that the Banach space assignment is finite dimensional we are in the setting of \cite[Section 5]{CCHS22}. We also assume familiarity with that setting in this remark. In particular, the choice of cross norm defining the tensor product will play no role in this remark. We note that we could reformulate the universal property above purely in terms of linear maps by replacing $V^{\times a}$ with the usual projective tensor product $V^{\proj a}$. Saying that $(f^a)_{a \in \Ob(\mcs)}$ is $\mcs$-symmetric is then the same as saying that for each $a, b \in \Ob(\mcs)$ and each $\gamma_{b,a} \in \hom_{\mcs}(a,b)$, the following diagram commutes. 
			\begin{center}
				\begin{tikzcd}
					V^{\proj a} \arrow[rr, swap,  "\gamma_{b,a}"'] \arrow[dr, "f^a"'] & &  V^{\proj b} \arrow[dl, swap,  "f^b"'] \\ & Z &
				\end{tikzcd}
			\end{center}
			In the language of category theory this says that $(Z, (f^a)_{a \in \Ob(\mcs)})$ is a cocone over the diagram with objects $\{V^{\proj a}: a \in \Ob(\mcs)\}$ and morphisms $\{ \gamma_{b,a} \in \hom_{\mcs}(a,b): a, b \in \Ob(\mcs)\}$. Similarly, $(V^{\proj \mcs}, (\otimes_{\mcs}^a)_{a \in \Ob(\mcs)})$ is a cocone over the same diagram. The existence of a unique factoring map $f$ for each cocone $(Z, (f^a)_{a \in \Ob(\mcs)})$ is then the assertion that $(V^{\proj \mcs}, (\otimes_{\mcs}^a)_{a \in \Ob(\mcs)})$ is the colimit in $\mathbf{TVec}$ of the diagram defined above.
			
			We furthermore note that if $\mcF$ is the functor of \cite{CCHS22} then $V^{\proj a} = \mcF(\{a\})$ where $\{a\}$ is the symmetric set whose only object is $a$ and whose only morphism is the identity. Given a symmetric set $\mcs$, one can check that $\mcs$ is the colimit in $\SSet$ of the diagram whose objects are $\{\{a\}: a \in \Ob(\mcs)\}$ and whose morphisms are $\{\gamma_{b,a} \in \hom_{\mcs}(a,b): a,b \in \Ob(\mcs)\}$ where $\gamma_{b,a}$ is interpreted as a morphism in $\SSet$. The image of this diagram under the functor $\mcF$ is the diagram in $\mathbf{TVec}$ considered above. Since $V^{\proj \mcs} = \mcF(\mcs)$, this means that (apart from the precise statements about norms) the fact that our universal property holds can be interpreted in the language of \cite{CCHS22} as saying that the functor $\mcF$ preserves a certain class of colimits.
		\end{remark}
\subsection{Operations on Tensor Products associated to Decorated Trees}\label{s: tensor_oper}

Our goal is now to use the construction of symmetrised tensor products according to their universal property to show that the usual operations on a regularity structure (abstract integration, abstract gradients, the tree product and suitable coproducts) are well-defined in our setting. We first define the the relevant spaces on which we will define these operations.
\begin{definition}\label{def:sets_of_trees}

 We denote by $\mcT$ the set of isomorphism classes of decorated trees. Given a complete, subcritical rule $R$, we denote by $\mcT(R)$ the set of elements of $\mcT$ that strongly conform to $R$ in the sense of \cite[Definition 5.8]{BHZ19}.
	
	We then let $\mcT_+(R)$ be the unital monoid (for the tree product) generated by
	$$\{X^k : k \in \mbN^d\} \sqcup \{ \mcI_{\mfl}^k \tau_i: \tau_i \in \mcT(R), |\tau_i|_\mfs + |\mfl|_\mfs - |k|_\mfs > 0\}$$
	where $|\cdot|_\mfs$ is the usual degree assignment on trees induced by the degree assignment on $\mfL$. We let $\mcT_-(R)$ be the set of elements of $\mcT(R)$ of negative degree. 
	
	For $\mcW \subseteq \mcT$, we let 
	\begin{equ}\label{eq:banach_space over sets of trees}
		V_\mcW = \bigoplus_{\boldsymbol{f} \in \mcW} V^{\proj \langle \boldsymbol{f} \rangle}
	\end{equ}
	equipped with the inductive limit topology (which in particular makes $V_\mcW$ a complete locally convex topological vector space). 
\end{definition} 

\begin{remark}
	It follows from \cite[Proposition 5.15]{BHZ19} that 
	$$V_{\mcT(R)} = \bigoplus_{\alpha \in A} \bigoplus_{\substack{\boldsymbol{\tau} \in \mcT(R), \\ |\boldsymbol{\tau}|_\mfs = \alpha}} V^{\proj \sttau}$$
	where $A$ is a subset of $\mbR$ which is locally finite and bounded below. Furthermore, $V_{\mcT_+(R)}$ is graded for $|\cdot|_\mfs$ by a locally finite subset of $[0, \infty]$.
\end{remark}
\begin{remark}
	There is an obvious set of canonical inclusions $V_{\mcT_-(R)} \hookrightarrow V_{\mcT(R)} \hookrightarrow V_\mcT$ and $V_{\mcT_+(R)} \hookrightarrow V_{\mcT}$. In what follows, we will often identify these spaces with their image under these inclusions without comment.
\end{remark}

\subsubsection{The Tree Product, Integration and Gradients}

Our goal in this subsection is to define in that order the operations of products, integration and gradients on tensor products arising from symmetric sets coming from combinatorial trees. 

We begin with the case of products. We recall from \cite{CCHS22} that there is a notion of tensor product of symmetric sets given by defining $\Ob(\mcs_1 \otimes \mcs_2) = \Ob(\mcs_1) \times \Ob(\mcs_2)$ and $$\hom_{\mcs_1 \otimes \mcs_2}((a_1, a_2), (b_1, b_2)) = \hom_{\mcs_1}(a_1, b_1) \times \hom_{\mcs_2}(a_2, b_2).$$ At the level of symmetric sets coming from combinatorial trees, this corresponds to drawing a forest as a disjoint union of two trees and declaring the symmetries of the forest to be exactly
 those symmetries arising from the symmetries of the individual trees. In this case, we will denote $(\tau, \sigma) \in \Ob(\sttau \otimes \sssigma)$ by $\tau \sqcup \sigma$.

The following lemma is then immediate from the fact that any pair of symmetries in $\sttau$ and $\sssigma$ respectively naturally induce a symmetry in $\langle \boldsymbol{\tau \sigma} \rangle$ by identifying the obvious copies of $\tau$ and $\sigma$ inside of the tree product $\tau \sigma$.
\begin{lemma}\label{lem:prod_set_up}
		Let $\boldsymbol{\tau}, \boldsymbol{\sigma} \in \mcT$. Then $(\otimes_{\langle \boldsymbol{\tau \sigma} \rangle}^{\tau \sigma})_{\tau \sqcup \sigma \in \Ob(\sttau \otimes \sssigma)}$ is a $\sttau \otimes \sssigma$-symmetric family of norm $1$ multilinear maps.
\end{lemma}
We then note that it follows from the definition of the tensor product of symmetric sets that $V^{\proj (\mcs_1 \otimes \mcs_2)} = V^{\proj \mcs_1} \proj  V^{\proj \mcs_2}$ and $\otimes_{\mcs_1 \otimes \mcs_2}^{(a,b)} = \left ( \otimes_{\mcs_1}^a \right ) \otimes \left ( \otimes_{\mcs_2}^b \right )$. As a consequence, by applying the universal property to the family of multilinear maps given in Lemma~\ref{lem:prod_set_up}, we obtain our analogue of the tree product. 
\begin{definition}
	We define $\star: V_\mcT \proj V_\mcT \to V_\mcT$ to be the unique map such that for each $\boldsymbol{\tau}, \boldsymbol{\sigma} \in \mcT$, $\star|_{V^{\proj \sttau}\proj V^{\proj \sssigma}}$ factors the family $(\otimes_{\langle \boldsymbol{\tau \sigma} \rangle}^{\tau \sigma})_{\tau \sqcup \sigma \in \Ob(\sttau \otimes \sssigma)}$.
\end{definition}

\begin{lemma}\label{lemma: product}
	$\star$ is commutative, associative and unital with unit $1 \in V^{\proj \langle \mathbf{1} \rangle} = \mbR$. Furthermore, the map $\star |_{V^{\proj \sttau} \otimes V^{\proj \sssigma}}$ has norm $1$.
\end{lemma}
\begin{proof}
The identification of the norm is immediate from the construction via the universal property. It remains to show the algebraic properties.
	Commutativity follows from the fact that the tree product for combinatorial trees is commutative and the canonical identification $V^{\proj \sttau} \proj V^{\proj \sssigma} \simeq V^{\proj \sssigma} \proj V^{\proj \sttau}$. That $\star$ is unital with unit $1$ follows from the fact that $\mathbf{1}$ is the unit for the tree product at the level of combinatorial trees, along with the canonical identification $V^{\proj \sttau} \proj \mbR \simeq V^{\proj \sttau}$.
	
	It then only remains to check that $\star$ is associative. It suffices to consider the action on $V^{\proj \sttaun{1}} \otimes V^{\proj \sttaun{2}} \otimes V^{\proj \sttaun{3}}$. We claim that acting on this space, $\star(\star \otimes \id)$ and $\star(\id \otimes \star)$ both factor the symmetric multilinear family $(\otimes_{\langle \boldsymbol{\tau_1 \tau_2 \tau_3} \rangle }^{\tau_1 \tau_2 \tau_3})_{\tau_1 \sqcup \tau_2 \sqcup \tau_3 \in \Ob(\sttaun{1} \otimes \sttaun{2} \otimes \sttaun{3})}$ through $V^{\proj (\sttaun{1} \otimes \sttaun{2} \otimes \sttaun{3})}$ so that by uniqueness of the factoring map in the universal property, they coincide. In the case of $\star(\star \otimes \id)$, we write
	\begin{align*}
		\star(\star \otimes \id) \otimes_{\sttaun{1} \otimes \sttaun{2} \otimes \sttaun{3}}^{\tau_1 \sqcup \tau_2 \sqcup \tau_3} = \star \left (\left ( \star \circ \otimes_{\sttaun{1} \otimes \sttaun{2}}^{\tau_1 \sqcup \tau_2} \right ) \otimes \left ( \otimes_{\sttaun{3}}^{\tau_3} \right )  \right ) 
	\end{align*}
	By the definition of $\star$, we have that
		$\star \circ  \otimes_{\sttaun{1} \otimes \sttaun{2}}^{\tau_1 \sqcup \tau_2} = \otimes_{\langle \boldsymbol{\tau_1 \tau_2} \rangle}^{\tau_1 \tau_2}$
	so that
		\begin{align*}
		\star(\star \otimes \id) \otimes_{\sttaun{1} \otimes \sttaun{2} \otimes \sttaun{3}}^{\tau_1 \sqcup \tau_2 \sqcup \tau_3} &= \star \circ \otimes_{\langle \boldsymbol{\tau_1 \tau_2} \rangle \otimes \sttaun{3}}^{\tau_1 \tau_2 \sqcup \tau_3} = \otimes_{\langle \boldsymbol{\tau_1 \tau_2 \tau_3} \rangle }^{\tau_1 \tau_2 \tau_3}
	\end{align*}
	where we again applied the definition of $\star$ to obtain the last equality. This completes the proof that $\star(\star \otimes \id)$ factors the aforementioned family of multilinear maps. The proof that $\star(\id \otimes \star)$ factors the same family is almost identical (since the underlying product of combinatorial trees is associative) so we omit the details. 
\end{proof}
	Next we define suitable abstract integration maps on $V_\mcT$. For given $\mfl \in \mfL$, we view $\mfl$ as a singleton typed set with the obvious type map. We denote by $\langle \mfl \rangle$ the corresponding one object symmetric set whose only symmetry is the identity map. We note that at the level of combinatorial trees, symmetries of $\mcI_\mfl \tau$ are in bijection with symmetries of $\tau$. Indeed, any symmetry of the former tree preserves incidence at the root and thus fixes the trunk. It then follows that $({\otimes}_{\langle \mathfrak{l} \rangle \otimes \sttau}^{\mathfrak{l} \sqcup \tau})_{ \sqcup \tau \in \text{Ob}(\langle \mathfrak{l} \rangle \otimes \sttau)}$ is a $\langle \mathfrak{l} \rangle \otimes \sttau$-symmetric family of maps.
\begin{definition} \label{d:int_tban}
	For each $\mfl \in \mfL$, we define $\mcI_\mfl: V_\mfl \proj V_\mcT \to V_\mcT$ componentwise via the universal property of $V^{\proj \langle \mfl \rangle \otimes \sttau}$ applied to the diagrams
	\begin{center}
		\begin{tikzcd}
			V_\mfl \times V^{\times \tau} \arrow[dr, "{\otimes}_{\langle \boldsymbol{ \mcI_\mfl \tau} \rangle}^{\mcI_\mfl \tau}"'] \arrow[r, "{\otimes}_{\langle \mfl \rangle \otimes \sttau}^{\mfl \sqcup \tau}"] & V^{\proj \langle \mfl \rangle \otimes \sttau} \arrow[d, dotted, "\mcI_\mfl"]
			\\
			& V^{\proj \langle \boldsymbol{\mcI_\mfl \tau} \rangle} \ .
		\end{tikzcd}
	\end{center}
	For $\zeta \in V_\mfl$, we define $\mcI^\zeta : V_\mcT \to V_\mcT$ via $\mcI^\zeta = \mcI_\mfl (\zeta \otimes \cdot)$. We remark that $\mcI_\mfl$ is a contraction by construction, whilst $\mcI^\zeta$ has norm $\|\zeta\|_{V_\mfl}.$
\end{definition}

In this subsection, it remains to provide suitable abstract gradients in our infinite dimensional setting. In the scalar valued setting, it is typical to define the gradient as acting only on planted trees $\mcI_\mfl \tau$ or monomials $X^k$. We will see later that in our setting it will be necessary to also take gradients of trees of the form $X^k \mcI_\mfl \tau$. Of course, the usual gradient in the scalar setting trivially extends to such trees via the Leibniz rule. We denote by $\mcT_{\text{plant}}$ the set of trees of the form $X^k \mcI_\mfl \tau$. 
	\begin{definition} \label{def:der_of_planted}
		Let $k \in \mathbb{N}^d$. We define $D^k : V_{\mcT_\mathrm{plant}} \to V_{\mcT_{\mathrm{plant}}}$ componentwise by the universal property of $V^{\proj \sttau}$ applied to the diagrams
		\begin{center}
			\begin{tikzcd}
				V^{\times \tau} \arrow[dr, "\hat{\otimes}_{\langle \boldsymbol{ D^k \tau} \rangle}^{D^k \tau}"'] \arrow[r, "\hat{\otimes}_{\sttau}^{ \tau}"] & V^{\sttau} \arrow[d, dotted, "D^k"]
				\\
				& V^{\proj \langle \boldsymbol{D^k \tau} \rangle}
			\end{tikzcd}
		\end{center}
		where we made use of the fact $V^{\times \tau} = V^{\times D^k \tau}$. We note that it is immediate that $D^k$ is an isometry.
	\end{definition}
	
	\subsubsection{Coproducts}
	
	In this subsection, we realise the Hopf algebra/comodule structure that is standard in the scalar valued setting (see \cite{BHZ19}) at the level of our symmetric tensor products. Here the main remaining ingredient is a suitable coaction $\Delta$ and coproduct $\Delta^+$ (defined on appropriate spaces). 
	
	Such coactions and coproducts are realised as contraction/extraction operations. With the viewpoint of our symmetric tensor products as gluing elements of the Banach spaces onto edges of the tree, it is convenient to have a mechanism that remembers which element to glue to which edge after the operation of contraction and extraction. To this end, for a fixed $\tau \in \boldsymbol{\tau} \in \mcT$ we have the following notion of $\tau$-identified trees.
		\begin{definition}\label{def:identified_tree}
			A $\tau$-identified tree is a pair $(\sigma, \iota)$ where $\sigma$ is a tree and  $\iota: E(\sigma) \to E(\tau)$ is an injective map which preserves edge types. We write $T_\tau$ for the free vector space generated by the set of $\tau$-identified trees.
		\end{definition}
		
	The $\tau$-identification is precisely the information needed to appropriately `glue in' components of an element of $V^{\times \tau}$ to yield an element of $V^{\proj \sssigma}$. In particular, we have the following definition of lifting maps.
	\begin{definition}
		We define $\mcL_\tau^{(k)}: T_\tau^{\otimes k} \to \{f: V^{\times \tau} \to V_\mcT^{\proj k} \}$ to be the unique linear map such that 
		$$\mcL_\tau^{(k)}\Big ( \bigotimes_{i=1}^k (\sigma_i, \iota_i) \Big )[(\zeta_e)_{e \in E(\tau)}] = \bigotimes_{i=1}^k \otimes_{\langle \boldsymbol{\sigma}_i \rangle}^{\sigma_i} [(\zeta_{\iota_i(e)})_{e \in E(\sigma_i)}].$$
	\end{definition}
	One important feature of these lifting maps is that they are invariant under a suitable notion of isomorphism between identified trees. This will be used to see that the lift of operations defined at the level of trees is independent of the choice of representative of the isomorphism class of a tree.
	
	\begin{definition}
		Let $\tau, \bar{\tau} \in \sttau \in \mathcal{T}$ and let $(\sigma, \iota), (\bar{\sigma}, \bar{\iota}) $ be $\tau$, $\bar{\tau}$ identified trees respectively. An isomorphism of identified trees between $(\sigma, \iota), (\bar{\sigma}, \bar{\iota})$ is a pair of tree isomorphisms $(\gamma, \eta)$ where $\gamma: \tau \to \bar{\tau}$ and $\eta: \sigma \to \bar{\sigma}$ are such that following diagram commutes.
		\begin{center}
			\begin{tikzcd}
				\bar{\tau} \arrow[r, "\gamma^{-1}"] & \tau \\
				\bar{\sigma} \arrow[u, "\bar{\iota}"] \arrow[r, "\eta^{-1}"] & \sigma \arrow[u, "\iota"]
			\end{tikzcd}
		\end{center}
		We will refer to the map $\gamma$ as the base map of the isomorphism.
	\end{definition}
We note that the diagram above is completely equivalent to a diagram involving $\gamma$ and $\eta$ rather than their inverses. The reason to write the diagram in this way is that the compositions appearing in this representation are precisely the ones that appear in the proof of Proposition~\ref{prop:lift_equiv} below.
	\begin{prop}\label{prop:lift_equiv}
			Suppose that for $i = 1, \dots, k$, $(\sigma_i, \iota_i), (\bar{\sigma}_i, \bar{\iota}_i)$ are $\tau$-identified and $\bar{\tau}$-identified trees respectively such that $(\sigma_i, \iota_i)$ is isomorphic to $(\bar{\sigma}_i, \bar{\iota}_i)$ via $(\gamma, \eta_i)$. Then	
			\begin{align*}
				\mcL_{\bar{\tau}}^{(k)} \left [\bigotimes_{i=1}^k (\bar{\sigma}_i, \bar{\iota}_i) \right ] \circ \gamma = \mcL_\tau^{(k)} \left [\bigotimes_{i=1}^k (\sigma_i, \iota_i) \right ].
			\end{align*}  
			In addition, if $(\sigma_1, \iota_1), \dots, (\sigma_k, \iota_k)$ are $\tau$-identified trees such that $\{\iota_i(E(\sigma_i)): i = 1, \dots, k\}$ forms a partition of $E(\tau)$ then $\mcL_\tau^{(k)} [ \bigotimes_{i=1}^k (\sigma_i, \iota_i)]$ is a
			continuous 
			multilinear map.
	\end{prop}
	\begin{proof}
		For the isomorphism equivariance, we take $(\zeta_e)_{e \in E(\tau)} \in V^{\times \tau}$ and unpack the definitions to write
		\begin{align*}
			&	\mcL_{\bar{\tau}}^{(k)} \left [\bigotimes_{i=1}^k (\bar{\sigma}_i, \bar{\iota}_i) \right ] \circ \gamma \left ( (\zeta_e)_{e \in E(\tau)} \right )  = \mcL_{\bar{\tau}}^{(k)} \left [ \bigotimes_{i=1}^k (\bar{\sigma}_i, \bar{\iota}_i) \right ] (\zeta_{\gamma^{-1}(e)} )_{e \in E(\bar{\tau})}		
				\\
				& = \bigotimes_{i=1}^k \otimes_{\langle \boldsymbol{\bar{\sigma}}_i \rangle}^{\bar{\sigma}_i} (\zeta_{\gamma^{-1}( \bar{\iota}_i(e))})_{e \in E(\bar{\sigma}_i)}
				 =  \bigotimes_{i=1}^k \otimes_{\langle \boldsymbol{\bar{\sigma}}_i \rangle}^{\bar{\sigma}_i} (\zeta_{ \iota_i(\eta_i^{-1}(e))})_{e \in E({\sigma}_i)}
			 = \bigotimes_{i=1}^k \otimes_{\langle \boldsymbol{\bar{\sigma}}_i \rangle}^{\bar{\sigma}_i} \circ \eta_i (\zeta_{ \iota_i(e)})_{e \in E(\bar{\sigma}_i)}
	\\
				&				
				= \bigotimes_{i=1}^k \otimes_{\langle \boldsymbol{{\sigma}}_i \rangle}^{{\sigma}_i} (\zeta_{ \iota_i(e)})_{e \in E({\sigma}_i)}
				 = \mcL_\tau^{(k)} \left [\bigotimes_{i=1}^k ({\sigma}_i, {\iota}_i) \right ] \left ( (\zeta_e)_{e \in E(\tau)} \right )
		\end{align*} 
		where the third equality followed by the definition of an isomorphism of identified trees and the fifth equality follows by symmetry of the maps $\otimes_{\langle \boldsymbol{{\sigma}}_i \rangle}^{\sigma_i}$.
		The claims about continuity and multilinearity are immediate from the definition and the multilinearity of $\bigotimes_{i=1}^k$ and $\otimes_{\langle \boldsymbol{\sigma_i} \rangle}^{\sigma_i}$.
	\end{proof}
	
	We now recall the definition of the coaction $\Delta$ at the level of combinatorial trees. We find it convenient to reformulate the more common definition in the literature in which one sums over certain subtrees containing the root as a definition in which one sums over cuts. These two formulations only differ aesthetically, however we find that the latter type of formulation will make computations in Section~\ref{s:ren_models} clearer and hence we already adopt it here for consistency. 
	
	\label{def:cut}
	Given $\tau \in \boldsymbol{\tau}$, we say that $C \subset E(\tau)$ is a cut of $\tau$ if any path in $E(\tau)$ with the root of $\tau$ as an endpoint contains at most one edge in $C$. We denote by $\mathbb{C}[\tau]$ the set of all cuts of $\tau$ and for $C \in \mathbb{C}[\tau]$ we write $\tau_{\ng C}$ for the subgraph consisting of those edges $e \in E(\tau)$ such that there does not exist $\tilde{e} \in C$ with $e \ge \tilde{e}$.
	We define $\underline{\mcT}(R) = \bigcup_{\boldsymbol{\tau} \in \mcT(R)} \boldsymbol{\tau}$ to consist of all drawings of (isomorphism classes of) trees in $\mcT(R)$ and define $\underline{\mcT}_+(R)$ similarly. We define $\underline{\Delta}: \langle \underline{\mcT}(R) \rangle \to \langle \underline{\mcT}(R) \rangle \otimes \langle \underline{\mcT}_+(R) \rangle$ via 
	\label{eq:underline_delta}
	\begin{align*}
		\underline{\Delta} (\tau, \mfn, \mfe) = \sum_{C \in \mathbb{C}[\tau]} \sum_{n_{\ng C}, \varepsilon_C} \frac{1}{\varepsilon_C !} { \mfn \choose n_{\ng C}} (\tau_{\ng C}, n_{\ng C} + \pi \varepsilon_C, \mfe) \otimes P_{\underline{\mcT}_+(R)}(\tau/ \tau_{\ng C}, [n-n_{\ng C}], \mfe + \varepsilon_C)
	\end{align*}
	where the sum over $n_{\ng C}$ is over node decorations $n_{\ng C}: N(\tau_{\ng C}) \to \mathbb{N}^d$, the sum over $\varepsilon_C$ is a sum over edge decorations $\varepsilon_C : C \to \mathbb{N}^d$, $\pi \varepsilon_C : N(\tau_{\ng C}) \to \mathbb{N}^d$ is the node decoration obtained via
	\begin{align} \label{eq:defpi1}
		\pi \varepsilon_C(v) = \sum_{e \in C: v \in e} \varepsilon_C(v),
	\end{align}
	$P_{\underline{\mcT}_+(R)}$ denotes the usual projection onto $\underline{\mcT}_+(R)$, $\tau/ \tau_{\ng C}$ denotes the quotient graph\footnote{To be completely precise, we should specify how to draw the quotient graph on $\mbN$ (though this detail is not too important at the level of exposition given here). One choice that ensures the properties of operations we need at the level of drawings is to define $N(\tau/\sigma) = \{\min N(\sigma)\} \sqcup (N(\tau) \setminus N(\sigma))$ and to move all edges that were ingoing to $\sigma$ to be connected to $\min N(\sigma)$.}of $\tau$ by $\tau_{\ng C}$ and $[n-n_{\ng C}]$ is the node decoration on this quotient graph defined via
	\begin{align*}
		[n-n_{\ng C}](v) = \sum_{w \in N(\tau): v = [w]} (n-n_{\ng C})(w)
	\end{align*}
	where $[w]$ denotes the equivalence class of $w$ viewed as a vertex in the quotient graph. We remark that each of the two trees appearing in a fixed summand of the right hand side of the definition of $\underline{\Delta}(\tau, \mfn, \mfe)$ can be viewed as being $\tau$-identified since their edge sets can naturally be viewed as forming a partition of $E(\tau)$. 
	
	\begin{prop}
		Given $\boldsymbol{\tau} \in \mcT(R)$, the family of maps $(\mcL_\tau^{(2)}[\underline{\Delta} \tau])_{\tau \in \boldsymbol{\tau}}$ is continuous, multilinear and $\sttau$-symmetric.
	\end{prop}
	\begin{proof}
		Multilinearity and continuity are immediate from Proposition~\ref{prop:lift_equiv}. Therefore it remains only to check $\sttau$-symmetry. To this end, we note that if $\tau, \bar{\tau} \in \boldsymbol{\tau}$ and $\gamma \in \hom_\sttau(\tau, \bar{\tau})$ then $\gamma$ induces a bijection $\gamma_\mathbb{C}: \mathbb{C}[\tau] \to \mathbb{C}[\bar{\tau}]$. Furthermore, for a fixed cut $C \in \mathbb{C}[\tau]$ if we set $\bar{C} = \gamma_{\mathbb{C}}(C)$ then we see that $\gamma$ induces also bijections $n_{\ng C} \mapsto n_{\ng  \bar{C}}$ and $\varepsilon_C \mapsto \varepsilon_{\bar{C}}$ appearing in the definition of $\underline{\Delta}^+$. Finally, one can check that if $C, n_{\ng C}, \varepsilon_{C}$ correspond to $\bar{C}, n_{\ng \bar{C}}, \varepsilon_{\bar{C}}$ via the above bijections then $\gamma$ also induces (via restriction or quotient) isomorphisms $$\eta_{\ng C}: (\tau_{\ng C}, n_{\ng C} + \pi \varepsilon_C, \mfe) \to (\bar{\tau}_{\ng \bar{C}}, n_{\ng \bar{C}} + \pi \varepsilon_{\bar{C}}, \bar{\mfe})$$ and $$\eta_{\tau/\tau_{\ng C}}: (\tau/ \tau_{\ng C}, [n-n_{\ng C}], \mfe + \varepsilon_C) \to (\bar{\tau}/ \bar{\tau}_{\ng \bar{C}}, [\bar{n}-n_{\ng \bar{C}}], \bar{\mfe} + \varepsilon_{\bar{C}})$$
		which are furthermore isomorphisms of identified trees with base maps $\gamma$.
		
		We are then in the situation of the first part of Proposition~\ref{prop:lift_equiv}. Indeed, we can write
		\begin{align*}
			\mcL_{\bar{\tau}}^{(2)}[\underline{\Delta} \bar \tau]\circ \gamma & = \sum_{\bar{C} \in \mathbb{C}[\bar{\tau}]} \sum_{n_{\ng \bar{C}}, \varepsilon_{\bar{C}}} \frac{1}{\varepsilon_{\bar{C}}} { \mfn \choose n_{\ng {\bar{C}}}} \mcL_{\bar{\tau}}^{(2)} \left [(\bar{\tau}_{\ng {\bar{C}}}, n_{\ng {\bar{C}}} + \pi \varepsilon_{\bar{C}}, \bar{\mfe}) \otimes P_{\underline{\mcT}_+(R)}(\bar{\tau}/ \bar{\tau}_{\ng {\bar{C}}}, [\bar{n}-n_{\ng {\bar{C}}}], \bar{\mfe} + \varepsilon_{\bar{C}}) \right ] \circ \gamma
			\\
			& = \sum_{C \in \mathbb{C}[\tau]} \sum_{n_{\ng C}, \varepsilon_C} \frac{1}{\varepsilon_C} { \mfn \choose n_{\ng C}} \mcL_\tau^{(2)} \left [ (\tau_{\ng C}, n_{\ng C} + \pi \varepsilon_C, \mfe) \otimes P_{\underline{\mcT}_+(R)}(\tau/ \tau_{\ng C}, [n-n_{\ng C}], \mfe + \varepsilon_C) \right ]
			\\
			& = \mcL_\tau^{(2)}[\underline{\Delta} \tau]
		\end{align*}
		where the second line follows by Proposition~\ref{prop:lift_equiv} and a reindexing of the sums based on the induced bijections between cuts and decorations.
	\end{proof}
	The following definition is then meaningful as a consequence of the universal property for $V^{\proj \sttau}$. 
	\begin{definition}\label{def:triangle}
		We define $\Delta: V_{\mcT(R)} \to V_{\mcT(R)} \hat{\otimes}_\pi V_{\mcT_+(R)}$ to be the unique continuous linear map such that for each $\tau \in \boldsymbol{\tau} \in \mcT(R)$, we have that the following diagram commutes.
		\begin{center}
			\begin{tikzcd}
				V^{\times \tau} \arrow[r, "\otimes_{\sttau}^\tau"] \arrow[dr, "\mcL_\tau^{(2)} (\underline{\Delta} \tau)"'] & V^{\proj \sttau} \arrow[d, dotted, "\Delta"]
				\\
				& V_{\mcT(R)} \hat{\otimes}_\pi V_{\mcT_+(R)}
			\end{tikzcd}
		\end{center}
	\end{definition}
	
	We similarly define a map $\underline{\Delta}^+ : \langle \underline{\mcT}_+(R) \rangle \to  \langle \underline{\mcT}_+(R) \rangle \otimes \langle \underline{\mcT}_+(R) \rangle$ via
	\begin{align*}
		\underline{\Delta}^+ (\tau, \mfn, \mfe) = \sum_{C \in \mathbb{C}[\tau]} \sum_{n_{\ng C}, \varepsilon_C} \frac{1}{\varepsilon_C} { \mfn \choose n_{\ng C}} P_{\underline{\mcT}_+(R)} (\tau_{\ng C}, n_{\ng C} + \pi \varepsilon_C, \mfe) \otimes P_{\underline{\mcT}_+(R)}(\tau/ \tau_{\ng C}, [n-n_{\ng C}], \mfe + \varepsilon_C) .
	\end{align*}
\begin{remark}
The above definition does coincide with that of \cite{BHZ19} despite the projections appearing to differ. This is due to the fact that in \cite{BHZ19}, the formula for $\underline{\Delta}$ and $\underline{\Delta}^+$ were unified by having the projections implicitly specified by the choice of codomain.
\end{remark}	
It follows in a similar way to the case of $\Delta$ that the following definition makes sense via the universal property.
	\begin{definition}\label{def:triangle+}
		We define $\Delta^+: V_{\mcT_+(R)} \to V_{\mcT_+(R)} \proj V_{\mcT_+(R)}$ to be the unique continuous linear map such that for each $\tau \in \boldsymbol{\tau} \in \mcT_+(R)$, we have that the following diagram commutes.
		\begin{center}
			\begin{tikzcd}
				V^{\times \tau} \arrow[r, "\otimes_{\sttau}^\tau"] \arrow[dr, "\mcL_\tau^{(2)} (\underline{\Delta}^+ \tau)"'] & V^{\proj \sttau} \arrow[d, dotted, "\Delta^+"]
				\\
				& V_{\mcT_+(R)} \proj V_{\mcT_+(R)}
			\end{tikzcd}
		\end{center}
	\end{definition}
	
	Our next order of business is to show that the fact that $V_{\mcT_+(R)}$ is a Hopf algebra with $V_{\mcT(R)}$ as a Hopf comodule can be lifted from the combinatorial setting. 
	
	\begin{lemma}\label{lem: coass}
		We have the relations
		\begin{align*}
			(\Delta^+ \otimes \id) \Delta^+ &= (\id \otimes \Delta^+) \Delta^+, \qquad 
			(\Delta \otimes \id)\Delta = (\id \otimes \Delta^+)\Delta.
		\end{align*}
	\end{lemma}	
	\begin{proof}
		Since the proofs are similar, we prove only the first equality. We recall from \cite{BHZ19} that the corresponding results hold for the underlined versions of $\Delta, \Delta^+$ acting on combinatorial trees.
		
		 It is enough for us to consider the restriction to a component $V^{\proj \sttau}$. The idea of the proof is to leverage the underlying combinatorial result by showing that both the sides of the equality claimed in the lemma factor the family of maps
		\begin{align*}
			\mcL_\tau^{(3)} \left ((\underline{\Delta}^+ \otimes \id) \underline{\Delta}^+ \tau \right ) = \mcL_\tau^{(3)} \left ((\id \otimes \underline{\Delta}^+) \underline{\Delta}^+ \tau \right ).
		\end{align*}
		Upon showing that this family is continuous, multilinear and $\sttau$-symmetric, the result will then follow by uniqueness of the factorisation. 
		
		We first show that both sides do indeed factor this family. We start by considering $(\Delta^+ \otimes \id)\Delta^+$. To fix notations, we write 
		$$\underline{\Delta}^+ \tau = \sum_{i=1}^n c_i (\tau_i \otimes \tau^i)$$
		where $\tau_i$ and $\tau^i$ are $\tau$-identified trees. We then write
		\begin{align*}
			(\Delta^+ \otimes \id) \Delta^+ \otimes_{\sttau}^\tau [ (\zeta_e)_{e \in E(\tau)}] &= (\Delta^+ \otimes \id)\mcL_{\tau}^{(2)} [ \underline{\Delta}^+ \tau] [ (\zeta_e)_{e \in E(\tau)}]
			\\
			& = \sum_{i=1}^n c_i  (\Delta^+ \otimes \id)  \left[ \bigotimes_{\sttaun{i}}^{\tau_i} [(\zeta_{\iota_{\tau_i}(e)})_{e \in E(\tau_i)}] \otimes \bigotimes_{\langle \boldsymbol{\tau}^i \rangle}^{\tau^i} [(\zeta_{\iota_{\tau^i}(e)})_{e \in E(\tau^i)}]\right ]
			\\ & = \sum_{i=1}^n c_i \mcL_{\tau_i}^{(2)}[\underline{\Delta}^+ \tau_i][(\zeta_{\iota_{\tau_i}(e)})_{e \in E(\tau_i)}] \otimes \bigotimes_{\langle \boldsymbol{\tau}^i \rangle}^{\tau^i} [(\zeta_{\iota_{\tau^i}(e)})_{e \in E(\tau^i)}]
			\\ & = \sum_{i=1}^n \sum_{j = 1}^k c_i c^j \bigotimes_{\langle \boldsymbol{\tau}_{ij}\rangle}^{\tau_{ij}} [(\zeta_{\iota_{\tau_i} \circ \iota_{\tau_{ij}}(e)})_{e \in E(\tau_{ij})}] \otimes \bigotimes_{\langle \boldsymbol{\tau}_{i}^j\rangle}^{\tau_{i}^j} [(\zeta_{\iota_{\tau_i} \circ \iota_{\tau_{i}^j}(e)})_{e \in E(\tau_{i}^j)}] \otimes \bigotimes_{\langle \boldsymbol{\tau}^i \rangle}^{\tau^i} [(\zeta_{\iota_{\tau^i}(e)})_{e \in E(\tau^i)}]
			\\& = \mcL_\tau^{(3)}[(\underline{\Delta}^+ \otimes \id) \underline{\Delta}^+ \tau][(\zeta_e)_{e \in E(\tau)}]
		\end{align*}
		where we wrote $\underline{\Delta}^+ \tau_i = \sum_{j = 1}^k c^j \tau_{ij} \otimes \tau_i^j$ and noted that e.g. $(\tau_{ij}, \iota_{\tau_{ij}})$ is a $\tau_i$-identified tree which induces the $\tau$-identified tree $(\tau_{ij}, \iota_{\tau_i} \circ \iota_{\tau_{ij}})$. The latter $\tau$-identified tree is precisely the one appearing in $(\underline{\Delta}^+ \otimes \id) \underline{\Delta}^+ \tau$. Since the claim for $(\id \otimes \Delta^+) \Delta^+$ has an almost identical proof, only changing the order of some terms everywhere we omit this detail.
		
		It remains only to check that $\mcL_\tau^{(3)} \left ((\underline{\Delta}^+ \otimes \id) \underline{\Delta}^+ \tau \right )$ is continuous, multilinear and $\sttau$-symmetric. This follows from the factorisation property above. Indeed, for continuity and multilinearity we note that e.g. $\mcL_\tau^{(3)}[(\underline{\Delta}^+ \otimes \id) \underline{\Delta}^+ \tau] = (\Delta^+ \otimes \id) \Delta^+ \otimes_{\sttau}^\tau$ is the composition of a continuous multilinear map with two continuous linear maps. For $\sttau$-symmetry, we simply note that $\otimes_{\sttau}^\tau$ is a $\sttau$-symmetric family.
		\end{proof}
		
		We define $\varepsilon: V_{\mcT_+(R)} \to \mathbb{R}$ to be the projection onto $V^{\proj \langle \mathbf{1} \rangle} \simeq \mathbb{R}$. 
		
		\begin{lemma}\label{lem:Hopf}
			 $(V_{\mcT_+(R)}, \ast, 1, \Delta^+, \eps)$ is a topological Hopf algebra\footnote{since this terminology seems to not be completely standardised in the literature, we refer the reader to Appendix~\ref{Hopf} for a definition and some useful properties.} with $V_{\mcT(R)}$ as a comodule.
		\end{lemma}
		\begin{proof}
			We note that by Lemma~\ref{bi_to_hopf} and Remark~\ref{rem:pol_bi}, for the first claim it suffices to prove that $V_{\mcT_+(R)}$ is a bialgebra. For this we need to check that $\Delta^+$ is an algebra morphism and that $\varepsilon$ satisfies the requirements for a counit.
			
			   We start by checking that $\Delta^+$ is multiplicative. We need to show that $$\Delta^+ \circ \star = (\star \otimes \star) \rho (\Delta^+ \otimes \Delta^+)$$ where $\rho : V_\mcT^{\proj 4} \to V_{\mcT}^{\proj 4}$ is the map that swaps the second and third components of elementary tensors.
			
			To do this, we again leverage the corresponding combinatorial result. Namely, we claim that when acting on $V^{\proj \sttau} \proj V^{\proj \sssigma}$ the two maps factor the families $\mcL_{\tau \sigma}^{(2)}[\underline{\Delta}^+ \tau \sigma]$ and $\mcL_{\tau \sigma}^{(2)}[(\mcM \otimes \mcM) \circ \underline{\rho} \circ (\underline{\Delta}^+ \otimes \underline \Delta^+)(\tau \otimes \sigma)]$ respectively where $\mcM$ is the (combinatorial) tree product and $\underline{\rho}$ is the combinatorial analogue of $\rho$. The result will then follow from uniqueness of the factorisation since the fact that $\underline{\Delta}^+$ is an algebra morphism implies that 
			$$\mcL_{\tau \sigma}^{(2)}[\underline{\Delta}^+ \tau \sigma] = \mcL_{\tau \sigma}^{(2)}[(\mcM \otimes \mcM) \circ \underline{\rho} \circ (\underline{\Delta}^+ \otimes \underline \Delta^+)(\tau \otimes \sigma)].$$
			
			We remark that it will follow from the factorisation that e.g. $\mcL_{\tau \sigma}^{(2)}[\underline{\Delta}^+ \tau \sigma]$ is continuous, multilinear and $\sttau \otimes \sssigma$-symmetric in the same way as in the corresponding part of the proof of Lemma~\ref{lem: coass}.
			
			We start by showing that $\Delta^+ \circ \star$ factors the left hand side. For this, we write
			\begin{equs}
				\Delta^+ \circ \star \circ \bigotimes_{\sttau \otimes \sssigma}^{\tau \sqcup \sigma} [(\zeta_e)_{e \in E(\tau \sqcup \sigma)}] &= \Delta^+ \otimes_{\langle \boldsymbol{\tau \sigma} \rangle}^{\tau \sigma} [(\zeta_e)_{e \in E(\tau \sqcup \sigma)}]
				 = \mcL_{\tau \sigma}^{(2)}[\underline{\Delta}^+ \tau \sigma] [(\zeta_e)_{e \in E(\tau \sqcup \sigma)}]
			\end{equs}
			as desired.
			
			We now turn to the corresponding result for the right hand side. We write
			\begin{equs}
				(\star \otimes \star)  \rho  (\Delta^+ \otimes \Delta^+) \otimes_{\sttau \otimes \sssigma}^{\tau \sqcup \sigma} [(\zeta_e)_{e \in E(\tau \sqcup \sigma)}] & =  (\star \otimes \star)  \rho (\Delta^+ \otimes_{\sttau}^\tau \otimes \Delta^+ \otimes_{\sssigma}^\sigma)[(\zeta_e)_{e \in E(\tau \sqcup \sigma)}]
				\\ & = (\star \otimes \star)\rho  (\mcL_\tau^{(2)}[\underline{\Delta}^+ \tau](\zeta_e)_{e \in E(\tau)} \otimes \mcL_\sigma^{(2)}[\underline{\Delta}^+ \sigma] (\zeta_e)_{e \in E(\sigma)}).
			\end{equs}
			Now we note that the term in the rightmost bracket can be written as
			\begin{equs}
				\sum_{i=1}^n \sum_{j = 1}^k c_i c^j 
				\bigotimes_{\sttaun{i}}^{\tau_i}(\zeta_{\iota_{\tau_i}(e)})_{e \in E(\tau)} 
				\otimes \bigotimes_{\langle \boldsymbol{\tau}^i \rangle}^{\tau^i}(\zeta_{\iota_{\tau^i}(e)})_{e \in E(\tau)}
				\otimes \bigotimes_{\langle \boldsymbol{\sigma}_j \rangle}^{\sigma_j}(\zeta_{\iota_{\sigma_j}(e)})_{e \in E(\sigma)}
				\otimes \bigotimes_{\langle \boldsymbol{\sigma}^j \rangle}^{\sigma^j}(\zeta_{\iota_{\sigma^j}(e)})_{e \in E(\sigma)}
			\end{equs}
			so that applying $\rho$ yields
			\begin{equs}
				\sum_{i=1}^n \sum_{j = 1}^k c_i c^j 
				\bigotimes_{\sttaun{i}}^{\tau_i}(\zeta_{\iota_{\tau_i}(e)})_{e \in E(\tau)}
				\otimes \bigotimes_{\langle \boldsymbol{\sigma}_j \rangle}^{\sigma_j}(\zeta_{\iota_{\sigma_j}(e)})_{e \in E(\sigma)}
				\otimes \bigotimes_{\langle \boldsymbol{\tau}^i \rangle}^{\tau^i}(\zeta_{\iota_{\tau^i}(e)})_{e \in E(\tau)}
				\otimes \bigotimes_{\langle \boldsymbol{\sigma}^j \rangle}^{\sigma^j}(\zeta_{\iota_{\sigma^j}(e)})_{e \in E(\sigma)} .
			\end{equs}
			We then note that
			\begin{equs}
				\star \left (\bigotimes_{\sttaun{i}}^{\tau_i}(\zeta_{\iota_{\tau_i}(e)})_{e \in E(\tau)} \otimes \bigotimes_{\langle \boldsymbol{\sigma}_j \rangle}^{\sigma_j}(\zeta_{\iota_{\sigma_j}(e)})_{e \in E(\sigma)} \right ) 
				& = \star \left ( \bigotimes_{\sttaun{i} \otimes \langle \boldsymbol{\sigma_j}}^{\tau_i \sqcup \sigma_j} [(\zeta_{\iota_{\tau_i} \sqcup \iota_{\sigma_j}(e)})_{e \in E(\tau \sqcup \sigma)}] \right )
				\\ & = \bigotimes_{\langle \boldsymbol{\tau_i \sigma_j} \rangle}^{\tau_i \sigma_j} [(\zeta_{\iota_{\tau_i} \sqcup \iota_{\sigma_j}(e)})_{e \in E(\tau \sqcup \sigma)}].
			\end{equs}
			Treating the second instance of $\star$ similarly we see that 
			\begin{equs}
				(\star \otimes \star)  \rho  (\Delta^+ \otimes \Delta^+) \otimes_{\sttau \otimes \sssigma}^{\tau \sqcup \sigma} [(\zeta_e)_{e \in E(\tau \sqcup \sigma)}] &= \sum_i \sum_j c_i c^j \bigotimes_{\langle \boldsymbol{\tau_i \sigma_j} \rangle}^{\tau_i \sigma_j} [(\zeta_{\iota_{\tau_i} \sqcup \iota_{\sigma_j}(e)})_{e \in E(\tau \sqcup \sigma)}] \otimes \bigotimes_{\langle \boldsymbol{\tau^i \sigma^j} \rangle}^{\tau^i \sigma^j} [(\zeta_{\iota_{\tau^i} \sqcup \iota_{\sigma^j}(e)})_{e \in E(\tau \sqcup \sigma)}]
				\\ & = \mcL_{\tau \sigma}^{(2)}[(\mcM \otimes \mcM)  \underline{\rho}  (\underline{\Delta}^+ \otimes \underline{\Delta}^+)(\tau \otimes \sigma)] (\zeta_e)_{e \in E(\tau \sqcup \sigma)}
			\end{equs}
			as required.
			
			It now remains to check that the counit $\varepsilon$ is a counitary algebra morphism. The algebra morphism property is immediate from the definitions. For the counitary property, we note that
			\begin{equs}
				(\varepsilon \otimes \id) \Delta^+ \otimes_{\sttau}^\tau [(\zeta_e)] & = (\varepsilon \otimes \id) \mcL_\tau^{(2)} [\underline{\Delta}^+ \tau] (\zeta_e)
				 = \sum_{i=1}^n c_i (\varepsilon \otimes \id) \left ( \bigotimes_{\langle \boldsymbol{\tau}_i \rangle}^{{\tau_i}} (\zeta_e) \otimes \bigotimes_{\langle \boldsymbol{\tau}^i \rangle}^{{\tau^i}} (\zeta_e) \right ) 
				 = \otimes_{\sttau}^\tau [(\zeta_e)]
			\end{equs}
			since the projection onto the linear span of $\boldsymbol{1}$ is the counit for $\underline{\Delta}^+$. The result for $(\id \otimes \eps)\Delta^+$ is similar.
			
			We now turn to the claim that $V_{\mcT(R)}$ is a comodule over $V_{\mcT_+(R)}$. Since Lemma~\ref{lem: coass} includes the fact that $(\id \otimes \Delta^+)\Delta = (\Delta \otimes \id)\Delta$, we need only show that $(\id \otimes \varepsilon)\Delta = \id$. This latter claim follows in a similar way to the proof of the counitary property above. For brevity, we omit the details.
			\end{proof}
			
			In particular, we now recover a regularity structure in the sense of the original definition given in \cite[Definition 2.1]{Hai14} which we emphasise is general enough to allow for the infinite dimensional components seen here. 
			
			\begin{definition} \label{d:tban}
				We define the regularity structure $\mcT^\mathrm{Ban} = (V_{\mcT(R)}, \mcG)$ where $\mcG$ is the group of continuous characters on $V_{\mcT_+(R)}$ acting on $V_{\mcT(R)}$ via $\Gamma_g v = (\id \otimes g) \Delta^+ v$.
			\end{definition}
			\begin{remark}
				The fact that $(g, v) \mapsto \Gamma_g v$ defines a (left) group action of $\mcG$ on $V_{\mcT_+(R)}$ follows from the comodule structure given in Lemma~\ref{lem:Hopf}
			\end{remark}

			It is clear that $\mcT^\mathrm{Ban}$ comes with a product $\star$ and a family of integration maps $\mcI^\zeta$ indexed by $\zeta \in V_\mfl$, $\mfl \in \mfL_+$. In the latter case, it still remains for us to check that these integration maps interact correctly with the action of the structure group.
			
			\begin{lemma}\label{lem:Gamma_Poly}
				For every $g \in \mcG$, $\mfl \in \mfL_+$, $\zeta \in V_\mfl$ and $\boldsymbol{\tau} \in \mcT(R)$ such that $\mcI_\mfl \boldsymbol{\tau} \in \mcT(R)$ we have that
				\begin{align*}
				\mcI^\zeta \Gamma_g - \Gamma_g \mcI^\zeta : V^{\proj \sttau} \to V_{\mcT_{\mathrm{poly}}} ,
				\qquad \text{where}
				\quad \mcT_{\mathrm{poly}} = \{X^k: k \in \mbN^d\}.
				\end{align*}
			\end{lemma}
			\begin{proof}
				It suffices to check the result for elements of the form $\otimes_{\sttau}^\tau [(\zeta_e)_{e \in E(\tau)}]$. We recall that 
				\begin{equ}
					\underline{\Delta} \mcI_\mfl \tau = (\mcI_\mfl \otimes \id) \underline{\Delta}^+ \tau + \sum_{|k|_{\mfs} < |\mcI_\mfl \tau|_\mfs} \frac{X^k}{k!} \otimes \mcI_{\mfl}^k \tau.
				\end{equ}
				As a result, we note that (writing $\zeta$ for the distribution associated to the trunk)
				\begin{equs}
					& \Delta \bigotimes_{\langle \mcI_\mfl \boldsymbol{\tau} \rangle}^{\mcI_\mfl \tau} [(\zeta_e)_{e \in E(\tau)} \sqcup \{\zeta\}] = \mcL_{\mcI_\mfl \tau}^{(2)} [ \underline{\Delta} \mcI_\mfl \tau] [(\zeta_e)_{e \in E(\tau)} \sqcup \{\zeta\}]
					\\
					& = \mcL_{\mcI_\mfl \tau}^{(2)}[(\mcI_\mfl \otimes \id) \underline{\Delta} \tau]  [(\zeta_e)_{e \in E(\tau)} \sqcup \{\zeta\}] + \sum_{k : |\mcI_{\mfl}^k \tau|_\mfs > 0} \mcL_{\mcI_\mfl \tau}^{(2)} \left [\frac{X^k}{k!} \otimes \mcJ_\mfl^k \tau \right ]  [(\zeta_e)_{e \in E(\tau)} \sqcup \{\zeta\}]
					\\
					& = \zeta \otimes \mcL_\tau^{(2)}[\underline{\Delta} \tau] (\zeta_e) + \sum_{k} \frac{X^k}{k!} \otimes \bigotimes_{\langle \mcI_{\mfl}^k \btau \rangle}^{\mcI_\mfl^k \tau} [(\zeta_e) \sqcup \{\zeta\}]
					\\
					& = (\mcI^{\zeta} \otimes \id) \Delta \bigotimes_{\sttau}^\tau [(\zeta_e)] + \sum_{k} \frac{X^k}{k!} \otimes \bigotimes_{\langle \mcI_{\mfl}^k \btau \rangle}^{\mcI_\mfl^k \tau} [(\zeta_e) \sqcup \{\zeta\}]. \label{comb_int_ss}
				\end{equs}
				Upon applying $(\id \otimes g)$ to both sides, the result immediately follows.
			\end{proof}
	
\section{Models on Regularity Structures with Infinite Dimensional Components}\label{s:Alg_Models}
In this section we recursively construct algebraically renormalised models on $\mcT^{\mathrm{Ban}}$ in the spirit of \cite{Bru18}. 
Here we consider only smooth models constructed from a fixed noise assignment $\xi$.

\subsection{Admissible Models on \texorpdfstring{$\mcT^\mathrm{Ban}$}{T}}\label{ss:Adm_mod}
We now define a notion of admissible model (see \cite[Definition 6.9]{BHZ19} for the usual scalar valued definition) that is well adapted to our setting. Throughout this section we will assume that a Banach space assignment $V$ has been fixed and will omit it from the notation for simplicity. 

We first fix an appropriate notion of kernel assignment that we wish to be realised by the abstract integration maps $\{\mcI^\zeta\}_{\zeta \in V_\mfl}$ that we constructed in the previous section. This should be nothing other than a way of assigning a translation invariant integral kernel to each element of $V_\mfl$ for each $\mfl \in \mfL_+$.

\begin{definition}\label{def:ka}
    A kernel assignment
	of order $\mco \in \mathbb{N}$
	on $\mcT^{\Ban} ( V )$
    is a collection of bounded linear maps $A_\mfl: V_\mfl \to \mcK_{\mco}^{|\mfl|}$ indexed by $\mfl \in \mfL_+$.
	We write $\mathcal{A}_+^\mco = \bigoplus_{\mfl \in \mfL_+} L ( V_\mfl, \mcK_{\mco}^{|\mfl|} )$ for the Banach space of all kernel assignments of order $\mco$
	equipped with the norm
		\begin{align*}
			 \| A \| _{\mathcal{A}_+^\mco} \coloneqq \sum_{\mfl \in \mfL_+} \, 
				\sup_{\substack{\zeta \in V_{\mfl},\\  \| \zeta \|_{V_{\mfl}} \leq 1}} \,
				\| A_{\mfl} ( \zeta ) \|_{ \mcK_{\mco}^{|\mfl|}}.
			\end{align*}
		Given $\zeta \in V_\mfl$, we write $A_\mfl(\zeta) = \sum_{n \ge 0} A_\mfl^{(n)}(\zeta)$ for the decomposition as in Remark~\ref{rem:dec2ker}.
		\end{definition}

We also introduce a corresponding notion of noise assignment on $\mcT^\Ban$. We will eventually be interested in specialising to the case where the Banach space assignment $V$ satisfies $V_\mfl = \mbR$ for each $\mfl \in \mfL_-$ in which case the notion of noise assignment introduced below will coincide with our notion of noise assignment on $\mcT^\eq$. However, since it introduces very little extra work to allow for more general noise assignments in this section and it is reasonable to imagine applications where one would wish to freeze coefficients at the level of the noise, we allow for the extra generality here.

\begin{definition} \label{d:noise_ass_TBan}
	A noise assignment on $\mcT^\Ban(V)$ is a collection of continuous linear maps $N_{\mfl} : V_\mfl \to \mcD'(\mbR^d)$ indexed by $\mfl \in \mfL_-$. We write $\mcA_-$ for the space of all noise assignments. A smooth noise assignment on $\mcT^\Ban(V)$ is a noise assignment such that $N_\mfl$ is a continuous map into $C^\infty(\mbR^d)$ with its usual Fr\'echet topology. We write $\mcA_-^\infty$ for the space of all smooth noise assignments.
\end{definition}
 We also recall that given a sector $W$ of a regularity structure $(\mcT, \mcG)$ then $(W,  \mcG)$ is also a regularity structure. This means that it makes sense to talk about models on a sector of a regularity structure. 
\begin{definition} \label{d:adm_on_tban}
	We say that a model $(\Pi, \Gamma)$ on a sector $W$ in $\mcT^\mathrm{Ban}$ is admissible with respect to a kernel assignment $A$ and a noise assignment $N$ if it satisfies the following three sets of relations. Firstly we require that
	    \begin{align*}
		        \Pi_x v &= N_\mfl(v) , \qquad \qquad
		        \Pi_x X^k = (\cdot - x)^k 
		    \end{align*}
	    for all $\mfl \in \mfL_-$, for all $v \in V^{\proj \langle \Xi_\mfl \rangle}  \cap W$ and all $k \in \mathbb{N}^d$ such that $X^k \in W$. Secondly, we require that 
	    \begin{align*}
	    	\Pi_x (X^k v) = (\cdot - x)^k \Pi_x v, \qquad \Pi_x D_i v = D_i \Pi_x v
	    \end{align*}
	    whenever $k \in \mathbb{N}^d$, $i \in \{1, \dots, d\}$ and $v \in W$ are such that $X^k v, D_i v \in W$. Finally
	    we require that for $\mfl \in \mfL_+$, we have that
	    \begin{align}\label{eq:adm_int}
		        \Pi_x \mcI^\zeta v &= A_\mfl(\zeta) \ast \Pi_x v - \sum_{|k|_\mfs < |\mcI^\zeta v|_\mfs} \frac{(\cdot - x)^k}{k!} (D^k A_\mfl(\zeta) \ast \Pi_x v)(x)
		 \end{align}
		 for every $v \in W$ and $\zeta \in V_\mfl$ such that $\mcI^\zeta v \in W$.    
\end{definition}
\begin{remark}\label{rem:order}
	We recall from \cite[Remark 5.10]{Hai14} that the correct way to interpret \eqref{eq:adm_int} is via	
	$$\Pi_x \mcI^\zeta v(\phi) = \sum_{n \ge 0} \int \mathrm{d} y  \, \phi(y) \, \Pi_x v \left ( A_\mfl^{(n)}(\zeta)(y - \cdot) - \sum_{|k|_\mfs< |\mcI^\zeta v|_\mfs} \frac{(y-x)^k}{k!} D^k A_\mfl^{(n)}(\zeta)(x - \cdot) \right ).$$
	In order for this expression to be well-defined for a generic model $(\Pi, \Gamma)$, it is necessary to require that $\mco > |k|_\mfs - a$ where $a$ is the lowest degree of a component in the smallest sector containing $v$. In particular, due to \eqref{eq:adm_int}, the definition of admissibility on a sector $W$ is a meaningful one only if $\mco$ is sufficiently large.  
\end{remark}
\subsection{Algebraic Renormalisation on \texorpdfstring{$\mcT^\mathrm{Ban}$}{T}}\label{ss: alg renorm}
 
Our goal in this section is to show that the construction of a model associated to a smooth driving noise from a preparation map as introduced in \cite{Bru18} can be adapted to our setting. This provides a generalisation of the results of \cite{Bru18} in two ways. Firstly, we do not assume that the kernel assignment is of infinite order. As discussed in Remark~\ref{rem:order}, this poses problems when constructing a model on the full regularity structure. One option to circumvent this issue is to exploit smoothness of the noise, translation invariance of the kernel and an integration by parts argument to move all derivatives in \eqref{eq:adm_int} onto $\Pi_x v$. This would lead to a new definition of admissibility for a smooth model that would allow for the construction of a model on the full regularity structure. This would have two disadvantages. Firstly, we will later want to talk about models built from preparation maps and non-translation invariant regularising kernels of finite order on $\mcT^\mathrm{eq}$. In this setting, the argument based on integration by parts is not available. Secondly, when later obaining quantitative estimates we will need tools such as the Extension Theorem provided in \cite[Theorem 5.14]{Hai14}. This result would not apply with the new definition of admissibility. In both cases, it would therefore still be necessary to restrict to a smaller sector. Therefore, rather than introducing a new notion of admissibility, we prefer to show that it is possible to construct models from preparation maps on suitable sectors rather than on the whole regularity structure immediately.

The second and more significant generalisation of \cite{Bru18} comes with our infinite dimensional setting. A first difference is that one no longer has access to a distinguished basis in the form of trees. Since elementary tensors do not form a linearly independent set, this means that more care is required when checking that the natural definition actually yields a well-defined object. The second difference is that the appropriate continuity in terms of the norm on $V^{\proj \sttau}$ is no longer a consequence of mere linearity and thus requires again more care. Both of these differences will be handled by appealing to the universal property given in Definition~\ref{def:tensor_product}.

We will use the following notations for projection operators on $\mcT^{\Ban} ( V )$.
Given $\btau \in \mcT ( R )$ we denote by $\mcQ_{\btau}$ the projection onto $V^{\proj \sttau}$.
Given $\eta \in \mathbb{R}$ we denote by $\mcQ_{\eta}$ the projection onto $\oplus_{| \btau |_{\mfs} = \eta} V^{\proj \sttau}$ 
and we define $\mcQ_{< \eta}, \mcQ_{\leq \eta}$, etc, analogously.

\begin{definition}\label{def:prep}
    A preparation map on $V_{\mcT(R)}$ on a sector $W$ is a bounded linear map $P: W \to V_{\mcT(R)}$ such that 
    $\mcQ_{\bsigma} (P - \operatorname{Id}) |_{V^{\proj \sttau}} \neq 0$ implies that
    \begin{enumerate}
	        \item $|\boldsymbol{\tau}|_{\mfs} < |\boldsymbol{\sigma}|_{\mfs}$,
	        \item $|\boldsymbol{\tau}|_{\#} > |\boldsymbol{\sigma}|_{\#}$ where $|\boldsymbol{\tau}|_{\#}$ denotes the number of noise edges in $\boldsymbol{\tau}$,
	        \item $\boldsymbol{\tau} \neq \boldsymbol{\Xi}_{\mfl}, X^k, \mcI_\mfl^k \boldsymbol{\boldsymbol{\sigma}}$
	    \end{enumerate}
    such that $P$ additionally has the commutativity properties
    $$P|_{V^{\proj \langle X^k \boldsymbol{\tau} \rangle}} = X^k P|_{V^{\proj \sttau}}, \qquad P \Gamma = \Gamma P \text{ for all } \Gamma \in G_+.$$
\end{definition}

In order to perform a recursive construction of a model with respect to a preparation map on a sector $W$, it will be important that $P : W \to W$. Since it will eventually be important to allow for different models to be constructed on the same sector (for example when obtaining convergence results), we here introduce a notion of sector that is appropriate for the constructions that follow in this section.

\begin{definition}\label{def:pgood}
	We say that a sector $W$ in $V_{\mcT(R)}$ is good if it can be written as $W = \bigoplus_{\boldsymbol{\tau} \in \mcB_W} V^{\proj \sttau}$ for some set of trees $\mcB_W \subset \mcT(R)$ (which we refer to as the generating set for $W$) with the property that $\mcI_\mfl^k \btau \in \mcB_W$ implies $\btau\in \mcB_W$ and $\btau \bsigma \in \mcB_W$ implies that $\btau \in \mcB_W$. Given a preparation map $P$ defined on $W$, we say that the sector $W$ is $P$-good if it is good and $P : W \to W$.
\end{definition}
We note that it is automatic that if $W$ is a $P$-good sector in $\mcT^\mathrm{Ban}$ then the set of trees $\mcB_W$ in the above definition spans a sector in $\mcT^\mathrm{eq}$. 
\begin{remark}
	If $\mcB$ is a finite subset of $\mcT(R)$ then it is possible to prove that if $W$ is the smallest $P$-good sector containing $\bigoplus_{\sttau \in \mcB} V^{\proj \sttau}$ then $\mcB_W$ is a finite set of trees. Since we will not explicitly need this fact in this section and will need a similar fact for a related construction adapated to the type of preparation maps that are actually used in practice, we do not provide a proof here and refer the interested reader to the related
	Lemma~\ref{l:finite}. 
\end{remark}

The following definition encodes the analogues in our setting of the usual recursive relations used to define a preparation map. The main difference to \cite{Bru18} is that the recursive definitions are not only imposed on a distinguished basis. 
\begin{definition}\label{def:mod_from_Pmap}
	Given a preparation map $P$ on a $P$-good sector $W$, a kernel assignment $A \in \mcA_+^\mco$ and a smooth noise assignment $\xi \in \mcA_-^\infty$, we say that a triple $(\Pi^P, \Pi^{P, \times}, \Gamma^P)$ is associated to $(P,A,\xi)$ on $W$ if
	\begin{enumerate}
		\item $\Pi_x^P, \Pi_x^{P, \times} : W \to C^\infty(\mathbb{R}^d)$ and $\Gamma_{xy}^P : W \to W$,
		\item $\Pi_x^{P, \times}v = N_\mfl(v)$ for each $\mfl$ such that $\Xi_\mfl \in \mcB_W$ and each $v \in V^{\proj \langle \Xi_\mfl \rangle}$,
		\item $\Pi_x^{P, \times} X^k = (\cdot - x)^k$ for each $k$ such that $X^k \in W$,
		\item $\Pi_x^{P, \times} (v \star w) = \Pi_x^{P, \times} v \cdot \Pi_x^{P, \times} w$ whenever $v, w, v\ast w  \in W$,
		\item If $v \in W$, $\zeta \in V_\mfl$ are such that $\mcI^{\zeta,k} v \in W$ then
		\begin{align*}
				\Pi_x^{P, \times} \mcI^{\zeta,k} v = D^k A_\mfl(\zeta) \ast \Pi_x^P v - \sum_{|j|_{\mfs} < |\mcI^{\zeta,k} v|_{\mfs}} \frac{(\cdot - x)^j}{j!} D^{k+j} A_\mfl(\zeta) \ast \Pi_x^P v(x),
		\end{align*}
		\item $\Pi_x^P = \Pi_x^{P, \times} P$,
		\item $\Gamma_{xy}^P = (\id \otimes (g_x^P)^{-1} (g_y^P)) \Delta$ where each $g_x^P \in \mcG$ satisfies $g_x^P(X_i) = -x_i$, and whenever $v \in W$, $\zeta \in V_\mfl$ and $k \in \mathbb{N}^d$ are such that $\mcI^{\zeta, k} v \in V_{\mcT_+(R)}$
		\begin{align*}
			g_x^P(\mcI^{\zeta,k} v) = - \sum_{| j |_{\mfs} \le |\mcI_\mfl^{\zeta,k} v|_\mfs} \frac{(-x)^j}{j!} D^{j+k} A_\mfl(\zeta_\mfl) \ast \Pi_x^P v (x) .
		\end{align*}
		\end{enumerate}
\end{definition}
\begin{remark}
	The final point of the last definition specifies the value of $g_x^P$ only on $V^{\proj \sttau}$ for $\sttau$ lying in the subalgebra generated by 
	\begin{align*}
		\bar{\mcB}_W = \{X^k: k \in \mathbb{N}^d\} \cup \{ \mcI_k^\mfl \boldsymbol{\tau} \in \mcT_+(R) : \boldsymbol{\tau} \in \mcB_W\}.
	\end{align*}
	However, any extension of $g_x^P$ defined in this way to an element of $\mcG$ will have the same action on $W$. Since the imporant object is the map $\Gamma_{xy}^P$ and not the corresponding characters, we do not distinguish between such extensions. This corresponds to viewing $W$ as a regularity structure whose structure group is the quotient of $\mcG$ by the kernel of the action of $\mcG$ on $W$. 
	
	One should of course check that at least one extension of $g_x^P$ to an element of $\mcG$ exists. To this end, we note that $\bar{\mcB}_W^c \coloneqq \mcT_+(R) \setminus \bar{\mcB}_W$ is such that
	\begin{align*}
		V_{\mcT_+(R)} = \left ( \bigoplus_{\boldsymbol{\tau} \in \bar{\mcB}_W} V^{\proj \sttau} \right ) \oplus \left ( \bigoplus_{\boldsymbol{\tau} \in \bar{\mcB}_W^c} V^{\proj \sttau}  \right )
	\end{align*}
	with $\bigoplus_{\boldsymbol{\tau} \in \bar{\mcB}_W^c} V^{\proj \sttau}$ being an algebra ideal in $V_{\mcT_+(R)}$. In particular, it follows that the projection $\pi$ onto $\bigoplus_{\boldsymbol{\tau} \in \bar{\mcB}_W} V^{\proj \sttau}$ along $\bigoplus_{\boldsymbol{\tau} \in \bar{\mcB}_W^c} V^{\proj \sttau} $ is continuous and multiplicative so that $g_x^P \circ \pi$ is a suitable extension.
	 \end{remark}
The following definition will be used to ensure that we assume a high enough order of kernel assignment to apply the Extension Theorem given in \cite[Theorem 5.14]{Hai14} on the sector $W$.
	 \begin{definition}\label{def:goodasdf}
	 	If $W$ is a good sector then we define its order via\footnote{The presence of the term $\max ( \mfs )$ may look surprising at first. When tracking the order of the kernels used in the Extension Theorem of \cite[Theorem 5.14]{Hai14}, this term comes from applying the anisotropic Taylor formula (see e.g.\  Lemma~\ref{Tay} or \cite[Proposition A.1]{Hai14}). We expect that this condition could be improved at the cost of replacing the norm on regularising kernels by suitable H\"older norms.}
	 	\begin{align*}
	 	\mathrm{ord}(W) 
	 	\coloneqq 
	 	\max \Big\lbrace 
					| \boldsymbol{\tau} |_{\mfs} + | \mfl |_{\mfs} + \max ( \mfs ) , 
				 	| k |_{\mfs}  - \lfloor \min \mcA_{W} \rfloor :
			\mcI_{\mfl}^k \boldsymbol{\tau} \in \mcB_W	 		
	 		\Big\rbrace .
	 	\end{align*}
	 	Here we recall that $\mcA_W$ denotes the set of degrees appearing in the sector $W$.
	 \end{definition} 
	 Our next goal is to prove the existence and uniqueness of a triple $(\Pi^P, \Pi^{P, \times}, \Gamma^P)$ as in Definition~\ref{def:mod_from_Pmap}. To do this, we adapt the construction from \cite{Bru18} in the scalar valued setting. The approach of \cite{Bru18} is to construct the model by induction with respect to the ordering $\prec$ defined below.
	 \begin{definition} \label{def:prec}
	 	We define a partial order $\preceq$ on $\mcT$ by declaring that
	 	\begin{align*}
	 		\boldsymbol{\sigma} \prec \boldsymbol{\tau} \iff |	\boldsymbol{\sigma}|_\# < |\boldsymbol{\tau}|_\# \text{ or }  |	\boldsymbol{\sigma}|_\# = |\boldsymbol{\tau}|_\#, |\boldsymbol{\sigma}|_\mfs < |\boldsymbol{\tau}|_\mfs .
	 	\end{align*}
	  \end{definition}
	  	\begin{remark}
	  			It is not the case that for any fixed tree $\btau$ there are finitely many trees $\bsigma$ such that $\bsigma \prec \btau$. For example, any tree with two noises is preceded by all trees with one noise. The reader may then worry that induction with respect to $\prec$ is not justified. However we note that it is the case that $\prec$ is a well-founded binary relation. Indeed, for any non-empty set of trees $\mcS_0$ one can first consider the subset $\mcS_1$ consisting of all trees in $\mcS_0$ with the fewest number of noises. Then one can consider the subset $\mcS_2$ of $\mcS_1$ consisting of those trees of least degree (this is possible by the assumption that $\mcA$ is locally finite and bounded below). Any element $\boldsymbol{\tau} \in \mcS_2$ then has the property that $\boldsymbol{\sigma} \prec \boldsymbol{\tau}$ implies that $\boldsymbol{\sigma} \not \in \mcS_0$ as required. This means that in place of usual induction, in this type of context it is possible to instead appeal to Noetherian induction.
	  		\end{remark}
	  		
	  		\begin{prop}\label{prop:mod_prep_1}
	  			Suppose that $P$ is a preparation map defined on a $P$-good sector $W$, $A \in \mcA_+^\mco$ with $\mco > \mathrm{ord}(W)$ and that $\xi \in \mcA_-^\infty$. Then there exists a unique triple $(\Pi^P, \Pi^{P, \times}, \Gamma^P)$ associated to $(P, A, \xi)$ on $W$.	 
	  		\end{prop}
	  		
	  		In preparation for the proof of this result, it will be useful to have prepared a result allowing us to construct multiplicative maps via the universal property contained in Definition~\ref{def:tensor_product}.
	  		
	  		\begin{lemma}\label{lem:multiplication}
	  		    Suppose that $Z$ is a complete commutative topological algebra and that $\boldsymbol{\tau} = \prod_{i=0}^n \boldsymbol{\tau}_i$ with $\boldsymbol{\tau}_i = \mcI_{\mfl_i}^{k_i} \boldsymbol{\sigma}_i$ for $i = 1, \dots, n$ and $\boldsymbol{\tau}_0 = X^k$. Suppose further that for each $i = 0, \dots, n$ we are given a continuous linear map $f_i : V^{\proj \sttaun{i}} \to Z$ such that $\sttaun{i} = \sttaun{j}$ implies that $f_i = f_j$. Then there exists a unique continuous linear
	  		    map $f^{\boldsymbol{\tau}}$ such that the following diagram commutes.
	  		    \begin{center}
	  			        \begin{tikzcd}
	  				            V^{\times \tau} \arrow[r, "\otimes_{\sttau}^\tau"] \arrow[dr, "\mcM_Z^{\otimes n } \circ \bigotimes_{i=0}^n (f_i \otimes_{\sttaun{i}}^{\tau_i})_{i = 0}^n"'] & V^{\proj \sttau} \arrow[d, dotted, "f^{\boldsymbol{\tau}}"] \\ & Z
	  				        \end{tikzcd}
	  			    \end{center}
	  		    Furthermore, $f^{\boldsymbol{\tau}}$ defined in this way is multiplicative in the stronger sense that if $\boldsymbol{\tau} = \boldsymbol{\sigma}_1 \boldsymbol{\sigma}_2$ and $v_i \in V^{\proj \langle \boldsymbol{\sigma}_i \rangle }$ then $f^{\boldsymbol{\tau}}(v_1 \star v_2) = f^{\boldsymbol{\sigma}_1}(v_1) f^{\boldsymbol{\sigma}_2}(v_2)$.
	  		\end{lemma}
	  		\begin{remark}
	  		    Since $\Xi_\mfl = \mcI_{\mfl} 1$, the above lemma covers all trees in $\mcT$. 
	  		\end{remark}
	  		\begin{proof}
	  		    We remark that by commutativity of $\ast_Z$, the leftmost arrow does not depend on the order of the $\tau_i$ and so is completely specified by the information $\tau \in \boldsymbol{\tau}$.
	  		
	  		    By appealing to the universal property of $V^{\proj \sttau}$, to show that $f$ is well-defined we only need to show that the $\tau$-indexed family of leftmost arrows is multilinear, continuous and $\sttau$-symmetric.
	  		
	  		    Multilinearity is clear due to the bilinearity of $\ast_Z$, linearity of $f_i$ and multilinearity of $\otimes_{\sttaun{i}}^{\tau_i}$. Continuity is clear by componentwise continuity and the continuity of the multiplication on $Z$. This leaves us only to consider symmetry. 
	  		    We fix $\tau, \bar \tau \in \boldsymbol{\tau}$ and 
	  		    $\gamma \in \hom_{\sttau} (\tau, \bar{\tau} )$. 
	  		    We need to show that 
	  		    \begin{align*}
	  			        \mcM_Z^{\otimes n }\bigotimes_{i=0}^n \left ( f_i \otimes_{\langle \bar{\boldsymbol{\tau}}_i \rangle}^{\bar{\tau}_i} \right ) \circ \gamma = \mcM_Z^{\otimes n }\bigotimes_{i=0}^n \left( f_i \otimes_{\sttaun{i}}^{\tau_i} \right).
	  			    \end{align*}
	  		    To that end, we write $\tau = \prod_{i = 0}^n \tau_i, \bar{\tau} = \prod_{i = 0}^n \bar{\tau}_i$ with $\tau_0 = \bar{\tau}_0 = X^k$, $\tau_i = \mcI_{\mfl_i}^{k_i} \sigma_i$ and $\bar{\tau}_i = \mcI_{\bar{\mfl}_i}^{\bar{k}_i} \bar{\sigma}_i$. We note that since $\gamma$ preserves incidency at the root, we can equivalently think of $\gamma$ as being specified by a bijection $\gamma_\rho: \{1, \dots, n\} \to \{1, \dots, n\}$ and a family of isomorphisms $\gamma_i \in \hom(\tau_{\gamma_\rho^{-1}(i)}, \bar{\tau}_i)$. The equivalence is obtained via the relation $\gamma(e) = \sum_{i} \gamma_i(e) \mathbf{1}_{e \in E(\tau_{\gamma_\rho^{-1}(i)})}$. With this notation in hand, we can compute
	  		    \begin{align*}
	  			    	\mcM_Z^{\otimes n }\bigotimes_{i=0}^n \left ( f_i \otimes_{\langle \bar{\boldsymbol{\tau}}_i \rangle}^{\bar{\tau}_i} \right ) \circ \gamma \left ( \zeta_e \right )_{e \in E(\tau)} &= \mcM_Z^{\otimes n}\bigotimes_{i=0}^n \left ( f_i \otimes_{\langle \bar{\boldsymbol{\tau}}_i \rangle}^{\bar{\tau}_i} \left ( \zeta_{\gamma_i^{-1}(e)} \right )_{e \in E(\tau_{\gamma_\rho^{-1}(i)})} \right )
	  			    	\\
	  			    	& = \mcM_Z^{\otimes n }\bigotimes_{i=0}^n \left ( f_i \otimes_{\langle {\boldsymbol{\tau}}_{\gamma_\rho^{-1}(i)} \rangle}^{{\tau}_{\gamma_\rho^{-1}(i)}} \left ( \zeta_{\gamma_e} \right )_{e \in E(\tau_{\gamma_\rho^{-1}(i)})} \right )
	  			    \end{align*}
	  		    where to reach the second line, we used $\sttau$-symmetry of the maps $\left ( \otimes_\sttau^\tau \right )_{\tau \in \boldsymbol{\tau}}$. By assumption, $f_i = f_{\gamma_\rho^{-1} (i)}$. Therefore
	  		      \begin{align*}
	  			    	\mcM_Z^{\otimes n }\bigotimes_{i=0}^n \left ( f_i \otimes_{\langle \bar{\boldsymbol{\tau}}_i \rangle}^{\bar{\tau}_i} \right ) \circ \gamma \left ( \zeta_e \right )_{e \in E(\tau)} 
	  			    	& = \mcM_Z^{\otimes n }\bigotimes_{i=0}^n \left ( f_{\gamma_\rho^{-1}(i)} \otimes_{\langle {\boldsymbol{\tau}}_{\gamma_\rho^{-1}(i)} \rangle}^{{\tau}_{\gamma_\rho^{-1}(i)}} \left ( \zeta_{\gamma_e} \right )_{e \in E(\tau_{\gamma_\rho^{-1}(i)})} \right )
	  			    	\\ & = \mcM_Z^{\otimes n} \bigotimes_{i=0}^n \left( f_i \otimes_{\sttaun{i}}^{\tau_i} \right) \left ( \zeta_e \right )_{e \in E(\tau)}    \end{align*}
	  		    as desired, where to pass to the second line we used commutativity of the product in $Z$.
	  		    
	  		    It remains to prove the stronger form of multiplicativity. For this, by appealing to the uniqueness part of the definition of $V^{\proj (\sssigman{1} \otimes \sssigman{2})}$, we need to show that
	  		    $f^{\boldsymbol{\tau}} \circ \star \circ \bigotimes_{\sssigman{1} \otimes \sssigman{2}}^{\sigma_1 \sqcup \sigma_2} = \mcM_Z (f^{\boldsymbol{\sigma}_1} \otimes f^{\boldsymbol{\sigma}_2})\circ \bigotimes_{\sssigman{1} \otimes \sssigman{2}}^{\sigma_1 \sqcup \sigma_2}$. We note that by definition of $\star$, we have that $\star \circ \bigotimes_{\sssigman{1} \otimes \sssigman{2}}^{\sigma_1 \sqcup \sigma_2} = \bigotimes_\sttau^\tau$ so that by definition of $f^{\boldsymbol{\tau}}$, the left hand side is nothing other than $\mcM_Z^{\otimes n} \big ( \otimes_{i=0}^n f_i \bigotimes_{\sttaun{i}}^{\tau_i} \big )$. Meanwhile, by construction of the tensor product of symmetric sets, the right hand side is $\mcM_Z (f^{\boldsymbol{\sigma}_1} \bigotimes_\sssigman{1}^{\sigma_1} \otimes f^{\boldsymbol{\sigma}_2} \bigotimes_\sssigman{2}^{\sigma_2})$. The result then follows from the definition of $f^{\boldsymbol{\sigma}_i}$ and the associativity of the product in $Z$.
	  		\end{proof}
	  		
	  		With this multiplication result in hand, we are now ready to proceed with the proof of Proposition~\ref{prop:mod_prep_1}.
	  		\begin{proof}[Proof of Proposition~\ref{prop:mod_prep_1}]
	  		 In a first part of the proof, we consider the maps $\Pi^P, \Pi^{P, \times}$. We define the statement $\Phi(\btau)$ to be that there exists a unique pair of continuous linear maps $$\Pi_x^P, \Pi_x^{P, \times} : \bigoplus_{\substack{\boldsymbol{\sigma} \in \mcB_W \\ \boldsymbol{\sigma} \preceq \boldsymbol{\tau}}}V^{\proj \sssigma} \to C^\infty(\mathbb{R}^d)$$ that satisfies the relations specified in Definition~\ref{def:mod_from_Pmap}.  We remark that it is straightforward to check that this set of relations indeed only refers to components corresponding to trees satisfying $\boldsymbol{\sigma} \prec \boldsymbol{\tau}$. We prove by Noetherian induction that $\Phi(\btau)$ holds for all $\btau \in \mcB_W$. This means we have to show that
	  			$$\big (\forall \boldsymbol{\sigma} \prec \boldsymbol{\tau}, \Phi(\boldsymbol{\sigma}) \big ) \implies \Phi(\boldsymbol{\tau}).$$
	  			It is straightforward to show that the set of statements $\Phi(\boldsymbol{\sigma})$ for $\bsigma \prec \btau$ implies the existence of a unique pair of maps $\Pi_x^P, \Pi_x^{P, \times}$ satisfying the recursive relations and defined on $\bigoplus_{\bsigma \prec \btau} V^{\proj \sssigma}$. Therefore, it suffices to show that each of these maps is uniquely defined and continuous on $V^{\proj \sttau}$. Furthermore, since $\mcQ_{\bsigma} (P - \id) v \ne 0$ for $v \in V^{\proj \sttau}$ implies that $\bsigma \prec \btau$, the relation $\Pi_x^P = \Pi_x^{P, \times} + \Pi_x^{P, \times} (P-\id)$ implies that $\Pi_x^P$ is uniquely defined and continuous on $V^{\proj \sttau}$ as soon as $\Pi_x^{P, \times}$ is. Therefore it remains only to consider $\Pi_x^{P, \times}$.
	  				  			
	  			As in the scalar valued setting, we can divide into the cases where $\boldsymbol{\tau} \in \{X^k, \Xi_\mfl\}$, where $\boldsymbol{\tau}$ is planted and where $\boldsymbol{\tau}$ is a non-trivial product of planted trees. The first case is immediate from the definitions so we focus on the latter two. In fact, the case of multiplication is precisely the situation of Lemma ~\ref{lem:multiplication} so that it remains to consider the case of planted trees.
	  			
	  			Our strategy is now to show that the recursive definition uniquely defines the required map via the universal property stated in Definition~\ref{def:tensor_product}. We write $\btau = \mcI_\mfl^k \bsigma$ and we will show that the family of maps
	  			\begin{align}\label{eq:a45}
	  				(\zeta_*) \sqcup (\zeta_e)_{e \in E(\sigma)} \mapsto \Pi_x^{P, \times} \mcI^{\zeta_*, k} \bigotimes_{\sssigma}^\sigma (\zeta_e)_{e \in E(\sigma)}
	  			\end{align}
	  			indexed by $\tau = \mcI_\mfl^k \sigma \in \btau$ is a multilinear, continuous, $\sttau$-symmetric family. Once this is proven, the universal property will yield $\Pi_x^{P, \times} : V^{\proj \sttau} \to C^\infty(\mbR^d)$ as the unique factoring map. We note that it follows by a simple argument using linearity, density and continuity that the required recursive relation then holds on an arbitrary element of the form $\mcI^\zeta v \in V^{\proj \sttau}$.
	  			
	  			Multilinearity follows from the fact that convolution is bilinear, $\otimes_{\sssigma}^\sigma$ is multilinear and $A_\mfl$ and $\Pi_x^P \boldsymbol{\sigma}$ are both linear. Symmetry follows from the fact that any symmetry between $\tau = \mcI_\mfl \sigma$ and $\bar{\tau} = \mcI_\mfl \bar{\sigma}$ maps the trunk of $\tau$ (the unique edge incident to the root of $\tau$) to the trunk of $\bar{\tau}$ and induces a symmetry from $\sigma$ to $\bar{\sigma}$. The result then follows from the fact that $\otimes_{\sssigma}^\sigma$ is a $\sssigma$-symmetric family of maps.
	  			    
	  			To complete the first part of the proof, we have to show that \eqref{eq:a45} is continuous from $V^{\times \tau} = V^{\times \sigma} \times V_{\mfl}$ to\footnote{We recall that we endow $C^{\infty} ( \mathbb{R}^d )$ with its natural Fr\'echet topology induced by the family of seminorms $\| \cdot \|_{C^r (\mfK )}$, where $r \in \mathbb{N}$ and $\mfK \subset \mathbb{R}^d$ runs through any (countable) exhaustion of compact sets.}
	  			$C^{\infty} ( \mathbb{R}^d )$.
	  			Since the map \eqref{eq:a45} is multilinear, it suffices to prove that the right-hand side thereof is bounded in each of the $\| \cdot \|_{C^r ( \mathfrak{K} )}$ norms by a constant multiple of $\| \zeta_{*}\| \prod_{e \in E ( \sigma )} \| \zeta_e \|$.
	  			To that effect, it suffices to prove that for all $r, \mathfrak{K}$, 
	  				\begin{align} \label{eq:a46}
	  						\| A_\mfl(\zeta_*) \ast D^k \Pi_x^P \bigotimes_\sssigma^\sigma (\zeta_e)_{e \in E(\sigma)}] \|_{C^r ( \mathfrak{K} )}
	  						& \lesssim_{r, \mathfrak{K}} \| \zeta_{*}\| \prod_{e \in E ( \sigma )} \| \zeta_e \| ,
	  					\end{align}
	  			since this is enough to control all terms appearing in the recursive definition of the right-hand side of \eqref{eq:a45}.
	  			For that purpose, on the one hand we know by assumption that $\Pi_x^P \bigotimes_\sssigma^\sigma \colon V^{\times \sigma} \to C^{\infty} ( \mathbb{R}^d )$ itself is continuous, so that for all $r, \mathfrak{K}$, 
	  				\begin{align*}
	  						\| D^k \Pi_x^P \bigotimes_\sssigma^\sigma (\zeta_e)_{e \in E(\sigma)}\|_{C^r ( \mathfrak{K} )}
	  						& \lesssim_{r, \mathfrak{K}} \prod_{e \in E ( \sigma )} \| \zeta_e \| .
	  					\end{align*}
	  			On the other hand, 
	  			by the Definition~\ref{d:rk3}, we obtain that given any kernel $K \in \mathcal{K}_0^{|\mfl|}$, for all $f \in C^{\infty} ( \mathbb{R}^d )$ 
	  			$\| K * f \|_{C^r ( \mfK)} \leq \| K \|_{L^1} \, \| f \|_{C^{r} ( \bar \mfK )} \lesssim \| K \|_{\mathcal{K}_0^{|\mfl|}} \, \| f \|_{C^{r} ( \bar \mfK )}$ by Young's convolution inequality.
	  			Applying this to $K = A_\mfl(\zeta_\mfl)$ and $f = D^k \Pi_x^P \bigotimes_\sssigma^\sigma (\zeta_e)_{e \in E(\sigma)}$, we get the desired inequality \eqref{eq:a46}, since by Definition~\ref{def:ka} we have $\| A_\mfl(\zeta_\mfl) \|_{\mathcal{K}_0^{|\mfl|}} \lesssim \| \zeta_{\mfl} \|$. This completes the first part of the proof.
	  			
	  			It remains to show that the maps $\Gamma^P$ are uniquely well-defined on $W$ by the last point in Definition~\ref{def:mod_from_Pmap}. This follows by a similar Noetherian induction to the case of $\Pi_x^P, \Pi_x^{P, \times}$. Since the new ideas in comparison to the scalar valued case are already demonstrated in the argument above, for brevity we omit the details here.
	  			\end{proof}
	  			
	  			It now remains to show that if $(\Pi^P, \Pi^{P, \times}, \Gamma^P)$ is the unique triple associated to $(P, A, \xi)$ on a $P$-good sector $W$ then $(\Pi^P, \Gamma^P)$ is in fact a model. To prove this, we have to check that the algebraic relations for $\Gamma^P$ hold and that both $\Pi^P$ and $\Gamma^P$ satisfy the required analytic estimates. At this stage these estimates are allowed to leverage the smoothness of the driving noise but importantly must be quantitative in the norm on $V^{\proj \sttau}$.
	  			
	  			In order to check that $\Pi_x^P \Gamma_{xy}^P = \Pi_y^P$, it is convenient to construct an intermediate functional $\boldsymbol{\Pi}^P$ such that $\Pi_x^P = \boldsymbol{\Pi} \Gamma_{g_x^P}$. Since $\boldsymbol{\Pi}^P$ (and its multiplicative analogue) will also be used to define the BPHZ model, we introduce it here. 
	  			
	  			\begin{definition}\label{def:pre_mod}
	  				Given a preparation map $P$ on a $P$-good sector $W$, a kernel assignment $A \in \mcA_+^\mco$ and a smooth noise assignment $\xi \in \mcA_-^\infty$, we say that a pair $(\boldsymbol{\Pi}^P, \boldsymbol{\Pi}^{P, \times})$ is associated to $(P,A,\xi)$ on $W$ if
	  				\begin{enumerate}
	  					\item $\boldsymbol{\Pi}^P, \boldsymbol{\Pi}^{P, \times} : W \to C^\infty(\mathbb{R}^d)$,
	  					\item $\boldsymbol{\Pi}^{P, \times} v = N_\mfl(v)$ for each $\mfl$ such that $\Xi_\mfl \in \mcB_W$ and each $v \in V^{\proj \langle \Xi_\mfl \rangle}$,
	  					\item $\boldsymbol{\Pi}^{P, \times} X^k = (\cdot - x)^k$ for each $k$ such that $X^k \in W$,
	  					\item $\boldsymbol{\Pi}^{P, \times} (v \star w) = \boldsymbol{\Pi}^{P, \times} v \cdot \boldsymbol{\Pi}^{P, \times} w$ whenever $v, w, v\ast w  \in W$,
	  					\item If $v \in W$, $\zeta \in V_\mfl$ are such that $\mcI^{\zeta,k} v \in W$ then
	  					$
	  						\boldsymbol{\Pi}^{P, \times} \mcI^{\zeta,k} v = D^k A_\mfl(\zeta) \ast \boldsymbol{\Pi}^P v 
	  					$,
	  					\item $\boldsymbol{\Pi}^P = \boldsymbol{\Pi}^{P, \times} P$.
	  				\end{enumerate}
	  			\end{definition}
	  			
	  			Similarly to Proposition~\ref{prop:mod_prep_1}, we have the following existence and uniqueness result for such pairs. 
	  			\begin{prop}\label{prop:mod_prep_2}
	  					Suppose that $P$ is a preparation map defined on a $P$-good sector $W$ and that $A \in \mcA_+^\mco$ with $\mco > \mathrm{ord}(W)$ and that $\xi \in \mcA_-^\infty$. Then there exists a unique pair $(\boldsymbol{\Pi}^P, \boldsymbol{\Pi}^{P, \times})$ associated to $(P, A, \xi)$ on $W$.	 
	  			\end{prop}
	  			Since the proof is very similar to that of Proposition~\ref{prop:mod_prep_1}, we omit the details here.
	  			\begin{lemma}
	  			    With the definitions above $\Pi_x^P \Gamma_{xy}^P = \Pi_y^P$. In addition, the map $\Gamma_{xy}^P$ is multiplicative with respect to $\star$ and satisfies the relations $\Gamma_{xy}^P \Gamma_{yz}^P = \Gamma_{xz}^P$ and $\Gamma_{xx}^P = \operatorname{Id}$.
	  			\end{lemma}
	  			\begin{proof}
	  			    The latter two claims follow from the fact that the precomposition by the antipode is the inverse map for the character group of a Hopf algebra equipped with the convolution product. Multiplicativity of $\Gamma_{xy}^P$ follows by writing 
	  			    $$\Gamma_{xy}^P = (\id \otimes g_x^P \mcA) \Delta^+ (\id \otimes g_y^P) \Delta^+$$ 
	  			    and exploiting multiplicativity of each of the constituent maps.
	  			    
	  			    Therefore we check only the first claim. To do this, we use the usual idea of showing that 
	  			    \begin{align}\label{eq:Pi_center}
	  			    \Pi_x^P = \boldsymbol{\Pi}^P \Gamma_{g_x^P}
	  			    \end{align}
	  			    for every $x$ which implies that
	  			    $\Pi_x^P \Gamma_{xy}^P = \boldsymbol{\Pi}^P \Gamma_{g_x^P \mcA \circ g_y^P} = \Pi_y^P.$
	  			
	  			    We show \eqref{eq:Pi_center} by Noetherian induction over $\mcB_W$ with respect to $\prec$. In fact, we prove also the claim that $\Pi_x^{P, \times} = \boldsymbol{\Pi}^{P, \times} \Gamma_{g_x^P}$. As before, we separate the cases where $\boldsymbol{\tau} \in \{X^k, \Xi_\mfl\}$, where $\boldsymbol{\tau}$ is planted and where $\boldsymbol{\tau}$ is a product of planted trees. The first case is straightforward since $\Delta^+ X_i = X_i \otimes 1 + 1 \otimes X_i, \Delta^+ v = v \otimes 1$ for each $v \in V^{\proj \langle \Xi_\mfl \rangle}$ (as in the scalar valued case) and $P$ acts trivially.
	  			We next consider the case where $\tau$ is a product of planted trees. We first consider the result for $\Pi_x^{P, \times}$. By the universal property, it is sufficient to consider elements of the form $\bigotimes_{\sttau}^\tau (\zeta_e)$. We then utilise first the definition of $\star$ and then the multiplicativity of $\boldsymbol{\Pi}^{P, \times}$ and $\Gamma_{g_x^P}$ to write  \begin{equs}
	  				    	\boldsymbol{\Pi}^{P, \times} \Gamma_{g_x^P} \bigotimes_{\sttau}^\tau  = \boldsymbol{\Pi}^{P, \times} \Gamma_{g_x^P} \star^{\otimes (n-1)} \bigotimes_{\otimes_{i=1}^n \sttaun{i}}^{\sqcup_{i=1}^n \tau_i} = \prod_{i=1}^n \boldsymbol{\Pi}^{P, \times} \Gamma_{g_x^P} \bigotimes_{\sttaun{i}}^{\tau_i}.
	  				    \end{equs}
	  			    Since $\boldsymbol{\tau}_i \prec \boldsymbol{\tau}$ for each $i$, by multiplicativity of $\Pi_x^{P, \times}$ we conclude that
	  			    \begin{equs}
	  				    		\boldsymbol{\Pi}^{P, \times} \Gamma_{g_x^P} \bigotimes_{\sttau}^\tau=\prod_{i=1}^n \Pi_x^{P, \times} \bigotimes_{\sttaun{i}}^{\tau_i}  = \Pi_x^{P, \times} \bigotimes_{\sttau}^\tau.
	  				    \end{equs}
	  			    
	  			     For the claim involving $\Pi_x^P$, we use the fact $$\boldsymbol{\Pi}^{P} \Gamma_{g_x^P} = \boldsymbol{\Pi}^{P, \times} \Gamma_{g_x^P} P$$
	  			    since $P$ commutes with the structure group. We then write $P = (P - \id) + \id$ and consider the resulting terms separately. The result for the second term is exactly the previous claim so it suffices to consider only the first. Here, we use the second enumerated point in Definition~\ref{def:prep} and the fact that we may assume the result for all $\boldsymbol{\sigma} \prec \boldsymbol{\tau}$ to conclude that
	  			    $\boldsymbol{\Pi} ^{P, \times} \Gamma_{g_x^P} (P - \id) = \Pi_x^{P, \times}(P - \id).$
	  			     Therefore, in this case we obtain
	  			    $\boldsymbol{\Pi}^{P, \times} \Gamma_{g_x^P}\bigotimes_{\sttau}^\tau = \Pi_x^{P, \times}\bigotimes_{\sttau}^\tau$
	  			    which implies that
	  			    $$\boldsymbol{\Pi}^P \Gamma_{g_x^P}\bigotimes_{\sttau}^\tau = \boldsymbol{\Pi}^{P, \times} P \Gamma_{g_x^P}\bigotimes_{\sttau}^\tau = \boldsymbol{\Pi}^{P, \times} \Gamma_{g_x^P} P\bigotimes_{\sttau}^\tau = \Pi_x^{P, \times} P\bigotimes_{\sttau}^\tau = \Pi_x^P\bigotimes_{\sttau}^\tau.$$
	  			    It remains to consider the case where $\boldsymbol{\tau} = \mcI_\mfl^j \boldsymbol{\sigma}$. By the universal property, it is again enough to consider elements of the form $\bigotimes_{\sttau}^\tau [(\zeta_e)]$. 
	  			
	  			    By applying the formula for $\Gamma_g \bigotimes_{\langle \boldsymbol{\mcI_\mfl^j \tau} \rangle}^{\mcI_\mfl^j \tau} [\{\zeta_\mfl\} \sqcup (\zeta_e)]$ which follows immediately from the equality \eqref{comb_int_ss} obtained in the proof of Lemma~\ref{lem:Gamma_Poly},
	  			    we see that
	  			    \begin{equs}
	  				        \boldsymbol{\Pi}^P \Gamma_{g_x^P} \bigotimes_{\langle \mcI_\mfl^j \boldsymbol{\tau} \rangle}^{\mcI_\mfl^j \tau} [\{\zeta_\mfl\} \sqcup (\zeta_e)]& = D^j A_\mfl(\zeta_\mfl) \ast \boldsymbol{\Pi}^P \Gamma_{g_x^P} \bigotimes_{\sttau}^\tau [(\zeta_e)] + \sum_k \frac{(\cdot)^k}{k!} g_x^P \left (\bigotimes_{\langle \mcI_{\mfl}^{k+j} \btau \rangle}^{\mcI_\mfl^{k+j}\tau } [\{\zeta_\mfl\} \sqcup (\zeta_e)]\right)
	  				        \\
	  				        &= D^j A_\mfl(\zeta_\mfl) \ast \Pi_x^P\bigotimes_{\sttau}^\tau [(\zeta_e)] - \sum_{k} \sum_{\ell} \frac{(\cdot)^k (-x)^\ell}{\ell! k!} D^{j + k+l} A_\mfl(\zeta_\mfl) \ast \Pi_x^P \bigotimes_{\sttau}^\tau [(\zeta_e)](x)
	  				        \\
	  				        &= D^j A_\mfl(\zeta_\mfl) \ast \Pi_x^P \bigotimes_{\sttau}^\tau [(\zeta_e)] - \sum_k \frac{(\cdot - x)^k}{k!} D^{k+j} A_\mfl(\zeta_\mfl) \ast \Pi_x^P \bigotimes_{\sttau}^\tau [(\zeta_e)](x)
	  				        \\ 
	  				        &= \Pi_x^P \bigotimes_{\langle \mcI_\mfl^j \boldsymbol{\tau} \rangle}^{\mcI_\mfl^j \tau} [\{\zeta_\mfl\} \sqcup (\zeta_e)]
	  				    \end{equs}
	  			    where the second line follows by the result for all $\boldsymbol{\sigma} \prec \boldsymbol{\tau}$. The result for $\Pi_x^{P, \times}$ follows in the same way because $P$ acts trivially on planted trees.
	  			\end{proof}
	  			
	  			In order to conclude the section on algebraic renormalisation it remains to show that $(\Pi^P, \Gamma^P)$ satisfy the analytic requirements for a model. We emphasise that at this stage, this is necessary only for a fixed smooth driving noise. In particular no uniformity with respect to something like a mollification scale is needed. However, we also point out that in comparison to the scalar valued setting some more care is needed since we cannot exploit finite dimensionality and linearity to get the right dependence on $\|v\|_{V^{\proj \sttau}}$. We will first prove the bounds on $\Pi^P$ and $\Pi^{P, \times}$ and then show that bounds on $\Gamma^P$ follow as a corollary in a more generic way.
	  			
	  			It is convenient to subdivide the model norm into its components corresponding to individual trees. To this end, we introduce
	  			\begin{align}
	  					\|\Pi\|_{\btau; \mfK} &= \sup_{\varphi \in \mcB^r} \sup_{\lambda \in (0,1]} \sup_{x \in \mfK} \sup_{v \in V^{\proj \sttau}} \frac{|\Pi_x v (\varphi_x^\lambda)|}{\lambda^{|\btau|_\mfs} \|v\|_{V^{\proj \sttau}}} \label{refined_pi_norm}
	  					\\
	  					\|\Gamma\|_{\btau; \mfK} &= \sup_{\substack{x,y \in \mfK \\ 0 < |x-y|_{\mfs} < 1}} \sup_{\alpha < |\btau|_\mfs} \sup_{v \in V^{\proj \sttau}} \frac{\|\Gamma_{xy}v \|_{\alpha}}{|x-y|_\mfs^{|\btau|_{\mfs} - \alpha} \|v\|_{V^{\proj \sttau}}} \label{refined_gamma_norm}
	  		\end{align}
	  		which are analogues of the usual model norms restricted to $V^{\proj \sttau}$, 
	  		and where we denoted $\| \cdot \|_{\alpha} = \| \mcQ_{\alpha} \cdot \|_{\oplus_{|\bsigma|_{\mfs} = \alpha} V^{\proj \sssigma}}$. In order to conclude the bounds on $\Gamma^P$, we will need the following lemma which will be useful several times in this paper.
	  		\begin{lemma}\label{lem:Gam_from_Pi}  				Suppose that for $i = 1, 2$, $(\Pi^i, \Gamma^i)$  satisfy all of the requirements of an admissible model on a good sector $W$ of order $\mathrm{ord}(W) < \mco$ except for possibly the analytic bounds where admissibility is with respect to 
			a pair of noise assignments $\xi^i \in \mcA_-$
	  		and a pair of kernel assignments $A^i \in \mcA_+^\mco$. Then for every $\btau \in \mcB_W$ there exists a finite set $\mcB_{\btau} \subset \mcB_W$ such that $\bsigma \in \mcB_{\btau}$ implies that $\bsigma \prec \btau$ and such that there exists $C_{\btau}, R_{\btau}, k_{\btau} > 0$ such that
	  				$$\|\Gamma^1\|_{\btau; \mfK} \lesssim C_{\btau} \Big (1 + \max_{\bsigma \in \mcB_{\btau}} \big ( \|\Pi^1\|_{\bsigma; \mfK+ B(R_{\btau})} + \|\Gamma^1\|_{\bsigma; \mfK+ B(R_\btau)} \big ) \Big  )^{k_{\btau}}.$$
	  				
	  				Furthermore for each $\btau \in \mcB_W$ 
	  				\begin{align*}
	  					\|\Gamma^1 -  \Gamma^2\|_{\btau; \mfK} 
	  					\lesssim 
	  					\|A^1 - A^2\|_{\mcA_+^\mco} +  \max_{\bsigma \in \mcB_W} \|\Pi^1 - \Pi^2\|_{\bsigma; \mfK + B(R_{\btau})} + \|\Gamma^{1} - \Gamma^2\|_{\bsigma, \mfK + B(R_{\btau})} ,
	  				\end{align*}
	  		where the implicit constant can be chosen 
to depend polynomially on
					\begin{align*}
						\max_{\bsigma \in \mcB_{\btau}}\max_i 
						\big( \|\Pi^i\|_{\bsigma, \mfK+B(R_{\btau})} + \|\Gamma^i\|_{\bsigma; \mfK+ B(R_{\btau})} + \|A^i\|_{\mcA_+^\mco} \big) .
					\end{align*}
	  			\end{lemma}
	  			\begin{proof}
	  				We prove both claims by Noetherian induction over $\mcB_W$ with respect to $\prec$. As usual we split into the cases where $\btau \in \{X^k, \Xi_\mfl\}$, where $\btau$ is planted and where $\btau$ is a product of planted trees. 
	  				
	  				In the first case, if $\btau = \Xi_\mfl$ it follows from the fact that $\Gamma_{xy}^i \in \mcG$ that $\Gamma_{xy}^i$ acts as the identity on $V^{\proj \sttau}$. In the case $\btau = X^k$, it follows from the fact that $\Gamma_{xy}^i \in \mcG$ and $\Pi_x \Gamma_{xy} = \Pi_y$ that $\Gamma_{xy}^i X^k = (X + x - y)^k$. Both claims for such $\btau$ then follow immediately. It remains to consider the cases where $\btau$ is planted or is a product of planted trees. Since $\mcQ_{\alpha} \Gamma_{xy}^1$ factors the multilinear, continuous and $\btau$-symmetric family $\mcQ_{\alpha} \Gamma_{xy}^1 \bigotimes_{\sttau}^\tau$, for the first claim it suffices to bound $\mcQ_{\alpha} \Gamma_{xy}^1 \bigotimes_{\sttau}^\tau$ as a multilinear map. A similar claim also holds for $\Gamma^1 - \Gamma^2$. 
	  				
	  				In the case where $\btau = \prod_{i=0}^n \btau_i$ is a product of planted trees, we use the definition of $\star$ to write
	  				\begin{align*}
	  					\Big\| \Gamma_{xy}^1 \bigotimes_\sttau^\tau (\zeta_e) \Big\|_{\alpha} 
	  					& = \Big \| \Gamma_{xy}^1 \star^{\otimes n} \Big (\bigotimes_{\sttaun{i}}^{\tau_i} (\zeta_e) \Big )_{i=0}^n \Big \|_{\alpha}
	  					= \Big \| \star^{\otimes n} \Gamma_{xy}^1 \Big (\bigotimes_{\sttaun{i}}^{\tau_i} (\zeta_e) \Big )_{i=0}^n \Big \|_{\alpha} \le \sum_{\alpha = \alpha_0 + \dots + \alpha_n} \prod_{i=0}^n \Big \| \Gamma_{xy}^1\bigotimes_{\sttaun{i}}^{\tau_i} (\zeta_e) \Big \|_{\alpha_i}
	  				\end{align*}
	  				where to obtain the final inequality we used the fact that $\star$ respects the grading and is a norm-$1$ bilinear map. Since each $\btau_i$ satisfies $\btau_i \prec \btau$, in this case we may set $\mcB_{\btau} = \{\btau_i: i = 0, \dots, n\}$ so that the inequality above immediately implies the desired bound.		
	  				For the difference $\Gamma^1 - \Gamma^2$ one argues similarly.
	  				It remains to consider the case where $\btau$ is a planted tree. As above, we may appeal to the universal property to reduce to obtaining bounds for symmetric elementary tensors. Since $\bigotimes_{\langle \mcI_\mfl^k \btau \rangle}^{\mcI_\mfl^k \tau} (\zeta_e)_{e \in E(\mcI_\mfl^k \tau)} = \mcI^{\zeta; k} \bigotimes_\sttau^\tau (\zeta_e)_{e \in E(\tau)}$ where $\zeta$ is the entry in $(\zeta_e)_{e \in E(\mcI_\mfl^k \tau)}$ associated to the trunk, this means that the bound for such a tree is essentially the situation of the Extension Theorem\footnote{We note that \cite[Theorem 5.14]{Hai14} assumes to be given a kernel of infinite order however examining the proof shows that on the sector $W$, $\mco > \mathrm{ord}(W)$ is all that is actually used.} given in \cite[Theorem 5.4]{Hai14} so that we can take $\mcB_{\mcI_\mfl^k \btau}$ to be the basis of trees for the smallest sector containing $\btau$ in $\mcT(R)$.  Whilst \cite[Theorem 5.14]{Hai14} does not cover the case of two models which are admissible with respect to different kernels, the result in this case is a minor adaptation which contributes the term $\|\mcA^1 - \mcA^2\|_{\mcA_+^\mco}$ above. Since the adaptation is straightforward, for brevity we omit the details.
	  				\end{proof}
	  			\begin{lemma} \label{l:model_from_pm}
	  			    	Suppose that $P$ is a preparation map defined on a $P$-good sector $W$, that $A \in \mcA_+^\mco$ with $\mco > \mathrm{ord}(W)$ and that $\xi \in \mcA_-^\infty$. If $(\Pi^P, \Pi^{P, \times}, \Gamma^P)$ is the unique triple associated to $(P, A, \xi)$ as in Proposition~\ref{prop:mod_prep_1} then $(\Pi^P, \Gamma^P)$ defined above is an admissible model on $W$ and $(\Pi^{P, \times}, \Gamma^P)$ is a model on $W$.
	  			\end{lemma}
	  			\begin{proof}
	  			It remains only to prove the analytic bounds. We start by considering the bounds for $\Pi^P$. Here, we will prove by Noetherian induction over $\mcB_W$ with respect to $\prec$ that the stronger bound
	  			$$|\Pi_x^P v(y)| + |\Pi_x^{P, \times} v(y)| \le C_\tau \|v\|_{V^{\proj \sttau}} |y-x|_{\mfs}^{|\boldsymbol{\tau}|_\mfs}$$
	  			holds
	  			uniformly over $v \in V^{\proj \sttau}$, locally uniformly in $x$ and uniformly over $y$ such that $|y-x|_{\mfs} \le 1$ for some constant $C_\tau$. 
	  			
	  			As usual, we divide into the cases where $\boldsymbol{\tau} \in \{\Xi, X^k\}$, $\btau$ is planted and where $\btau$ is a product of planted trees. We note that the case where $\boldsymbol{\tau} = X^k$ coincides with the scalar valued case and the case where $\btau = \Xi_\mfl$ follows immediately from the fact that for each compact set $\mfK$, $\sup_{y \in \mfK} |N_\mfl(v)(y)| \lesssim \|v\|_{V^{\proj \sttau}}$ by continuity of the map $N_\mfl: V_\mfl \to C^\infty(\mbR^d)$. Therefore, we need only consider the cases of planted trees and products of planted trees in what follows.
	  			
	  			For the case where $\btau$ is planted, $\Pi_x^P$ and $\Pi_x^{P, \times}$ coincide on $V^{\proj \sttau}$ so it suffices to consider only the first of them.  We want to simply appeal to Taylor's Theorem. To get the correct dependence on $\|v\|_{V^{\proj \sttau}}$ we recall that by construction, 
	  			for all $x \in \mathbb{R}^d$, $\boldsymbol{\tau} \in \mathcal{T} ( R )$ 
	  			the map $\Pi_x^P \boldsymbol{\tau} \colon V^{\proj \sttau} \to C^{\infty} ( \mathbb{R}^d )$ is linear and continuous,
	  			that is, 
	  			for all $r \in \mathbb{N}$, all compacts $\mathfrak{K} \subset \mathbb{R}^d$,
	  			and $v \in V^{\proj \sttau}$
	  				\begin{align*}
	  						\| \Pi_x^P v \|_{C^{r} ( \mfK )} 
	  						\lesssim_{r, \mfK, \boldsymbol{\tau}} \| v \|_{V^{\proj \sttau}} .
	  					\end{align*}
	  			Therefore it is indeed the case that the desired bound in this case follows from Taylor's Theorem since the definition of an admissible model involves exactly the right Taylor jet for trees of this type. Therefore, it remains only to treat the case where $\boldsymbol{\tau} = \prod_{i=0}^n \boldsymbol{\tau}_i$ where $\boldsymbol{\tau}_0 = X^k$ and for $i \ge 1$, $\boldsymbol{\tau}_i = \mcI_\mfl^k \boldsymbol{\sigma}_i$. We first consider $\Pi_x^{P, \times}$ and we note that if $\mathrm{ev}_y: C^\infty(\mbR^d) \to \mbR$ is the evaluation at $y$ then $\mathrm{ev}_y \circ \Pi_x^{P, \times} : V^{\proj \sttau} \to \mbR$ factors the multilinear, continuous and $\sttau$-symmetric family $\Big ( \mathrm{ev}_y \circ \Pi_x^{P, \times} \circ \bigotimes_\sttau^\tau \Big )_{\tau \in \btau}$. Therefore, by appealing to the statement regarding norms in Definition~\ref{def:tensor_product} it suffices to consider the case where $v$ is an elementary symmetric tensor. Therefore we assume that $v = \otimes_{\sttau}^\tau (\zeta_e)_{e \in E(\tau)}$. In this case, we apply multiplicativity and the induction hypothesis to write
	  			\begin{align*}
	  					|\Pi_x^{P, \times} \otimes_\sttau^\tau (\zeta_e)_{e \in E(\tau)} (y) | & = \left | \prod_{i=1}^n \Pi_x^{P, \times} \otimes_{\sttaun{i}}^{\tau_i} (\zeta_e)_{E(\tau_i)}(y) \right | .
	  			\end{align*}
	  			Since each $\boldsymbol{\tau}_i$ satisfies $\boldsymbol{\tau}_i \prec \boldsymbol{\tau}$, we obtain 
	  			\begin{align*}
	  						|\Pi_x^{P, \times} \otimes_\sttau^\tau (\zeta_e)_{e \in E(\tau)} (y) |  \lesssim \prod_{i=1}^n |y-x|_{\mfs}^{|\boldsymbol{\tau}_i|} \| \otimes_{\sttaun{i}}^{\tau_i} (\zeta_e)_{e \in E(\tau_i)} \|_{V^{\proj \sttaun{i}}}  \lesssim  |y-x|_{\mfs}^{|\boldsymbol{\tau}|} \prod_{e \in E(\tau)} \|\zeta_e\|_{V_{\mfl(e)}}
	  			\end{align*}
	  			where the last inequality follows since $\otimes_{\sttaun{i}}^{\tau_i}$ is a norm-$1$ multilinear map.
	  			
	  			For the corresponding bound for $\Pi_x^P v(y)$ note that
	  			\begin{align*}
	  				\Pi_x^P v(y) = \Pi_x^{P, \times} v(y) + \Pi_x^{P, \times} (P- \id) v(y).
	  			\end{align*}
	  			Since we already obtained the required bound for $\Pi_x^{P, \times}$, it suffices to consider the latter term on the right hand side. For this, we note that by the definition of a preparation map, we can write
	  			\begin{align*}
	  				|\Pi_x^{P, \times} (P- \id) v(y)| \le \sum_{\bsigma} |\Pi_x^{P, \times} \mcQ_{\bsigma} P v (y)|
	  			\end{align*}
	  			where the sum runs over those $\bsigma$ such that $\mcQ_{\bsigma} (P- \id) v \ne 0$. Then since $P$ is continuous and such $\bsigma$ necessarily satisfy $\bsigma \prec \btau$, we obtain by induction hypothesis that
	  			\begin{align*}
	  				|\Pi_x^{P, \times} (P - \id)v(y)| \lesssim \sum_{\bsigma} \|v\|_{V^{\proj \sttau}} |y-x|_{\mfs}^{|\bsigma|_\mfs} \lesssim \|v\|_{V^{\proj \sttau}} |y-x|_{\mfs}^{|\btau|_\mfs}
	  			\end{align*}
	  			where the last inequality follows from the assumption that $|y-x|_{\mfs} \le 1$. 
	  			
	  			It now remains to prove the corresponding bound for $\Gamma_{xy}^P$. However these follow by a straightforward (Noetherian) induction from the corresponding bounds on $\Pi^P$ by use of Lemma~\ref{lem:Gam_from_Pi}.
\end{proof}

\section{Stochastic Estimates for Models on \texorpdfstring{$\mcT^\Ban$}{T Ban}}\label{s:Kolmogorov}

In this section, our goal is to show that stochastic estimates for a particular choice of renormalised model on $\mcT^\Ban$ (the BPHZ model) can be lifted from versions of the same estimate for suitably constructed scalar valued regularity structures. Since this requires a more precise choice of kernel and noise assignment we specialise here to the setting of Theorem~\ref{theo:main_result_1}.
From this point onwards
 we work only with scalar noise assignments, i.e.\ we make the assumption that $V_\mfl = \mbR$ for each $\mfl \in \mfL_-$. We will often view $\otimes_{\sttau}^\tau$ as a map with domain $\prod_{e \in E_+(\tau)} V_\mfl$ by abuse of notation by writing $\otimes_{\sttau}^\tau (\zeta_e)_{e \in E_+(\tau)} \coloneq \otimes_{\sttau}^\tau (\bar{\zeta}_e)_{e \in E(\tau)}$ where 
$$\bar{\zeta}_e = \begin{cases} \zeta_e, \qquad e \in E_+(\tau), \\ 1 \qquad \text{otherwise.} \end{cases}$$
In particular, this implies that the notions of noise assignment $\mcA_-$ on $\mcT^\Ban$ and $\mcA_-^{\eq}$ on $\mcT^\eq$ coincide. 
We expect that a similar strategy with suitable higher dimensional noise assignments could be implemented in future work.

\subsection{Stochastic Estimates for the BPHZ Model on \texorpdfstring{$\mcT^\mathrm{Ban}$}{T Ban} -- Annealed Estimates} \label{ss:annealed}

We begin by defining a particular class of preparation maps that will include the BPHZ choice. As in the scalar valued setting (cf \cite{Bru18, BB21}), this is done by introducing an analogue of the extraction/contraction map used in \cite{BHZ19} that only extracts subtrees at the root. Since by this stage we have made many similar applications of our universal property, we omit the proof that this map is well-defined.

\begin{definition} \label{d:delta_r_-}
	We define $\Delta_r^-: V_\mcT \to V_\mcT \proj V_\mcT$ to be the unique
	continuous
	linear map such that for each $\tau \in \boldsymbol{\tau} \in \mcT$ the following diagram commutes 
	\begin{center}
		\begin{tikzcd}
			V^{\times \tau} \arrow[r, "\otimes_{\sttau}^\tau"] \arrow[dr, "\mcL_\tau^{(2)} (\underline{\Delta}_r^- \tau)"'] & V^{\proj \sttau} \arrow[d, dotted, "\Delta_r^-"]
			\\
			& V_\mcT \proj V_\mcT
		\end{tikzcd}
	\end{center}
	where $\underline{\Delta}_r^-$ is the coproduct defined at the level of combinatorial trees by
	\begin{equ}
		\underline{\Delta}_r^- (\tau, \mfn, \mfe) = \sum_{C \in \mathbb{C}[\tau]} \sum_{n_{\ng C}, \varepsilon_C} \frac{1}{\varepsilon_C} { \mfn \choose n_{\ng C}} P_{\underline{\mcT}_-(R)} (\tau_{\ng C}, n_{\ng C} + \pi \varepsilon_C, \mfe) \otimes (\tau/ \tau_{\ng C}, [n-n_{\ng C}], \mfe + \varepsilon_C)
	\end{equ}
	where we adopt the same notations as in the definition of $\underline{\Delta}$ and $\underline{\Delta}^+$ and further have defined $\underline{\mcT}_-(R) = \bigcup_{\boldsymbol{\tau} \in \mcT(R)} \boldsymbol{\tau}$.
\end{definition}

\begin{remark} 
	The notation $\underline{\Delta}_r^-$ is reserved for the coproduct acting on drawing of trees.
	Note however that by taking the `trivial' Banach space assignment $V = ( \mathbb{R} )_{\mfl \in \mfL}$, the above yields the version of this coproduct acting on isomorphism classes of trees, which we will also henceforth denote by $\Delta_r^- \colon \langle \mcT ( R) \rangle \to  \langle \mcT_-( R) \rangle \otimes \langle \mcT ( R) \rangle$.
\end{remark}

The BPHZ model on $\mcT^{\Ban}$ will be an instance of a preparation map of the form 
$(\ell \otimes \id ) \Delta_r^-$, with $\ell$ a suitable functional on $\mcT_- ( R )$.
However,
if we wish to appeal to the results of
Section~\ref{ss: alg renorm},
we have to be careful about the following points.
On the one hand, the choice of preparation map will typically depend on a noise assignment which in turn in practice depends on the choice of ultraviolet cut-off at scale $\varepsilon$. In order to make sense of convergence of the corresponding models, it is necessary to construct them on the same sector thus necessitating the construction of a sector that is simultaneously $P^\varepsilon$-good for every $\varepsilon$.
A second technical issue comes from the fact that the BPHZ preparation map $\hat{P}$ (at fixed $\varepsilon$) is not given a priori but rather is constructed in tandem with the BPHZ model $\hat{\Pi}$. 

This leads us to introduce the notion of the \emph{history} of a set of trees in $\mcT(R)$. The definition is designed to provide a sector that is $P$-good for every $P$ of the form $(\ell \otimes \id)\Delta_r^-$ such that the sector is also closed under the operation $(\id \otimes \mathbf{1})\Delta_r^-$ where $\mathbf{1}$ is the constant map. The latter point is required for the intertwined construction of the BPHZ model and preparation map. 

\begin{definition} \label{d:history}
	Given a set of trees $\mcS \subset \mcT(R)$, we denote by $\mathrm{Hist}(\mcS)$ the smallest set of trees such that the following properties hold:
	\begin{enumerate}
		\item $\mcS \subset \mathrm{Hist} (\mcS)$.
		\item $\mathrm{Hist}(\mcS)$ is the generating set of a good sector in $\mcT^\mathrm{Ban}$.
		\item For all $\ell \in \langle \mcT ( R ) \rangle^*$, $\big [(\id \otimes \ell) \Delta_r^- ( \mathrm{Hist}(\mcS)) \big ] \cup \big [ (\ell \otimes \id ) \Delta_r^- ( \mathrm{Hist}(\mcS)) \big ] \subset \langle \mathrm{Hist}(\mcS) \rangle$.
	\end{enumerate}
\end{definition}

\begin{definition}\label{def:historic}
	We say that a set $\mcS \subset \mcT(R)$ is historic if $\mathrm{Hist} (\mcS) = \mcS$.
	Analogously, we say that a good sector $W$ of $\mcT^{\Ban}$ 
	is historic if $\mathrm{Hist} ( \mcB_W) = \mcB_W$.
\end{definition}

\begin{lemma} \label{l:finite}
	If $\mcS \subset \mcT ( R )$ is finite then so is $\mathrm{Hist}(\mcS)$. 
\end{lemma}

\begin{proof}
	$\mathrm{Hist}(\mcS)$ can be constructed explicitly as
	\begin{align} \label{eq:constr_hist}
		\mathrm{Hist}(\mcS) = \bigcup_{n \ge 0} \mcH_n(\mcS)
	\end{align}
	where $\mcH_0(\mcS) = \mcS$ and for $k \ge 0$, 
	\begin{align*}
		\mcH_{3k+1} (\mcS) &= \{ \boldsymbol{\sigma} \in \mcT(R): \exists \Gamma \in \mcG, \boldsymbol{\tau} \in \mcH_{3k}(\mcS) \text{ such that } \mcQ_{\boldsymbol{\sigma}} \Gamma \boldsymbol{\tau} \ne 0 \} \\ 
		\mcH_{3k+2} (\mcS) &= \{ \boldsymbol{\sigma}  \in \mcT(R): \exists \ell \in \langle \mcT ( R ) \rangle^* \text{ and } \boldsymbol{\tau} \in \mcH_{3k+1}(\mcS) \text{ such that } \mcQ_{\boldsymbol{\sigma}} (\ell \otimes \id ) \Delta_r^- \boldsymbol{\tau} \ne 0 \} \\ 
		& \quad \cup \{ \boldsymbol{\sigma} \in \mcT(R): \exists \ell \in \langle \mcT ( R ) \rangle^* \text{ and }  \boldsymbol{\tau} \in \mcH_{3k+1}(\mcS) \text{ such that } \mcQ_{\boldsymbol{\sigma}} (\id \otimes \ell) \Delta_r^- \boldsymbol{\tau} \ne 0 \} \\
		\mcH_{3k+3} (\mcS) &= \{  \boldsymbol{\sigma}  \in \mcT(R): \exists \mfl, k \text{ such that } \mcI_\mfl^k \boldsymbol{\sigma} \in \mcH_{3k+2}(\mcS) \} 
		\\ & \quad \cup \{ \boldsymbol{\sigma} \in \mcT(R): \exists \btau  \text{ such that }  \boldsymbol{\tau} \boldsymbol{\sigma} \in \mcH_{3k+2}(\mcS) \} .
	\end{align*} 
	We first claim that the map $k \mapsto \mcH_k(\mcS)$ is eventually constant. If not, there would exist an infinite sequence $\btau_j \in \mcH_{k_{j+1}}(\mcS) \setminus \mcH_{k_j}(\mcS)$ for some strictly increasing sequence $j \mapsto k_j$. It is then easy to check by division of cases that the map $j \mapsto (|\btau_j|_\#, |\btau_j|_\mfs)$ is a strictly decreasing sequence in $\mathbb{N}^2$ with the lexicographic ordering. But no such sequence exists, leading to a contradiction. The stated result then follows by noting that a straightforward induction (dividing into cases) implies that if $\mcS$ is finite then so is $\mcH_{k+1}(\mcS) \setminus \mcH_k(\mcS)$. 
\end{proof}

\begin{lemma} \label{l:pfl}
	Let $W$ be a historic sector on $\mcT^{\Ban}$.
	Suppose that $\ell \in ( V_{\mcT_-(R)} \cap W )^*$ is such that $\ell \restriction_{V^{\proj \langle \boldsymbol{\tau} \rangle}} = 0$ for every $\boldsymbol{\tau} \in \mcT(R)$ of the form $\mcI \boldsymbol{\sigma}, X^k \boldsymbol{\sigma}$. 
	Then $P_{\ell} \coloneqq (\ell \otimes \id) \Delta_r^-$ is a preparation map 
	on $W$.
	Furthermore, the sector $W$ is $P_{\ell}$-good as per Definition~\ref{def:pgood}.
\end{lemma}
\begin{proof}
	We start with the claim that $P_{\ell}$ is a preparation map on $W$.
	For brevity, 
	we demonstrate only that $P_\ell$ commutes with the structure group. It suffices to show that 
	$(\Delta_r^- \otimes \id) \Delta^+ = (\id \otimes \Delta^+) \Delta_r^-.$
	This follows from the purely combinatorial analogue of this result (which is the contents of \cite[Proposition 4.6]{Bru18}) in the same way as in the proof of coassociativity for $\Delta^+$ so we omit the details.
	Finally, 
	the fact that $W$ is $P_{\ell}$-good 
	directly follows from the assumption that $W$ is historic.    
\end{proof}

We are now ready to introduce the partial order with respect to which we will inductively construct the BPHZ model on $\mcT^{\Ban}$.

\begin{definition}\label{d:age}
	We define the age of a tree $\boldsymbol{\tau} \in \mcT ( R )$ to be the length of the longest chain in 
	the (finite) poset 
	$(\mathrm{Hist} ( \lbrace \boldsymbol{\tau} \rbrace ), \prec)$.
	More precisely, 
	\begin{align*}
		\mathrm{Age} ( \boldsymbol{\tau} ) = \max \big\lbrace n \in \mathbb{N} : \exists \boldsymbol{\sigma}_1, \cdots , \boldsymbol{\sigma}_n  \in \mathrm{Hist} ( \lbrace \boldsymbol{\tau} \rbrace ) \text{ with } \boldsymbol{\sigma}_1 \prec \cdots \prec \boldsymbol{\sigma}_n \big\rbrace .
	\end{align*}
\end{definition}

This notion comes with the following convenient inductive properties.

\begin{lemma} \label{l:age_ind_prop}
	Let $\boldsymbol{\tau} \in \mcT ( R )$.
	For all $\boldsymbol{\sigma} \in \mathrm{Hist} ( \lbrace \boldsymbol{\tau} \rbrace ) \setminus \{\btau\}$, $\boldsymbol{\sigma} \prec \boldsymbol{\tau}$ and therefore  $\mathrm{Age} ( \boldsymbol{\sigma} ) < \mathrm{Age} ( \boldsymbol{\tau} )$.
	In particular, the following properties hold. 
	\begin{enumerate}
		\item \label{it:ap} If $\boldsymbol{\tau} = \prod_{i=0}^n \boldsymbol{\tau}_{i}$ then $\mathrm{Age} ( \boldsymbol{\tau}_{i} ) < \mathrm{Age} ( \boldsymbol{\tau} )$ for all $i \in \lbrace 0, \cdots , n \rbrace$.
		
		\item \label{it:ai} If $\boldsymbol{\tau} = \mcI_\mfl^k \boldsymbol{\sigma}$ for some $\mfl \in \mfL_+$, $k \in \mathbb{N}^d$, then 
		$\mathrm{Age} ( \boldsymbol{\sigma} ) < \mathrm{Age} ( \boldsymbol{\tau} )$.
		
		\item \label{it:ad} If $\bsigma$ is of the form $\bsigma= (\ell \otimes \id ) \mathring{\Delta}_r^- \btau$ 
		or
		$\bsigma= (\id \otimes \ell) \mathring{\Delta}_r^- \btau$ for some $\ell \in \langle \mcT ( R ) \rangle^*$, 
		then $\mathrm{Age} ( \boldsymbol{\sigma} ) < \mathrm{Age} ( \boldsymbol{\tau} )$.
		Here we have written 
		$\mathring{\Delta}_r^- = \Delta_r^- - (\id \otimes \mathbf{1}) - (\mathbf{1} \otimes \id)$.
	\end{enumerate}
\end{lemma}

\begin{proof}
	The first claim follows by the construction \eqref{eq:constr_hist} of $\mathrm{Hist} ( \lbrace \boldsymbol{\tau} \rbrace )$, since by case separation one readily establishes that if $\boldsymbol{\sigma} \in \mcH_{n+1} \setminus \mcH_n$ then\footnote{We note that here we use Assumption~\ref{ass:reg} to deal with the case of integration.} $\boldsymbol{\sigma} \prec \mcH_n$.
	To obtain the second claim, it suffices to note that by the first claim, if $\boldsymbol{\sigma}_1 \prec \cdots \prec \boldsymbol{\sigma}_n$ is any chain in $\mathrm{Hist} ( \lbrace \boldsymbol{\sigma} \rbrace )$, then $\boldsymbol{\sigma}_1 \prec \cdots \prec \boldsymbol{\sigma}_n \prec \boldsymbol{\tau}$ is a strictly longer chain in $\mathrm{Hist} ( \lbrace \boldsymbol{\tau} \rbrace )$. Finally, the items~\ref{it:ap} to \ref{it:ad} are an immediate consequence of the second claim and the Definition~\ref{d:history} of $\mathrm{Hist} ( \lbrace \boldsymbol{\tau} \rbrace )$.
\end{proof}

Given a historic sector $W$ of $\mcT^{\Ban}$, 
it will be convenient to `filter' it with respect to the age as follows.

\begin{definition}\label{def:W_n}
	Let $\mcB_W$ be the set of generating trees of a historic sector $W$ of $\mcT^{\Ban}$. We decompose
	$W = \bigcup_{n \in \mathbb{N}} W_n$
	by setting $W_n$ to be the historic sector with generating set $\mcB_{W_n} \coloneqq \big\lbrace \boldsymbol{\tau} \in \mcB_W : \mathrm{Age} ( \boldsymbol{\tau} ) \leq n \big\rbrace$.
\end{definition}

Thus far, we have worked in a purely deterministic setting. We now turn to defining the BPHZ model associated to a random smooth noise assignment. 
\begin{definition} \label{def:BPHZ}
	Let $W$ be a historic sector, $A \in \mcA_+^{\mco}$ be a kernel assignment 
	and $\xi: \Omega \to  \mcA_-^{\eq, \infty}$ be a random smooth noise assignment.
	We say that a deterministic functional $\ell \in ( V_{\mcT_-(R)} \cap W )^*$ 
	is the BPHZ functional associated to $(A, \xi)$ on $W$ 
	if $\ell \restriction_{V^{\proj \langle \boldsymbol{\tau} \rangle}} = 0$ for every $\boldsymbol{\tau} \in \mcT_-(R) \cap \mcB_W$ of the form $\mcI \boldsymbol{\sigma}, X^k \boldsymbol{\sigma}$, 
	and if furthermore $\mbE[ \boldsymbol{\Pi}^{P_\ell} v (0) ] = 0$ for all $v\in V^{\proj \sttau}$ with $\boldsymbol{\tau} \in \mcT_- ( R ) \cap \mcB_W$.
\end{definition}

Since this definition
requires that the expectation actually makes sense, we need to show that integrability of the driving noise propagates to integrability of $\boldsymbol{\Pi}^{P}$. Since $\boldsymbol{\Pi}^P$ was constructed realisation-by-realisation, it is useful to have a result that connects realisation-by-realisation constructions using the universal property to their $L^p(\Omega)$ analogues (where $(\Omega, \mathcal{F}, \mathbb{P})$ will always denote the underlying probability space).

\begin{lemma}\label{lem:UP_in_Lp}
	Suppose that $X$ is a Banach space, that $\mcs$ is a symmetric set and that for each $\omega \in \Omega$, $(f^a(\omega, \cdot): V^{\times a} \to X)_{a \in \Ob(\mcs)}$ is a $\mcs$-symmetric family of continuous and multilinear maps. Suppose additionally that for each $a$, $f^a: V^{\times a} \to L^p(\Omega; X)$ is continuous. For each $\omega \in \Omega$, let $f(\omega, \cdot)$ be the unique map such that for each $a \in \Ob(\mcs)$, the following diagram commutes.
	\begin{center}
		\begin{tikzcd}
			V^{\times a} \arrow[r, "\otimes_{\mcs}^a"] \arrow[dr, "{f^a(\omega, \cdot)}"'] & V^{\proj \mcs} \arrow[d, dotted, "{f(\omega, \cdot)}"]
			\\
			& X
		\end{tikzcd}
	\end{center}
	Then  $f : V^{\proj \mcs} \to L^p(\Omega; X)$ continuously and for each $a \in \Ob(\mcs)$ the following diagram commutes.
	\begin{center}
		\begin{tikzcd}
			V^{\times a} \arrow[r, "\otimes_{\mcs}^a"] \arrow[dr, "{f^a}"'] & V^{\proj \mcs} \arrow[d, dotted, "{f}"]
			\\
			& L^p(\Omega; X)
		\end{tikzcd}
	\end{center}
\end{lemma}
\begin{proof}
	By the universal property applied to the $\mcs$-symmetric family of continuous multilinear maps $(f^a)_{a \in \mcs}$, we see that there exists a unique continuous map $g: V^{\proj \mcs} \to L^p(\Omega; X)$ such that the diagram
	\begin{center}
		\begin{tikzcd}
			V^{\times a} \arrow[r, "\otimes_{\mcs}^a"] \arrow[dr, "{f^a}"'] & V^{\proj \mcs} \arrow[d, dotted, "{g}"]
			\\
			& L^p(\Omega; X)
		\end{tikzcd}
	\end{center}
	commutes. We simply need to show that $f(v) = g(v)$ almost surely for each $v \in V^{\proj \mcs}$. It follows from the definitions that for each $(\zeta_e)_{e \in \mcs}$, we have that $g(\otimes_\mcs^a (\zeta_e))(\omega) = f(\omega,  \otimes_\mcs^a (\zeta_e))$ for almost all $\omega \in \Omega$. Then by linearity, we see that for each $v \in \operatorname{span} \operatorname{Ran} (\otimes_\mcs^a)$, we have that $g(v)(\omega) = f(\omega, v)$ for almost all $\omega \in \Omega$ (where we emphasise that the nullset is allowed to depend on $v$). Finally, for generic $v \in V^{\proj \sttau}$, we appeal to density of $\operatorname{span} \operatorname{Ran} (\otimes_\mcs^a)$ to find a sequence $v_n  \in \operatorname{span} \operatorname{Ran} (\otimes_\mcs^a)$ such that $v_n \to v$ in $V^{\proj \mcs}$. It then follows that $g(v_n) \to g(v)$ in $L^p(\Omega; X)$. By passing to a subsequence and relabelling we may assume that this convergence holds almost surely. We also know that for each fixed $\omega \in \Omega$, $f(\omega, v_n) \to f(\omega, v)$. Since $g(v_n) = f(\cdot, v_n)$ almost surely, we conclude that also $g(v) = f(\cdot, v)$ almost surely, as was required.
\end{proof}

We are now ready to show existence and the uniqueness of the BPHZ preparation map on a historic sector. 
\begin{prop}\label{lem: BPHZ exists}
	Let $W$ be a historic sector.
	Fix a kernel assignment $A \in \mcA_+^{\mco}$ with $\mco > \mathrm{ord} ( W )$.
	Fix a random smooth noise assignment $\xi: \Omega \to \mcA_-^{\eq; \infty}$ such that for all $z \in \mathbb{R}^d$ 
	and $\mfl \in \mfL_-$, the random variable 
	$\xi_{\mfl} ( z )$
	has moments of all orders and such that the field $\xi_{\mfl}$ is translation invariant in law. 
	Then there exists a unique 
	BPHZ functional $\hat{\ell}$ associated to $(A, \xi)$ on $W$.
\end{prop}
We will call the preparation map
given by $\hat{P} = (\hat{\ell} \otimes \id) \Delta_r^-$
the BPHZ preparation map associated to $(A, \xi)$ on $W$ 
and the corresponding model $(\hat{\Pi}, \hat{\Gamma})$ the BPHZ model on $W$.
\begin{proof}
	We let $W_n$ be as in Definition~\ref{def:W_n}. We prove by induction in $n \in \mathbb{N}$ that
	\begin{enumerate}
		\item \label{it:bphz1} there is a unique BPHZ functional $\ell_n$ on $W_n$,
		\item \label{it:bphz2} the map $v \mapsto D^k \boldsymbol{\Pi}^{P_{\ell_{n}}, \times} v (0)$ is continuous from $V^{\proj \sttau}$ to $L^p(d\mathbb{P})$ for every $p \ge 1$, $k \in \mathbb{N}^d$ and $\boldsymbol{\tau} \in \mcB_{W_n}$,
		\item \label{it:bphz3} the map $v \mapsto D^k \boldsymbol{\Pi}^{P_{\ell_{n}}} v (0)$ is continuous from $V^{\proj \sttau}$ to $L^p(d\mathbb{P})$ for every $p \ge 1$, $k \in \mathbb{N}^d$ and $\boldsymbol{\tau} \in \mcB_{W_n}$
	\end{enumerate}
	where $(\boldsymbol{\Pi}^{P_{\ell_n}}, \boldsymbol{\Pi}^{P_{\ell_n}, \times})$ is the pair associated to $(P_{\ell_n},A,\xi)$ on $W_{n}$ as defined in
	Definition~\ref{def:pre_mod}.
	
	In the base case $n=1$, we note that $W_1 \subseteq V^{\proj \langle \mathbf{1} \rangle} \simeq \mbR$ so the result is clear. 
	We turn to the induction step, assuming the statements for $k \le n$ and establishing them for $n+1$. 
	We first establish the second of the induction hypotheses by noting that the definition of $\boldsymbol{\Pi}^{P_{\ell_{n+1}}, \times}$ on $W_{n+1}$ given in Definition~\ref{def:pre_mod} depends only on $\ell_n$ so that it remains to prove the continuity property.
	Let $\boldsymbol{\tau} \in \mcB_{W_{n+1}}$.
	If $\boldsymbol{\tau}$ is a noise or a polynomial, the result is clear, so that we may assume that $\btau$ is planted or a product of planted trees.
	Furthermore, we note that for each fixed $\omega \in \Omega$, since
	evaluation at $0$ is a continuous map from $C^\infty(\mathbb{R})$ to $\mathbb{R}$, we have that 
	$\boldsymbol{\Pi}^{M_{\ell_{n+1}}, \times} \cdot (0)$ factors the $\sttau$-symmetric family of continuous multilinear maps $\boldsymbol{\Pi}^{P_{\ell_{n+1}}, \times} \otimes_\sttau^\tau \cdot (0)$. By Lemma~\ref{lem:UP_in_Lp} with $f^\tau = \boldsymbol{\Pi}^{P_{\ell_{n+1}}, \times} \otimes_{\sttau}^\tau \cdot (0)$ it thus suffices to  check that for each $\tau \in \boldsymbol{\tau}$, we have that $(\zeta_e)_{e \in E_+(\tau)} \mapsto \boldsymbol{\Pi}^{P_{\ell_{n+1}}, \times} \bigotimes_{\sttau}^\tau (\zeta_e)_{e \in E_+(\tau)} (0)$ is a continuous map from $V^{\times \tau}$ to $L^p(\Omega)$ (and by multilinearity it in turn suffices to check only boundedness).
	
	To check this boundedness, we divide as usual into cases.
	In the case where $\tau = \prod_{i = 1}^n \tau_i$ with $i > 1$ and $\tau_i \ne 1$,
	then we have $\boldsymbol{\tau_i} \in \mcB_{W_{n}}$ and we note that 
	\begin{equ}
		\boldsymbol{\Pi}^{P_{\ell_{n+1}}, \times} \bigotimes_\sttau^\tau (\zeta_e)_{e \in E_+(\tau)} (0) = \prod_{i=1}^n \boldsymbol{\Pi}^{P_{\ell_{n}}, \times} \bigotimes_{\langle \boldsymbol{\tau_i} \rangle}^{\tau_i} (\zeta_e)_{e \in E_+(\tau_i)}(0)
	\end{equ}
	so that the result follows from continuity of the product map $\prod_{i=1}^n L^{p_i} \to L^p$ for $p^{-1} = \sum_{i=1}^n p_i^{-1}$.
	
	It remains only to consider the case where $\tau$ is planted. In this case, 
	\begin{equs}
		\mathbb{E}^{\frac{1}{p}} \Big[ \Big| \boldsymbol{\Pi}^{P_{\ell_{n+1}}, \times} \bigotimes_{\langle \boldsymbol{\mcI_\mfl^k \tau} \rangle}^{\mcI_\mfl^k \tau} \left ( (\zeta_e)_{e \in E_+(\tau)}\sqcup \{\zeta\} \right ) (0) \Big|^p \Big] 
		& \le \int \big| D^k A_\mfl(\zeta)(-y) \big| \mathrm{d}y \, \, \mathbb{E}^{\frac{1}{p}} \Big[ \Big| \boldsymbol{\Pi}^{P_{\ell_{n}}} \bigotimes_{\sttau}^{\tau} (\zeta_e)_{e \in E_+(\tau)} (0) \Big|^p \Big]
	\end{equs}
	where we used Minkowski's integral inequality and translation invariance of $\boldsymbol{\Pi}^{P, \times} \bigotimes_{\sttau}^{\tau} (\zeta_e)_{e \in E_+(\tau)}$.
	The result then follows from the induction hypothesis and the fact that $\|D^k A_{\mfl}\|_{L^1} \lesssim \| A \|_{\mcA_+^{\mco}} < \infty$ by definition of a kernel assignment and Assumption~\ref{ass:reg}.
	
	Still in the induction step, we turn to the proof of item~\ref{it:bphz1}. 
	For that purpose, it suffices to see that $\ell_{n+1}$ is uniquely characterised on $V_{\mcT_-} \cap W_{n+1}$ by the expression
	\begin{equ}
		\ell_{n+1} (v) = - \mathbb{E}[\boldsymbol{\Pi}^{P_{\ell_n}, \times} v (0) ] - (\ell_n \otimes \mathbb{E}[\boldsymbol{\Pi}^{P_{\ell_n}, \times} \cdot (0) ]) \mathring{\Delta}_r^- (v) ,
	\end{equ}
	which defines a continuous linear form by item~\ref{it:bphz2} of the induction loop.
	Here we use the fact that $\mathring{\Delta}_r^-$ strictly reduces the age to appeal to the induction hypothesis.
	
	Finally, item~\ref{it:bphz3}
	follows from $\boldsymbol{\Pi}^{P_{\ell_{n+1}}} = \boldsymbol{\Pi}^{P_{\ell_{n+1}}, \times} (\ell_{n+1} \otimes \id) \Delta_r^-$
	and the fact that $(\ell_{n+1} \otimes \id) \Delta_r^-$ is continuous since $\ell_{n+1}$ is.
\end{proof}

Whilst the above result shows that the BPHZ lift of a smooth driving noise is well-defined in our infinite dimensional setting, it does not show the required stochastic estimates to obtain stability as a smoothing mechanism is removed.
Our goal for the remainder of this subsection is to obtain the annealed forms of such stochastic estimates for the BPHZ model on $\mcT^\mathrm{Ban}$ for an arbitrary Banach space assignment $V$ and kernel assignment $A$.

Our strategy to do this is to lift the statement from the corresponding scalar valued results given in any of \cite{CH16, HS23, BH23}. The key idea is to construct for each tree $\btau$ a scalar-valued regularity structures according to the recipe of \cite{BHZ19} which contains a sufficient amount of information to obtain the estimates on the range of $\otimes_{\btau}^\tau$. We now do this for fixed  $\tau \in \boldsymbol{\tau} \in \mcT(R)$. We define a new set of edge types 
$$\mfL(\tau) = \{(\mfl, e): e \in E_+(\tau), \mfl = \mfl_\tau(e)\}$$
equipped with the degree map induced by the degree map on $\mfL$
and rule
$$R_\tau(\mfl, e) = \big\{ \big((\alpha_1, k_1), \dots, (\alpha_n , k_n)\big): n \in \mbN, k_i \in \mbN^d, \alpha_i = (\mfl_i, e_i) \in \mfL(\tau), \big((\mfl_1, k_1), \dots,(\mfl_n, k_n) \big) \in R(\mfl)\big\}$$
which we note is automatically subcritical, normal and complete so long as the same was true of the rule $R$. We denote by $T^\tau$ the reduced regularity structure (in the sense of \cite{BHZ19}) corresponding to this type set/rule pair. This construction is designed in such a way that if $\sigma \in \boldsymbol{\sigma} \in \mcT(R)$ is such that $$\mfl(E(\sigma)) \subseteq \mfl(E(\tau))$$
then $T^\tau$ contains a distinct copy of $\sigma$ for each way of assigning an edge of $\tau$ of type $\mfl$ to each edge of $\sigma$ of type $\mfl$. In particular, if $\sigma \in \bsigma \in \mcT(R)$ is equipped with a $\tau$-identification $\iota: \sigma \to \tau$ then we obtain an element $\sigma^*$ of $T^\tau$ by assigning to each edge $e$ of $\sigma$ the new type $(\mfl(e), \iota(e))$ and then defining $\sigma^*$ to be the resulting isomorphism\footnote{We note that for elements of $T^\tau$ we have broken our convention that isomorphism classes receive bold symbols. This is justified by the fact that we will only be interested in trees in $T^\tau$ in which each edge receives a distinct type so that each drawing has no non-trivial symmetries.} class.

We recall that a kernel assignment
of order $\mco$
on $T^\tau$ is simply a choice of $A^\tau(\mfl,e)  \in \mcK_{\mco}^{|\mfl|}$ for each $(\mfl, e) \in \mfL_+(\tau)$. Given $(\zeta_e)_{e \in E_+(\tau)} \in V^{\times \tau}$ and a kernel assignment $A \in \mcA_{+}^{\mco}$, 
we obtain an induced kernel assignment of order $\mco$ and noise assignment
on $T^\tau$ by setting 
\begin{equ}\label{eq:induced assigment}
	A^\tau(\mfl, e) = A_\mfl(\zeta_e), \qquad \xi(\mfl, e)= \xi_\mfl\ .
\end{equ}

Let $W$ be a historic sector of $\mcT^{\Ban}$ of order $\mathrm{ord} (W) < \mco$ and $\tau \in \btau \in \mcB_{W}$.
For a fixed tuple $(\zeta_e)_{e \in E_+(\tau)}$, a kernel assignment $A \in \mcA_{+}^{\mco}$ 
and a 
random smooth noise assignment $\xi \colon \Omega \to \mcA_{-}^{\infty}$ satisfying the assumptions of Proposition~\ref{lem: BPHZ exists}, we define $(\hat{\Pi}^\tau, \hat{\Gamma}^\tau)$ to be the BPHZ model for the kernel assignment $A^\tau$ and noise $\xi$ on the regularity structure $T^\tau$ as defined in \cite[Theorem 6.18]{BHZ19}. 
The purpose of this construction is that the distribution $\hat{\Pi}_x \otimes_\sttau^\tau (\zeta_e)_{e \in E_+(\tau)}$ can be recovered from $\hat{\Pi}^\tau$.

Finally, we observe that for each tree $\sigma^*\in \mathrm{Hist}(\tau^*)$ and each drawing $\sigma \in \sigma^*$  the map 
\begin{equ}\label{eq: identification}
	E(\sigma)\to E(\tau),\qquad e\mapsto \pi_2(\mfl_{\sigma}(e))
\end{equ}
is injective. Therefore each drawing $\sigma \in \sigma^*$ induces a $\tau$-identified tree $(\bar{\sigma}, \iota_{\bar{\sigma}})$
where $\bar{\sigma}$ has the same edge set as $\sigma$ but equipped with the $\mfL$-valued type map $\pi_1\circ \mfl_{\sigma}$ and where the $\tau$-identification $\iota_{\bar{\sigma}}$ is the map given in \eqref{eq: identification}.

\begin{lemma}\label{lem:up_down}
	Let $W$ be a historic sector of $\mcT^{\Ban}$, let $\tau \in \btau \in \mcB_{W}$
	and let $(\zeta_e)_{e \in E_+(\tau)} \in V^{\times \tau}$.
	Fix a kernel assignment $A \in \mcA_+^{\mco}$ with $\mco > \mathrm{ord} ( W)$ and a random smooth noise assignment $\xi \colon \Omega \to \mcA_-^{\infty; \eq}$ on $\mcT^{\Ban}$ satisfying the assumptions of Proposition~\ref{lem: BPHZ exists}.
	Fix the induced noise assignment and kernel assignment on $T^\tau$ as in \eqref{eq:induced assigment}. 
	Then, the BPHZ model $(\hat{\Pi}^\tau, \hat{\Gamma}^\tau)$ on $\langle \mathrm{Hist}(\tau^*) \rangle \subset T^\tau$ is well defined and it holds that for any drawing  $\sigma$ of $\sigma^* \in \langle \mathrm{Hist}(\tau^*) \rangle $
	\begin{equ}\label{eq:local eq bphz}
		\hat{\Pi}_x^\tau \sigma^*=\hat{\Pi}_x \otimes_{\langle{\bar{\bsigma}}\rangle}^{\bar{\sigma}} (\zeta_{\iota_{\bar{\sigma}}(e)})_{e \in E_+(\bar{\sigma})} \ .
	\end{equ}	  
\end{lemma}

\begin{proof}
We prove by standard induction with respect to $\mathrm{Age}(\sigma^*)$ the following five claims.
\begin{enumerate}
	\item $\hat{\Pi}_x^{\times} \otimes_{\langle{\bar{\bsigma}}\rangle}^{\bar{\sigma}} (\zeta_{\iota_{\bar{\sigma}}(e)})_{e \in E_+(\bar{\sigma})} = \hat{\Pi}_x^{\tau, \times} \sigma^*$,
	\item $\hat{\boldsymbol{\Pi}}^{\times} \otimes_{\langle{\bar{\bsigma}}\rangle}^{\bar{\sigma}} (\zeta_{\iota_{\bar{\sigma}}(e)})_{e \in E_+(\bar{\sigma})}= \hat{\boldsymbol{\Pi}}^{\tau, \times} \sigma^*$,
	\item if $\sigma^*$ is of negative degree then $\hat{\ell} \big (\otimes_{\langle{\bar{\bsigma}}\rangle}^{\bar{\sigma}} (\zeta_{\iota_{\bar{\sigma}}(e)})_{e \in E_+(\bar{\sigma})}\big )= \hat{\ell}^\tau(\sigma^*)$ where $\hat{\ell}, \hat{\ell}^\tau$ denote the BPHZ 
	functionals
	on $\mcT^{\mathrm{Ban}}$ and $T^\tau$ respectively,
	\item $\hat{\boldsymbol{\Pi}} \otimes_{\langle{\bar{\bsigma}}\rangle}^{\bar{\sigma}} (\zeta_{\iota_{\bar{\sigma}}(e)})_{e \in E_+(\bar{\sigma})} = \hat{\boldsymbol{\Pi}}^{\tau} \sigma^*$,
	\item $\hat{\Pi}_x \otimes_{\langle{\bar{\bsigma}}\rangle}^{\bar{\sigma}} (\zeta_{\iota_{\bar{\sigma}}(e)})_{e \in E_+(\bar{\sigma})}= \hat{\Pi}_x^{\tau} \sigma^*$.
\end{enumerate}

When either  $\sigma = X^k$ for some $k \in \mathbb{N}^d$ (which includes the base case $\sigma=\mathbf{1}$) or $\sigma = \Xi_\mft$ for some $\mft \in \mfL_-$ all five parts of the induction hypothesis follows immediately from admissibility of all models involved and the fact that the 
BPHZ functionals
vanish on such trees by assumption.

In the case where $\sigma^*$ is planted, we can write $\sigma = \mcI_\mfl^k w$ for some $\mfl \in \mfL_+$, $k \in \mathbb{N}^d$. We note that the first and the second points coincide with the fourth and the fifth points respectively. Furthermore, the third condition is trivial since the 
functionals
again vanish by assumption. We demonstrate how to show the fourth part of the induction hypothesis in this case since the proof of the fifth is similar. It follows from admissibility and the induction hypothesis applied to $w$ that
\begin{align*}
	\hat{\boldsymbol{\Pi}} \otimes_{\langle{\bar{\bsigma}}\rangle}^{\bar{\sigma}} (\zeta_{\iota_{\bar{\sigma}}(e)})_{e \in E_+(\bar{\sigma})} = D^k A_\mfl(\zeta_{\iota_{\bar{\sigma}}(e_*)}) \ast \hat{\boldsymbol{\Pi}} 
	\otimes_{\langle{\bar{\pmb{w}}}\rangle}^{\bar{w}} (\zeta_{\iota_{\bar{w}}(e)})_{e \in E_+(\bar{w})}
	= D^k A^\tau(\mfl, \iota_{\bar{\sigma}}(e_*)) \ast \hat{\boldsymbol{\Pi}}^\tau w^*
	= \hat{\boldsymbol{\Pi}}^\tau \sigma^*.
\end{align*}
where $e_*$ denotes the trunk of $\sigma$ which we recall is of type $(\mfl, \iota_{\bar{\sigma}}(e_*)) \in \mfL(\tau)$ by definition of $\iota_{\bar{\sigma}}$. 

It now remains only to treat the case where $\sigma = \prod_{i = 0}^n \sigma_i$. Here we prove each of the induction hypotheses in the order they are presented. Since the second is similar to the first and the fourth is similar to the fifth, we will omit the details for the second and fourth claims. For the first of the induction hypotheses, we note that due to its definition via Lemma~\ref{lem:multiplication}, we can write 
\begin{align*}
	\hat{\Pi}_x^\times \otimes_{\langle{\bar{\bsigma}}\rangle}^{\bar{\sigma}} (\zeta_{\iota_{\bar{\sigma}}(e)})_{e \in E_+(\bar{\sigma})} 
	&= \prod_{i = 0}^n \hat{\Pi}_x^\times  \otimes_{\langle{\bar{\bsigma}_i}\rangle}^{\bar{\sigma}_i} (\zeta_{\iota_{\bar{\sigma}_i}(e)})_{e \in E_+(\bar{\sigma}_i)} 
	= \prod_{i = 0}^n \hat{\Pi}_x^{\tau, \times} \sigma_i^* 
	= \hat{\Pi}_x^{\tau, \times} \sigma^*
\end{align*}
where the second equality follows from the induction hypothesis and Lemma~\ref{l:age_ind_prop}.

For the third part of the induction hypothesis, we note that by definition we have that
\begin{align}\label{ud01}
	\hat{\ell}\big ( \otimes_{\langle{\bar{\bsigma}}\rangle}^{\bar{\sigma}} (\zeta_{\iota_{\bar{\sigma}}(e)})_{e \in E_+(\bar{\sigma})} \big ) 
	= - \mathbb{E}[ \hat{\boldsymbol{\Pi}}^\times \otimes_{\langle{\bar{\bsigma}_i}\rangle}^{\bar{\sigma}_i} (\zeta_{\iota_{\bar{\sigma}_i}(e)})_{e \in E_+(\bar{\sigma}_i)} (0)] 
	- (\hat{\ell} \otimes \mathbb{E}[ \hat{\boldsymbol{\Pi}}^\times \cdot (0)]) \mathring{\Delta}_r^- \otimes_\sttaun{i}^{\tau_i} (\zeta_e)_{e \in E_+(\tau_i)}.
\end{align}
By the second claim in the induction hypothesis, we can immediately identify the first term as $- \mathbb{E}[ \hat{\boldsymbol{\Pi}}^{\tau, \times} \sigma^* (0)]$. To conclude that 
the second terms agrees with
$ 	(\hat{\ell}^\tau \otimes \mathbb{E}[ \hat{\boldsymbol{\Pi}}^{\tau, \times} \cdot (0)]) \mathring{{\Delta}}_r^- \sigma^*$, we simply appeal to the induction hypothesis and the third point of Lemma~\ref{l:age_ind_prop}. This allows us to conclude that
$
\hat{\ell}\big ( \otimes_{\langle{\bar{\bsigma}}\rangle}^{\bar{\sigma}} (\zeta_{\iota_{\bar{\sigma}}(e)})_{e \in E_+(\bar{\sigma})} \big ) 
= \hat{\ell}^\tau(\sigma^*)
$
as required.	
It remains to prove the fifth claim in the induction hypothesis in the case $\sigma = \prod_{i = 0}^n \sigma_i$. This now follows from the definitions and the previous hypotheses since
\begin{align*}
	\hat{\Pi}_x  \otimes_{\langle{\bar{\bsigma}}\rangle}^{\bar{\sigma}} (\zeta_{\iota_{\bar{\sigma}}(e)})_{e \in E_+(\bar{\sigma})} = (\hat{\ell} \otimes \hat{\Pi}_x^\times) \Delta_r^-  \otimes_{\langle{\bar{\bsigma}}\rangle}^{\bar{\sigma}} (\zeta_{\iota_{\bar{\sigma}}(e)})_{e \in E_+(\bar{\sigma})}= (\hat{\ell}^\tau \otimes \hat{\Pi}_x^{\tau, \times}) {\Delta}_r^- \sigma^* = \hat{\Pi}_x^\times \sigma^*.
\end{align*}
\end{proof}

With this identification in hand, the annealed forms of the stochastic estimates can be lifted as a corollary of their scalar valued counterparts, so long as those scalar valued counterparts are assumed to be suitably uniform in the kernel assignment and to hold on the regularity structures $T^\tau$ defined above rather than only on $\mcT^\mathrm{eq}$.

\begin{lemma}\label{l:annealed_mb} 
Let $W$ be a historic sector of $\mcT^{\Ban}$, let $\tau \in \btau \in \mcB_{W}$
and let $(\zeta_e)_{e \in E_+(\tau)} \in V^{\times \tau}$.
Fix a kernel assignment $A \in \mcA_+^{\mco}$ with $\mco > \mathrm{ord} ( W)$ and a random smooth noise assignment $\xi \colon \Omega \to \mcA_-^{\infty; \eq}$ satisfying the assumptions of Proposition~\ref{lem: BPHZ exists}.
The BPHZ model on $W$
satisfies the estimate 
\begin{equ} \label{eq:annealed_1}
	\sup_{v \in V^{\proj \sttau}} \frac{\mathbb{E}^{\frac{1}{p}} \left [ \big| \hat{\Pi}_x v \big( \phi_x^{\lambda} \big) \big|^p \right ]}{\| v \|_{V^{\proj \sttau}}} 
	\le \|A\|_{\mcA_+^{\mco}}^{|E_+(\tau)|} \sup_{A^\tau} \frac{\mathbb{E}^{\frac{1}{p}} \left [ \big| \hat{\Pi}_x^\tau \tau^* \big( \phi_x^{\lambda} \big) \big|^p \right ]}{ \prod_{e \in E_+(\tau)} \|A^\tau(\mfl(e))\|_{\mathcal{K}_{\mco}^{|\mfl(e)|}}}
\end{equ}
for each 
$x \in \mathfrak{K}, \phi \in \mathcal{B}^r, \lambda \in (0,1]$,  $1 \le p < \infty$ and for each compact $\mfK \subset \mathbb{R}^d$
where the supremum over $A^\tau$ on the right hand side runs over all kernel assignments on $T^\tau$ and $\hat{\Pi}_x^\tau$ denotes the BPHZ model on $T^\tau$ for the kernel assignment $A^\tau$. 

Furthermore, given a pair of 
random smooth noise assignments 
$\xi^1, \xi^2 \colon \Omega \to \mcA_-^{\infty; \eq}$
and a fixed kernel assignment $A \in \mathcal{A}_+^{\mco}$, writing now $\hat{\Pi}_x^i$ for the BPHZ models on $\mcT^\mathrm{Ban}$ with respect to the noise assignments $\xi^i$ and kernel assignment $A$, we have that 
\begin{align}\label{eq:annealed_2}
	\sup_{v \in V^{\proj \sttau}} \frac{\mathbb{E}^{\frac{1}{p}} \left [ \big| (\hat{\Pi}_x^1 - \hat{\Pi}_x^2) v \big( \phi_x^{\lambda} \big) \big|^p \right ]}{\| v \|_{V^{\proj \sttau}}} 
	\le \|A\|_{\mcA_+^{\mco}}^{|E_+(\tau)|} \sup_{A^\tau} \frac{\mathbb{E}^{\frac{1}{p}} \left [ \big| (\hat{\Pi}_x^{\tau,1} - \hat{\Pi}_x^{\tau, 2}) \tau^* \big( \phi_x^{\lambda} \big) \big|^p \right ]}{ \prod_{e \in E_+(\tau)} \|A^\tau(\mfl(e))\|_{\mathcal{K}_{\mco}^{|\mfl(e)|}}}
\end{align}
where $\hat{\Pi}_x^{\tau, i}$ denotes the BPHZ model on $T^\tau$ for the kernel assignment $A$ and noise assignment $\xi^i$. 
\end{lemma}

\begin{proof}
We start with the first claim where we recognise the left hand side as the norm of the map $v \mapsto \hat{\Pi}_x v (\phi_x^\lambda) \in L^p(\Omega)$. By Lemma~\ref{lem:UP_in_Lp}, it then suffices to prove the bound in the case where $v = \otimes_\sttau^\tau (\zeta_e)$ is a symmetric elementary tensor and with $\|v\|_{V^{\proj \sttau}}$ replaced by $\prod_{e \in E_+(\tau)} \|\zeta_e\|_{V_{\mfl(e)}}$.	
To do this we note that if we write $\hat{\Pi}_x^\tau[A^\tau]$ for the BPHZ model on $T^\tau$ with kernel assignment $A^\tau(\mfl, e) = A_\mfl(\zeta_e)$
of order $\mco$,	
then by Lemma~\ref{lem:up_down} we have that
\begin{align*}
	\mathbb{E}^{\frac{1}{p}} [ |\hat{\Pi}_x v (\phi_x^\lambda)|^p] 
	&= \mathbb{E}^{\frac{1}{p}} [|\hat{\Pi}_x^\tau [A^\tau] \tau^* (\phi_x^\lambda)|^p]
	\le C \prod_{e \in E_+(\tau)} \| A_\mfl ( \zeta_e) \|_{\mcK_{\mco}^{|\mfl(e)|}} 
	\le C \|A\|_{\mcA_+^{\mco}}^{|E_+(\tau)|} \prod_{e \in E_+(\tau)} \|\zeta_e\|_{V_{\mfl(e)}} 
\end{align*}
for all $x \in \mfK$ and $\lambda \in (0,1]$ where $C$ denotes the right hand side of \eqref{eq:annealed_1}. This yields the desired estimate immediately. 

Since the case of two noise assignments follows almost exactly as in the case of a single model, we omit the details of this case. 
\end{proof}

\subsection{Stochastic Estimates for the BPHZ Model on \texorpdfstring{$\mcT^\mathrm{Ban}$}{T Ban} -- Quenched Estimates} \label{ss:quenched}

Let 
$\mco \in \mathbb{N}$.
We now consider the 
Banach space assignment 
$V = ( \mathcal{K}_{\mco}^{| \mfl |} )_{\mfl \in \mfL_+}$,
and the kernel assignment of order $\mco$ defined by
$A_{\mfl} = \mathrm{Id} \colon \mathcal{K}_{\mco}^{| \mfl |} \to \mathcal{K}_{\mco}^{| \mfl |}$
on the regularity structure
$\mcT^{\Ban} ( V )$.
Let $W$ be a historic sector of order $\mathrm{ord} (W)<\mco$
in $\mcT^{\Ban} ( V )$.

The goal of this section is to deduce quenched, i.e.\ almost sure, estimates in the model seminorms on $(\hat{\Pi}, \hat{\Gamma})$
and on $(\hat{\Pi}^1 - \hat{\Pi}^2, \hat{\Gamma}^1 - \hat{\Gamma}^2)$, in terms of the suprema over $x, \lambda$ and $\phi$ of the left-hand side of \eqref{eq:annealed_1} and \eqref{eq:annealed_2} respectively. This essentially involves passing suprema through expectations.
For the suprema over $x$, $\lambda$, and $\phi$ we will be able to proceed similarly to as in the scalar-valued case, see for instance the Kolmovorov-type theorem for models in \cite[Theorem~10.7]{Hai14} 
or \cite[Theorem~B.1]{HS23}.
Here the main difficulty comes from the supremum over $v$, which needs to be pulled inside the expectation.

Let us briefly describe our strategy, which is inspired by the one employed in \cite{GH19}.
Since Kolmogorov criteria for stochastic processes indexed by Banach spaces are very subtle, our almost sure estimate will not be uniform over $v \in \mathcal{K}_{\mco}^{|\mfl|}$.
Instead, we now fix throughout this section a kernel assignment $K \in \mcA_+^{\eq; \mcO}$ of order $\mcO \in \mathbb{N}$ on $\mcT^{\eq} ( R )$. 
Within $\mathcal{K}_{\mcO}^{|\mfl|}$, we consider the curve $(K_{\mfl}^c)_{c \in \mathbb{R}^d}$, which we embed into a linear space by duality, i.e.\ by considering the
vector space $\lbrace \zeta ( K_{\mfl}^{(\cdot)} ) \rbrace$ spanned by a suitable space of distributions $\zeta$ acting on the upper variable.

One of the reasons to restrict ourselves to the curve 
$(K_{\mfl}^c)_{c \in \mathbb{R}^d}$
in kernel space 
is the following identity, 
rigorously established
in Lemma~\ref{l:semi-annealedII} 
below,
\begin{align} \label{eq:a39}
	\hat{\Pi}_x v 
	& = v \Big( c \mapsto \hat{\Pi}_x \bigotimes_{\sttau}^{\tau} \big( K_{\mfl_{e}}^{c_e} \big)_{e \in E_+ ( \tau )} \Big) ,
	\qquad \text{for } v \in  (C_w^{\alpha} )^{\proj \sttau}
\end{align}
where in the right-hand side $\hat{\Pi}$ denotes the BPHZ model on
(an appropriate historic sector of)
the regularity structure
$\mathcal{T}^{\mathrm{Ban}} \big( ( \mcK_{\mco}^{| \mfl |} )_{\mfl \in \mfL_+} )$, 
while in the left-hand side $\hat{\Pi}$ denotes the BPHZ model on
(an appropriate historic sector of)
$\mathcal{T}^{\mathrm{Ban}} \big( ( C_w^{\alpha} )_{\mfl \in \mfL_+} )$ for an appropriate choice of $\alpha$. Here $w$ is a weight\footnote{For notational consistency, from Lemma~\ref{l:pa} onwards, the letter $w$ will be typically reserved for weights which decay at infinity, whilst $w^*$ will denote weights which grow at infinity.} of the form 
$w(x) \coloneqq (1 + |x|_\mfs)^l$ where $l \in \mbR$ and the corresponding spaces $C_w^\alpha$ are defined below.
\begin{definition}
\label{def:wHoel}
Let $\alpha \in \mathbb{R} \setminus \mathbb{N}$.
We define $C_{w}^{\alpha} \subset \mathcal{D}^{\prime} ( \mathbb{R}^d )$ to be
the 
Banach space 
of distributions $\zeta$ when $\alpha < 0$
and function $g$ when $\alpha > 0$, 
such that, 
with $r \coloneqq \min \lbrace j \in \mathbb{N} : j >-\alpha \rbrace$,
\begin{alignat*}{2}
	\| \zeta \|_{C_{w}^{\alpha}} & \coloneqq \sup_{x \in \mathbb{R}^d} \sup_{\lambda \in (0, 1]} \sup_{\varphi \in \mathcal{B}^r} \frac{| \zeta ( \varphi_x^{\lambda} ) |}{w ( x ) \, \lambda^{\alpha}}< +\infty , \quad && \text{ if $\alpha < 0$ } , \\
	\| g \|_{C_{w}^{\alpha}} & \coloneqq \sup_{x \in \mathbb{R}^d} \frac{| g ( x ) |}{w ( x )} 
		+ \sup_{x \in \mathbb{R}^d} \sup_{| h |_{\mfs} \leq 1} \frac{\big| g ( x + h ) - \sum_{| {k} |_{\mfs} < \alpha} \frac{1}{{k} !} \partial^{k} g ( x ) \, h^{k} \big|}{w (x ) \, | h |_{\mfs}^{\alpha}}< + \infty , \quad && \text{ if $\alpha > 0$} .
\end{alignat*}
\end{definition}

The proof of the following lemma is standard, e.g.\ using (Daubechies) wavelets.
\begin{lemma}
\label{l:pa}
Let $\alpha < 0$, $\alpha^* > 0$ along with weights $w, w^*$ satisfying
$
	\alpha + \alpha^* > 0$ and $\int_{\mathbb{R}^d} \mathrm{d} x \, w ( x ) w^* ( x ) < \infty .
$
Then the usual distributional pairing extends (uniquely) to a pairing such that for $\zeta \in C_w^{\alpha}$, $g \in C_{w^*}^{\alpha^*}$
\begin{align*}
	| \zeta ( g ) | 
	\lesssim
	\| \zeta \|_{C_{w}^{\alpha}} \| g \|_{C^{\alpha^*}_{w^*}} . 
\end{align*}
\end{lemma}

We come back to our discussion of $(\hat{\Pi}, \hat{\Gamma})$.
We shall show that
the BPHZ model
$(\hat{\Pi}, \hat{\Gamma})$, 
defined on 
an appropriate historic sector of
$\mathcal{T}^{\mathrm{Ban}} \big( ( \mathcal{K}_{\mco}^{|\mfl|} )_{\mfl \in \mfL_+} \big)$,
gives rise to an almost surely admissible model on (historic sectors of) the \emph{smaller} regularity structure $\mathcal{T}^{\mathrm{Ban}} \big( ( C_w^{\alpha} )_{\mfl \in \mfL_+} \big)$ 
(for  $\alpha, w$ to be fixed later)
over the Banach space assignment $(C_w^{\alpha})_{\mfl \in \mfL_+}$,
with respect to the kernel assignment 
\begin{align} \label{eq:a34b} 
\begin{array}[t]{lrcl}
	A_{\mfl} : & C_w^{\alpha} & \longrightarrow & \mathcal{K}_{\mco}^{|\mfl|} \\
	& \zeta & \longmapsto & 
	K_{\mfl}^{\zeta} ,
\end{array}
\qquad \text{where} 
\qquad 
\begin{array}[t]{lrcl}
	K_{\mfl}^{\zeta} : & \mathbb{R}^d & \longrightarrow & \mathbb{R} \\
	& x & \longmapsto & \zeta ( K_{\mfl}^{(\cdot)} ( x ) ) .
\end{array}
\end{align}
More precisely, 
$K_{\mfl}^{\zeta}$ is the kernel decomposition with components 
$K_{\mfl, n}^{\zeta} ( x ) = \zeta ( K_{\mfl, n}^{(\cdot)} ( x ) )$
for $n \in \mathbb{N}$.

\begin{lemma} \label{l:ka}
Assume that
 $\mcO \geq \mco + \max (\mfs)$, $\mcO \geq - \alpha + \max (\mfs)$ and $\int_{\mathbb{R}^d} \mathrm{d} x \, w ( x ) < \infty .$
Then 
$A$
defined in \eqref{eq:a34b} is a kernel assignment of order $\mco$ on $\mathcal{T}^{\mathrm{Ban}} \big( ( C_w^{\alpha} )_{\mfl \in \mfL_+} \big) $ 
in the sense of Definition~\ref{def:ka}.
Furthermore, the
linear map 
$\mcA_+^{\eq; \mcO} \ni K\mapsto A \in\mcA_+^{\mco}$
is continuous.
\end{lemma}

\begin{proof}
We fix $\mfl \in \mfL_+$ and suppress appearances of $\mfl$ in the notation in this proof.
Let $\zeta \in C_w^{\alpha}$.
We estimate the kernel norm $\| A \|_{\mcK^{| \mfl|}_{\mco}}$.
By the anisotropic Taylor formula of Lemma~\ref{Tay}
and the fact that $K \in C^{\mcO} ( \mcK_{\mcO}^{\mfl} )$, 
we see that
the $n$-th dyadic component
$A_n \coloneqq K_{\mfl, n}^{\zeta}$ is  
of regularity $C^{2\mco}$
provided $\mcO\geq\mco + \max(\mfs)$,
and
\begin{align*}
	\partial^k A_n ( x )
	& = \zeta \big( \partial^k K_n^{(\cdot)} ( x ) \big)
	\qquad \text{for $| k |_{\mfs} \leq 2 \mco$} .
\end{align*}
Furthermore,
let $w^* \equiv 1$, and choose $\alpha^*>0$ with $\alpha^* + \alpha > 0$ sufficiently small so that by the assumption $\mcO \geq - \alpha + \max ( s )$ and the anisotropic Taylor formula of Lemma~\ref{Tay}, 
we have for any $x \in \mathbb{R}^d$, 
\begin{align*}
	\| \partial^k K_n^{(\cdot)} ( x ) \|_{C_{w^*}^{\alpha^*}}
	& \lesssim \| \partial^k K_n^{(\cdot)} ( x ) \|_{C^{\mcO}} 
	\lesssim \| K \|_{C^{\mcO} ( \mathcal{K}_{\mcO}^{| \mfl | } )} \, 2^{(| \mfs | - |\mfl| + | k |_{\mathfrak{s}}) n} .
\end{align*}
In turn, by the pairing Lemma~\ref{l:pa}, $| \partial^k A_n ( x ) |
	\lesssim \| K \|_{\mcA_+^{\eq; \mcO}} \, \| \zeta \|_{C_w^{\alpha}} \, 2^{(| \mfs | - |\mfl| + | k |_{\mathfrak{s}}) n}$
which implies that
$
	\| A ( \zeta ) \|_{\mcK^{| \mfl|}_{\mco}}
 \lesssim \| K \|_{\mcA_+^{\eq; \mcO}} \| \zeta \|_{C_w^{\alpha}} .
$
This yields the desired boundedness of both linear maps $\zeta \mapsto A (\zeta)$ and $K \mapsto A$.
\end{proof}

In preparation for our almost-sure model estimate, let us now recall a version of Kolmogorov's continuity theorem in weighted spaces, the proof of which we omit since it can be obtained in the same fashion as the usual (local) Kolmogorov Continuity Theorem.
\begin{lemma}
\label{l:wk}
Let $N \in \mathbb{N}$ and let $\mfs \in \mathbb{R}_+^N$ be a scaling in $\mathbb{R}^N$.
Let $\beta > 0$, $p> \max( 1, 2|\fraks|/\beta)$,
and let $w^*$ be a weight such that $\int_{\mathbb{R}^N} \mathrm{d} x \, w^* ( x )^{- p} < \infty.$ 
Let $X \colon \mathbb{R}^N \to \mathbb{R}$ be a random function such that 
$\partial^k X_c$ exists for all $|k|_{\mfs} < \beta$.
Then, 
\begin{align} \label{a8}
	\mathbb{E}^{\frac{1}{p}} \Big[ \| X \|_{C_{w^*}^{\beta- 2|\fraks|/p}}^p \Big] 
	& \lesssim
	\sup_{c \in \mathbb{R}^N} \mathbb{E}^{\frac{1}{p}} [ | X_c |^{p} ] 
		+ \sup_{c \in \mathbb{R}^N} \sup_{| h |_{\mfs} \leq 1}  
		\frac{ \mathbb{E}^{\frac{1}{p}} \big[ \big| X_{c + h} - \sum_{| {k} |_{\mfs} < \beta} \frac{1}{{k} !} \partial^{k} X_c \, h^{k} \big|^p \big]}{ | h |_{\mfs}^{\beta}} . 
\end{align}
\end{lemma}
We now apply the above Kolmogorov Continuity Theorem in the situation of interest here.
\begin{lemma} \label{l:semi-annealed}
Let $W$ be a historic sector in $\mathcal{T}^{\mathrm{Ban}} \big( ( \mathcal{K}_{\mco}^{| \mfl |} )_{\mfl \in \mfL_+} \big)$
of order $\mathrm{ord} ( W ) < \mco$.
Let $( \hat{\Pi}, \hat{\Gamma} )$ be the BPHZ model on $W$ associated to the identity kernel assignment and a random smooth noise assignment $\xi \colon \Omega \to \mcA_-^{\infty; \eq}$.
Let $K \in \mcA_+^{\eq; \mcO}$ and $\tau \in \boldsymbol{\tau}\in \mcB_W$.
Assume that 
$\mcO \geq \mco + \max ( \mfs )$ 
and 
let $\alpha^* > 0$, $p \geq 1$ be such that 
$\alpha^* \leq \mcO - \max ( \mfs) - 2 | \mfs| | E_+ ( \tau ) |/p$.
Let $w^*$ be a weight such that 
$\int_{\mathbb{R}^{d | E_+ ( \tau ) |}} \mathrm{d} x \, w^* ( x )^{- p} < \infty$.
Then for all 
$x \in \mathbb{R}^d$,
$\lambda \in ( 0, 1]$
and $\phi \in \mcB^r$,
\begin{align} \label{eq:a37}
	& \mathbb{E}^{\frac{1}{p}} \bigg[ \Big\| (\mathbb{R}^d )^{E_+ ( \tau )} \ni c \mapsto \hat{\Pi}_x \bigotimes_{\sttau}^{\tau} \big( K_{\mfl_{e}}^{c_e} \big)_{e \in E_+ ( \tau )} \big( \phi_x^{\lambda} \big) \Big\|_{C_{w^*}^{\alpha^*}}^p \bigg]
	\lesssim
	\| K \|_{\mcA_+^{\eq; \mcO}}^{| E_+ ( \tau ) |} 
	\sup_{v \in (\mathcal{K}_{\mco}^{| \mfl |})^{\proj \sttau}} \frac{\mathbb{E}^{\frac{1}{p}} \big[ \big| \hat{\Pi}_x v \big( \phi_x^{\lambda} \big) \big|^p \big]}{\| v \|_{(\mathcal{K}_{\mco}^{|\mfl|})^{\proj \sttau}} } .
\end{align}
Moreover, if $\hat{\Pi}^1$
and $\hat{\Pi}^2$ denote the BPHZ models on 
$W$
with respect to 
two smooth noise assignments $\xi^1, \xi^2$,
and if $K^1,K^2 \in \mcA_+^{\eq; \mcO}$,
then 
\begin{align} \label{eq:a37b}
	& \mathbb{E}^{\frac{1}{p}} \bigg[ \Big\| (\mathbb{R}^d )^{E_+ ( \tau )} \ni c \mapsto \Big( \hat{\Pi}_x^1 \bigotimes_{\sttau}^{\tau} \big( K_{\mfl_{e}}^{c_e; 1} \big)_{e \in E_+ ( \tau )} - \hat{\Pi}_x^2 \bigotimes_{\sttau}^{\tau} \big( K_{\mfl_{e}}^{c_e; 2} \big)_{e \in E_+ ( \tau )}  \Big) \big( \phi_x^{\lambda} \big) \Big\|_{C_{w^*}^{\alpha^*}}^p \bigg] \\
	& \lesssim
	\| K^1 - K^2 \|_{\mcA_+^{\eq; \mcO}} 
	+ \sup_{v \in (\mathcal{K}_{\mco}^{| \mfl |})^{\proj \sttau}} \frac{\mathbb{E}^{\frac{1}{p}} \big[ \big| \big( \hat{\Pi}_x^1 - \hat{\Pi}_x^2 \big) v \big( \phi_x^{\lambda} \big) \big|^p \big]}{\| v \|_{(\mathcal{K}_{\mco}^{|\mfl|})^{\proj \sttau}} } ,
	\nonumber
\end{align}
where the implicit constant can be chosen to depend polynomially on
\begin{align*}
	\| K^1 \|_{\mcA_+^{\eq; \mcO}} + \| K^2 \|_{\mcA_+^{\eq; \mcO}}
	+ \sup_{v \in (\mathcal{K}_{\mco}^{| \mfl |})^{\proj \sttau}} \frac{\mathbb{E}^{\frac{1}{p}} \big[ \big| \hat{\Pi}_x^1 v \big( \phi_x^{\lambda} \big) \big|^p \big]}{\| v \|_{(\mathcal{K}_{\mco}^{|\mfl|})^{\proj \sttau}} } 
	+ \sup_{v \in (\mathcal{K}_{\mco}^{| \mfl |})^{\proj \sttau}} \frac{\mathbb{E}^{\frac{1}{p}} \big[ \big| \hat{\Pi}_x^2 v \big( \phi_x^{\lambda} \big) \big|^p \big]}{\| v \|_{(\mathcal{K}_{\mco}^{|\mfl|})^{\proj \sttau}} }.
\end{align*}
\end{lemma}

\begin{proof}
We only prove \eqref{eq:a37} since \eqref{eq:a37b} can be proved with an analogous argument using the multilinearity of the map $\hat{\Pi}_x \bigotimes_{\sttau}^{\tau} \colon ( (\mathcal{K}_{\mco}^{| \mfl |})_{\mfl \in \mfL_+} )^{\times \tau} \to \mathcal{D}^{\prime}$.
Hence, we consider the random field
\begin{align*}
	X \colon (\mathbb{R}^d )^{E_+ ( \tau )} \ni c \mapsto \hat{\Pi}_x \bigotimes_{\sttau}^{\tau} \big( K_{\mfl_{e}}^{c_e} \big)_{e \in E_+ ( \tau )} \big( \phi_x^{\lambda} \big) ,
\end{align*}
to which we apply the Kolmogorov continuity theorem of Lemma~\ref{l:wk}, 
where we identify $(\mathbb{R}^d )^{E_+ ( \tau )} \simeq \mathbb{R}^{N \coloneqq d | E_+ ( \tau ) |}$,
in combination with the anisotropic Taylor formula of Lemma~\ref{Tay},
yielding
	\begin{align*}
	\mathbb{E}^{\frac{1}{p}} \Big[ \| X \|_{C_{w^*}^{\alpha^*}}^p \Big] 
	& \lesssim
	\max_{| k_{e} |_{\mfs} \leq \mcO}
	\sup_{c \in \mathbb{R}^N} \mathbb{E}^{\frac{1}{p}} \Big[ \Big| \hat{\Pi}_x \bigotimes_{\sttau}^{\tau} \big( \partial_{c_e}^{k_e} K_{\mfl_{e}}^{c_e} \big)_{e \in E_+ ( \tau )} \big( \phi_x^{\lambda} \big) \Big|^{p} \Big] . 
	\end{align*}
To conclude, it suffices 
to appeal to
the fact that the symmetric set tensor product $(\otimes_\mcs^a)_{a \in \mcs}$ is a collection of norm $1$ multilinear maps 
(recall Definition~\ref{def:tensor_product}),
so that for all $c \in \mathbb{R}^{d | E_+ ( \tau ) |}$,
and $(k_e)_{e \in E_+ ( \tau )}$ with $| k_e |_{\mfs} \leq \mcO$,
\begin{align*}
	\Big\| \bigotimes_{\sttau}^{\tau} \big( \partial_{c_e}^{k_e} K_{\mfl_{e}}^{c_e} \big)_{e \in E_+ ( \tau )} \Big \|_{( (\mathcal{K}_{\mco}^{| \mfl |})_{\mfl \in \mfL_+} )^{\proj \sttau}}
	& \leq \prod_{e \in E_+ ( \tau )} \big\| \partial_{c_e}^{k_e} K_{\mfl_{e}}^{c_e} \big\|_{\mcK_{\mco}^{| \mfl |}}
	\leq \| K \|_{\mcA_+^{\eq; \mcO}}^{| E_+ (\tau) |} .
\end{align*}
This yields the claimed bound \eqref{eq:a37}.
\end{proof}
We now
recover and estimate the BPHZ model
on 
(appropriate historic sectors of)
the regularity structure $\mathcal{T}^{\mathrm{Ban}} ( C_w^{\alpha} )$ with the kernel assignment \eqref{eq:a34b}.
For that purpose, we will use the following pairing between symmetric-set tensor products of weighted H\"older spaces.
\begin{lemma}
\label{l:pa2}
Let $\mcs$ be a symmetric set and $a \in \Ob(\mcs)$.
Let $\alpha < 0$, $\alpha^* >0$ along with weights $w \colon \mathbb{R}^d \to \mathbb{R}$, $w^* \colon \mathbb{R}^{d | a |} \to \mathbb{R}$, satisfying
\begin{align*}
\alpha^* + | a | \alpha > 0 , 
\qquad \int_{\mathbb{R}^{d | a |}} \mathrm{d} x \, w^{\otimes | a |} ( x ) \, w^{*} ( x ) < \infty .
\end{align*}
Then there exists a unique (bilinear) pairing 
\begin{align} \label{eq:a47}
(C_w^{\alpha} ( \mathbb{R}^d ))^{\proj \mcs} \times C_{w^*}^{\alpha^*} ( \mathbb{R}^{d | a |} ) \to \mathbb{R} ,
\end{align}
making the following family of diagrams indexed by $g \in C_{w^*}^{\alpha^*} ( \mathbb{R}^{d | a |} )$ commute
\begin{center}
\begin{tikzcd}
	(C_w^{\alpha})^{\times a}
	\arrow[dr, "(\pi_{a} ( \otimes \cdot )) ( g )"'] 
	\arrow[r, "\hat\otimes_\mcs^{a}"] 
	& (C_w^{\alpha})^{\proj \mcs} 
	\arrow[d, dotted, "\cdot ( g )"]
	\\ & \mathbb{R}
\end{tikzcd}
\end{center}
where 
$\pi_{a} = | \hom_{\mcs} (a, a) |^{-1} \sum_{\gamma \in \hom_{\mcs} ( a, a)} \gamma$
is the symmetrisation map from $\mcs$.
Furthermore
\begin{align} \label{eq:a42}
| \zeta ( g ) | 
& \lesssim
\| \zeta \|_{(C_w^{\alpha})^{\proj \mcs}} \, 
\| g \|_{C_{w^*}^{\alpha^*}},
\end{align}
uniformly over $\zeta \in (C_w^{\alpha} (\mathbb{R}^d))^{\proj \mcs}$ and $g \in C_{w^*}^{\alpha^*} ( \mathbb{R}^{d | a |} ) $.
\end{lemma}

We note that in this statement we have performed a slight abuse of notation in writing $\mathbb{R}^{d | a |}$,
since the order of the variables does matter to identify the pairing and should be considered to be fixed.
\begin{proof}
By our universal property of Definition~\ref{def:tensor_product}
it suffices to prove that the maps $(\pi_{a} \otimes \cdot) ( g )$ are well-defined, $\mcs$-symmetric, multilinear and continuous.
To prove that it is well-defined, we appeal to the pairing Lemma~\ref{l:pa}: it suffices to prove that given 
$\zeta = (\zeta_1, \cdots , \zeta_{n \coloneqq | a |}) \in (C_w^{\alpha})^{\times a}$, one has $\pi_{a} (\otimes f) \in C_{\tilde{w}}^{\tilde{\alpha}}$ for $\tilde{w} = w^{\otimes n}$ and $\tilde{\alpha} = n \alpha$.
By definition, for such a $\zeta$, 
\begin{align} \label{eq:a40}
\pi_{a} ( \otimes \zeta )
& = | \hom_{\mcs} (a, a) |^{-1} \sum_{\gamma \in \hom_{\mcs} (a, a)} \zeta_{\gamma ( 1 )} \otimes \cdots \otimes \zeta_{\gamma ( n )} \in (C_{w}^{\alpha})^{\otimes n} .
\end{align}
Recall e.g.\ the wavelet characterisation
of $C_{w}^{\alpha}$ \cite[Proposition~2.4]{HL15},
which in particular implies a 
`tensorisation property' of the H\"older norms, 
namely for $\alpha < 0$ that $(C_{w}^{\alpha})^{\otimes n} \subset C_{w^{\otimes n}}^{n \alpha}$ with 
\begin{align*}
\| \zeta_1 \otimes \cdots \otimes \zeta_n \|_{C_{w^{\otimes n}}^{n \alpha}} 
\lesssim
\prod_{i = 1}^n \| \zeta_i \|_{C_{w}^{\alpha}} .
\end{align*}		
We deduce that $\pi_{a} (\otimes \zeta) \in C_{w^{\otimes n}}^{n \alpha}$ with 
\begin{align*}
\| \pi_{a} (\otimes \zeta) \|_{C_{w^{\otimes n}}^{n \alpha}} 
\lesssim
\prod_{i = 1}^n \| \zeta_i \|_{C_{w}^{\alpha}} .
\end{align*}
It follows that the map $(\pi_{a} (\otimes \cdot)) ( g )$ is well-defined.
It is $\mcs$-symmetric by \eqref{eq:a40}.
It is also clearly multilinear and by Lemma~\ref{l:pa} bounded with norm
\begin{align*}
\| (\pi_{a} (\otimes \cdot)) ( g ) \|
\lesssim
\| g \|_{C_{w^*}^{\alpha^*}} ,
\end{align*}
so that the continuity estimate \eqref{eq:a42} 
follows from the fact that the factoring map
$\cdot (g)$
has the same norm as each of the factored maps 
$\big( (\pi_{a} (\otimes \cdot)) ( g ) \big)_{a \in \Ob(\mcs)}$, recall Definition~\ref{def:tensor_product}.

By construction, the pairing \eqref{eq:a47} is linear in the first variable.
The fact that it is also linear in the second variable is a direct consequence of the uniqueness of the factoring map in Definition~\ref{def:tensor_product}.
\end{proof}
One simple but useful observation about the pairing defined above is that if $g \in C_{w^*}^{\alpha^*}(\mathbb{R}^{d|a|})$ is $\mcs$-symmetric then the pairing takes a particularly simple form.
\begin{lemma}\label{lem:symm_test}
Suppose that $\mcs$ is a symmetric set and $a \in \Ob(\mcs)$. If $g \in C_{w^*}^{\alpha^*}(\mathbb{R}^{d|a|})$ is $\mcs$-symmetric in the sense that $g \circ \gamma = g$ for all $\gamma \in \hom_\mcs(a,a)$ then for $(\zeta_e)_{e \in a} \subset C_w^\alpha(\mathbb{R}^d)$,
\begin{align*}
\bigotimes_\mcs^a (\zeta_e) (g) = \bigotimes_{e \in a} (\zeta_e) (g).
\end{align*}
\end{lemma}
\begin{proof}
By its definition we can write
\begin{align*}
\bigotimes_\mcs^a (\zeta_e) (g) &= |\hom_\mcs(a,a)|^{-1} \sum_{\gamma \in \hom_\mcs(a,a)} \bigotimes_{e \in a} (\zeta_{\gamma^{-1}(e)}) (g) =|\hom_\mcs(a,a)|^{-1} \sum_{\gamma \in \hom_\mcs(a,a)} \bigotimes_{e \in a} (\zeta_{e}) (g \circ \gamma)
\\
&= \bigotimes_{e \in a} (\zeta_{e}) (g ) ,
\end{align*}
as claimed.
\end{proof}

Equipped with the pairing of Lemma~\ref{l:pa2}, 
we now prove that the BPHZ models on 
appropriate 
historic sectors of
$\mathcal{T}^{\mathrm{Ban}} ( \mcK_{\mco}^{| \mfl |} )$
and
$\mathcal{T}^{\mathrm{Ban}} ( C_w^{\alpha} )$ 
are related by \eqref{eq:a39},
and furthermore establish a corresponding estimate.
In order to clarify the notations due to working in these two structures, in the remainder of this section, we will decorate by the superscripts $\mcK$ and $C$
objects that correspond to 
these respective regularity structures.

More precisely, we fix a historic sector $W = W^{\mcK}$ of order $\mathrm{ord} (W^{\mcK}) < \mco$ in $\mathcal{T}^{\mathrm{Ban}} \big( ( \mathcal{K}_{\mco}^{| \mfl |} )_{\mfl \in \mfL_+} \big)$.
We will use the notation $(\hat{\Pi}^{\mcK}, \hat{\Gamma}^{\mcK})$ to denote the BPHZ model on $W$ with identity kernel assigment and a given random smooth noise assignment $\xi \colon \Omega \to \mcA_-^{\infty; \eq}$.
Furthermore, given a kernel assigment $K \in \mcA_+^{\eq; \mcO}$ and parameters $\alpha, w$ satisfying the assumptions of Lemma~\ref{l:ka}, 
we will use the notation 
$(\hat{\Pi}^{C}, \hat{\Gamma}^C )$ to denote the BPHZ model on the historic sector 
(which we will denote by $W^{C}$)
of 
$\mathcal{T}^{\mathrm{Ban}} \big( ( C_w^{\alpha} )_{\mfl \in \mfL_+} )$ with
generating set of trees $\mcB_W$, 
with respect to the kernel assignment \eqref{eq:a34b} and noise assignment $\xi$. 
Similarly, when considering two noise assigments $\xi^1, \xi^2$ and two kernel assignments $K^1, K^2 \in \mcA_+^{\eq; \mcO}$, we will use the notations 
$\hat{\Pi}^{C; 1}, \hat{\Pi}^{C; 2}$
and 
$\hat{\Pi}^{\mcK; 1}, \hat{\Pi}^{\mcK; 2}$
to denote the corresponding models. 

\begin{lemma} \label{l:semi-annealedII}
Let $K \in \mcA_+^{\eq; \mcO}$ and $\tau \in \boldsymbol{\tau}\in \mcB_W$.
Assume that 
$\mcO \geq \mco + \max ( \mfs )$,
and
\begin{align} \label{eq:a43}
0 > \alpha > \frac{\max ( \mfs ) - \mcO}{| E_+ ( \tau ) |} + 2 | \mfs | ,
\qquad \int_{\mathbb{R}^d} \mathrm{d} x \, w ( x ) \, (1 + | x |_{\mfs})^{2 | \mfs |} < \infty .
\end{align}
Then for all 
$x \in \mathbb{R}^d$, 
$\lambda \in (0, 1]$, 
$\phi \in \mcB^r$
and $v \in (C_w^{\alpha} )^{\proj \sttau}$
one has the relation 
\begin{align} \label{eq:a39b}
\hat{\Pi}_x^{C} v ( \phi )
& = v \Big( c \mapsto \hat{\Pi}_x^{\mcK} \bigotimes_{\sttau}^{\tau} \big( K_{\mfl_{e}}^{c_e} \big)_{e \in E_+ ( \tau )} \big( \phi \big) \Big) ,
\end{align}
and moreover, for all $1 \leq p < \infty$
\begin{align} \label{eq:a41}
\mathbb{E}^{\frac{1}{p}} \left[ \left| \sup_{v \in ( C_w^{\alpha} )^{\proj \sttau}} \frac{| \hat{\Pi}_x^{C} v ( \phi_x^{\lambda} ) |}{\| v \|_{( C_w^{\alpha} )^{\proj \sttau}}} \right|^p \right]
& \lesssim
\| K \|_{\mcA_+^{\eq; \mcO}}^{| E_+ ( \tau ) |} 
\sup_{v \in (\mathcal{K}_{\mco}^{| \mfl |})^{\proj \sttau}} \frac{\mathbb{E}^{\frac{1}{p}} \big[ \big| \hat{\Pi}_x^{\mcK} v \big( \phi_x^{\lambda} \big) \big|^p \big]}{\| v \|_{(\mathcal{K}_{\mco}^{|\mfl|})^{\proj \sttau}} }.
\end{align}
Furthermore, 
in the case of two models, 
\begin{align} \label{eq:a41b}
& \mathbb{E}^{\frac{1}{p}} \left[ \left| \sup_{v \in ( C_w^{\alpha} )^{\proj \sttau}} \frac{| ( \hat{\Pi}_x^{C; 1} - \hat{\Pi}_x^{C; 2} ) v ( \phi_x^{\lambda} ) |}{\| v \|_{( C_w^{\alpha} )^{\proj \sttau}}} \right|^p \right]
\lesssim
\| K^1 - K^2 \|_{\mcA_+^{\eq; \mcO}}
+ \sup_{v \in (\mathcal{K}_{\mco}^{| \mfl |})^{\proj \sttau}} \frac{\mathbb{E}^{\frac{1}{p}} \big| \big( \hat{\Pi}_x^{\mcK; 1} - \hat{\Pi}_x^{\mcK; 2} \big) v \big( \phi_x^{\lambda} \big) \big|^p}{\| v \|_{(\mathcal{K}_{\mco}^{|\mfl|})^{\proj \sttau}} } ,
\end{align}
where the implicit constant can be chosen 
to depend polynomially on
\begin{align*}
\| K^1 \|_{\mcA_+^{\eq; \mcO}} + \| K^2 \|_{\mcA_+^{\eq; \mcO}} 
+ \sup_{v \in (\mathcal{K}_{\mco}^{| \mfl |})^{\proj \sttau}} \frac{\mathbb{E}^{\frac{1}{p}} \big[ \big| \hat{\Pi}_x^{\mcK; 1} v \big( \phi_x^{\lambda} \big) \big|^p \big]}{\| v \|_{(\mathcal{K}_{\mco}^{|\mfl|})^{\proj \sttau}} } 
+ \sup_{v \in (\mathcal{K}_{\mco}^{| \mfl |})^{\proj \sttau}} \frac{\mathbb{E}^{\frac{1}{p}} \big[ \big| \hat{\Pi}_x^{\mcK; 2} v \big( \phi_x^{\lambda} \big) \big|^p \big]}{\| v \|_{(\mathcal{K}_{\mco}^{|\mfl|})^{\proj \sttau}} }.
\end{align*}
\end{lemma}

\begin{proof}
The fact that 
the right-hand side of
\eqref{eq:a39} is well-defined with the bounds 
\eqref{eq:a41} and \eqref{eq:a41b}
follows from combining Lemma~\ref{l:semi-annealed} and Lemma~\ref{l:pa2}, where we have chosen $w^* ( x_1, \cdots, x_{| E_+ ( \tau ) |} ) = \prod_{i=1}^{| E_+ ( \tau ) |} ( 1 + | x_i |_{\mfs} )^{| \mfs |}$ therein.

It remains to prove that $\eqref{eq:a39}$ coincides with the BPHZ model on
the historic sector of $\mathcal{T}^{\mathrm{Ban}} ( C_w^{\alpha} )$ with generating set of trees $\mcB_W$.
More precisely we will prove by Noetherian induction 
over $\mcB_W$ with respect to $\prec$
that for every elementary symmetric tensor $\otimes_{\sttau}^\tau (\zeta_e)$ we have that
\begin{enumerate}
\item $\hat{\Pi}_x^{\times} \otimes_{\sttau}^\tau (\zeta_e)(y) = \bigotimes_{\sttau}^\tau (\zeta_e) \left [ c \mapsto \hat{\Pi}_x^{\times} \otimes_{\sttau}^\tau (K_{\mfl_e}^{c_e})(y) \right ]$,
\item $\hat{\boldsymbol{\Pi}}^\times \otimes_{\sttau}^\tau (\zeta_e)(y) = \bigotimes_{\sttau}^\tau (\zeta_e) \left [ c \mapsto \hat{\boldsymbol{\Pi}}^\times \otimes_{\sttau}^\tau (K_{\mfl_e}^{c_e})(y) \right ]$,
\item $\hat{\ell}(\bigotimes_\sttau^\tau (\zeta_e)) = \otimes_{\sttau}^\tau (\zeta_e) \left [ c \mapsto \hat{\ell}( \otimes_\sttau^\tau (K_{\mfl_e}^{c_e}) \right ]$,
\item $\hat{\boldsymbol{\Pi}}\otimes_{\sttau}^\tau (\zeta_e)(y) = \bigotimes_{\sttau}^\tau (\zeta_e) \left [ c \mapsto \hat{\boldsymbol{\Pi}} \otimes_{\sttau}^\tau (K_{\mfl_e}^{c_e})(y) \right ]$,
\item $\hat{\Pi}_x \otimes_{\sttau}^\tau (\zeta_e)(y) = \bigotimes_{\sttau}^\tau (\zeta_e) \left [ c \mapsto \hat{\Pi}_x \otimes_{\sttau}^\tau (K_{\mfl_e}^{c_e})(y) \right ]$
\end{enumerate}
where for a given tree $\boldsymbol{\tau}$, the statements will be verified in the order that they are listed above, 
and where we have omitted the superscripts $C$ and $\mcK$ in the corresponding models $\hat{\Pi}$ for notational convenience.
As usual, we treat separately the cases where $\boldsymbol{\tau} \in \{\Xi, X^k\}$, where $\boldsymbol{\tau}$ is planted and where $\boldsymbol{\tau}$ is a product of planted trees.

In the first case $\boldsymbol{\tau} \in \{\Xi, X^k\}$ so that each of the statements above follows from the definitions immediately. Therefore it suffices to fix a tree $\boldsymbol{\tau}$ which we can assume is either planted $\boldsymbol{\tau} = \mathcal{I}_\mfl^k \boldsymbol{\sigma}$ or is a non-trivial product of planted trees $\boldsymbol{\tau} = \prod_{i=1}^n \boldsymbol{\tau}_i$ with $n > 1$ and $\boldsymbol{\tau}_i = \mathcal{I}_{\mfl_i}^k \boldsymbol{\sigma}_i$.

In the former case, statement $3$ above is trivial since both sides vanish and statements $1$ and $2$ coincide with statements $4$ and $5$. Hence in this case, we can consider only the first two statements. Since the proofs are very similar (only differing by the potential presence of recentering in the admissibility condition), we demonstrate only the first of these claims in this case. We fix a drawing $\tau$ of $\boldsymbol{\tau}$ and consider $(\zeta_e)_{e\in E_+(\tau)} \in V^{\times \tau}$. We denote by $\zeta$ the value of $\zeta_e$ for $e$ being the trunk of $\tau$. We then have
\begin{align}\label{ko01}
\hat{\Pi}_x^{\times} \otimes_{\sttau}^\tau (\zeta_e)(y) 
& 
= D^k A_\mfl(\zeta) \ast \hat{\Pi}_x \otimes_{\sssigma}^\sigma (\zeta_e)(y) - \sum_{|j| < |\mcI_\mfl^k \tau|_{\mfs}} \frac{(y-x)^j}{j!} D^{k+j} A_\mfl(\zeta) \ast \hat{\Pi}_x \otimes_{\sssigma}^\sigma (\zeta_e)(x).
\end{align}
By the induction hypothesis and the definition of $A_\mfl$, we can rewrite the first term on the right hand side as
\begin{align*}
\zeta(D^k K_\mfl^{\bullet}) \ast \otimes_{\sssigma}^\sigma (\zeta_e) \left [ \hat{\Pi}_x \otimes_\sssigma^\sigma (K_{\mfl_e}^{\bullet_e})\right ](y) = \zeta \otimes \bigotimes_{\sssigma}^\sigma (\zeta_e) \left [ D^k K_\mfl^{\bullet} \ast \hat{\Pi}_x \otimes_{\sssigma}^\sigma (K_{\mfl_e}^{\bullet_e}) (y) \right ]
\end{align*}
where the distributional pairing is in the variables denoted by $\bullet$ and the convolution is in the remaining variables. Rewriting the summand in the same way, we then note that the right hand side of \eqref{ko01} is
\begin{align*}
\zeta \otimes \bigotimes_{\sssigma}^\sigma (\zeta_e) \left [ D^k K_\mfl^{\bullet} \ast \hat{\Pi}_x \otimes_{\sssigma}^\sigma (K_{\mfl_e}^{\bullet_e}) (y) \right ] - \sum_{|j|_{\mfs} < |\mcI_\mfl^k \tau|}\frac{(y-x)^j}{j!} \zeta \otimes \bigotimes_{\sssigma}^\sigma (\zeta_e) \left [ D^{k+j} K_\mfl^{\bullet} \ast \hat{\Pi}_x \otimes_{\sssigma}^\sigma (K_{\mfl_e}^{\bullet_e}) (x) \right ].
\end{align*}
By admissibility of the BPHZ model on 
$W$,
we recognise this as the right hand side of the first equation in our induction hypothesis.

It remains to treat the case where $\sttau$ is a non-trivial product of planted trees. In this case, we must check all five of our induction hypotheses in order. We now consider the first of our five hypotheses. We note that $\hat{\Pi}_x^{\times} \otimes_{\sttau}^\tau (K_{\mfl_e}^{\bullet_e})(y)$ is a $\sttau$-symmetric test function in the $\bullet_e$-variables so that by Lemma~\ref{lem:symm_test}, we can write
\begin{align*}
\bigotimes_{\sttau}^\tau (\zeta_e) \left [\hat{\Pi}_x^{\times} \otimes_{\sttau}^\tau (K_{\mfl_e}^{\bullet_e})(y) \right ] &= 
\bigotimes_{e\in E_+(\tau)} (\zeta_e) \left [\hat{\Pi}_x^{\times} \otimes_{\sttau}^\tau (K_{\mfl_e}^{\bullet_{e}})(y) \right ]  = \prod_{i = 1}^n \bigotimes_{e\in E_+(\tau_i)} (\zeta_e) \left [\hat{\Pi}_x^{\times} \otimes_{\langle \boldsymbol{\tau}_i\rangle}^{\tau_i}(K_{\mfl_e}^{\bullet_{e}})(y) \right ] 
\end{align*}
wherefor the second equality we used that by the multiplicative definition of $\hat{\Pi}_x^\times$,
\begin{align*}
\left [ c \mapsto \hat{\Pi}_x^{\times} \otimes_{\langle \boldsymbol{\tau}\rangle}^{\tau}(K_{\mfl_e}^{c_{e}})(y) \right] = \otimes_{i=1}^n \left [ c \mapsto \hat{\Pi}_x^{\times} \otimes_{\langle \boldsymbol{\tau}_i\rangle}^{\tau_i}(K_{\mfl_e}^{c_{e}})(y) \right ].
\end{align*}
Lemma~\ref{lem:symm_test} then shows that 
\begin{align*}
\bigotimes_{e\in E_+(\tau_i)} (\zeta_e) \left [\hat{\Pi}_x^{\times} \otimes_{\langle \boldsymbol{\tau}_i\rangle}^{\tau_i}(K_{\mfl_e}^{\bullet_{e}})(y) \right ] = \bigotimes_{\langle \boldsymbol{\tau}_i \rangle}^{\tau_i} (\zeta_e) \left [\hat{\Pi}_x^{\times} \otimes_{\langle \boldsymbol{\tau}_i\rangle}^{\tau_i}(K_{\mfl_e}^{\bullet_{e}})(y) \right ] 
\end{align*}
so that we may deduce from the induction hypothesis that
\begin{align*}
\bigotimes_{\sttau}^\tau (\zeta_e) \left [\hat{\Pi}_x^{\times} \otimes_{\sttau}^\tau (K_{\mfl_e}^{\bullet_e})(y) \right ] &= \prod_{i = 1}^n \hat{\Pi}_x^\times \bigotimes_{e\in E_+(\tau_i)} (\zeta_e) (y)
= \hat{\Pi}_x^\times \bigotimes_\sttau^\tau (\zeta_e)(y).
\end{align*}
This completes the proof of the first claim in our induction hypothesis in this case.

Since the second claim in the induction hypothesis follows by essentially the same approach (replacing $\hat{\Pi}_x$ by $\hat{\boldsymbol{\Pi}}$) we omit the details for this part and move to the third claim in the induction hypothesis. We note that by definition of the BPHZ model, if $\mathring{\Delta}_r^- = \Delta_r^- - \id \otimes 1 - 1 \otimes \id$ then
\begin{align*}
\hat{\ell}\left (\bigotimes_\sttau^\tau (\zeta_e) \right ) &= - \mathbb{E}\left [ \boldsymbol{\Pi}^\times \bigotimes_\sttau^\tau(\zeta_e)(0)\right ] - \sum_{\tau_1, \tau_2} \hat{\ell} \left ( \bigotimes_{\langle \boldsymbol{\tau}_1 \rangle}^{\tau_1} (\zeta_e)\right )\mathbb{E} \left [ \boldsymbol{\Pi}^\times \bigotimes_{\langle \boldsymbol{\tau}_2 \rangle}^{\tau_2}(\zeta_e)(0) \right] 
\end{align*}
By the induction hypothesis and the previous parts of this induction step, the right hand side is equal to
\begin{align*}
- \mathbb{E}\left [ \bigotimes_\sttau^\tau(\zeta_e) \left ( \boldsymbol{\Pi}^\times \bigotimes_\sttau^\tau(K_{\mfl_e}^{\bullet_e}) (0) \right ) \right ] - \sum_{\tau_1, \tau_2} \bigotimes_{\langle \boldsymbol{\tau}_1 \rangle}^{\tau_1} (\zeta_e) \left ( \hat{\ell} \left ( \bigotimes_{\langle \boldsymbol{\tau}_1 \rangle}^{\tau_1} (K_{\mfl_e}^{\bullet_e})\right ) \right )\mathbb{E} \left [\bigotimes_{\langle \boldsymbol{\tau}_2 \rangle}^{\tau_2}(\zeta_e) \left ( \boldsymbol{\Pi}^\times \bigotimes_{\langle \boldsymbol{\tau}_2 \rangle}^{\tau_2}(K_{\mfl_e}^{\bullet_e})(0) \right ) \right]
\end{align*}
By Bochner integrability (which can also be checked by Noetherian induction), we may commute expectations and the application of distributions to the effect of
\begin{align*}
\hat{\ell}\left (\bigotimes_\sttau^\tau (\zeta_e) \right ) = & - \bigotimes_\sttau^\tau(\zeta_e) \left ( \mathbb{E}\left [  \boldsymbol{\Pi}^\times \bigotimes_\sttau^\tau(K_{\mfl_e}^{\bullet_e}) (0)  \right ]\right ) - \sum_{\tau_1, \tau_2} \bigotimes_{\langle \boldsymbol{\tau}_1 \rangle}^{\tau_1} (\zeta_e) \left ( \hat{\ell} \left ( \bigotimes_{\langle \boldsymbol{\tau}_1 \rangle}^{\tau_1} (K_{\mfl_e}^{\bullet_e})\right ) \right )\bigotimes_{\langle \boldsymbol{\tau}_2 \rangle}^{\tau_2}(\zeta_e) \left (\mathbb{E} \left [ \boldsymbol{\Pi}^\times \bigotimes_{\langle \boldsymbol{\tau}_2 \rangle}^{\tau_2}(K_{\mfl_e}^{\bullet_e})(0) \right] \right ).
\end{align*}
By Lemma~\ref{lem:symm_test}, we recognise the right hand side as
\begin{align*}
& - \bigotimes_\sttau^\tau(\zeta_e) \left ( \mathbb{E}\left [  \boldsymbol{\Pi}^\times \bigotimes_\sttau^\tau(K_{\mfl_e}^{\bullet_e}) (0)  \right ]\right ) - \sum_{\tau_1, \tau_2} \bigotimes_{e \in E_+(\tau_1)} (\zeta_e) \left ( \hat{\ell} \left ( \bigotimes_{\langle \boldsymbol{\tau}_1 \rangle}^{\tau_1} (K_{\mfl_e}^{\bullet_e})\right ) \right )\bigotimes_{e \in E_+(\tau_2)}(\zeta_e) \left (\mathbb{E} \left [ \boldsymbol{\Pi}^\times \bigotimes_{\langle \boldsymbol{\tau}_2 \rangle}^{\tau_2}(K_{\mfl_e}^{\bullet_e})(0) \right] \right )
\\ & = - \bigotimes_\sttau^\tau(\zeta_e) \left ( \mathbb{E}\left [  \boldsymbol{\Pi}^\times \bigotimes_\sttau^\tau(K_{\mfl_e}^{\bullet_e}) (0)  \right ]\right ) - \sum_{\tau_1, \tau_2} \bigotimes_{e \in E_+(\tau)} (\zeta_e) \left ( \hat{\ell} \left ( \bigotimes_{\langle \boldsymbol{\tau}_1 \rangle}^{\tau_1} (K_{\mfl_e}^{\bullet_e})\right ) \mathbb{E} \left [ \boldsymbol{\Pi}^\times \bigotimes_{\langle \boldsymbol{\tau}_2 \rangle}^{\tau_2}(K_{\mfl_e}^{\bullet_e})(0) \right] \right )
\\ & =
- \bigotimes_\sttau^\tau(\zeta_e) \left ( \mathbb{E}\left [  \boldsymbol{\Pi}^\times \bigotimes_\sttau^\tau(K_{\mfl_e}^{\bullet_e}) (0)  \right ]\right ) - \bigotimes_{e \in E_+(\tau)} (\zeta_e) \left ( \mathbb{E}[ \boldsymbol{\Pi}^\times (\hat{\ell}\otimes \id) \mathring{\Delta}_r^- \bigotimes_{\sttau}^\tau (K_{\mfl_e}^{\bullet_e}) (0) ]\right ).
\end{align*}
By applying Lemma~\ref{lem:symm_test} again, we see the right hand side is nothing other than $\bigotimes_{\sttau}^\tau (\zeta_e) \left ( \hat{\ell} \left ( \bigotimes_\sttau^\tau (K_{\mfl_e}^{\bullet_e} \right )\right ) $, completing our proof of the third part of our induction step.
\end{proof}

At this stage, we know that the BPHZ model 
$(\hat{\Pi}^{C}, \hat{\Gamma}^{C})$ 
on
a historic sector $W^C$ of $\mathcal{T}^{\mathrm{Ban}} ( C_w^{\alpha} )$ 
of order $\mathrm{ord} ( W^C )<\mco$, 
with respect to the kernel assignment $A$ defined in \eqref{eq:a34b}, satisfies the `semi-annealed' estimate \eqref{eq:a41}.
We now post-process this estimate into an almost-sure estimate on the model norm of $(\hat{\Pi}, \hat{\Gamma})$, by means of a Kolmogorov-type theorem for models.
This is a standard argument once \eqref{eq:a41} is known, see e.g.\ \cite[Theorem~10.7]{Hai14}
which relies on a countable characterisation of the model norm based on wavelets,
and which we (partially) reproduce for convenience.
See also \cite[Theorem~B.1]{HS23} for an alternative wavelet-free approach which could be implemented here with a bit more technical effort.

\begin{lemma}\label{l:quenched}
Let $W^C$ be a historic sector in $\mathcal{T}^{\mathrm{Ban}} \big( (C_w^{\alpha} )_{\mfl \in \mfL_+} \big)$ 
of order $\mathrm{ord} ( W^C ) < \mco$, 
with finite generating set of trees $\mcB_W$.
Let $K \in \mcA_+^{\eq; \mcO}$.
Let $\kappa > 0$.
Assume that 
 $\mcO \geq \mco + \max ( \mfs )$,
and 
that the assumption \eqref{eq:a43} of Lemma~\ref{l:semi-annealedII} holds for all $\tau \in \boldsymbol{\tau} \in \mcB_W \cap \mcT_-(R)$.
Then
for all compact sets $\mathfrak{K} \subset \mathbb{R}^d$, and all 
$p \in [1, \infty)$,
\begin{align*}
\mathbb{E}^{\frac{1}{p}} \big[ \| ( \hat{\Pi}^{C}, \hat{\Gamma}^{C} ) \|_{W^C; \mfK}^p \big]
\lesssim 
1 ,
\end{align*}
where the implicit constant can be chosen 
to depend polynomially on
\begin{align*}
\| K \|_{\mcA_+^{\eq; \mcO}} 
+ 
\sup_{\boldsymbol{\tau} \in \mcB_W}
\sup_{x \in \mfK + B ( R_{W} )} 
\sup_{\lambda \in (0, 1]}	
\sup_{\phi \in \mcB^r}	
\sup_{v \in (\mathcal{K}_{\mco}^{| \mfl |})^{\proj \sttau}}
\frac{\mathbb{E}^{\frac{1}{q_{\boldsymbol{\tau}, p}}} \big[ \big| \hat{\Pi}_x^{\mcK} v \big( \phi_x^{\lambda} \big) \big|^{q_{\boldsymbol{\tau}, p}} \big]}{\| v \|_{(\mathcal{K}_{\mco}^{|\mfl|})^{\proj \sttau}} \, \lambda^{| \boldsymbol{\tau} |_{\mfs} + \kappa}} 
\end{align*}
for some choice of radius $R_W > 0$ and integrability exponents $q = q_{\boldsymbol{\tau}, p} \in [1, \infty)$.
Furthermore, 
\begin{align*}
& \mathbb{E}^{\frac{1}{p}} \big[ \| ( \hat{\Pi}^{C; 1}, \hat{\Gamma}^{C; 1} ) ; ( \hat{\Pi}^{C;2}, \hat{\Gamma}^{C;2} ) \|_{W^C; \mfK}^p \big]
\\
& 
\lesssim 		
\| K^1 - K^2 \|_{\mcA_+^{\eq; \mcO}}
+ \sup_{\boldsymbol{\tau} \in \mcB_W} 
\sup_{x \in \mfK + B ( R_{W} )}
\sup_{\lambda \in (0, 1]}
\sup_{\phi \in \mcB^r} 	
\sup_{v \in (\mathcal{K}_{\mco}^{| \mfl |})^{\proj \sttau}}
\frac{\mathbb{E}^{\frac{1}{q_{\boldsymbol{\tau}, p}}} \big[ \big| \big( \hat{\Pi}_x^{\mcK; 1} - \hat{\Pi}_x^{\mcK; 2} \big) v \big( \phi_x^{\lambda} \big) \big|^{q_{\boldsymbol{\tau}, p}} \big]}
{\| v \|_{(\mathcal{K}_{\mco}^{|\mfl|})^{\proj \sttau}} \, \lambda^{| \boldsymbol{\tau} |_{\mfs} + \kappa}} ,
\end{align*}
where the implicit constant can be chosen 
to depend polynomially on
\begin{align*}
\| K^1 \|_{\mcA_+^{\eq; \mcO}} + \| K^2 \|_{\mcA_+^{\eq; \mcO}} 
+ 
\sup_{i \in \lbrace 1, 2 \rbrace}
\sup_{\boldsymbol{\tau} \in \mcB_W} 
\sup_{x \in \mfK + B ( R_{W} )} 
\sup_{\lambda \in (0, 1]}	
\sup_{\phi \in \mcB^r}	
\sup_{v \in (\mathcal{K}_{\mco}^{| \mfl |})^{\proj \sttau}} \frac{\mathbb{E}^{\frac{1}{q_{\boldsymbol{\tau}, p}}} \big[ \big| \hat{\Pi}_x^{\mcK; i} v \big( \phi_x^{\lambda} \big) \big|^{q_{\boldsymbol{\tau}, p}} \big]}{\| v \|_{(\mathcal{K}_{\mco}^{|\mfl|})^{\proj \sttau}} \lambda^{| \boldsymbol{\tau} |_{\mfs} + \kappa}} .
\end{align*}

\end{lemma}

\begin{proof}
In the proof, we suppress the superscripts $C$ and $\mcK$ from the notations for convenience.
We only deal with the case of one model, since the adaptation to the comparison of two models proceeds analogously.
More precisely, we prove by Noetherian induction on $\mcB_W$ with respect to $\prec$ the property, 
which we call $\Phi ( \boldsymbol{\tau} )$,
that for all $p \in [ 1, \infty )$ and $\mfK \subset \mathbb{R}^d$, it holds that
$\mathbb{E}^{1/p} [ \| \hat{\Pi} \|_{\boldsymbol{\tau}; \mfK}^p ] \lesssim 1$ and $\mathbb{E}^{1/p} [ \| \hat{\Gamma} \|_{\boldsymbol{\tau}; \mfK}^p ] \lesssim 1$ , 
with the implicit prefactor being as in the statement.
Let us denote 
by $\Phi^{\Pi} ( \boldsymbol{\tau} )$ the claim on $\hat{\Pi}$ 
and by $\Phi^{\Gamma} ( \boldsymbol{\tau} )$ the claim on $\hat{\Gamma}$
in $\Phi ( \boldsymbol{\tau} )$.
We first observe that for all $\boldsymbol{\tau} \in \mcB_W$, 
$\lbrace \Phi ( \boldsymbol{\sigma} ) \rbrace_{\boldsymbol{\sigma} \prec \boldsymbol{\tau}}$
implies $\Phi^{\Gamma} ( \boldsymbol{\tau} )$: indeed, this is a straightforward consequence of Lemma~\ref{lem:Gam_from_Pi}
in combination with H\"older's inequality.

Thus, it remains to prove that for all $\boldsymbol{\tau} \in \mcB_W$, 
the properties 
$\lbrace \Phi ( \boldsymbol{\sigma} ) \rbrace_{\boldsymbol{\sigma} \prec \boldsymbol{\tau}}$ and $\Phi^{\Gamma} ( \boldsymbol{\tau} )$ 
imply the property $\Phi^{\Pi} ( \boldsymbol{\tau} )$.
To that effect, we divide into cases depending on the sign of $| \boldsymbol{\tau} |_{\mfs}$.
In the case where $| \boldsymbol{\tau} |_{\mfs} = 0$, then $\boldsymbol{\tau} = \mathbf{1}$ and the model bound is clear.
If $| \boldsymbol{\tau} |_{\mfs} > 0$, then by tracking for which trees\footnote{here we use the triangularity of $\hat{\Gamma}$ with respect to $\prec$} we require control of the model in the proof
of \cite[Proposition~3.31]{Hai14}, 
we have
\begin{align*}
\| \hat{\Pi} \|_{\boldsymbol{\tau}; \mfK}
& \lesssim \sup_{\boldsymbol{\sigma} \in \mcB_W, \boldsymbol{\sigma} \prec \boldsymbol{\tau}} \| \hat{\Pi} \|_{\boldsymbol{\sigma}; \bar{\mfK}} \, \|\hat{\Gamma} \|_{\boldsymbol{\tau}; \bar{\mfK}} ,
\end{align*}
whence the desired $\mathbb{E}^{1/p} [ \| \hat{\Pi} \|_{\boldsymbol{\tau}; \mfK}^p ] \lesssim 1$ by our induction hypothesis and H\"older's inequality.

Finally, consider the case where $| \boldsymbol{\tau} |_{\mfs} < 0$.
Then, 
by \cite[Proposition~3.32]{Hai14} (again tracking for which trees we need to control the model in the proof)
\begin{align*}
\| \hat{\Pi} \|_{\boldsymbol{\tau}; \mathfrak{K}}
& \lesssim \big( 1 + \| \hat{\Gamma} \|_{\boldsymbol{\tau}; \mathfrak{K}} \big) \,
\sup_{n \geq 0}  \,
\sup_{x \in \Lambda_{\mathfrak{s}}^n \cap \bar{\mathfrak{K}} } \,
\sup_{\substack{\boldsymbol{\sigma} \in \mcB_W, \\ |\boldsymbol{\sigma}|_{\mfs} \leq | \boldsymbol{\tau} |_{\mfs}}} \,
\sup_{v \in (C_w^{\alpha})^{\proj \sssigma}} 
2^{n |\boldsymbol{\sigma}|_{\mfs} } 
\frac{\big| \hat{\Pi}_x v \big( \varphi_x^{2^{-n}} \big) \big|}{\| v \|_{( C_w^{\alpha} )^{\proj \sssigma}}} .
\end{align*}
We recall that in the above
$\varphi$ denotes the scaling function of some fixed sufficiently regular multiresolution analysis,
$\Lambda_{\mathfrak{s}}^{n} = \lbrace \sum_{j = 1}^d 2^{- n \mathfrak{s}_j} k_j e_j , k_j \in \mathbb{Z} \rbrace$ is the dyadic lattice associated to the scaling $\mathfrak{s}$, 
and $\bar{\mathfrak{K}}$ is the 1-fattening of the compact set $\mathfrak{K}$.
Now taking the $p$-th power, replacing the suprema over $n, x, \boldsymbol{\sigma}$ by sums, and applying Cauchy--Schwarz, we deduce that
\begin{align*}
\mathbb{E} \big[ \| \hat{\Pi} \|_{\boldsymbol{\tau}; \mathfrak{K}}^p \big]
& \lesssim
\mathbb{E}^{\frac{1}{2}} \big( 1 + \| \hat{\Gamma} \|_{\boldsymbol{\tau}; \mathfrak{K}} \big)^{2p} \,
\sum_{n \geq 0}  \,
\sum_{x \in \Lambda_{\mathfrak{s}}^n \cap \bar{\mathfrak{K}} } \,
\sum_{\substack{\boldsymbol{\sigma} \in \mcB_W, \\ |\boldsymbol{\sigma}|_{\mfs} \leq | \boldsymbol{\tau} |_{\mfs}}} \,
2^{n p |\boldsymbol{\sigma}|_{\mfs} } 
\,
\mathbb{E}^{\frac{1}{2}} \left[ \left| \sup_{v \in (C_w^{\alpha})^{\proj \sssigma}} \frac{\big| \hat{\Pi}_x v \big( \varphi_x^{2^{-n}} \big) \big|}{\| v \|_{( C_w^{\alpha} )^{\proj \sssigma}}} \right|^{2p} \right] .
\end{align*}
On the one hand, by the induction hypothesis, 
$\mathbb{E}^{1/2} \big[ ( 1 + \| \hat{\Gamma} \|_{\boldsymbol{\tau}; \mathfrak{K}} \big)^{2p} \big] \lesssim 1$.
On the other hand, by assumption, 
we may apply Lemma~\ref{l:semi-annealedII} to $\boldsymbol{\sigma} \in \mcB_W \cap \mcT_- ( R )$, 
yielding
\begin{align*}
\mathbb{E} \big[ \| \hat{\Pi} \|_{\boldsymbol{\tau}; \mathfrak{K}}^p \big]
& \lesssim
\sum_{n \geq 0} 
2^{n ( |\mfs| - p \kappa )} ,
\end{align*}
which is finite for $p$ large enough in function of $\kappa$, which concludes the inductive step.
\end{proof}

\begin{remark}
We note that this result is actually not restricted to the choice of the BPHZ model but would apply 
more generally to any pair 
of models related by the expression \eqref{eq:a39b}.
\end{remark}
\section{(Pointed) Modelled Distributions}\label{sec:mod_dist}

In this section we discuss the behaviour of modelled distributions on $\mcT^{\mathrm{Ban}}$. 
Throughout, we will use 
the definition of modelled distributions given in \cite[Definition 3.1]{Hai14}, 
which we emphasise
is blind to the dimension of the homogeneous components of the given regularity structure and so makes perfect sense in the setting of $\mathcal{T}^\mathrm{Ban}$.
In particular, for a model $Z = ( \Pi, \Gamma )$ and a sector $W$, 
we will denote by 
$\mcD^{\gamma} ( Z, W )$
the space of modelled distributions 
with respect to $Z$ valued in $W$,
and will use the shorter notation 
$\mcD^{\gamma} ( Z)$, 
$\mcD^{\gamma} ( W)$
or simply $\mcD^{\gamma}$
whenever the sector and/or model is clear from the context.
We will denote by $\mcR$ the corresponding reconstruction map.

\subsection{Multiplication and Multi-Level Schauder on \texorpdfstring{$\mcT^{\Ban}$}{T}}

In what follows, we show that the usual calculus of modelled distributions is compatible with the norm structure on $\mcT^\mathrm{Ban}$.
Since the multiplication results of \cite[Theorem 4.7, Proposition 4.10]{Hai14} apply verbatim in our setting, we only provide a slight adaptation of the usual multi-level Schauder estimates.
\begin{theorem}\label{theo: multi-level Schauder}
	Let $V = ( V_{\mfl} )_{\mfl \in \mfL_+}$ be a Banach space assignment, 
	and let $A \in \mcA_+^{\mco}$ be a kernel assignment of order $\mco \in \mathbb{N}$.
	Let $\zeta \in V_\mfl$ with $|\mfl| = \beta$, 
	let $W \subset \mcT^{\Ban} ( V )$ be a sector of regularity $a \in \mathbb{R}$ supporting the abstract integration map $\mcI^{\zeta}$.
	Fix an admissible model $Z = (\Pi, \Gamma)$
	on $\mcT^{\Ban} ( V )$
	and $\gamma > 0$.
	For $f \in \cD^\gamma(Z, W)$, let $\cK_\gamma^\zeta f$ be defined by
    \begin{equ}
        \mcK_\gamma^\zeta f(x) = \mcI^\zeta f(x) + J^\zeta(x) f(x) + (\mcN_\gamma^\zeta f)(x)
    \end{equ}
    where $J^\zeta$, $\mcN_\gamma^\zeta$ are defined as in \cite[Equations (5.11), (5.16)]{Hai14} with respect to the kernel $A_\mfl(\zeta)$. 
    Suppose that $\gamma + \beta \not \in \mathbb{N}$
    and that 
  $\mco > \max \lbrace \gamma + \beta + \max ( \mfs ), - a \rbrace$.
    Then 
    \begin{equ}
  	  \|\cK_\gamma^\zeta f\|_{\gamma + \beta; \mfK} 
  	  \lesssim
  	  	\|\zeta\|_{V_\mfl} \|f\|_{\gamma; \bar \mfK} 
  	  	\|\Pi\|_{\gamma; \bar{\mfK}} 
  	  	(1+ \|\Gamma\|_{\gamma + \beta; \bar \mfK}) 
  	  	( 1 + \| A \|_{\mcA_+^{\mco}} ) ,
    \end{equ}
    and $\mcR \mcK_\gamma^\zeta f = A_\mfl(\zeta) \ast \mcR f$.

    Furthermore, given a second 
    kernel assignment $\bar{A} \in \mcA_+^{\mco}$,
    admissible model $\bar{Z} = (\bar \Pi, \bar \Gamma)$, 
   	modelled distribution $\bar{f} \in \cD^\gamma(\bar{Z}, W)$ 
   	and $\bar{\zeta} \in V_\mfl$ we have that
    \begin{equ}\label{eq:2ModInt}
        \|\mcK_\gamma^{\zeta} f ; \bar{\mcK}_\gamma^{\bar{\zeta}} \bar{f}\|_{\gamma + \beta; \mfK} 
        \lesssim 
        	\|f; \bar{f}\|_{\gamma; \bar \mfK} 
        	+ \|\Pi - \bar{\Pi}\|_{\gamma; \bar \mfK} + \|\Gamma - \bar{\Gamma}\|_{\gamma + \beta; \bar \mfK} 
        	+ \| \zeta - \bar \zeta\|_{V_\mfl} 
        	+ \| A - \bar{A} \|_{\mcA_+^{\mco}}
    \end{equ}
where the implicit constant can be chosen 
to depend polynomially on
    \begin{equs}
        \|f\|_{\gamma; \bar \mfK} + \|\bar{f}\|_{\gamma; \bar \mfK} + \|\Pi\|_{\gamma; \bar\mfK} + \|\bar \Pi\|_{\gamma ; \bar \mfK} 
        + \|\Gamma\|_{\gamma + \beta; \bar \mfK} 
        + 
        \|\bar \Gamma\|_{\gamma + \beta; \bar \mfK} 
        + \|\zeta\|_{V_\mfl} + \|\bar \zeta\|_{V_\mfl}
        + \| A \|_{\mcA_+^{\mco}} + \| \bar{A} \|_{\mcA_+^{\mco}}.
    \end{equs}
\end{theorem}
\begin{proof}
	This is mostly just a minor adaptation of the proof of \cite[Theorem 5.12]{Hai14} making use of the fact that 
$\|\mcI^\zeta\| + \|A_\mfl(\zeta)\|_{\mcK_{\mco}^{\beta}} \lesssim ( 1 + \| A \|_{\mcA_+^{\mco}} ) \|\zeta\|_{V_\mfl}$ 
and being careful to track the norms of the trees involved. The only real difference when proving \eqref{eq:2ModInt} is that terms in the difference of kernels appear. For example, when applying expressions of the type \cite[Equation 5.47]{Hai14} one will pick up an extra term of the form 
	\begin{equ}
		\bar{\Pi}_x \big (\mcQ_\delta(\bar{f}(x) - \Gamma_{xy} \bar{f}(y)) \big ) \big(D^k K_n(x - \cdot) - D^k \bar{K}_n(x - \cdot)\big).
	\end{equ} 
Since 
$\|K - \bar{K}\|_{\mcK_{\mco}^{\beta}} = \|A_\mfl(\zeta - \bar{\zeta}) + (A_{\mfl} - \bar{A}_{\mfl}) ( \bar{\zeta} )\|_{\mcK_{\mco}^{\beta}} \le \| A \|_{\mcA_+^{\mco}} \|\zeta - \bar{\zeta}\|_{V_\mfl} + \| \bar{\zeta} \|_{V_{\mfl}} \| A - \bar{A} \|_{\mcA_+^{\mco}}$,
the only effect of such terms is contributing instances of the final two terms on the right hand side of \eqref{eq:2ModInt}.

Finally, let us note that the condition 
$\mco > \max \lbrace \gamma + \beta + \max ( \mfs ), - a \rbrace$
does not appear in \cite{Hai14}, since therein the kernel is assumed to belong to $\bigcap_{\mco \in \mathbb{N}} \mcK_{\mco}^{\beta}$, 
but it follows from tracking the regularity of kernels that is needed in the computations (see also \cite[Theorem~4.13]{BCZ24} where such a condition appears in the case of the Euclidean scaling).
\end{proof}

We record the following corollary, 
which immediately follows 
by applying \eqref{eq:2ModInt} with $Z = \bar{Z}$ and $f = \bar{f}$.
\begin{corollary} \label{cor:sl}
In the setting of
Theorem~\ref{theo: multi-level Schauder}, 
the map 
    	\begin{align*}
		\begin{array}[t]{lrcl}
		& L ( V_{\mfl}, \mcK_{\mco}^{| \mfl |} ) \times V_{\mfl} & \longrightarrow & L(\mcD^\gamma( Z, W ) , \mcD^{\gamma + | \mfl |}( Z, W ) ) \\
		& ( A_{\mfl} , \zeta ) & \longmapsto & \mcK_{\gamma}^{\zeta} ,
		\end{array}  
		\end{align*} 
is
affine\footnote{The reason why this map is affine and not linear in $A_{\mfl}$ is due to the presence of the term $\mcI^{\zeta}$ in the definition of $\mcK_{\gamma}^{\zeta}$.}
in $A_{\mfl}$, 
linear in $\zeta$ 
and continuous.
\end{corollary}

We now recall from \cite{HS23} the notion of pointed modelled distribution, which will be the key analytic tool to push models on $\mcT^\mathrm{Ban}$ to models on $\mcT^\mathrm{eq}$. Here, our setting is in one way somewhat simpler than that of \cite{HS23} since we only need H\"older pointed modelled distributions.

\begin{definition}
\label{d:pointed_md}
    Given $x \in \mathbb{R}^d$, $\gamma, \nu \in \mathbb{R}$ and a model $Z = (\Pi, \Gamma)$, we define $\mcD^{\gamma, \nu; x}(Z)$ to be the space of those $f \in \mcD^\gamma(Z)$ such that\footnote{in line with the language of \cite{HL17}, we will refer to the leftmost estimate here as the `local bound' and the rightmost estimate as the `translation bound'}  
    \begin{equ}
        \sup_{y \in B_{\mfs} (x, \lambda)} \|f(y)\|_{\zeta} \lesssim \lambda^{\nu - \zeta}, \qquad \sup_{y \in B_{\mfs}(x, \lambda)} \sup_{|h|_{\mfs} \le \lambda} \frac{\|f(y+h) - \Gamma_{y+h, y} f(y) \|_{\zeta}}{|h|_{\mfs}^{\gamma - \zeta}} \lesssim \lambda^{\nu - \gamma} 
    \end{equ}
    uniformly over $\lambda \in (0,1]$. We define $\|f\|_{\gamma, \nu; x}$ to be the smallest implicit constant in the above.

    Given a second model $\bar{Z} = (\bar \Pi, \bar \Gamma)$ and $\bar{f} \in \mcD^{\gamma, \nu; x}(\bar{Z})$, we write
      \begin{equ}
        \|f; \bar{f}\|_{\gamma, \nu; x} 
        = \sup_{\zeta < \gamma} \sup_{\lambda \in (0,1]} \sup_{y \in B_{\mfs}(x, \lambda)} \frac{\|(f - \bar{f})(y)\|_{\zeta}}{\lambda^{\nu - \zeta}} 
        	+ \sup_{\zeta < \gamma} \sup_{\lambda \in (0,1]} \sup_{y \in B_{\mfs}(x, \lambda)} \sup_{|h|_{\mfs} \le \lambda} \frac{\|\bar{\Delta}_h f(y) \|_{\zeta}}{|h|_{\mfs}^{\gamma - \zeta} \lambda^{\nu - \gamma} }  
    \end{equ}
    where 
    $$\bar{\Delta}_h f(y) = f(y+h) - \Gamma_{y+h, y} f(y) - \bar{f}(y+h) + \bar{\Gamma}_{y+h, y} f(y).$$
\end{definition}

As in the non-pointed case, the standard multiplication results apply verbatim in our setting
since the proof and statement of \cite[Proposition 3.11]{HS23} apply even in the case of Banach components.
We now reproduce the contents of 
this statement
for the convenience of the reader, 
and refer to \cite{HS23} for the proof.
\begin{theorem}
Let $V = ( V_{\mfl} )_{\mfl \in \mfL_+}$ be a Banach space assignment.
Let $\gamma_1, \nu_1, \gamma_2, \nu_2 \in \mathbb{R}$.
Let $W_1, W_2$ be sectors of $\mcT^{\Ban} ( V )$ of regularity $a_1$ and $a_2 \in \mathbb{R}$
such that $(W_1, W_2)$ is $\gamma$-regular, where 
$\gamma=(\gamma_1+a_2) \wedge (\gamma_2+a_1)$.
Given a model $Z = (\Pi, \Gamma)$
and modelled distributions 
$f_1 \in \mcD^{\gamma_1, \nu_1; x} (Z, W_1)$, 
$f_2 \in \mcD^{\gamma_2, \nu_2; x} (Z, W_2)$, 
then
	\begin{align*}
		f_1 \star_{\gamma} f_2 
		\coloneqq \mcQ_{< \gamma} ( f_1 \star f_2 ) 
		\in \mcD^{\gamma, \nu_1+\nu_2; x} ( Z ) ,
	\end{align*}
with 
	\begin{align*}
		\| f_1 \star_{\gamma} f_2 \|_{\gamma, \nu_1+\nu_2; x}
		& \lesssim 
			\| f_1 \|_{\gamma_1, \nu_1; x} \| f_2 \|_{\gamma_2, \nu_2; x} ( 1 + \| \Gamma \|_{\gamma_1 \vee \gamma_2, B_{\mfs} ( x, 2 )} )^2 .
	\end{align*}
Furthermore, given a second model 
$\bar{Z} = ( \bar{\Pi}, \bar{\Gamma} )$
and modelled distributions
$\bar{f}_1 \in \mcD^{\gamma_1, \nu_1; x} (\bar{Z}, W_1)$, 
$\bar{f}_2 \in \mcD^{\gamma_2, \nu_2; x} (\bar{Z}, W_2)$,
then 
	\begin{align*}
		\| f_1 \star_{\gamma} f_2;  \bar{f}_1 \star_{\gamma} \bar{f}_2 \|_{\gamma, \nu_1+\nu_2; x}
		& \lesssim 
			\| f_1; \bar{f}_1 \|_{\gamma_1, \nu_1; x} 
			+ \| f_2 ; \bar{f}_2 \|_{\gamma_2, \nu_2; x}
			+ \| \Gamma - \bar{\Gamma} \|_{\gamma_1 \vee \gamma_2, B_{\mfs} ( x, 2 )} ,
	\end{align*}
where the implicit constant can be chosen 
to depend polynomially on
	\begin{align*}
	\| f_1 \|_{\gamma_1, \nu_1; x} + \| \bar{f}_1 \|_{\gamma_1, \nu_1; x} 
			+ \| f_2 \|_{\gamma_2, \nu_2; x} + \| \bar{f}_2 \|_{\gamma_2, \nu_2; x} 
			+ \| \Gamma \|_{\gamma_1 \vee \gamma_2, B_{\mfs} ( x, 2 )} + \| \bar{\Gamma} \|_{\gamma_1 \vee \gamma_2, B_{\mfs} ( x, 2 )}.
	\end{align*}
\end{theorem}
In this work, we will need a variant of the reconstruction theorem that is not included in \cite{HS23}. This is because \cite{HS23} is primarily interested in a $\mcB_{2, \infty}^\gamma$-type variant of pointed modelled distributions and needs to obtain strong estimates for H\"older type norms of the reconstruction which can be expected only for a limited range of values of $\gamma$ (see \cite[Theorem 3.15]{HS23} for a precise statement). 
Since here we use H\"older-type norms for both distributions and modelled distributions that restriction can be lifted.
\begin{theorem}\label{theo: pointed recon}
    Suppose that $\gamma > 0$, $\nu \in \mbR$ and $f \in \mcD^{\gamma, \nu; x}$ takes values in a sector of regularity $a \le 0$. 
    Then
    \begin{equs}
        |\langle \mcR f, \phi_x^\lambda \rangle | 
        \lesssim \lambda^{\nu} \|f\|_{\gamma, \nu; x} \|\Pi\|_{\gamma; B_{\mfs} (x, 2)} \big( 1 + \|\Gamma\|_{\gamma; B_{\mfs} (x, 2 )}^2 \big) ,
    \end{equs}
uniformly over $\lambda \in (0, 1]$ and $\phi \in \mcB^r$, 
provided $r > - a$.

Furthermore, given a second model $\bar{Z} = ( \bar{\Pi}, \bar{\Gamma} )$
with associated reconstruction operator $\bar{\mcR}$, 
and a second modelled distribution $\bar{f} \in \mcD^{\gamma, \nu; x}$, then
 	\begin{equs}
        |\langle \mcR f - \bar{\mcR} \bar{f}, \phi_x^\lambda \rangle | 
        \lesssim \lambda^{\nu}
        	\big(  \|f; \bar{f} \|_{\gamma, \nu; x} + \|\Pi - \bar{\Pi} \|_{\gamma; B_{\mfs} (x, 2)}
        	+ \|\Gamma - \bar{\Gamma} \|_{\gamma; B_{\mfs} (x, 2 )}
        	\big) ,
    \end{equs}
    where the implicit constant can be chosen 
to be uniform
over $\lambda \in (0, 1]$, $\phi \in \mcB^r$ 
(again provided $r > - a$) and
to depend polynomially on
	\begin{align*}
		\|f\|_{\gamma, \nu; x} + \|\bar{f}\|_{\gamma, \nu; x} + \|\Pi\|_{\gamma; B_{\mfs}(x,2)} + \|\bar \Pi\|_{\gamma ; B_{\mfs}(x,2)} + \|\Gamma\|_{\gamma; B_{\mfs}(x,2)} + \|\bar \Gamma\|_{\gamma; B_{\mfs}(x,2)}.
	\end{align*}
\end{theorem}
\begin{proof}
We give the proof only in the case of a single model since that of two models follows similarly.
    As in the proof of \cite[Theorem 3.15]{HS23}, by localisation we can find $f_\lambda \in \mcD^{\gamma, \nu; x}$ such that $f_\lambda$ is supported in $B_{\mfs}(x,2\lambda)$, $f_\lambda = f$ on $B_{\mfs}(x,\lambda)$ and 
$$\|f_\lambda\|_{\gamma; B_{\mfs}(x, 2)} \lesssim \lambda^{\nu - \gamma} \|f\|_{\gamma, \nu; x} (1 + \|\Gamma\|_{\gamma; B_{\mfs} (x, 2 )}^2)$$ 
where $\|\cdot\|_{\gamma; \mfK}$ is the usual seminorm for H\"older modelled distributions.

    Then by the Reconstruction Theorem in its original form \cite[Theorem 3.10]{Hai14}, we deduce that
    \begin{equs}
        |\langle \mcR f - \Pi_x f(x), \phi_x^\lambda \rangle | 
        &= |\langle \mcR f_\lambda - \Pi_x f_\lambda(x), \phi_x^\lambda \rangle | 
        \lesssim \lambda^\nu \| \Pi \|_{\gamma; B_{\mfs}(x,2)} \|f\|_{\gamma, \nu; x} \big( 1 + \|\Gamma\|_{\gamma; B_{\mfs} (x, 2 )}^2 \big).
    \end{equs}
    We then note that the local bound for pointed modelled distributions implies that $\mcQ_{< \nu} f(x) = 0$. 
    Meanwhile, for $\zeta \ge \nu$ the local bound with $\lambda = 1$ implies that $\|f(x)\|_{\zeta} \le \|f\|_{\gamma, \nu; x}$. In particular,
    we can deduce that
    \begin{equs}
        |\langle \Pi_x f(x) , \phi_x^\lambda \rangle | 
        &\lesssim \sum_{\nu \le \zeta < \gamma} \lambda^{\zeta} \| \Pi \|_{\gamma; B_{\mfs} ( x, 2 )} \|f\|_{\gamma, \nu; x}
        \lesssim \lambda^\nu \| \Pi \|_{\gamma; B_{\mfs}(x, 2)} \|f\|_{\gamma, \nu; x} .
    \end{equs}
    The result then follows by combining these two estimates and the triangle inequality.
\end{proof}

We now provide an adaptation of the
pointed multi-level Schauder estimates of\footnote{we note that the condition $0<\gamma<\min(W\setminus\mathbb{N})$ appearing in \cite[Theorem~3.19]{HS23} is superfluous in our setting due to the absence of the corresponding parameter restriction in our reconstruction result.} \cite[Theorem~3.19]{HS23}.
\begin{theorem} \label{thm:pointed_mls} 
	Let $V = ( V_{\mfl} )_{\mfl \in \mfL_+}$ be a Banach space assignment, 
	and let 
	$A \in \mcA_+^{\mco}$
	be a kernel assignment of order $\mco \in \mathbb{N}$.
	Let $\zeta \in V_\mfl$ with $|\mfl| = \beta$, 
	let $W \subset \mcT^{\Ban} ( V )$ be a sector of regularity $a \leq 0$ supporting the abstract integration map $\mcI^{\zeta}$.
	Fix an admissible model $Z = (\Pi, \Gamma)$
	on $\mcT^{\Ban} ( V )$
	and $\gamma > 0$, $\nu \in \mathbb{R}$. 
	For $f \in  \cD^{\gamma, \nu; x}(Z)$ let $\cK_{\gamma, \nu}^{\zeta, x} f$ be defined by
    \begin{equ}
        \mcK_{\gamma,\nu}^{\zeta,x} f(y) = \mcI^\zeta f(y) + J^\zeta(y) f(y) + (\mcN_\gamma^\zeta f)(y) - T_{\gamma, \nu}^{\zeta, x} f(y)
    \end{equ}
    where
    \begin{equ} 
        T_{\gamma, \nu}^{\zeta, x} f(y) = \mcQ_{< \gamma + \beta} \left ( \sum_{|k|_{\mfs} < \nu + \beta} \frac{(X + y - x)^k}{k!} (D^k A_\mfl(\zeta) \ast \mcR f)(x) \right ) .
    \end{equ}
    Suppose that $\gamma + \beta \not \in \mathbb{N}$
    and that 
$\mco > \max \lbrace \gamma + \beta + \max ( \mfs ), \nu + \beta + \max ( \mfs ) , - a \rbrace$.
    Then
    \begin{equ}
		\|\cK_{\gamma, \nu}^{\zeta,x} f\|_{\gamma + \beta, \nu + \beta; x} 
		\lesssim
		\|\zeta\|_{V_\mfl} 
		\|f\|_{\gamma, \nu; x} 
		\|\Pi\|_{\gamma; B_{\mfs}(x,2)}
		(1+ \|\Gamma\|_{\gamma + \beta; B_{\mfs}(x,2)}) 
		( 1 + \| A \|_{\mcA_+^{\mco}}), 
	\end{equ}
	and 
	\begin{equs}
		\mcR \cK_{\gamma, \nu}^{\zeta,x} f
		& = A_{\mfl} ( \zeta ) \ast \mcR f
			- \sum_{| k |_{\mfs} < \nu + \beta} \frac{(\cdot - x)^k}{k!} (D^k A_\mfl(\zeta) \ast \mcR f)(x) .
	\end{equs}
	
    Furthermore, given a second 
    kernel assignment $\bar{A} \in \mcA_+^{\mco}$,
    admissible model $\bar{Z} = (\bar \Pi, \bar \Gamma)$, 
	modelled distribution $\bar{f} \in \cD^{\gamma, \nu; x}(\bar{Z})$ 
	and $\bar{\zeta} \in V_\mfl$ we have that
    \begin{equs}
        \|\cK_{\gamma, \nu}^{\zeta,x} f ; \bar{\cK}_{\gamma, \nu}^{\bar{\zeta},x} \bar{f}\|_{\gamma + \beta, \nu + \beta; x} 
        \lesssim 
        	\|f; \bar{f}\|_{\gamma, \nu; x} 
        	+ \|\Pi - \bar{\Pi}\|_{\gamma; B_{\mfs}(x,2)} 
        	+ \|\Gamma - \bar{\Gamma}\|_{\gamma + \beta; B_{\mfs}(x,2)} 
        	+ \| \zeta - \bar \zeta\|_{V_\mfl}
        	+ \| A - \bar{A} \|_{\mcA_+^{\mco}} ,
    \end{equs}
    where the implicit constant can be chosen 
to depend polynomially on
    \begin{equs}
        \|f\|_{\gamma, \nu; x} + \|\bar{f}\|_{\gamma, \nu; x} 
        + \|\Pi\|_{\gamma; B_{\mfs}(x,2)} + \|\bar \Pi\|_{\gamma ; B_{\mfs}(x,2)} 
        + & \|\Gamma\|_{\gamma + \beta; B_{\mfs}(x,2)} + \|\bar \Gamma\|_{\gamma + \beta; B_{\mfs}(x,2)} 
        \\
        & 
        + \|\zeta\|_{V_\mfl} + \|\bar \zeta\|_{V_\mfl} 
		+ \| A \|_{\mcA_+^{\mco}} + \| \bar{A} \|_{\mcA_+^{\mco}}
        .
    \end{equs}
\end{theorem}
\begin{proof}
    This follows by a minor adaptation of the proof of \cite[Theorem 3.19]{HS23} in the same way that the proof of Theorem~\ref{theo: multi-level Schauder} is an adaptation of the proof of \cite[Theorem 5.12]{Hai14}.
    We recall that the fact that the convolution in $T_{\gamma, \nu}^{\zeta, x}$ is well-defined is not immediately obvious, however it can be justified thanks to the assumption on $\mco$,
    see \cite[Remark~3.17]{HS23}.
\end{proof}

\subsection{Variable Coefficient Abstract Integration on \texorpdfstring{$\mcT^{\Ban} ( C_w^{\alpha})$}{T}}\label{ss:VML_Schauder}

Our eventual goal is to construct a model on $\mcT^\mathrm{eq}$ by defining $\Pi^\mathrm{eq}$ as the reconstruction of suitable modelled distributions on $\mcT^\mathrm{Ban} ( C_w^{\alpha} )$. 
Admissibility of this model on $\mcT^\mathrm{eq}$ will be defined with respect to 
a given
kernel assignment $K \in \mcA_+^{\eq}$.
As a result, we need to lift convolution with 
$K_{\mfl}$
(as a function of $z, \bar{z})$ to an operation on modelled distributions on $\mcT^\mathrm{Ban}$. 
In this subsection,
we assume to have fixed a regularity exponent $\alpha \in \mathbb{R}\setminus \mathbb{N}$ and a weight $w$,
and define the corresponding 
Banach space assignment  
$V = (C_w^\alpha)_{\mfl \in \mfL_+}$, 
as introduced in Subsection~\ref{ss:quenched} above.
We also fix a kernel assignment $K \in \mcA_+^{\eq; \mcO}$ in $\mcT^{\eq}$ of given order $\mcO \in \mathbb{N}$.
We then endow the
Banach space assignment $V = (C_w^\alpha)_{\mfl \in \mfL_+}$ 
with the kernel assignment $A$ defined
from $K$ by 
Lemma~\ref{l:ka}.
We recall
that it is well-defined and of order $\mco \in \mathbb{N}$
provided 
$\mcO \geq \mco + \max (\mfs)$, $\mcO \geq - \alpha + \max (\mfs)$ and $\int w < \infty$,
which we assume to hold throughout this subsection.
We also fix some $\mfl \in \mfL_+$ and denote $\beta = |\mfl|$.

\begin{theorem}\label{theo: VML Schauder}
Let $Z = (\Pi, \Gamma)$ be an admissible model
on $\mcT^{\Ban} ( C_w^\alpha )$ with respect to the kernel assignment $A$, 
and let $W$ be a sector of regularity $a \le 0$ supporting the abstract integration maps $\mcI^{\zeta}$.
	Fix $\gamma > 0$ 
	and $m \in \mathbb{N}$. 
	For $f \in \cD^{\gamma}(Z, W)$,
	define $\cK_{\gamma, m} f$ via
	\begin{equ} \label{eq:a44}
		\cK_{\gamma, m} f(x) = \mcQ_{<\bar{\gamma}} \left ( \sum_{|k|_{\mfs} < m} \frac{X^k}{k!} \cK_{\gamma}^{(-\partial)^k \delta_x} f(x) \right ) . 
	\end{equ}
    where $\bar{\gamma} = \min\{\gamma + \beta, m,  m + a + \beta\}$.
    Assume also that $\gamma + \beta \notin \mathbb{N}$, 
    $\bar{\gamma} > 0$,
   $0<m\leq - \alpha - |\mfs| - \max ( \mfs )$
    and 
$\mco > \max \lbrace \gamma+\beta + \max ( \mfs ), -a \rbrace$.
	  Then
	\begin{align}
		\|\cK_{\gamma, m} f\|_{{\bar\gamma}; \mfK} \lesssim \|f\|_{\gamma; \bar \mfK} \|\Pi\|_{\gamma; \bar\mfK} (1+ \|\Gamma\|_{\gamma + \beta; \bar \mfK}) 
		( 1 + \| K \|_{\mcA_+^{\eq; \mcO}})
		\label{eq:1ModVInt}
	\end{align}
    and 
	\begin{align*}
		\mcR \mcK_{\gamma,m} f(\phi) 
		= K_{\mfl} \mcR f(\phi) 
		\coloneqq \mcR f \left ( \int K_{\mfl}^{z}(z-\cdot) \phi(z) \, \mathrm{d}z \right ).
	\end{align*}	    

    Furthermore, given a second
   	kernel assignment $K \in \mcA_+^{\eq; \mcO}$,
    admissible model $\bar{Z} = (\bar \Pi, \bar \Gamma)$
    and modelled distribution $\bar{f} \in \cD^\gamma(\bar{Z})$ we have that
    \begin{equ}\label{eq:2ModVInt}
        \|\mcK_{\gamma, m} f ; \bar{\mcK}_{\gamma, m} \bar{f}\|_{{\bar{\gamma}}; \mfK} 
        \lesssim 
        	\|f; \bar{f}\|_{\gamma; \bar \mfK} 
        	+ \|\Pi - \bar{\Pi}\|_{\gamma; \bar \mfK} 
        	+ \|\Gamma - \bar{\Gamma}\|_{\gamma + \beta; \bar \mfK}
        	+ \| K - \bar{K} \|_{\mcA_+^{\eq; \mcO}}
    \end{equ}
    where the implicit constant can be chosen 
to depend polynomially on
    \begin{equ}
        \|f\|_{\gamma; \bar \mfK} + \|\bar{f}\|_{\gamma; \bar \mfK} 
        + \|\Pi\|_{\gamma; \bar\mfK} + \|\bar \Pi\|_{\gamma ; \bar \mfK} 
        + \|\Gamma\|_{\gamma + \beta; \bar \mfK} + \|\bar \Gamma\|_{\gamma + \beta; \bar \mfK}
        + \| K \|_{\mcA_+^{\eq; \mcO}} + \| \bar{K} \|_{\mcA_+^{\eq; \mcO}}.
    \end{equ}
\end{theorem}

\begin{remark} \label{rem:on_m}
We have stated Theorem~\ref{theo: VML Schauder} in terms of an arbitrary parameter $m \in \mathbb{N}$ for the sake of generality.
Note however that the expression \eqref{eq:a44}
does not depend on $m$ as soon as 
$m \geq \gamma + \max \lbrace \beta, - a \rbrace$.
\end{remark}

\begin{remark}
The condition  $m\leq - \alpha - |\mfs| - \max ( \mfs )$ in 
Theorem~\ref{theo: VML Schauder}
is used to ensure that $(-\partial)^k \delta_x \in C_w^{\alpha}$ for all $| k |_{\mfs} < m$, 
so that in turn $\mcK_{\gamma}^{(-\partial)^k \delta_x}$ is well-defined on $\mcT^{\Ban} ( C_w^{\alpha} )$ by Theorem~\ref{theo: multi-level Schauder}.
\end{remark}

\begin{proof}[Proof of Theorem~\ref{theo: VML Schauder}]
    We begin with the single model case and first consider the local bound. For a compact set $\mfK \subseteq \mbR^{d}$, $x, y \in \mfK$ and $\eta < \bar{\gamma}$, we write, 
by definition of the modelled-distribution norm,
    \begin{equs}
        \left \| \sum_{|k|_1 < m} \frac{X^k}{k!} \cK_\gamma^{(-\partial)^k \delta_x} f(x) \right \|_\eta 
       & \le \sum_{|k|_1 < m} \sup_{x \in \mfK} \| \cK_\gamma^{(-\partial)^k \delta_x} f \|_{{\gamma + \beta}; \mfK} \\
       & 
      	\leq \sum_{|k|_1 < m} \sup_{x \in \mfK}
      	\|(-\partial)^k \delta_x\|_{C_w^{\alpha}} 
      	\|f\|_{\gamma; \bar \mfK} 
  	  	\|\Pi\|_{\gamma; \bar{\mfK}} 
  	  	(1+ \|\Gamma\|_{\gamma + \beta; \bar \mfK}) 
  	  	( 1 + \| K \|_{\mcA_+^{\eq; \mcO}} ) ,
	\end{equs}
where in the second inequality we have applied Theorem~\ref{theo: multi-level Schauder}.
The desired bound follows from noting that 
$\sup_{x \in \mfK} \| (-\partial)^k \delta_x \|_{C_w^{\alpha}} < \infty$, 
since we have assumed 
 $m\leq - \alpha - |\mfs| - \max ( \mfs )$.

    We turn to the translation bound.
    Using the multiplicativity of $\Gamma$, we first decompose
    \begin{equs}
        \cK_{\gamma, m} f(x) - \Gamma_{xy} \cK_{\gamma, m} f(y) 
        = & T_1 + T_2 ,
    \end{equs}
   	where 
   	\begin{align*}
   		T_1 & \coloneqq \mcQ_{< \bar{\gamma}} \left [ \sum_{|k|_{\mfs} < m} \frac{X^k}{k!} \cK_{\gamma}^{(-\partial)^k \delta_x} f(x) - \frac{\Gamma_{xy} X^k}{k!} \Gamma_{xy} \cK_{\gamma}^{(-\partial)^k \delta_y} f(y)\right ] , \\ 
        T_2 & \coloneqq \left (\mcQ_{<\bar{\gamma}} \Gamma_{xy} - \Gamma_{xy} \mcQ_{< \bar{\gamma}} \right ) \sum_{|k|_{\mfs} < m} \frac{X^k}{k!} \cK_{\gamma}^{(-\partial)^k \delta_y} f(y) .
   	\end{align*}
	The desired bound on the `truncation error' $T_2$ follows as for the local bound after noting that $\mcQ_{<\bar{\gamma}} \Gamma_{xy} - \Gamma_{xy} \mcQ_{< \bar{\gamma}} = \mcQ_{<\bar{\gamma}} \Gamma_{x y} \mcQ_{\geq\bar{\gamma}}$.  
    Hence we focus on $T_1$, 
    which we further decompose as
    \begin{equs}
    T_1  
    & = \sum_{|k|_{\mfs} < m} \frac{X^k}{k!} \left [ \cK_\gamma^{(-\partial)^k \delta_x} f(x) - \sum_{|k+l|_{\mfs} < m} \frac{(x-y)^l}{l!} \Gamma_{xy} \cK_\gamma^{(-\partial)^{k+l} \delta_y} f(y) \right ] \\
 	& = \sum_{| k |_{\mfs} < m} \frac{X^k}{k!} ( T_{1 , 1} ( k ) + T_{1 , 2} ( k ) ) ,
	\end{equs}
    with
    \begin{equs}
        T_{1,1}(k) & \coloneqq \cK_\gamma^{(-\partial)^k \delta_x} f(x) - \sum_{|k+l|_{\mfs} < m} \frac{(x-y)^l}{l!} \cK_\gamma^{(-\partial)^{k+l} \delta_y} f(x) 
        = \mcK_{\gamma}^{\zeta_{k, m}^{x, y}} f ( x ) , \\
        T_{1,2}(k) & \coloneqq \sum_{|k+l|_{\mfs} < m} \frac{(x-y)^l}{l!} \left ( \cK_\gamma^{(-\partial)^{k+l} \delta_y} f(x) - \Gamma_{xy} \cK_\gamma^{(-\partial)^{k+l} \delta_y} f(y) \right) ,
    \end{equs}
    where in the definition of $T_{1, 1}$ we have written
	\begin{equ} \label{eq:dz1}
    	\zeta_{k, m}^{x, y} \coloneqq (-\partial)^k \delta_x - \sum_{|k + l|_{\mfs} < m} \frac{(x-y)^l}{l!} (-\partial)^{k+l} \delta_y 
    \end{equ}    
    and used linearity of the map $\zeta \mapsto \mathcal{K}_\gamma^\zeta$.  We note that the powers of $-1$ appearing in
    \eqref{eq:dz1}
    are exactly such that this term acts on test functions via Taylor expansion.
    We bound the terms in $T_{1,1}$ and $T_{1,2}$ separately.
    On the one hand, the desired bound on $T_{1,2}$ follows from  applying Theorem~\ref{theo: multi-level Schauder} to $\cK_\gamma^{(-\partial)^k \delta_y} f$ with $y$ fixed.
	On the other hand, applying again Theorem~\ref{theo: multi-level Schauder} 
	we obtain
    \begin{equs}
        \left \| \sum_{|k|_{\mfs} < m} \frac{X^k}{k!} T_{1,1}(k) \right \|_\eta  
       & \le \sum_{|k|_{\mfs} < m} \left \| T_{1,1}(k) \right \|_{\eta - |k|} \\ 	
       & \lesssim \left \| \zeta_{k, m}^{x, y} \right \|_{C_w^\alpha} \|f\|_{\gamma; \bar{\mfK}} \|\Pi\|_{\gamma; \bar{\mfK}} (1 + \|\Gamma\|_{\gamma + \beta; \bar \mfK})
       ( 1 + \| K \|_{\mcA_+^{\eq; \mcO}} )  \\ 
        & \lesssim_{\mathfrak{K}} |x-y|_{\mfs}^{m - |k|_{\mfs}} \|f\|_{\gamma; \bar{\mfK}} \|\Pi\|_{\gamma; \bar{\mfK}} (1 + \|\Gamma\|_{\gamma + \beta; \bar \mfK}) ( 1 + \| K \|_{\mcA_+^{\eq; \mcO}} ) ,
    \end{equs}
	where
    in the third inequality we made use of the anisotropic Taylor formula of Lemma~\ref{Tay},
    the fact that $x,y \in \mfK$ 
   	and the assumption  $m\leq - \alpha - |\mfs| - \max ( \mfs )$.
     Finally, we note that the summands in the left-hand side only contribute if $\eta > |k|_{\mfs} + (a + \beta) \wedge 0$ since the latter term is a lower bound on the degree of any symbol in $X^k T_{1,1}(k)$. Hence this bound is of order $|x-y|_{\mfs}^{m + (a +\beta) \wedge 0 - \eta}$ as required.   

We turn to the proof that $\cR \cK_{\gamma, m} f = K \cR f$ 
(where we write $K$ in place of $K_{\mfl}$ for notational convenience).
Fix $\phi \in C_c^{\infty}$.
By the uniqueness part of the Reconstruction Theorem, it suffices to prove that 
$| ( \Pi_x \cK_{\gamma, m} f ( x ) - K \cR f ) ( \phi_x^{\lambda} ) | = o_{\lambda \to 0}(1)$
locally uniformly in $x$.
We write
	\begin{equs}
		\Pi_x  ( \cK_{\gamma, m} f ( x ) ) ( \phi_x^{\lambda} )
		& = \sum_{|k|_{\mfs} < m} \Pi_x \left( \frac{X^k}{k!} \cK_{\gamma}^{(-\partial)^k \delta_x} f(x) \right) ( \phi_x^{\lambda} ) 
		- \Pi_x \left( \mcQ_{\geq \bar{\gamma}} \sum_{|k|_{\mfs} < m} \frac{X^k}{k!} \cK_{\gamma}^{(-\partial)^k \delta_x} f(x) \right) ( \phi_x^{\lambda} ) .  
	\end{equs}
The second term
in the right-hand-side
is already of order $\lambda^{\bar{\gamma}}$
(locally uniformly in $x$) 
by the local bounds and the definition of a model.
Thus, we may consider only the first term, 
for which we write
by the assumption of admissibility of $(\Pi, \Gamma)$ (see Definition~\ref{d:adm_on_tban})
	\begin{equs}
		\sum_{|k|_{\mfs} < m} \Pi_x \left( \frac{X^k}{k!} \cK_{\gamma}^{(-\partial)^k \delta_x} f(x) \right) ( \phi_x^{\lambda} ) 
		& = \sum_{|k|_{\mfs} < m} \Pi_x \left( \cK_{\gamma}^{(-\partial)^k \delta_x} f(x) \right) \left( \frac{(\cdot - x)^k}{k !} \phi_x^{\lambda} (\cdot) \right) .
	\end{equs}
Now applying the reconstruction statement of the usual multi-level Schauder estimates of Theorem~\ref{theo: multi-level Schauder} at fixed $x$ and $k$,
we arrive at
	\begin{equs}
		\big( \Pi_x (\cK_{\gamma, m} f ( x ) ) -  K \cR f \big) ( \phi_x^{\lambda} )
		& = \tilde{K}^{x; m} \cR f ( \phi_x^{\lambda} ) + O ( \lambda^{\bar{\gamma}} ) ,
	\end{equs}
where we have written for $z, \bar{z} \in \mathbb{R}^d$
	\begin{equs}
		\tilde{K}^{x; m} (z, \bar{z}) 
		& \coloneqq - K^{z} ( z - \bar{z} ) + \sum_{|k|_{\mfs} < m} \frac{( z - x )^k}{k !} \partial_x^k K^{x} ( z - \bar{z} )
		= - K^{\zeta_{0, m}^{z, x}} ( z - \bar{z} ) ,
	\end{equs}
where $\zeta_{0, m}^{z, x} \in C_{w}^{\alpha}$ is defined in \eqref{eq:dz1}.
The remainder of the argument is analogous to the classical Schauder estimates, we spell out some of the details for completeness.
Using the dyadic decomposition of $K$, we express
	\begin{align*}
		\tilde{K}^{x; m} \cR f ( \phi_x^{\lambda} )
		& = \sum_{n \geq 0} \cR f \left(  \int \mathrm{d} z \, \phi_x^{\lambda} ( z ) K_n^{\zeta_{0, m}^{z, x}} ( z - \cdot ) \right) .
	\end{align*}
Thanks to our assumptions, we can readily check that 
	\begin{align*}
		\int \mathrm{d} z \, \phi_x^{\lambda} ( z ) K_n^{\zeta_{0, m}^{z, x}} ( z - \bar{z} )
		& = \lambda^m 2^{- n \beta} \, \big( \psi^{[x, \lambda, n]} \big)_{x}^{2^{- n} \vee \lambda} ( \bar{z} ) ,
	\end{align*}
for test functions $\psi^{[x, \lambda, n]}$ such that 
$\mathrm{supp} ( \psi^{[x, \lambda, n]} ) \subset B_{\mfs} ( 0, 2)$
and $\sup_{x \in \mfK} \sup_{\lambda \in (0, 1]} \sup_{n \geq 0} \| \psi^{[x, \lambda, n]} \|_{C^{\mco}} < \infty$.
Since $\mcR f \in C^{a}$ and $\mco > - a$, we deduce after summing the geometric series that 
$\tilde{K}^{x; m} \cR f ( \phi_x^{\lambda} ) = O ( \lambda^{\bar{\gamma}} )$
locally uniformly in $x$, yielding the desired reconstruction bound.

We now turn to the case of two models, where for brevity we provide only a sketch of the translation bound since the local bound follows by easier applications of the same ideas. We write
\begin{equs}
   & \mcK_{\gamma, m} f(x) - \Gamma_{xy} \mcK_{\gamma, m} f(y) - \bar{\mcK}_{\gamma, m} \bar{f}(x) + \bar{\Gamma}_{xy} \bar{\mcK}_{\gamma, m} \bar{f}(y) \\ 
    & = \mcQ_{<\bar{\gamma}} \sum_{|k|_{\mfs}< m} \Bigg[ \frac{X^k}{k!} (\mcK_\gamma^{(-\partial)^k \delta_x} f(x) - \bar{\mcK}_\gamma^{(-\partial)^k \delta_x} \bar{f}(x)) 
    	- \frac{\Gamma_{xy} X^k}{k!} (\Gamma_{xy} \mcK_\gamma^{(-\partial)^k \delta_y} f(y) - \bar{\Gamma}_{xy} \bar{\mcK}_\gamma^{(-\partial)^k \delta_y} \bar{f}(y)) \Bigg] \\
    & \quad + (\mcQ_{< \bar{\gamma}} \Gamma_{xy} - \Gamma_{xy} \mcQ_{<\bar{\gamma}}) \mcK_{\gamma, m}f(y) 
    - (\mcQ_{< \bar{\gamma}} \bar{\Gamma}_{xy} - \bar{\Gamma}_{xy} \mcQ_{<\bar{\gamma}}) \bar{\mcK}_{\gamma, m}\bar{f}(y)
\end{equs}
where we have used that $\Gamma$ and $\bar{\Gamma}$ have the same action on the polynomial structure.

Since the first 
term in the right-hand side
is treated via the same strategy as in the single model case (making use of the parts of Theorem~\ref{theo: multi-level Schauder} that apply to the case of two models 
and two kernel assignments), 
we treat only the 
last line of the above display.
We can write
\begin{equs}
    \big\|(\mcQ_{< \bar{\gamma}} \Gamma_{xy} - \Gamma_{xy} \mcQ_{<\bar{\gamma}}) \mcK_{\gamma, m}f(y) & - (\mcQ_{< \bar{\gamma}} \bar{\Gamma}_{xy} - \bar{\Gamma}_{xy} \mcQ_{<\bar{\gamma}}) \bar{\mcK}_{\gamma, m}\bar{f}(y) \big\|_\eta \\ 
	    & \le \big\| (\mcQ_{< \bar{\gamma}} \Gamma_{xy} - \Gamma_{xy} \mcQ_{<\bar{\gamma}} - \mcQ_{< \bar{\gamma}} \bar{\Gamma}_{xy} + \bar{\Gamma}_{xy} \mcQ_{<\bar{\gamma}}) \mcK_{\gamma, m} f(y) \big\|_\eta \\ 
   		& \quad + \big\| (\mcQ_{< \bar{\gamma}} \bar{\Gamma}_{xy} - \bar{\Gamma}_{xy} \mcQ_{<\bar{\gamma}}) (\mcK_{\gamma, m}f(y) - \bar{\mcK}_{\gamma, m} \bar{f}(y)) \big\|_\eta .
\end{equs}
Noting that $\mcQ_{< \bar{\gamma}} \Gamma_{xy} - \Gamma_{xy} \mcQ_{< \bar{\gamma}} = \mcQ_{<\bar{\gamma}} \Gamma_{xy} \mcQ_{\ge \bar{\gamma}}$ (and similarly for the first term on the right hand side of the above display), we conclude via the local bounds that the right hand side of the previous display is bounded by a term of order 
\begin{equ}
    \big(
    \| \Pi - \bar{\Pi} \|_{\gamma, \bar{\mfK}} +
    \|\Gamma - \bar{\Gamma}\|_{\gamma + \beta; \bar\mfK} + \|f; \bar{f}\|_{\gamma; \bar{\mfK}}
    + \| K - \bar{K} \|_{\mcA_+^{\eq; \mcO}}
    \big) \, |x-y|_{\mfs}^{\bar{\gamma} - \eta}
\end{equ}
as required.
\end{proof}

We now turn to the corresponding result for pointed modelled distributions. We note that the proof of the above result relied only on the constant coefficient abstract integration and the fact that the map $\zeta \mapsto \mcK_\gamma^\zeta$ is a continuous linear map. This suggests that one should define
\begin{align*}
	\tilde{\cK}_{\gamma, \nu, m}^{x} f(y) 
	= \mcQ_{<\bar{\gamma}} \left ( \sum_{|k|_{\mfs} < m} \frac{X^k}{k!} \cK_{\gamma, \nu}^{(-\partial)^k \delta_y, x} f(y) \right ) .
\end{align*}
However, we are ultimately interested in defining a model on $\mcT^\mathrm{eq}$ by setting
	$\Pi_x \boldsymbol{\tau} = \mathcal{R} f_x^{\boldsymbol{\tau}}$
for a suitably defined pointed modelled distribution $f_x^{\boldsymbol{\tau}}$ on $\mcT^\mathrm{Ban} ( C_w^{\alpha} )$.
Given $f_x^{\boldsymbol{\tau}}$, it is natural to postulate that $f_x^{\mcI \boldsymbol{\tau}} = \tilde{\cK}_{\gamma, \nu, m}^{x} f_x^{\boldsymbol{\tau}}$. We note that with the above definition, this would lead to
\begin{align*}
	\Pi_x \mcI \boldsymbol{\tau} ( y )
	= K^y \ast \Pi_x \boldsymbol{\tau} (y) - \sum_{|k|_{\mfs} < |\mcI \boldsymbol{\tau}|_{\mfs}} \frac{(y-x)^k}{k!} D^k K^y \ast \Pi_x \boldsymbol{\tau} (x)
\end{align*}
whilst the correct admissibility condition would be
\begin{align*}
	\Pi_x \mcI \boldsymbol{\tau} ( y ) &= [K\Pi_x \boldsymbol{\tau}](y) - \sum_{|k|_{\mfs} < |\mcI \boldsymbol{\tau}|_{\mfs}} \frac{(y-x)^k}{k!} D^k [K \Pi_x \boldsymbol{\tau}] (x) \\ &= K^y \ast \Pi_x \boldsymbol{\tau}(y) - \sum_{|k|_{\mfs} < |\mcI \boldsymbol{\tau}|_{\mfs}} \frac{(y-x)^k}{k!} D^k [z \mapsto K^z \ast \Pi_x \boldsymbol{\tau} (z)]_{z = x}
	\\ & = K^y \ast \Pi_x \boldsymbol{\tau} (y) - \sum_{|k|_{\mfs} < |\mcI \boldsymbol{\tau}|_{\mfs}} \frac{(y-x)^k}{k!} \sum_{j \le k} {k \choose j} D^j K^{(-\partial)^{k-j} \delta_x} \ast \Pi_x \boldsymbol{\tau} (x).
\end{align*}
This means that we will have to tweak the definition of our variable coefficient abstract integration to realise the correct Taylor jet. Since it will be convenient to use $\tilde{K}_{\gamma, \nu, m}^x$ as a reference point when obtaining the desired estimates, we first record the multi-level Schauder estimate for $\tilde{K}_{\gamma, \nu, m}^x$. Since the result follows from the constant coefficient case in essentially the same way as the proof of Theorem~\ref{theo: VML Schauder}, we omit the proof.

\begin{theorem}\label{theo: VML Pointed Schauder 1}
Let $Z = (\Pi, \Gamma)$ be an admissible model
on $\mcT^{\Ban} ( C_w^\alpha )$ with respect to the kernel assignment $A$, 
and let $W$ be a sector of regularity $a \le 0$ supporting the abstract integration maps $\mcI^{\zeta}$.
	Fix $\gamma > 0$, $\nu \in \mathbb{R}$ 
	and $m \in \mathbb{N}$. 
	For $f \in \cD^{\gamma, \nu; x}(Z, W)$,
	define $\tilde{\cK}_{\gamma, \nu, m}^x f$ via
	\begin{equ}
		\tilde{\cK}_{\gamma, \nu, m}^{x} f(y) = \mcQ_{<\bar{\gamma}} \left ( \sum_{|k|_{\mfs} < m} \frac{X^k}{k!} \cK_{\gamma, \nu}^{(-\partial)^k \delta_y, x} f(y) \right ) .
	\end{equ}
	where $\bar{\gamma} = \min\{\gamma + \beta, m,  m + a + \beta\}$.
	Assume also that $\gamma + \beta \notin \mathbb{N}$, 
    $\bar{\gamma} > 0$,
$0<m\leq -\alpha - | \mfs | - \max ( \mfs )$
    and 
$\mco > \max \lbrace \gamma+\beta + \max (\mfs) , \nu+\beta + \max ( \mfs ), -a \rbrace$.
	Then
	\begin{equ}
	\|\tilde{\cK}_{\gamma, \nu, m}^x f\|_{{\bar\gamma}, \nu + \beta; x} \lesssim \|f\|_{\gamma, \nu; x} \|\Pi\|_{\gamma; B_{\mfs}(x,2)} (1+ \|\Gamma\|_{\gamma + \beta; B_{\mfs}(x,2)}) 
	( 1 + \| K \|_{\mcA_+^{\eq; \mcO}}) .
	\end{equ}
	
	 Furthermore, given a second
	 kernel assignment $\bar{K} \in \mcA_+^{\eq; \mcO}$,
    admissible model $\bar{Z} = (\bar \Pi, \bar \Gamma)$
    and modelled distribution $\bar{f} \in \cD^{\gamma, \nu; x}(\bar{Z})$ we have that
	\begin{equ}
		\|\tilde{\mcK}_{\gamma, \nu, m}^x f ; \tilde{\bar{\mcK}}_{\gamma, \nu, m}^x \bar{f}\|_{{\bar\gamma}, \nu + \beta; x} 
		\lesssim 
			\|f; \bar{f}\|_{\gamma, \nu; x} 
			+ \|\Pi - \bar{\Pi}\|_{\gamma; B_{\mfs}(x,2)} 
			+ \|\Gamma - \bar{\Gamma}\|_{\gamma + \beta; B_{\mfs}(x,2)} 
			+ \| K - \bar{K} \|_{\mcA_+^{\eq; \mcO}}
	\end{equ}
	where the implicit constant can be chosen to depend polynomially on
	\begin{equ}
		\|f\|_{\gamma, \nu; x} + \|\bar{f}\|_{\gamma, \nu; x} 
		+ \|\Pi\|_{\gamma; B_{\mfs}(x,2)} + \|\bar \Pi\|_{\gamma ; B_{\mfs}(x,2)} 
		+ \|\Gamma\|_{\gamma + \beta; B_{\mfs}(x,2)} + \|\bar \Gamma\|_{\gamma + \beta; B_{\mfs}(x,2)}
		+ \| K \|_{\mcA_+^{\eq; \mcO}} + \| \bar{K} \|_{\mcA_+^{\eq; \mcO}}.
	\end{equ}
\end{theorem}

With this preparatory result in hand, we now turn to correcting the $x$-centered Taylor jet realised by the abstract integration map. 

\begin{theorem}\label{theo: VML Pointed Schauder 2}
	Let $Z = (\Pi, \Gamma)$ be an admissible model
	on $\mcT^{\Ban} ( C_w^\alpha )$ with respect to the kernel assignment $A$, 
and let $W$ be a sector of regularity $a \le 0$ supporting the abstract integration maps $\mcI^{\zeta}$.
	Fix $\gamma > 0$, $\nu \in \mathbb{R}$ 
	and $m \in \mathbb{N}$. 
	For $f \in \cD^{\gamma, \nu; x}(Z, W)$,
	define $\cK_{\gamma, \nu, m}^x f$ via
	\begin{equ}
		\cK_{\gamma, \nu, m}^{x} f(y) = \mcQ_{<\bar{\gamma}} \left ( \sum_{|k|_{\mfs} < m} \frac{X^k}{k!} \cK_{\gamma}^{(-\partial)^k \delta_y} f(y)  - \sum_{j, l: |j+l|_{\mfs} < \nu + \beta} \frac{\Gamma_{yx} X^{j+l}}{j! l!} D^j K^{(- \partial)^{l} \delta_x} \ast \mathcal{R}f(x) \right ) .
	\end{equ}
	where $\bar{\gamma} = \min\{\gamma + \beta, m,  m + a + \beta\}$.
	Assume also that $\gamma + \beta \notin \mathbb{N}$, 
    $\bar{\gamma} > 0$,
 $\nu + \beta \leq m\leq -\alpha - | \mfs | - \max ( \mfs )$
    and 
$\mco > \max \lbrace \gamma+\beta + \max ( \mfs ), \nu + \beta + \max ( \mfs ) , - a \rbrace$.
	Then
	\begin{equ}
	\|\cK_{\gamma, \nu, m}^x f\|_{{\bar\gamma}, \nu + \beta; x} \lesssim \|f\|_{\gamma, \nu; x} \|\Pi\|_{\gamma; B_{\mfs}(x,2)} (1+ \|\Gamma\|_{\gamma + \beta; B_{\mfs}(x,2)}) ( 1 + \| K \|_{\mcA_+^{\eq; \mcO}}) , 
	\end{equ}
	and
	\begin{equ}
		\mcR \cK_{\gamma, \nu, m}^x f
		= K_{\mfl} \mcR f -\sum_{| k |_{\mfs}<\nu+\beta} \frac{( \cdot - x )^k}{k!} D^k ( K_{\mfl} \mcR f ) ( x ) .
	\end{equ}

	 Furthermore, given a second
	 kernel assignment $\bar{K} \in \mcA_+^{\eq; \mcO}$,
    admissible model $\bar{Z} = (\bar \Pi, \bar \Gamma)$
    and modelled distribution $\bar{f} \in \cD^{\gamma, \nu; x}(\bar{Z})$ we have that
	\begin{equ}
		\|\mcK_{\gamma, \nu, m}^x f ; \bar{\mcK}_{\gamma, \nu, m}^x \bar{f}\|_{{\bar\gamma}, \nu + \beta; x} 
		\lesssim 
			\|f; \bar{f}\|_{\gamma, \nu; x} 
			+ \|\Pi - \bar{\Pi}\|_{\gamma; B_{\mfs}(x,2)} 
			+ \|\Gamma - \bar{\Gamma}\|_{\gamma + \beta; B_{\mfs}(x,2)} 
			+ \| K - \bar{K} \|_{\mcA_+^{\eq; \mcO}}
	\end{equ}
	where the implicit constant can be chosen to depend polynomially on
	\begin{equ}
		\|f\|_{\gamma, \nu; x} + \|\bar{f}\|_{\gamma, \nu; x} 
		+ \|\Pi\|_{\gamma; B_{\mfs}(x,2)} + \|\bar \Pi\|_{\gamma ; B_{\mfs}(x,2)} 
		+ \|\Gamma\|_{\gamma + \beta; B_{\mfs}(x,2)} + \|\bar \Gamma\|_{\gamma + \beta; B_{\mfs}(x,2)} 
		+ \| K \|_{\mcA_+^{\eq; \mcO}} + \| \bar{K} \|_{\mcA_+^{\eq; \mcO}}.
	\end{equ}
\end{theorem}

\begin{remark} 
Just as in Remark~\ref{rem:on_m}, 
note that although we have stated 
Theorems~\ref{theo: VML Pointed Schauder 1} and \ref{theo: VML Pointed Schauder 2} 
in terms of an arbitrary parameter $m \in \mathbb{N}$,
these statements do not depend on $m$ anymore as soon as
$m \geq \max \lbrace \gamma + \beta, \gamma - a, \nu + \beta \rbrace$.
\end{remark}

\begin{proof}
We give the proof only in the case of a single model since that of two models follows similarly.
	We note that we can rewrite
	\begin{align*}
		\mcK_{\gamma, \nu, m}^x f(y) = \tilde{\mcK}_{\gamma, \nu, m}^x f(y) + \mcQ_{< \bar{\gamma}} \Bigg ( \sum_{|k|_{\mfs}< m} & \sum_{|j|_{\mfs} < \nu + \beta} \frac{X^k}{k!} \frac{\Gamma_{yx} X^j}{j!} D^j K^{(-\partial)^k \delta_y} \ast \mathcal{R} f(x) \\ & - \sum_{j,l: |j+l|_{\mfs} < \nu + \beta} \frac{\Gamma_{yx} X^{j+l}}{j! l!} D^j K^{(- \partial)^{l} \delta_x} \ast \mathcal{R}f(x) \Bigg ).
	\end{align*}
	By Theorem~\ref{theo: VML Pointed Schauder 1}, it suffices for us to consider $H(y) = \mcK_{\gamma, \nu, m}^x f(y) - \tilde{\mcK}_{\gamma, \nu, m}^x f(y)$.
	
	For this, we first consider the local bounds. Since $H(y)$ is valued in the polynomial part of the structure, this means that we want to estimate its coefficients in each $X^\eta$ term for $|\eta|_{\mfs} < \bar{\gamma}$. We write $\mcQ_{X^\eta} H(y) = T_1 + T_2$ where
	\begin{align*}
		T_1 &= \mcQ_{X^\eta} \Bigg( \sum_{|j|_{\mfs} < \nu + \beta} \frac{\Gamma_{xy} X^j}{j!} \sum_{k : |k+j|_{\mfs} < \nu + \beta} \frac{X^k}{k!} \Bigg [ D^j K^{(-\partial)^k \delta_y} \ast \mathcal{R} f(x) 
		\\
		& \qquad \qquad \qquad \qquad \qquad \qquad 
		- \sum_{l : |k+j+l|_{\mfs} < \nu + \beta} \frac{(y-x)^l}{l!} D^j K^{(-\partial)^{k+l} \delta_x} \ast \mathcal{R} f(x) \Bigg ] \Bigg )  \\
		T_2 & = \mcQ_{X^\eta} \left (\sum_{|j|_{\mfs} < \nu + \beta} \sum_{\nu + \beta \le |k+j|_{\mfs} < m + | j |_{\mfs} } \frac{X^k}{k!} \frac{\Gamma_{yx} X^j}{j!} D^j K^{(- \partial)^k 
			\delta_y} \ast \mathcal{R}f(x) \right ).
	\end{align*}
	We treat the terms $T_1$ and $T_2$ separately. For $T_1$, we write
	\begin{align*}
		T_1 = \sum_{j,k: |j+k|_{\mfs} < \nu + \beta} {j \choose \eta - k} \frac{(y-x)^{j + k - \eta}}{j!k!} D^j K^{\zeta_{k,j}^{x,y}} \ast \mathcal{R}f(x) 
	\end{align*}
	with 
	\begin{align} \label{mls04}
		\zeta_{k,j}^{x,y} = (- \partial)^k \delta_y - \sum_{l: |k+j + l|_{\mfs}< \nu + \beta} \frac{(y-x)^l}{l!} (-\partial)^{k+l} \delta_x.
	\end{align}
	We now note that for any $\zeta \in C_w^\alpha$, by
	recalling 
	Definition~\ref{d:rk2n}
	and
	Theorem~\ref{theo: pointed recon} we have 
	for $| j |_{\mfs}<\nu+ \beta$
	that 
	\begin{equs}
		|D^j K^\zeta \ast \mathcal{R}f(x)| &= \left |\sum_{n \ge 0} \mathcal{R}f (D^j K_n^\zeta(x - \cdot))\right |
		\\
		& \lesssim \sum_{n \ge 0} 2^{-n ( \beta + \nu - |j|_{\mfs})} \|\Pi\|_{\gamma; B_{\mfs} (x, 2 )}
		( 1 +\| \Gamma \|_{\gamma; B_{\mfs} (x, 2 )} )
		\|f\|_{\gamma, \nu; x} \|K^\zeta\|_{\mcK_r^\beta} \\
		&
		\lesssim \|\Pi\|_{\gamma; B_{\mfs} (x, 2 )}
		( 1 +\| \Gamma \|_{\gamma; B_{\mfs} (x, 2 )} )
		\|f\|_{\gamma, \nu; x} 
		\| K \|_{\mcA_+^{\eq; \mcO}} 
		\| \zeta \|_{C_w^{\alpha}}, \label{mls03}
	\end{equs}
	where we used Lemma~\ref{l:ka} and also the 
	assumption on $\mco$
	so that we could appeal to Theorem~\ref{theo: pointed recon}. 
 	
	Since
	$\|\zeta_{k,j}^{x,y}\|_{C_w^\alpha} \lesssim |y-x|_{\mfs}^{\nu + \beta - |k + j|_{\mfs}}$ 
	uniformly in $y \in B_{\mfs} ( x, 1 )$, 
	by a straightforward computation using the 
	assumption 
$\nu + \beta \leq - \alpha - | \mfs | - \max ( \mfs )$,
	we get
	\begin{align}\label{mls01}
		|D^j K^{\zeta_{k,j}^{x,y}} \ast \mathcal{R}f(x) | 
		\lesssim 
		\|\Pi\|_{\gamma; B_{\mfs} (x, 2 )}
		( 1 +\| \Gamma \|_{\gamma; B_{\mfs} (x, 2 )} )
		\|f\|_{\gamma, \nu; x}
		\| K \|_{\mcA_+^{\eq; \mcO}} \,
		|y-x|_{\mfs}^{\nu + \beta - |k+j|_{\mfs}} ,
	\end{align}
	which gives us	
	\begin{align*}
		|T_1| 
		&\lesssim \sum_{j,k: |j+k|_{\mfs} < \nu + \beta} 
		\|\Pi\|_{\gamma; B_{\mfs} (x, 2 )}
		( 1 +\| \Gamma \|_{\gamma; B_{\mfs} (x, 2 )} )
		\|f\|_{\gamma, \nu; x}
		\| K \|_{\mcA_+^{\eq; \mcO}} \,
		|y-x|_{\mfs}^{\nu + \beta - |k+j|_{\mfs}} |y-x|_{\mfs}^{|j + k|_{\mfs} - \eta} 
		\\
		& \lesssim 
		\|\Pi\|_{\gamma; B_{\mfs} (x, 2 )}
		( 1 +\| \Gamma \|_{\gamma; B_{\mfs} (x, 2 )} )
		\|f\|_{\gamma, \nu; x}
		\| K \|_{\mcA_+^{\eq; \mcO}} \,
		|y-x|_{\mfs}^{\nu + \beta - \eta} 
	\end{align*}
	as required. To complete the proof of the local bounds it remains to treat $T_2$. 
	To this effect, we write
	\begin{align*}
		|T_2| &= \left | \sum_{|j|_{\mfs} < \nu + \beta} \sum_{k : \nu + \beta \le |j+k|_{\mfs} < m + | j |_{\mfs}} {j \choose \eta - k} \frac{(y-x)^{j + k - \eta}}{j! k!} D^j K^{(- \partial)^k \delta_y} \ast \mathcal{R}f(x) \right |
		\\
		& \lesssim |y-x|_{\mfs}^{\nu + \beta - \eta} 
			\|\Pi\|_{\gamma; B_{\mfs} (x, 2 )}
			( 1 +\| \Gamma \|_{\gamma; B_{\mfs} (x, 2 )} )
			\|f\|_{\gamma, \nu; x} 
			\| K \|_{\mcA_+^{\eq; \mcO}}
	\end{align*}
	as required,
	where we used \eqref{mls03} and the fact that $\|(- \partial)^{k} \delta_y\|_{C_w^\alpha} \lesssim 1$ over $y \in B_{\mfs}(x, 1)$.
	
	It now remains to treat the translation bounds. 
	We define
	\begin{align*}
		\mathring{H}(y) = \sum_{|k|_{\mfs}< m} & \sum_{|j|_{\mfs} < \nu + \beta} \frac{X^k}{k!} \frac{\Gamma_{yx} X^j}{j!} D^j K^{(-\partial)^k \delta_y} \ast \mathcal{R} f(x)  
		- \sum_{j,k: |j+k|_{\mfs} < \nu + \beta} \frac{\Gamma_{yx} X^{j+k}}{j!k!} D^j K^{(- \partial)^{k} \delta_x} \ast \mathcal{R}f(x) 
	\end{align*}
	so that
	$H = \mcQ_{< \bar{\gamma}} \mathring{H}$ and
	\begin{align}\label{mls02}
		\mcQ_{X^\eta} [H(y) - \Gamma_{yz}H(z)] 
		& = \mcQ_{X^{\eta}} [\mathring{H}(y) - \Gamma_{yz}\mathring{H}(z)]
		+ \mcQ_{X^\eta} (\mcQ_{< \bar{\gamma}} \Gamma_{yz} - \Gamma_{yz}\mcQ_{< \bar \gamma}) \mathring{H}(z). 
	\end{align}
	We now estimate each term on the right hand side 
	separately, 
	starting with the `truncation error'
	\begin{align*}
		\mcQ_{X^\eta} (\mcQ_{< \bar{\gamma}} \Gamma_{yz} - \Gamma_{yz}\mcQ_{< \bar \gamma}) \mathring{H}(z) 
		= \mcQ_{X^\eta} \Gamma_{yz} \mcQ_{\ge \bar{\gamma}} \mathring{H}(z). 
	\end{align*}
	We now note that we can further decompose $\mathring{H}(y) = \mathring{H}_1(y) + \mathring{H}_2(y)$ where
	\begin{align*}
		\mathring{H}_1(y) &= \sum_{\substack{j,k \\ |j+k|_{\mfs} < \nu + \beta}} \frac{X^k}{k!} \frac{\Gamma_{yx} X^j}{j!} D^j K^{(-\partial)^k \delta_y} \ast \mathcal{R} f(x) - \frac{\Gamma_{yx} X^{j+k}}{j! k!} D^j K^{(- \partial)^{l} \delta_x} \ast \mathcal{R}f(x), 
		\\
		\mathring{H}_2(y) &= \sum_{|j|_{\mfs} < \nu + \beta} \sum_{k: \nu + \beta \le |j + k|_{\mfs} \le m + |j|_{\mfs}} \frac{X^k}{k!} \frac{\Gamma_{yx}X^j}{j!} D^j K^{(-\partial)^k \delta_y} \ast \mathcal{R} f(x). 
	\end{align*}
	We will again estimate each of these terms separately. 
	For $\mathring{H}_1$, we note that
	\begin{align*}
		\mathring{H}_1(y) 
		= \sum_{\substack{j,k \\ |j+k|_{\mfs} < \nu + \beta}} \frac{X^k}{k!} \frac{\Gamma_{yx} X^j}{j!} D^j K^{\zeta_{k,j}^{x,y}} \ast \mathcal{R} f(x) 
	\end{align*}
	with $\zeta_{k,j}^{x,y}$ defined in \eqref{mls04}.
	A short computation shows that
	\begin{align*}
		\mcQ_{X^\eta} \Gamma_{yz} \mcQ_{\ge \gamma} \mathring{H}_1(z) = \sum_{\substack{j,k \\ |j+k|_{\mfs} < \nu + \beta}} \sum_{\substack{l \le j \\ |l|_{\mfs} \ge \bar{\gamma} - |k|_{\mfs}}} \sum_{\substack{\hat{k}, \hat{l} \\ |\hat{k} + \hat{l}|_{\mfs} = \eta}} \frac{1}{\hat{k}! \hat{l}!} \frac{(y-z)^{k+l - \hat{k} - \hat{l}}}{(k-\hat{k})! (l - \hat{l})!} \frac{(z-x)^{j - l}}{(j-l)!}
		D^j K^{\zeta_{k,j}^{x,z}} \ast \mathcal{R} f(x) . 
	\end{align*}
	Therefore, making use of \eqref{mls01}, we obtain the estimate
	\begin{align*}
		|\mcQ_{X^\eta} \Gamma_{yz} \mcQ_{\ge \gamma} \mathring{H}_1(z)|
		& \lesssim \|\Pi\|_{\gamma; B_{\mfs} (x, 2 )}
			( 1 +\| \Gamma \|_{\gamma; B_{\mfs} (x, 2 )} )
			\|f\|_{\gamma, \nu; x}
			\| K \|_{\mcA_+^{\eq; \mcO}} \, \\
		& \qquad \qquad \sum_{\substack{j,k \\ |j+k|_{\mfs} < \nu + \beta}} \sum_{\substack{l \le j \\ |l|_{\mfs} \ge \bar{\gamma} - |k|_{\mfs}}} \sum_{\substack{\hat{k}, \hat{l} \\ |\hat{k} + \hat{l}|_{\mfs} = \eta}}
		|y-z|_{\mfs}^{|k+l|_{\mfs} - \eta} |z-x|_{\mfs}^{|j-l|_{\mfs} + \nu + \beta - |j+k|_{\mfs}}.
	\end{align*}
	Since we are interested in the regime $|y-z|_{\mfs}, |z-x|_{\mfs} \le \lambda$, we have the bound
	\begin{align*}
			|\mcQ_{X^\eta} \Gamma_{yz} \mcQ_{\ge \gamma} \mathring{H}_1(z)| 
			\lesssim 
			\|\Pi\|_{\gamma; B_{\mfs} (x, 2 )}
			( 1 +\| \Gamma \|_{\gamma; B_{\mfs} (x, 2 )} )
			\|f\|_{\gamma, \nu; x}
			\| K \|_{\mcA_+^{\eq; \mcO}} \,
			|y-z|_{\mfs}^{\bar{\gamma} - \eta} \lambda^{\nu + \beta - \bar{\gamma}}
	\end{align*}
	which was the desired estimate. Therefore to complete our estimate of the term coming from the truncation error, it remains to estimate the term coming from $\mathring{H}_2$. 
	We can again compute that
	\begin{align*}
		& \mcQ_{X^\eta} \Gamma_{yz}\mcQ_{\ge \bar{\gamma}} \mathring{H}_2(z) \\
		& = \sum_{|j|_{\mfs} \le \nu + \beta} \sum_{\substack{k \\ \nu + \beta \le |j+k|_{\mfs} \le m + |j|_{\mfs}}} \sum_{\substack{l \le j \\ |l|_{\mfs} \ge \bar{\gamma} - |k|_{\mfs}}} \sum_{|n|_{\mfs} = \eta} {l+k \choose n} 
		\frac{(z-x)^{j-l} (y-z)^{k+l - \eta}}{k! l! (j-l)!} 
		D^j K^{(-\partial)^k \delta_z} \ast \mathcal{R}f(x).
	\end{align*}
From this, we immediately obtain the estimate
\begin{align*}
	|\mcQ_{X^\eta} \Gamma_{yz}\mcQ_{\ge \bar{\gamma}} \mathring{H}_2(z)| 
	\lesssim 
	\|\Pi\|_{\gamma; B_{\mfs} (x, 2 )}
			( 1 +\| \Gamma \|_{\gamma; B_{\mfs} (x, 2 )} )
			\|f\|_{\gamma, \nu; x}
			\| K \|_{\mcA_+^{\eq; \mcO}} \,
	|y-z|_{\mfs}^{\bar{\gamma} - \eta} \lambda^{\nu + \beta - \bar{\gamma}}
\end{align*}
on the domain of $y,z$ of interest. This completes our treatment of the truncation error. 

We are now left to treat the first term on the right hand side of \eqref{mls02}. This term can be rewritten as
\begin{align*}
	\mcQ_{X^{\eta}} [\mathring{H}(y) - \Gamma_{yz}\mathring{H}(z)]
	&
	=
	\mcQ_{X^\eta} \Bigg ( \sum_{|j|_{\mfs}< \nu + \beta} \sum_{|k|_{\mfs}< m} \frac{X^k}{k!} \frac{\Gamma_{yx} X^j }{j!} \Bigg [ D^j K^{\tilde{\zeta}_{k,j}^{y,z}} \ast \mathcal{R}f(x)\Bigg ] \Bigg )
	\\
	& = \sum_{|j|_{\mfs}< \nu + \beta} \sum_{|k|_{\mfs}< m} \sum_{\substack{0 \le l \le j \\ |l+k|_{\mfs} = \eta}} \frac{(y-x)^{j-l} }{k! l! (j-l)!} D^j K^{\tilde{\zeta}_{k,j}^{y,z}} \ast \mathcal{R}f(x) ,
\end{align*}
where
\begin{align*}
 	\tilde{\zeta}_{k,j}^{y,z} = (-\partial)^k \delta_y - \sum_{l: |k+j+l|_{\mfs} < m} \frac{(y-z)^l}{l!} (-\partial)^{k+l} \delta_z.
\end{align*}
We deduce that
	\begin{align*}
		& | \mcQ_{X^{\eta}} [\mathring{H}(y) - \Gamma_{yz}\mathring{H}(z)] | \\
		& \lesssim \|\Pi\|_{\gamma; B_{\mfs} (x, 2 )}
			( 1 +\| \Gamma \|_{\gamma; B_{\mfs} (x, 2 )} )
			\|f\|_{\gamma, \nu; x} 
			\| K \| _{\mcA_+^{\eq; \mcO}}
			\sum_{|j|_{\mfs} < \nu + \beta} \sum_{|k|_{\mfs}< m} \sum_{\substack{0 \le l \le j \\ |l+k|_{\mfs} = \eta}}
			|y-x|_{\mfs}^{|j-l|_{\mfs}} |y-z|_{\mfs}^{m- |k|_{\mfs}} ,
	\end{align*}
which concludes the proof since
the assumption 
$m \geq \nu + \beta$
implies that
each of the summands is of order
$| y - z |_{\mfs}^{\bar{\gamma} - \eta} \lambda^{\nu + \beta - \bar{\gamma}}$.
\end{proof}
\section{Renormalised models on \texorpdfstring{$\mcT^{\mathrm{eq}}$}{T eq}} \label{s:ren_models}

In this section, we show that one can leverage the calculus of pointed modelled distributions provided in the previous section to obtain models on the usual scalar-valued regularity structure $\mcT^\mathrm{\eq}$ associated to a complete, subcritical rule $R$ from models on an appropriate instance of $\mcT^\mathrm{Ban}(C_w^\alpha)$.

 As in Section~\ref{s:Alg_Models}, in order to work with kernels for which we assume control on only finite many derivatives, we need to restrict our attention to carefully chosen sectors. We fix a finite set of trees $\mcB \subset \mcT(R)$ that is the generating set for a historic sector. The reader should have in mind that we aim to construct $\hat{\Pi}_x \btau$ for $\btau \in \mcB$ where $\hat{\Pi}_x$ is (part of) a model coming from a preparation map on $\mcT(R)$. Our strategy will be to use the results of the previous section to construct a family of modelled distributions $f_x^{\btau}$ on $\mcT^\mathrm{Ban}(C_w^\alpha)$ indexed by $x \in \mbR^d$ and $\btau \in \mcB$ and define $\hat{\Pi}_x \btau = \mcR f_x^{\btau}$. In order to do this we will need to ensure that we are given a model on a good sector of $\mcT^\mathrm{Ban}(C_w^\alpha)$ that contains the range of each $f_x^{\btau}$. Whilst there is room to optimise here, for simplicity we will appeal to a relatively crude construction involving only the regularity of the modelled distributions.
	
\begin{definition}\label{def: gamma_def}
	Given $\gamma_0 > 0$, we define $\gamma_{\btau} \in \mathbb{R}$ recursively via
	\begin{equs}
	 \gamma_\tau = \gamma_0 \text{ for } \tau \in \{\Xi, X^k\} , \qquad
	 \gamma_{\tau \sigma} = \min \{ \gamma_\tau + a_*, \gamma_\sigma + a_*\}, \qquad
	 \gamma_{\mcI_\mfl^k \tau} = \gamma_\tau + |\mfl|_\mfs - |k|_\mfs 
	\end{equs} 
	where $a_* = \min \mcA$ where $\mcA$ is the grading for $\mcT^\eq$. We say that $\gamma_0$ is admissible for $\mcB$ if $\gamma_{\btau} \in \mbR_+\setminus\mbN$ for all $\btau \in \mcB$.
\end{definition}

At this point, we would like to mimic the construction of \cite[Equations (4.1 - 4.4)]{HS23} adapting for the fact that we wish to describe a model itself rather than the corresponding Malliavin derivative as is done there. One subtlety in comparison to \cite{HS23} is that due to the definition of the variable coefficient abstract integration map given in Theorem~\ref{theo: VML Pointed Schauder 2}, in the setting considered here, the pointed modelled distribution associated to a tree $\mathcal{I}\tau$ (defined below) will take values in the sector spanned by polynomials, planted trees and products of polynomials with a single planted tree rather than the sector spanned by only the first two of these. In particular, it is necessary to extend the abstract gradient to act on trees of the form $X^k \mathcal{I}_\mfl \tau$. At the analytic level, we can simply enforce the Leibniz rule which causes no problems in the construction of the model since trees of the form $X^k \mathcal{I}_\mfl \tau$ see no renormalisation at their root. This means that one only has to check that the trees that result from formally applying the Leibniz rule conform to the resulting rule. Since derivatives appearing when applying the Leibniz rule may hit the polynomials rather than the edge itself, this requires us to know that our rule is stable under reducing the derivative decoration on edges.

\begin{definition}\label{def: derivative_completion}
	Given a rule $R$, we define its derivative completion $\bar{R}$ to be the rule defined by
	\begin{align*}
		\bar{R}(\mft) = \{ (\mft_i, k_i)_{i=1}^n : \text{ there exist } j_i \in \mathbb{N^d} \text{ such that } (\mft_i, j_i)_{i=1}^n \in R(\mft) \text{ and } k_i \le j_i \}.
	\end{align*}
\end{definition}

Replacing $R$ by its derivative completion causes no problems due to the following technical result. Since details about rules are not the main focus of this paper and the main challenge of the proof is in appropriately reformulating definitions of \cite{BHZ19}, we provide only a sketch of proof here.

\begin{lemma}
	Suppose that $R$ is a complete and subcritical rule. Then so is its derivative completion $\bar{R}$.
\end{lemma}
\begin{proof}[Sketch Proof]
	Normality (cf. \cite[Definition 5.7]{BHZ19}) of $\bar{R}$ is immediate from the definition. To check subcriticality, one simply notes that any regularity assignment  as in \cite[Definition 5.14]{BHZ19} exhibiting subcriticality of $R$ also exhibits subcriticality of $\bar{R}$ since removing derivatives increases regularity. Checking $\Theta$-completeness of $\bar{R}$ (cf. \cite[Definition 5.20]{BHZ19}) is more tedious. Roughly speaking, one can prove that a normal rule is $\Theta$-complete if and only if the set of trees conforming to the rule is closed under the operation of contracting subtrees of negative degree and distributing derivatives appropriately on the outgoing edges of the contracted subtree. Since one can also prove that the set of trees conforming to $\bar{R}$ are precisely those trees that can be obtained by reducing edge decorations of a tree conforming to $R$, the result then follows in a straightforward manner. 
\end{proof}

We recall that we have fixed a complete subcritical rule $R$ generating the regularity structure $\mcT^\eq$. We fix $\bar{R}$ to be the derivative completion of $R$. All instances of $\mcT^\Ban$ in what follows will be built over the underlying set of combinatorial trees $\mcT(\bar{R})$. 

Given $\gamma_0$ that is admissible for $\mcB$, we define the historic sector $W_{\gamma_0}$ of $\mcT^\mathrm{Ban}$ to be the historic sector with generating set of trees given by $\mathrm{Hist}(\{\btau \in \mcT(\bar{R}): |\btau|_{\mfs} < \gamma_*\})$ where $\gamma_* = \max\{\gamma_{\btau}: \exists j \text{ such that } D^j\btau \in \mcB\}$ where we have extended the definition of $\gamma_\btau$ via $\gamma_{\mcI_\mfl^k \btau} = \gamma_{\mcI_\mfl^{j} \btau} + |j - k|_{\mfs}$ if $k \le j$ and $\mcI_\mfl^j \btau \in \mcB$. This construction accounts for the fact that $\mcB$ need not be closed under the operation of removing derivatives at the trunk. We also define $a_*$ to be the regularity of the sector $W_{\gamma_0}$ and $m_* = \max\{\gamma_* - a_*, |\mcI_\mfl \btau|_{\mfs} : \exists k \in \mathbb{N}^d, \mcI_\mfl^k \btau \in \mcB \}$.

         \begin{prop}\label{prop: PMD def}
Let $\gamma_0$ be admissible for $\mcB$ and suppose that 
 $- \alpha > m_* + | \mfs | + \max ( \mfs )$.
Then there exists $\mco > 0$ depending only on $\gamma_*$ and the rule $R$ such that if
$Z = (\Pi, \Gamma)$ is a model on $W_{\gamma_0}$ that is admissible with respect to $A \in \mcA_+^\mco$ and $\xi \in \mcA_-^\eq$ then there exists a unique family of pointed modelled distributions $f_x^{\btau} \in \mcD^{\gamma_{\btau}, |\btau|_\mfs; x}$ satisfying the recursive relations
     \begin{equs}
         &f_x^{\btau} = \mcQ_{< \gamma_0} \Gamma_{yx} \tau \text{ for } \tau \in \{\Xi, X^k\} , \qquad          
        &&f_x^{\btau \bsigma} =      
          f_x^{\btau} \star_{\gamma_{\btau \bsigma}} f_x^{\bsigma} , \qquad
     	&& f_x^{\mcI_\mfl^k \btau} = D^k \mcK^{x; \mfl}_{\gamma_{\btau}, | \btau |_\mfs, m_*} f_x^{\btau}
    \end{equs} 
   where $\mcK^{x; \mfl}_{\gamma, \nu,m}$ denotes the integration operator as in Theorem~\ref{theo: VML Pointed Schauder 2} corresponding to $\mcI_\mfl$. Furthermore, given a second model $\bar{Z}$ on $W_{\gamma_0}$ which is admissible for $\bar{A} \in \mcA_+^\mco$ and $\bar{\xi} \in \mcA_-^\eq$, denoting by $\bar{f}_x^{\btau}$ the respective modelled distributions, there is an $r>0$ 
 such that for every 
 $\btau \in \mcB$ 
it holds that
$$
 \|f_x^{\btau}\|_{\gamma_{\btau}, |\btau|_\mfs; x} \lesssim 1, \qquad
\|f_x^{\btau}; \bar{f}_x^{\btau} \|_{\gamma_{\btau}, |\btau|_\mfs; x} \lesssim {\|Z; \bar{Z}\|_{W_{\gamma_0}, B_{\mfs}(x,r)}} + \|A - \bar{A}\|_{\mcA_+^\mco}
$$
where the implicit constants can be chosen to depend polynomially on
$\|Z\|_{W_{\gamma_0}, B_{\mfs}(x,r)} +  \|\bar{Z}\|_{W_{\gamma_0}, B_{\mfs}(x,r)} + \|A\|_{\mcA_+^\mco} + \|\bar{A}\|_{\mcA_+^\mco}$.
\end{prop}
We remind the reader that for $\btau \in \{X^k, \Xi_\mfl\}$, $V^{\proj \sttau} \simeq \mbR$ so that it is legitimate to identify $\btau$ with $1 \in V^{\proj \sttau}$ (which matches the usual scalar valued setting).

\begin{proof}
    At this point, we have all the ingredients to reduce the proof to a (Noetherian) induction over $\mcB$. For $\btau \in \{\Xi, X^k\}$, one can check that $\mcQ_{< \gamma} \Gamma_{\cdot, x} \tau \in \mcD^{\gamma, |\tau|_{\mfs}; x}$ for every $\gamma > 0$. The claim in the case of two models for these values of $\btau$ follows by noting that
    \begin{align*}
    	f_x^{\btau}(y) - \bar{f}_x^{\btau}(y) =  \mcQ_{<\gamma_0} (\Gamma_{yx}- \bar{\Gamma}_{yx})\btau.
    \end{align*} 
    Stability under integration and differentiation follows from Theorem~\ref{theo: VML Pointed Schauder 2}  (recalling Assumption~\ref{ass:reg}) whilst stability under multiplication follows from \cite[Proposition 3.11]{HS23}. Here we make use of the fact that $\mcB$ generates a good sector to appeal to the induction hypothesis. 
\end{proof}

For the remainder of this section, we assume that a finite set of trees $\mcB \subset \mcT(R)$ and an admissible\footnote{Note that such a $\gamma_0$ exists since we have assumed that $\mcB$ is finite.} $\gamma_0$ for $\mcB$ have been fixed. We assume also that 
$- \alpha > m_* + | \mfs | + \max ( \mfs )$.

\begin{definition}
	Given a model $Z$ on $W_{\gamma_0}$, we define $\Phi_x(Z) : \mcB \to \mcD^\prime(\mbR^d)$ via $\Phi_x(Z)\btau = \mcR f_x^{\btau}$. 
\end{definition}

The idea is that we will eventually show that $\Phi_x(Z) = \Pi_x^\mathrm{eq} |_\mcB$ for a model $Z^\mathrm{eq} = (\Pi_x^\mathrm{eq}, \Gamma^{\mathrm{eq}})$ on $ \langle \mcB \rangle \subset \mcT^\mathrm{eq}$.  In particular, bounds obtained for $\Phi_x(Z)$ will also immediately imply bounds for the model $Z^\mathrm{eq}$. With this in mind, we note here that combining the reconstruction theorem for pointed modelled distributions given in Theorem~\ref{theo: pointed recon} with the result of Proposition~\ref{prop: PMD def} immediately yields the following quantitative estimates. In the following statement $\|\cdot\|_{\btau; \mfK}$ is as defined in \eqref{refined_pi_norm}.

\begin{lemma}\label{lem:transfer_down}
	Given a pair of models $Z, \bar{Z}$ on $W_{\gamma_0}$ that are admissible with respect to $A, \bar{A} \in \mcA_+^\mco$ and $\xi, \bar{\xi} \in \mcA_-^\eq$ respectively with $\mco$ as in Proposition~\ref{prop: PMD def}, for each $\btau \in \mcB$
	\begin{align*}
		\|\Phi(Z)\|_{\btau; \mfK} \lesssim 1, \qquad  \|\Phi(Z) - \Phi(\bar{Z}) \|_{\btau; \mfK} \lesssim {\|Z; \bar{Z}\|_{W_{\gamma_0}, \mfK + B_{\mfs}(x,r+1)}} + \|A - \bar{A}\|_{\mcA_+^\mco}
	\end{align*}
	where the implicit constants 
	can be chosen to depend polynomially on $\|Z\|_{W_{\gamma_0}, B_{\mfs}(x,r)} +  \|\bar{Z}\|_{W_{\gamma_0}, B_{\mfs}(x,r)} + \|A\|_{\mcA_+^\mco} + \|\bar{A}\|_{\mcA_+^\mco}$.
\end{lemma}

\subsection{An Exact Formula for Models on \texorpdfstring{$\mcT^\mathrm{eq}$}{T eq}}\label{ss: mod_eq}
In this section,
we assume to have fixed parameters $\mcO \geq \mco + \max (\mfs)$, $\mcO \geq - \alpha + \max (\mfs)$, $\int w < \infty$,
a smooth noise assignment $\xi \in \mcA_-^{\infty; \eq}$ and a kernel assignment $K \in \mcA_+^{\eq, \mcO}$ on $\mcT^{\eq}$. 
Furthermore we recall that if $\mcT^\Ban = \mcT^\Ban_{\bar{R}}(C_w^\alpha)$ then the kernel assignment $K$ also induces a kernel assignment $A: \bigoplus_{\mfl \in \mfL_+} C_w^\alpha \to \mcK_\mco^{|\mfl|_{\mfs}}$ as described in Section~\ref{ss:quenched}. This kernel assignment is such that $K_\mfl^y(y-x) = A_\mfl(\delta_y)(y-x)$. Given this data, we aim to exhibit a preparation map on the historic sector\footnote{We note that the usual scalar valued regularity structures are a special case of our vectorial construction with the Banach space assignment $V_\mfl = \mbR$ so that it makes sense to talk about historic sectors in this context.} of $\mcT^\eq$ generated by $\mcB$ such that if we denote the corresponding model by $(\Pi^\eq, \Gamma^\eq)$ then we have $\Pi_x^\eq \btau = \Phi_x(Z^\Ban)\btau$ for all $\btau \in \mcB$ where $Z^\Ban = (\Pi^\Ban, \Gamma^\Ban)$ is the BPHZ model on the historic sector of $\mcT^\Ban$ generated by $\mcB$. In fact, the results of this section would apply to any preparation map on $\mcT^\Ban$ of the form specified in Lemma~\ref{l:pfl}.

We recall here that the correct notion of preparation map on $\mcT^\eq$ is that of a state-space dependent preparation map as is defined in \cite[Definition 5.1]{BB21}. We remark that in principle the construction of a model from a state-space dependent preparation map in our setting does not quite follow from \cite{BB21} due to our weaker assumptions on kernel assignments (required to deal with the coefficient field in $C^r$ rather than $C^\infty$ topologies). However, the adaptation to \cite{BB21} runs along similar lines to the adaptation to the translation invariant setting considered in Section~\ref{s:Alg_Models}. The only significant difference is that one does not expect $\Pi_x \btau$ to be smooth in this setting since we cannot appeal to integration by parts to move derivatives off of the kernels $K_\mfl$. However this issue can be circumvented by instead constructing $\Pi_x \btau \in L_\mathrm{loc}^{\infty} \Big (\frac{dy}{|y-x|_{\mfs}^{|\btau|_{\mfs}}} \Big )$ and noting that in place of appealing to Taylor expansion to control $\Pi_x \mcI_\mfl^k \btau$, one can instead appeal to a straightforward adaptation of the proof of the Extension Theorem given in \cite[Theorem 5.14]{Hai14} to obtain the right mapping properties of $K_\mfl$ on these spaces. In order to avoid significant repetition, we leave the details to the reader. 

For our purposes, the most important class of state-space dependent preparation maps are those arising from $\Delta_r^-$ in a similar way to Lemma~\ref{l:pfl}. We recall that given a state-space dependent functional $\ell_y : \mcT_-^\eq \to \mbR$ such that the dependence on $y$ is continuous, $P(y, \tau) = (\ell_y \otimes \id) \Delta_r^-$ defines a state-space dependent (strong) preparation map. This means that in order to exhibit a preparation map on $\mcT^\eq$, it suffices to exhibit $\ell_y^\eq$. 

To do this, in order to capture the difference between taking derivatives on $\mcT^\Ban$ where the ``upper slot'' is frozen and on $\mcT^\eq$ where derivatives may hit the ``upper slot'', it is convenient to introduce an intermediary class of trees with an extra decoration that accounts for the use of the Leibniz rule on $\mcT^\eq$ to split derivatives between the upper and lower slot before passing to $\mcT^\Ban$. 

\begin{definition}
	An over-decorated tree is a decorated tree\footnote{Strictly speaking, a decorated tree will be an isomorphism class of such trees where the relevant notion of isomorphism also preserves the `over-decoration'} $\btau = (\btau, \mfn, \mfe)$ equipped with an additional edge decoration $\mck : E_+(\tau) \to \mbN^d$ such that for each $e \in E_+(\btau)$, 
	$|k_e|_{\mfs} < m_*$. 
	We will often write $\btau = (\btau, \mfn, \mfe, \mck)$ to denote an over-decorated tree and will write 
	$\tilde{\mcT}(R) =  \{(\btau, \mfn, \mfe, \mck): (\btau, \mfn, \mfe) \in \mcT({R}) \}.$ We extend the degree map to over-decorated trees by writing 
	$$|(\btau, \mfn, \mfe, \mck)|_\mfs = |(\btau, \mfn, \mfe)|_\mfs$$
	and write $\tilde{\mcT}_-(R)$ for the set of elements of $\tilde{\mcT}(R)$ of negative degree.
\end{definition}

We then define a linear map $\iota_x: \langle \tilde{\mcT} \rangle \to \mcT^\Ban$ via
\begin{equs} \label{eq:iota}
	\iota_x 
(\btau, \mfn, \mfe, \mck) =  \otimes_{\langle \btau \rangle}^{\tau} \big[\big((-\partial)^{\mck_e}\delta_x \big)_{e\in E_+(\tau)}\big] 
\end{equs}
which we note is independent of the choice of $\tau \in (\btau, \mfn, \mfe, \mck)$.

With this notation in hand, we are ready to define our state-space dependent preparation map on $\mcT^\eq$. We define $\ell_y^\eq$ on $\mcT_-^\eq$ via $\ell_y^\eq = \ell_y^\Ban \circ \iota_y \circ \mfD$ where
\begin{align}\label{eq:D_def}
	\mfD (\tau, \mfn, \mfe) = \sum_{\mck, \ell} \frac{1}{k!} {\mfe \choose \ell} \mfp_{\tilde{\mcT}_-} (\tau, \mfn + \pi \mck, \mfe - \ell, \mck + \ell)
\end{align}
where the sum ranges over all edge decorations $\mck, \ell$ on $(\btau,\mfn, \mfe)\in \mathcal{T}^{\mathrm{eq}}_-\subset\mcT^{\eq}$. We note that the sum appearing above is a finite sum due to the binomial coefficient and the projection onto $\tilde{\mcT}_-(R)$. We remind the reader that the map $\pi$ converts edge decorations of a tree to node decorations via $$\pi \mck(v) = \sum_{\substack{w \in N (\tau): v < w \, \\ \, e = (v, w) \in E ( \tau )}} \mck ( e )$$
where $<$ denotes the usual partial order on the vertices of a rooted tree. 
Our goal in this section is then to establish the following proposition.
\begin{prop}\label{prop:f_algebraic_renormalisation}
		Suppose that $\mcB \subset \mcT(R)$ is the generating set of an historic sector in $\mcT^\eq$. Fix a smooth noise assignment $\xi \in \mcA_-^{\infty; \eq}$ and a kernel assignment $K \in \mcA_+^{\eq; \mcO}$. 
		Denote by $A \in \mcA_+^\mco$ the kernel assignment on $\mcT^\Ban(C_w^\alpha)$ associated to $K$ as described in Section~\ref{ss:quenched}. 
		
		Fix $\gamma_0$ that is admissible for $\mcB$ and suppose that $\mco > \mathrm{ord}(W_{\gamma_0})$
		 and that $\mcO > \mathrm{ord}(\langle \mcB \rangle)$. Let $Z^\eq = (\Pi^\eq, \Gamma^\eq)$ denote the model on $\langle \mcB \rangle$ associated to the preparation map $P_y = (\ell_y^\eq \otimes \id) \Delta_r^-$, the noise assignment $\xi$ and the kernel assignment $K$. Then we have that
		\begin{align*}
			\Pi_x^\eq \btau  = \Phi_x(Z^\Ban) \btau
		\end{align*}
		for every $\btau \in \mcB$ where $Z^\Ban$ is the BPHZ model on $W_{\gamma_0}$ with noise assignment $\xi$ and kernel assignment $A$. 
\end{prop}

We note that it straightforward to establish this proposition when $Z^\Ban$ and $Z^\eq$ are replaced with the corresponding canonical lifts. As such, the main challenge is dealing with renormalisation at the algebraic level.  We note that the definition of $\Phi_x(Z^\Ban)$ via the recursive definition of the corresponding $f_x^{\btau}$ is not particularly convenient for dealing with this challenge. As a result, before turning to the proof of Proposition~\ref{prop:f_algebraic_renormalisation}, we will reformulate the definition of $f_x^{\btau}$ in terms of a coproduct-type operation that can be formulated non-recursively. 

\subsection{An Algebraic Representation of \texorpdfstring{$f_x^{\tau }$}{f x tau}}\label{ss:alg_rep}

In this subsection we derive an explicit algebraic expression of 
the pointed modelled distribution
$f^{\btau}$
defined in Proposition~\ref{prop: PMD def}, 
based on a new
character $\chi$
and a new map $\tilde{\Delta}$ which acts on trees by a suitable
cut operations reminiscent of positive renormalisation. We assume that all objects as in the statement of Proposition~\ref{prop:f_algebraic_renormalisation} are fixed.
We define 
 $$ \mathcal{T}_{\sim}(R) = \big\{ \btau=X^j \prod_{i} \mcI_\mfl^{k_i} \btau_i \ : \ \btau \text{ conforms to } R , \  \gamma_{\mathcal{I}_\mfl^{k_i} \btau}>0 \big\} $$
  and 
denote by $\mfp_{\mcT_{\sim}}:  \langle \mathcal{T} ( R ) \rangle \to \langle \mathcal{T}_{\sim}(R) \rangle $ the associated projection. We note that $\langle \mcT_\sim \rangle$ is an algebra for the tree product. Furthermore, we define $\mcB_\sim \subset \mcT_\sim$ to consist of those $\tau$ as above such that each $\btau_i$ lies in $\mcB$. The following definition will be used to describe the coefficients of $f_x^{\btau}$ in the polynomial part of the regularity structure.

\begin{definition}\label{def:chi}
	Given $x, y \in \mathbb{R}^d$, we define an algebra morphism
	$\chi_y^x: \langle \mathcal{B}_{\sim} \rangle \to \mathbb{R}$
	by setting
	$
	\chi_{y}^x(\mathbf{X}_i)=y_i-x_i $
	and 
	\begin{equs}
		& 
		\chi_y^x ( \mcI_\mfl^{k}\btau )
		\\
		& 
		= k ! \, \mcQ_{X^k} \Bigg( & \sum_l \frac{X^l}{l!} \bigg( J^{(- \partial)^l \delta_y; \mfl} ( y ) f_x^{\tau} ( y ) + \mcN_{\gamma_{\tau}}^{(- \partial)^l \delta_y; \mfl} f_x^{\tau} ( y ) \bigg)
		- \sum_{| j + l |_{\mfs} < | \tau |_{\mfs} + |\mfl|_\mfs} \frac{\Gamma_{y x} X^{j + l}}{j! l!} \big( D^j A_\mfl((-\partial)^l \delta_x) * \mcR f_x^{\tau} \big) ( x ) \Bigg) .
	\end{equs}
\end{definition}

Note that we can write
\begin{align*}
	f_x^{\mcI_\mfl^j \btau}(y) = \mcQ_{<\gamma_{\mcI_\mfl^j \btau}} \left ( \sum_{|k|_{\mfs} < m_*} \left [ D^j \left ( \frac{X^k}{k!} \mcI_{\mfl}^{(- \partial)^k \delta_y} f_x^{\btau}(y) \right ) + \frac{D^j X^k}{k!} \chi_y^x(\mcI_\mfl^k \btau) \right ]\right )
\end{align*}
which has polynomial part
\begin{align*}
	\mcQ_{< \gamma_{\mcI_\mfl^j \btau}} \left ( \sum_{|k+j|_{\mfs} < m_*} \frac{X^k}{k!} \chi_y^x (\mcI_\mfl^{k+j} \btau ) \right).
\end{align*}

For a tree $\btau \in \tilde{\mcT}$, we will write $\mcI_\mfl^{e, k} \btau$ for its planted version where the trunk carries decoration $e$ and over-decoration $k$.

\begin{definition}\label{def:delta_tilde}
	Let $\tilde{\Delta}: \langle \mcT(R) \rangle \to \langle \tilde{\mcT}(\bar{R}) 
	\rangle \otimes \langle {\mcT}_{\sim}(R) \rangle$ be the linear map recursively defined by
	\begin{enumerate}
		\item $\tilde{\Delta} \Xi = \Xi \otimes \mathbf{1}$ and $\tilde{\Delta} \mathbf{X}_i= \mathbf{X}_i\otimes \mathbf{1}+ \mathbf{1}\otimes \mathbf{X}_i$,
		\item $\tilde{\Delta} (\btau \cdot \bsigma)=  \big( \mcQ_{< \gamma_{\btau \cdot \bsigma}} \otimes \id \big) \, \big( \tilde{\Delta} \btau \cdot \tilde{\Delta} \bsigma \big)$,
		\item\label{item:integ}
		$\tilde{\Delta} \mathcal{I}_\mfl^j \btau = \big( \mcQ_{< \gamma_{\mcI_\mfl^j \btau}} \otimes \id \big) \bigg( \sum_{l \le j} \sum_{|k+l|_{\mfs} < m_*} {j \choose l} \big(\frac{X^k}{k!}  \mathcal{I}_\mfl^{j-l,k+l} \otimes \id \big) \tilde{\Delta} \btau + \sum_{|k+j|_{\mfs} < m_*} \frac{X^{k}}{k!} \otimes \mathcal{I}_\mfl^{k+j} \btau \bigg)$,
	\end{enumerate}
\end{definition}
We remark that a straightforward induction shows that $\tilde{\Delta} (\mcB) \subset \langle \tilde{\mcT}(\bar{R}) \rangle \otimes \langle \mcB_\sim \rangle$ so that in particular, the expression $(\id \otimes \chi_y^x) \tilde{\Delta}$ is well defined on $\langle \mcB \rangle$. 

\begin{lemma}\label{lemma:f_and_delta}
	For any
	$x, y \in \mathbb{R}^d$
	and any
	$\btau \in \mcB$,
	the following identity holds 
	\begin{equ}\label{eq:f_and_delta}
		f^{\btau}_x(y)= (\iota_{y} \otimes \chi_{y}^x ) \tilde{\Delta} \btau \ .
	\end{equ}
\end{lemma}
\begin{proof}
	We proceed by Noetherian induction over $\mcB$. It is clear that \eqref{eq:f_and_delta} holds if $\btau \in \{\Xi_\mfl, X^k\}$ so that it suffices to consider separately the cases where $\btau$ is planted and where $\btau$ is a product of planted trees. In the latter case, the result is straightforward since all objects involved are multiplicative. Therefore, we consider only the case of planted trees. In this case, we write
	\begin{align*}
		& (\iota_y \otimes \chi_y^x)\tilde{\Delta} \mcI_\mfl^k \btau  = 	(\iota_y \otimes \chi_y^x)\big( \mcQ_{< \gamma_{\mcI_\mfl^j \btau}} \otimes \id \big) \bigg( \sum_{l \le j} \sum_{|k+l|_{\mfs} < m_*} {j \choose l} \Big(\frac{X^k}{k!}  \mathcal{I}_\mfl^{j-l,k+l} \otimes \id \Big) \tilde{\Delta} \btau + \sum_{|k+j|_{\mfs} < m_*} \frac{X^{k}}{k!} \otimes \mathcal{I}_\mfl^{k+j} \btau \bigg)
		\\
		&= \big( \mcQ_{< \gamma_{\mcI_\mfl^j \btau}} \otimes \id \big) \bigg( \sum_{l \le j} \sum_{|k+l|_{\mfs} < m_*} {j \choose l} \Big(\frac{X^k}{k!} D^{j-l} \mathcal{I}_\mfl^{(-\partial)^{k+l} \delta_y} \otimes \id \Big) (\iota_y \otimes \chi_y^x) \tilde{\Delta} \btau + \sum_{|k+j|_{\mfs} < m_*} \frac{X^{k}}{k!} \chi_y^x \big ( \mathcal{I}_\mfl^{k+j} \btau \big ) \bigg).
	\end{align*}
	We then write
	\begin{align*}
		 \sum_{l \le j} \sum_{|k+l|_{\mfs} < m_*} {j \choose l} \Big(\frac{X^k}{k!} D^{j-l} \mathcal{I}_\mfl^{(-\partial)^{k+l} \delta_y} \otimes \id \Big) (\iota_y \otimes \chi_y^x) \tilde{\Delta} \btau &=  \sum_{l \le j} \sum_{|\bar{k}|_{\mfs} < m_*} {j \choose l} 
		  \frac{D^l X^{\bar{k}}}{\bar{k}!} D^{j-l} \mathcal{I}_\mfl^{(-\partial)^{\bar{k}} \delta_y}f_x^{\btau}(y)
		 \\
		 & = D^j \bigg ( \sum_{|\bar{k}|_{\mfs} < m_*} \frac{X^{\bar{k}}}{\bar{k}!}  \mathcal{I}_\mfl^{(-\partial)^{\bar{k}} \delta_y}f_x^{\btau}(y) \Bigg )
	\end{align*}
	where we made use of the induction hypothesis to pass to the second line.
	Treating the term with $\chi_y^x(\mcI_\mfl^{k+j}\btau)$ similarly, we see that the induction hypothesis implies that
	\begin{align*}
		(\iota_y \otimes \chi_y^x)\tilde{\Delta} \mcI_\mfl^k \btau  
		&= D^j \mcQ_{< \gamma_{\mcI_\mfl \btau}}\Bigg ( \sum_{|\bar{k}|_{\mfs} < m_*} \frac{X^{\bar{k}}}{\bar{k}!}  \mathcal{I}_\mfl^{(-\partial)^{\bar{k}} \delta_y}f_x^{\btau}(y) + \sum_{|k|_{\mfs} < m_*} \frac{X^{k}}{k!} \chi_y^x \big ( \mathcal{I}_\mfl^{k} \btau \big ) \Bigg )
		\\
		&
		= D^j \mcK_{\gamma_\btau, |\btau|_{\mfs}, m_*}^{x; \mfl} f_x^\btau(y) = f_x^{\mcI_\mfl^j \btau}(y)
	\end{align*}
	as desired.
\end{proof}

We now show that in a similar way to $\Delta$ and $\Delta^+$, it is possible to reformulate the definition of $\tilde{\Delta}$ non-recursively. This will be essential to identify how it interacts with $\Delta_r^-$. Since the operation $\tilde{\Delta}$ acts solely at the level of combinatorial trees, rather than on $\mcT^\Ban$, in the remainder of this section we are in a setting analogous to the one considered in \cite{BHZ19} where it suffices to work with the distinguished basis of (isomorphism classes of) combinatorial trees. As such, the difference between drawings and isomorphism classes becomes less important in the arguments that follow. In order to simplify notations and bring our proofs more in line with the similar arguments given in \cite{BHZ19} we will therefore make sense of operations on isomorphism classes at the level of drawings in the remainder of this section with the implicit understanding that the operations considered are independent of the choice of drawing and thus define operations on isomorphism classes.

Similarly to \cite{BHZ19}, it will be convenient to formulate the action of $\tilde{\Delta}$ in terms of colourings. 
A coloured tree is then an (over-)decorated tree $\tau$ together with a distinguished subset of edges $\hat{\tau}\subset E(\tau)$
which is the edge set of a subtree of $\tau$ containing its root. We will write $(\tau, \hat{\tau}, \mfn, \mfe)$ and $ (\tau, \hat{\tau}, \mfn, \mfe, \mck)$ for coloured decorated and over-decorated trees respectively. 
Finally, we define the contraction operation $\mathfrak{C}$ via
\begin{align}\label{eq:K_def}
	\mathfrak{C} (\tau, \hat{\tau}, \mfn, \mfe, \mck) = (\tau / \hat{\tau} , [\mfn]_{\hat{\tau}}, \mfe, \mck)
\end{align}
where $[\mfn]_{\hat{\tau}}(v) = \sum_{w \in \cup_{e \in \hat{\tau} } e } \mfn(w)$. Here we implicitly identify $E(\tau / \sigma)$ as a subset of $E(\tau)$ in order to restrict edge decorations to the quotient graph. We extend this definition to decorated coloured trees by identifying a decorated coloured tree with its over-decorated analogue with $\mck \equiv 0$.

\begin{prop}\label{prop: tDelta_form}
	The map $\tilde{\Delta}$ has the explicit form 
\begin{equs}
	\tilde{\Delta} (\tau, \mfn,\mfe)
	= \sum_{\tilde{C}\in \mathbb{C}_+[\tau]} \sum_{n_{ \ng \tilde{C}}}
	\sum_{\ell_{\ng \tilde{C}}, \mck_{\ng \tilde{C} }}  
	\sum_{\eps_{\tilde{C} }}  
	& {\mfn \choose n_{\ng \tilde{C}} } {\mfe \choose \ell_{\ng \tilde{C}}}  
	\frac{1}{\mck_{\ng \tilde{C}} !} \frac{1}{\eps_{\tilde{C}} !} 
	\label{eq:dt_expl}
	\\ 
	&
\mcQ_{< \gamma_{(\tau, \mfn,\mfe)}} (\tau_{\ng \tilde{C}} , n_{\ng \tilde{C}} + \pi (\mck_{\ng \tilde{C}} +{\eps_{\tilde{C}}}) , \mfe - \ell_{\ng \tilde{C}} , \mck_{\ng \tilde{C}} + \ell_{\ng \tilde{C}} ) 
	\\
	& \otimes \mfp_{\mcT_{\sim}} \mathfrak{C} (\tau, \tau_{\ng \tilde{C}}, \mfn- n_{\ng \tilde{C}}, \mfe +\eps_{\tilde{C}}  )  \ ,
\end{equs}
where $\mathbb{C}_+[\tau] = \mathbb{C}[\tau] \cap \mcP(E_+(\tau))$, the second sum is over maps ${n_{\ng \tilde{C}}}: N_{\tau_{\ng \tilde{C}}}\to \mathbb{N}^d$, the third sum is over maps 
$\ell_{\ng \tilde{C}}, \mck_{\ng \tilde{C} }:  E({\tau_{\ng \tilde{C}}}) \to \mathbb{N}^d$, and the last sum is over maps $\eps_{\tilde{C} }: \tilde{C}\to \mathbb{N}^d$.
\end{prop}
\begin{proof}
We prove that 
the right-hand side of \eqref{eq:dt_expl}
satisfies the inductive identities of Definition~\ref{def:delta_tilde}.
We cover separately the various cases, leaving the cases where $(\tau,\mfn,\mfe)$ is either a noise or monomial to the reader.

\medskip
\noindent
\textit{Case of $\tilde{\Delta} (\tau \sigma)$}:
We assume to be given $\tau = ( \tau, \mfn_{\tau}, \mfe_{\tau} )$
and $\sigma = ( \sigma, \mfn_{\sigma}, \mfe_{\sigma} )$,
such that
$\tilde{\Delta}\tau$ and $\tilde{\Delta} \sigma$
satisfy
the explicit expression \eqref{eq:dt_expl}. In the definition of $\tilde{\Delta} \tau \sigma$, we note that we can write $\mbC_+[\tau \sigma] = \mbC_+[\tau] \oplus \mbC_+[\sigma]$ where we set $\mbC_+[\tau] \oplus \mbC_+[\sigma] = \{C_1 \sqcup C_2: C_1 \in \mbC_+[\tau], C_2\in \mbC_+[\sigma] \}$. Furthermore, for a fixed $C = C_\tau \sqcup C_\sigma \in \mbC_+[\tau] \oplus \mbC_+[\sigma]$, we can write $\ell_{\ng C} = \ell_{\ng C_\tau} + \ell_{\ng C_{\sigma}}$ with $\ell_{\ng C_\tau}(e) = \ell_{\ng C}(e) \mathbf{1}_{e \ng C_\tau, e \in E(\tau)}$ and similarly for $\mck_{\ng C}$. We also write $\eps_C = \eps_{C_\tau} + \eps_{C_\sigma}$ with $\eps_{C_\tau}(e) = \eps_{C}(e) \mathbf{1}_{e \in C_\tau}$. Finally, we write $n_{\ng C} = \mathring{n}_{\ng C_\tau} + \mathring{n}_{\ng C_\sigma} + n^{\rho}$ where $n^{\rho}(v) = n_{\ng C}(\rho) \mathbf{1}_{v = \rho}$ and $\mathring{n}_{\ng C_\tau}^\tau (v) = n_{\ng C}(v) \mathbf{1}_{v \ng C_\tau, v \in N(\tau) \setminus \{\rho\}}$ where $\rho$ denotes the root of $\tau \sigma$.

Substituting in these decompositions, and adopting the usual convention that sums run over decorations with the correct support, the expression \eqref{eq:dt_expl} can be rewritten as
\begin{equs}
	\tilde{\Delta} ( \tau\sigma )
	& = \sum_{\substack{C_{\tau} \in \mathbb{C}_+[\tau] \\ C_{\sigma} \in \mathbb{C}_+[\sigma]}} 
	\sum_{\substack{\mathring{n}_{ \ng C_{\tau}} \\
			\mathring{n}_{\ng C_{\sigma}} \\ n^{\rho}}}
	\sum_{\substack{\ell_{\ng C_{\tau}}, \mck_{\ng C_{\tau}} \\ \ell_{\ng C_{\sigma}}, \mck_{\ng C_{\sigma}} }}
	\sum_{\substack{\eps_{C_{\tau}} \\ \eps_{C_{\sigma}}}}  
	{\mfn_{\tau} \choose \mathring{n}_{\ng C_{\tau}} }
	{\mfn_{\sigma} \choose \mathring{n}_{\ng C_{\sigma}} }
	{\mfn_{\tau} + \mfn_{\sigma}\choose n^{\rho}  } 	{\mfe_{\tau} \choose \ell_{\ng C_{\tau}}}
	{\mfe_{\sigma} \choose \ell_{\ng C_{\sigma}}} 
	\frac{1}{\mck_{\ng C_{\tau}} !} 
	\frac{1}{\mck_{\ng C_{\sigma}} !} 
	\frac{1}{\eps_{C_{\tau}} !} 
	\frac{1}{\eps_{C_{\sigma}} !} 
	\\ 
	& \quad
	\mcQ_{< \gamma_{\tau \sigma}} 
	\Big( \tau_{\ng C_{\tau}}  \sigma_{\ng C_{\sigma}} , 
	\mathring{n}_{\ng C_{\tau}} + \mathring{n}_{\ng C_{\sigma}} + n^\rho
	+ \pi (\mck_{\ng C_{\tau}} +{\eps_{C_{\tau}}})
	+ \pi (\mck_{\ng C_{\sigma}} +{\eps_{C_{\sigma}}}) , 
	\\
	& 
	\phantom{\quad \mcQ_{< \gamma_{\tau \sigma}} \Big(} 
	\quad
	( \mfe_{\tau} - \ell_{\ng C_{\tau}} )
	+( \mfe_{\sigma} - \ell_{\ng C_{\sigma}} ), 
	(\mck_{\ng C_{\tau}} + \ell_{\ng C_{\tau}})
	+ (\mck_{\ng C_{\sigma}} + \ell_{\ng C_{\sigma}}) \Big) 
	\\
	& 
	\quad
	\otimes \mfp_{\mcT_{\sim}} 
	\mathfrak{C} \Big( \tau \sigma
	,\tau_{\ng C_{\tau}} \cup \sigma_{\ng C_{\sigma}}, 
	(\mfn_{\tau} - \mathring{n}_{\ng C_{\tau}}) 
	+ (\mfn_{\sigma} - \mathring{n}_{\ng C_{\sigma}}), 
	( \mfe_{\tau} +\eps_{C_{\tau}} )
	+ ( \mfe_{\sigma} +\eps_{C_{\sigma}} ) \Big) \ .
\end{equs}
We now use the Chu--Vandermonde identity in the form
\begin{equ}
	{\mfn_{\tau}  + \mfn_{\sigma} \choose n^\rho }
	= \sum_{n^\rho = n_\tau^\rho + n_\sigma^\rho}  
	{\mfn_{\tau} \choose n_{\tau}^\rho} \, 
	{\mfn_{\sigma} \choose n_\sigma^\rho}.
\end{equ}
The result then follows by substituting in the result of Chu-Vandermonde, again re-indexing the summation via $n_{\ng C_\tau} = \mathring{n}_{\ng C_\tau} + n_\tau^\rho$ and similarly for $\sigma$ and noting the multiplicativity
properties
\begin{equs}
	& \big( \tau_{\ng C_{\tau}} \sigma_{\ng C_{\sigma}} , 
	\mathring{n}_{\ng C_{\tau}} + \mathring{n}_{\ng C_{\sigma}} + n^\rho
	+ \pi (\mck_{\ng C_{\tau}} +{\eps_{C_{\tau}}})
	+ \pi (\mck_{\ng C_{\sigma}} +{\eps_{C_{\sigma}}}) , 
	\\
	& 
	\phantom{\quad Q_{< \gamma_{\tau \cdot \sigma}} \big(} 
	\quad
	( \mfe_{\tau} - \ell_{\ng C_{\tau}} )
	+( \mfe_{\sigma} - \ell_{\ng C_{\sigma}} ), 
	(\mck_{\ng C_{\tau}} + \ell_{\ng C_{\tau}})
	+ (\mck_{\ng C_{\sigma}} + \ell_{\ng C_{\sigma}}) \big) 
	\\
	& = \big( \tau_{\ng C_{\tau}} , 
	n_{\ng C_{\tau}}
	+ \pi (\mck_{\ng C_{\tau}} +{\eps_{C_{\tau}}}) , 
	\mfe_{\tau} - \ell_{\ng C_{\tau}}, 
	\mck_{\ng C_{\tau}} + \ell_{\ng C_{\tau}} \big) \\
	& \quad \cdot 
	\big( \sigma_{\ng C_{\sigma}} , 
	n_{\ng C_{\sigma}}
	+ \pi (\mck_{\ng C_{\sigma}} +{\eps_{C_{\sigma}}}) , 
	\mfe_{\sigma} - \ell_{\ng C_{\sigma}} , 
	\mck_{\ng C_{\sigma}} + \ell_{\ng C_{\sigma}} \big) ,
\end{equs}
and 
\begin{equs}
	& \mathfrak{C} \big( \tau \sigma, 
	\tau_{\ng C_{\tau}} \sqcup \sigma_{\ng C_{\sigma}}), 
	(\mfn_{\tau} - n_{\ng C_{\tau}}) 
	+ (\mfn_{\sigma} - n_{\ng C_{\sigma}}), 
	( \mfe_{\tau} +\eps_{C_{\tau}} )
	+ ( \mfe_{\sigma} +\eps_{C_{\sigma}} ) \big)
	\\	& = \mathfrak{C} \big( \tau , 
\tau_{\ng C_{\tau}}, 
	\mfn_{\tau} - n_{\ng C_{\tau}}, 
	\mfe_{\tau} +\eps_{C_{\tau}} \big) 
	\cdot 
	\mathfrak{C} \big( \sigma , 
	\sigma_{\ng C_{\sigma}}, 
	\mfn_{\sigma} - n_{\ng C_{\sigma}}, 
	\mfe_{\sigma} +\eps_{C_{\sigma}} \big) ,
\end{equs}
along with the multiplicativity of $\mfp_{\tilde{\mcT}}$ and the identity 
\begin{equ}
	\mcQ_{< \gamma_{\tau \sigma}} ( \tilde{\tau} \tilde{\sigma} )
	= \mcQ_{< \gamma_{\tau \sigma}} \big( ( \mcQ_{< \gamma_{\tau}} \tilde{\tau} ) \cdot ( \mcQ_{< \gamma_{\sigma}} \tilde{\sigma} ) \big) ,
\end{equ}
which follows from the definition of $\gamma_{\tau \sigma}$.
\\ \\
\textit{Case of $\tilde{\Delta} \mcI_\mfl^j \tau$:}
We assume $\tau = (\tau, \mfn_{\tau}, \mfe_{\tau})$ is 
such that 
$\tilde{\Delta} \tau$ satisfies \eqref{eq:dt_expl}. We write $\rho$ for the root of $\mcI_\mfl^j \tau$ and $e_*$ for the unique edge adjacent to $\rho$. In particular, $e_*$ has type $\mfl$ and decoration $j$ so that we can write 
$\mcI_\mfl^j \tau = ( \mcI_\mfl \tau, \mfn_{\tau}, \mfe_{\tau} + j \mathbf{1}_{e_*} )$ where we extend $\mfn_\tau, \mfe_\tau$ to $\mcI_\mfl^j \tau$ by $0$. We then divide the contribution of the sum over cuts in the expression \eqref{eq:dt_expl} for $\mcI_\mfl^j \tau$ into the case where the cut $C = \{e_*\}$ and the remaining cases where $C \subset E(\tau)$. 

In the former case, since $\mfn(\rho) = 0$ we have that the sums over $\ell_{\ng C}, \mfk_{\ng C}$ and $n_{\ng C}$ are trivial. Furthermore, for this cut $((\mcI_\mfl \tau)_{\ng C}, \mfn + \pi \eps_{C}, \mfe, 0) = X^\mck$. Therefore this case contributes the term
\begin{align*}
	\sum_{k} \frac{1}{k!} \mcQ_{<\gamma_{\mcI_\mfl^j \tau}} X^k \otimes \mfp_{\tilde{\mcT}} (\mcI_\mfl \tau, \mfn, \mfe + k + j) = \sum_{k: |k+j|_{\mfs} < m_*} \frac{1}{k!} \mcQ_{< \gamma_{\mcI_\mfl^j \tau}} X^k \otimes  (\mcI_\mfl \tau, \mfn, \mfe + k + j) 
\end{align*}
where the projection $\mfp_{\tilde{\mcT}}$ can be removed since the same condition is enforced by the projection $\mcQ_{< \gamma_{\mcI_\mfl^j \btau}}$ on the other term and the fact that the effective range of summation is over $k$ which satisfy $|k+j|_{\mfs}< m_*$ follows from the fact that the projection on the first term enforces that $|k+j|_{\mfs} < \gamma_{\mcI_\mfl \tau} < m_*$ by the definition of $m_*$. We recognise this as the second summation in the recursive definition \eqref{item:integ} so it remains to show that the contribution where the cut is above the trunk contributes the remaining terms in that recursive definition.

The contribution where the cut is above the trunk can be written as
\begin{align}
	\nonumber \sum_{C\in \mbC_+[\tau]} \sum_{n_{\ng C}} \sum_{\mck_{\ng C}, \ell_{\ng C}} \sum_{\eps_C} {\mfn_\tau \choose n_{\ng C}} {\mfe_\tau + j \mathbf{1}_{e_*} \choose \ell_{\ng C}} \frac{1}{\mck_{\ng {C}}!} \frac{1}{\eps_C!} & \mcQ_{\gamma_{\mcI_\mfl^j \tau}} (\mcI_\mfl \tau_{\ng C}, n_{\ng C} + \pi(\mck_{\ng C} + \eps_{C}), \mfe_\tau + j \mathbf{1}_{e*} - \ell_{\ng C}, \mck_{\ng C}+ \ell_{\ng C}) 
	\\
	& \otimes \mfp_{\tilde{\mcT}} \mathfrak{C} (\tau,  \tau_{\ng C}, \mfn_\tau - n_{\ng C}, \mfe_\tau + \eps_C)  \label{eq:cut_above}
\end{align}
where to exclude the trunk in the second component of the tensor, we used the fact that the node decoration at the root of $\mcI_\mfl^j \tau$ vanishes so that removing the trunk leads to the same tree after contraction.

We now decompose $\ell_{\ng C} = \ell_{\ng C}^\tau + \ell^*$ where $\ell^* = \ell_{\ng C}\mathbf{1}_{e_*}$ and $\mck_{\ng C} = \mck_{\ng C}^\tau + \mck^*$ where $\mck^* = \mck_{\ng C}\mathbf{1}_{e_*}$. Due to the disjoint supports, we can write 
\begin{align*}
	{\mfe_\tau + j \mathbf{1}_{e_*} \choose \ell_{\ng C}} = {\mfe_\tau  \choose \ell_{\ng C}^\tau} {j \choose \ell^*}, \qquad \frac{1}{\mck_{\ng C}!} = \frac{1}{\mck_{\ng C}^\tau!} \frac{1}{\mck^*!}.
\end{align*}
Furthermore, we have that 
\begin{align*}
	(\mcI_\mfl \tau_{\ng C},  n_{\ng C} + \pi(\mck_{\ng C} + \eps_{C}), \mfe_\tau + & j \mathbf{1}_{e*} - \ell_{\ng C}, \mck_{\ng C}+ \ell_{\ng C}) \\ &= X^{\mck^*} \mcI_\mfl^{j- \ell^*, \mck^* + \ell^*} ( \tau_{\ng C}, n_{\ng C} + \pi(\mck_{\ng C}^\tau + \eps_{C}), \mfe_\tau - \ell_{\ng C}^\tau, \mck_{\ng C}^\tau + \ell_{\ng C}^\tau).
\end{align*}
Therefore, we can rewrite \eqref{eq:cut_above} as
\begin{align*}
	(\mcQ_{\gamma_\mfl^j \tau} \otimes \id) \left [ \sum_{\mck^*, \ell^*} {j \choose \mck^*} \left ( \frac{X^{\mck^*}}{\mck^*!} \mcI_\mfl^{j - \ell^*, \mck^* + \ell^*} \otimes \id \right ) \tilde{\Delta} \tau \right ].
\end{align*}
To complete the proof, it remains to see that the effective range of summation satisfies $|\mck^* + \ell^*|_{\mfs} < m_*$. We note that the projection on the first component of the tensor enforces that $a_* + \mck^* + \ell^* + |\mfl|_\mfs - |j|_{\mfs} < \gamma_{\mcI_\mfl^j \tau}$. Since $|\mfl|_\mfs - |j|_{\mfs} > 0$ by Assumption~\ref{ass:reg} and $\gamma_{\mcI_\mfl^j \tau} \le m_* + a_*$ by assumption, the desired restriction on the range of summation follows.
\end{proof}

\subsection{The Proof of Proposition~\ref{prop:f_algebraic_renormalisation}}\label{ss:mod_ident}

With the ingredients of the previous subsection in place, we are in position to prove Proposition~\ref{prop:f_algebraic_renormalisation}. The proof will be by induction in $\mathrm{Age}(\tau)$ over $\tau \in \mcB$. The part of the inductive step that deals with the case of products of planted trees will involve a lengthy combinatorial computation which we separate out into a lemma which is stated and proven after the remainder of the proof which is presented immediately below.

\begin{proof}[Proof of Proposition~\ref{prop:f_algebraic_renormalisation}]
	We proceed by induction in the age and also enhance the induction hypothesis to include the statement
	\begin{align*}
		\Pi_x^{\mathrm{eq},\times} \btau = \mathcal{R}^{\times} f_x^\btau 
	\end{align*}
	where $\mathcal{R}^\times$ denotes the reconstruction operator associated to the model $Z^{\Ban, \times} = (\Pi^{\Ban, \times}, \Gamma^{\Ban})$. We note that since both models have second component $\Gamma^\Ban$, the notion of modelled distributions agree for $Z^\Ban$ and $Z^{\Ban, \times}$ so that $\mathcal{R}^\times f_x^{\btau}$ is well-defined.
	
	In any of the cases where $\btau \in \{X^k, \Xi_\mfl\}$ (which includes the base case $\btau = \mathbf{1}$), the result follows immediately from the definitions. Therefore, as usual, we need only consider separately the case of planted trees and products of planted trees. 
	
	In the case of planted trees, we note that $\Pi_x^{\eq, \times}$ and $\Pi_x^\eq$ agree on such trees and that $f_x^\btau$ takes values in the subsector of $W_{\gamma_0}$ generated by trees of the form $X^k \mcI_\mfl \sigma, X^k$. Since $\Pi^\Ban$ and $\Pi^{\Ban, \times}$ agree on this sector it suffices only to prove the result comparing $\Pi^\eq$ and $\Pi^\Ban$ for such trees.
	
	For this, we write
	\begin{align*}
		\mcR f_x^{\mcI_\mfl^j \btau} = \mcR D^j \mcK_{\gamma_\btau, |\tau|_{\mfs}, m_*}^{x; \mfl} f_x^\btau = D^j \mcR \mcK_{\gamma_\btau, |\btau|_{\mfs}, m_*}^{x; \mfl} f_x^\btau
	\end{align*}
	By Theorem~\ref{theo: VML Pointed Schauder 2}, this is nothing but 
	\begin{align*}
		\mcR f_x^{\mcI_\mfk^j \btau} = D^jK_\mfl \mcR f_x^\btau -  \sum_{|k|_{\mfs} < |\mcI_\mfl^j \btau |_{\mfs}} \frac{(\cdot - x)^k}{k!} D^{k+j} K_\mfl \mcR f_x^\btau(x)
	\end{align*}
	so that by the induction hypothesis and admissibility of $\Pi_x^\eq$, the result for the case of planted trees follows.
	
	In the case where $\btau = \prod_{i=0}^n \btau_i$ is a product of planted trees, we treat $\Pi_x^{\eq, \times}$ and $\Pi_x^\eq$ separately. In the former case, we write
	\begin{align*}
		\Pi_x^{\eq, \times} \btau(y) = \prod_{i=0}^n \Pi_x^{\eq, \times}\btau_i (y) = \prod_{i=0}^n \Pi_y^{\Ban, \times} f_x^{\btau_i}(y) (y) = \Pi_y^{\Ban, \times} f_x^{\btau_i}(y)(y) = \mcR^\times f_x^\btau (y)
	\end{align*}
	where we have made use of the fact that for smooth models, the reconstruction is given by diagonal evaluation of the model applied to the modelled distribution.
	
	Therefore, it remains only to consider $\Pi_x^\eq$. To this end, we use the definition of $\Pi_x^\eq$, the induction hypothesis and Lemma~\ref{lemma:f_and_delta} to write
	\begin{align*}
		\Pi_x^\eq \btau (y) &= (\ell_y^\eq \otimes \Pi_x^{\eq, \times}) \Delta_r^- \btau(y)  = (\ell_y^\eq \otimes \Pi_y^{\Ban, \times} f_x^{\cdot}(y)(y)) \Delta_r^- \btau
		\\
		&= (\ell_y^\eq \otimes [\Pi_y^{\Ban, \times} \circ \iota_y [\cdot] (y)] \otimes \chi_y^x)(\id \otimes \tilde{\Delta})\Delta_r^- \btau.
	\end{align*}
	On the other hand, we have
	\begin{align*}
		\mcR f_x^\btau(y) = ([\Pi_y^\Ban \circ \iota_y] \otimes \chi_x^y) \tilde{\Delta}\btau(y) = ([\ell^\Ban \otimes \Pi_y^{\Ban, \times}) \Delta_r^- \iota_y] \otimes \chi_y^x) \tilde{\Delta} \btau(y).
	\end{align*}
	We then extend $\Delta_r^-$ to $\tilde{\mcT}(\bar{R})$ by simply having it ignore\footnote{If $\Delta_r^- \btau = \bsigma^1 \otimes \bsigma^2$ in Sweedler's notation, then the edge sets of $\bsigma^i$ can be naturally identified with those of $\btau$. We simply carry the over-decoration along this identification.} the over-decoration. We note that with this definition,
	\begin{align*}
		\Delta_r^- \iota_y (\btau, \mfn, \mfe, \mck) = \Delta_r^- \bigotimes_\sttau^\tau [(-\partial)^{\mck_e} \delta_y]_{e \in E_+(\tau)} = \mcL_\tau^{(2)} [\Delta_r^- \tau] ((-\partial)^{\mck_e} \delta_y)_{e \in E_+(\tau)} = (\iota_y \otimes \iota_y) \Delta_r^- (\btau, \mfn, \mfe, \mck).
	\end{align*}
	Therefore, we obtain
	\begin{align*}
		\mcR f_x^\btau(y) = (\ell^\Ban \circ \iota_y \otimes \Pi_y^{\Ban, \times} \circ \iota_y [\cdot](y) \otimes \chi_y^x) (\Delta_r^- \otimes \id) \tilde{\Delta} \btau.
	\end{align*}
	Therefore the result follows if we can show that
	\begin{align*}
		(\ell_y^\eq \otimes [\Pi_y^{\Ban, \times} \circ \iota_y [\cdot] (y)] \otimes \chi_y^x)(\id \otimes \tilde{\Delta})\Delta_r^- \btau = \mcR f_x^\btau(y) = (\ell^\Ban \circ \iota_y \otimes \Pi_y^{\Ban, \times} \circ \iota_y [\cdot](y) \otimes \chi_y^x) (\Delta_r^- \otimes \id) \tilde{\Delta} \btau.
	\end{align*}
	This is the result of a combinatorial computation that we separate out in to Lemma~\ref{lem: Delta_comb} below.
\end{proof}

\begin{lemma}\label{lem: Delta_comb}
	In the context of Proposition~\ref{prop:f_algebraic_renormalisation}, suppose that $\btau \in \mcB$. Then
	\begin{align*}
		(\ell_y^\eq \otimes [\Pi_y^{\Ban, \times} \circ \iota_y [\cdot] (y)] \otimes \chi_y^x)(\id \otimes \tilde{\Delta})\Delta_r^- \btau = (\ell^\Ban \circ \iota_y \otimes \Pi_y^{\Ban, \times} \circ \iota_y [\cdot](y) \otimes \chi_y^x) (\Delta_r^- \otimes \id) \tilde{\Delta} \btau.
	\end{align*}
\end{lemma}
Before proving this lemma, let us introduce some pieces of notation which we will need.
For a given tree $\tau$ and an edge decoration $f$ (i.e. a function $f: E_{\tau}\to \mathbb{N}^{d+1}$) we shall write for a given set of edges $C\subset E $
\begin{equ}\label{eq:notation edge dec}
 f^{>C}(e):= f(e) \mathbf{1}_{e\in \cup_{c\in C} E_{>c}}, \qquad f^{=C}(e):= f(e) \mathbf{1}_{e\in C}  
 \end{equ}
as well as 
$f^{\geq C}= f^{>E}+f^{=C}$ and $ f^{\ng C}(e)= f- f^{\geq C}$.

Next we define an extension of $\tilde{\Delta}$ to coloured trees by 
\begin{equs}
\tilde{\Delta} (\tau, \hat{\tau}, \mfn, \mfe )
= \sum_{\substack{
\tilde{C}\in \mathbb{C}_+[\tau]\\
 \tilde{C}\cap E_{\hat{\tau}}= \emptyset }
} \sum_{n_{\ng \tilde{C}}}
			\sum_{\ell^{\geq C}_{\ng \tilde{C}}, \mfk^{\geq C}_{\ng \tilde{C} }}  
			\sum_{\mck_{\tilde{C} }}  &
			 {\mfn \choose n_{\ng \tilde{C}} } {\mfe \choose \ell^{\geq C}_{\ng \tilde{C}}}  
			\frac{1}{\mfk^{\geq C}_{\ng \tilde{C}} !} \frac{1}{\mck_{\tilde{C}} !}  \label{eq:triangle extenstion colored} \\ 
			&
			\mcQ_{< \gamma_{\mathfrak{C}(\tau, {\tau'}, \mfn, \mfe )}} (\tau_{\ng \tilde{C}} ,\hat{\tau},  n_{\ng \tilde{C}} + \pi (\mfk^{\geq C}_{\ng \tilde{C}} +{\mck_{\tilde{C}}}) , \mfe - \ell^{\geq C}_{\ng \tilde{C}} , \mfk^{\geq C}_{\ng \tilde{C}} + \ell^{\geq C}_{\ng \tilde{C}} ) 
			\\
			 & \otimes \mfp_{\mcT_{\sim}} \mathfrak{C} (\tau, {\tau_{\ng \tilde{C}}}, \mfn- n_{\ng \tilde{C}}, \mfe +\mck_{\tilde{C}}  )  \ ,
\end{equs}
where in the middle line we used the extension of the projection maps $\mcQ_{< \gamma}$ to coloured trees given by
\begin{equ}\label{eq:proj extsionsion}
\mcQ_{<\gamma}(\tau,\hat{\tau}, \mfn, \mfe )= \mathbf{1}_{| \mathfrak{C} (\tau,\hat{\tau}, \mfn, \mfe )|<\gamma  }  (\tau,\hat{\tau}, \mfn, \mfe)\ .
\end{equ} These extensions are chosen such that the following lemma holds.
\begin{lemma}\label{lem:identity contractions}
For every coloured tree $(\tau, \hat{\tau}, \mfn, \mfe)$, such that $(\tau, \mfn, \mfe) \in \mathcal{B}$
\begin{equ}
\tilde{\Delta} \circ \mathfrak{C} (\tau, \hat{\tau}, \mfn, \mfe)= (\mathfrak{C} \otimes \id)\tilde{\Delta} (\tau, \hat{\tau}, \mfn, \mfe) \ .
\end{equ} 
\end{lemma}
\begin{proof}
First note that by definition $\mathfrak{C} \mcQ_{<\gamma}(\tau,\hat{\tau}, \mfn, \mfe )= \mcQ_{<\gamma} \mathfrak{C}(\tau,\hat{\tau}, \mfn, \mfe )$. By unraveling the definition of the right hand side we then find that
\begin{align*}
(\mathfrak{C} \otimes \id) \tilde{\Delta} (\tau, \hat{\tau}, \mfn, \mfe )
&= \sum_{\substack{
\tilde{C}\in \mathbb{C}_+[\tau]\\
 \tilde{C}\cap E(\hat{\tau})= \emptyset }
} \sum_{n_{\ng \tilde{C}}}
			\sum_{\ell^{\geq C}_{\ng \tilde{C}}, \mfk^{\geq C}_{\ng \tilde{C} }}  
			\sum_{\mck_{\tilde{C} }}  
			 {\mfn \choose n_{\ng \tilde{C}} } {\mfe \choose \ell^{\geq C}_{\ng \tilde{C}}}  
			\frac{1}{\mfk^{\geq C}_{\ng \tilde{C}} !} \frac{1}{\mck_{\tilde{C}} !}  \\ 
			&
			\mcQ_{< \gamma_{\mathfrak{C}(\tau, \hat{\tau}, \mfn, \mfe )}} \mathfrak{C} (\tau_{\ng \tilde{C}} , \hat{\tau},  n_{\ng \tilde{C}} + \pi (\mfk^{\geq C}_{\ng \tilde{C}} +{\mck_{\tilde{C}}}) , \mfe - \ell^{\geq C}_{\ng \tilde{C}} , \mfk^{\geq C}_{\ng \tilde{C}} + \ell^{\geq C}_{\ng \tilde{C}} ) 
			\\
			 & \otimes \mfp_{\mcT_{\sim}} \mathfrak{C} (\tau, {\tau_{\ng \tilde{C}}}, \mfn- n_{\ng \tilde{C}}, \mfe +\mck_{\tilde{C}}  )  \ .			 
\end{align*}
Comparing to $\tilde{\Delta}  \mathfrak{C} (\tau,\hat{\tau}, \mfn, \mfe )$ the claim follows directly from the formula  \eqref{eq:dt_expl} 
by recalling the canonical bijection between the edges of the tree $\mathfrak{C} (\tau, \hat{\tau})$ and the set $E\setminus E(\hat{\tau})$ and a simple application of the Chu-Vandermonde identity in order to see that the node decoration at the root agrees on both sides.
\end{proof}

\begin{proof}[Proof of Lemma~\ref{lem: Delta_comb}]

Our strategy of proof is to unravel the definitions of $(\Delta_r^-  \otimes \id) \tilde{\Delta}\tau $ and
$ \big(\id  \otimes \tilde{\Delta}\big) \Delta_r^- \tau$ respectively and perform several changes of variables in the resulting sums. This step is analagous to the proof of \cite[Proposition 3.11]{BHZ19} with the main difference being that we do not have the extended decoration $\mfo$ considered there but do have an additional `over-decoration'. Unlike in \cite{BHZ19}, it is not true that a procedure of repeated substitutions will show that the two terms match since negative renormalisation does not commute with the projection $\mcQ_{< \gamma_\tau}$ and we do not introduce an extended decoration to handle this mismatch. Instead, we show that due to the definition of $\gamma_\tau$, this mismatch vanishes when $\Pi_y^\Ban [\cdot](y)$ is applied. 

Throughout the proof, we will adopt the convention that $C$ denotes a generic element of $\mbC[\tau]$ and $\tilde{C}$ denotes a generic element of $\mbC_+[\tau]$. For a pair of cuts $C_1, C_2$, we write $C_1 \ge C_2$ if for every $e_1 \in C_1$ there exists an $e_2 \in C_2$ with $e_1 \ge e_2$. We define $C_1 \le C_2$ similarly. We warn the reader that the quantifiers appearing mean that it is not true that $C_1 \ge C_2$ if and only if $C_2 \le C_1$. By Proposition~\ref{prop: tDelta_form} and Definition~\ref{d:delta_r_-}, we have 
\begin{align*} 
	&(\Delta_r^-  \otimes \id)\tilde{\Delta} (\tau, \mfn,\mfe)\\  &=  \sum_{C \ng \tilde{C}} \sum_{\substack{n_{\ng \tilde{C}} \\ n_{\ng C \cup \tilde{C}}}}\sum_{ \mck_{\ng \tilde{C} } ,\ell_{\ng \tilde{C}}}
	\sum_{\eps_C, \eps_{\tilde{C} } }   
	{\mfn \choose n_{\ge \tilde{C}} }  {n_{\ng \tilde{C}} + \pi (\mck_{\ng \tilde{C}} +{\eps_{\tilde{C}}})  \choose n_{\ng {C \cup \tilde{C}}} }
	{\mfe \choose \ell_{\ng \tilde{C}}}  
	\frac{1}{\mck_{\ng \tilde{C}} !}  \frac{1}{\eps_C!} \frac{1}{\eps_{\tilde{C}} !} 
	\mathbf{1}_{|(\tau_{\ng \tilde{C}}, n_{\ng \tilde{C}} + \pi(\mck_{\ng \tilde{C}} + \eps_{\tilde{C}}, \mfe - \ell_{\ng \tilde{C}})|_\mfs < \gamma_{(\tau, \mfn, \mfe)}}
	\\
	& \qquad \qquad\qquad \qquad\qquad \qquad 
	\mfp_{\tilde{\mcT}_{-}}(\tau_{\ng C \cup \tilde{C}} , n_{\ng C \cup \tC} + \pi\eps_C
	, \mfe - \ell_{\ng \tilde{C}} , \mck_{\ng \tilde{C}} + \ell_{\ng \tilde{C}} ) \\
	& \qquad \qquad\qquad \qquad\qquad \qquad 
	\otimes
	\mathfrak{C} (\tau_{\ng \tilde{C}} ,\tau_{\ng C} , n_{\ng \tilde{C}} - n_{\ng C \cup \tilde{C}} + \pi (\mck_{\ng \tilde{C}} +{\eps_{\tilde{C}}}) 
	, \mfe - \ell_{\ng\tilde{C}} +\eps_C , \mck_{\ng \tilde{C}} + \ell_{\ng \tilde{C}} ) \\
	& \qquad \qquad\qquad \qquad\qquad \qquad 
	\otimes \mfp_{\mcT_{\sim}} \mathfrak{C} (\tau , \tau_{\ng \tilde{C}}, \mfn- n_{\ng \tilde{C}}, \mfe 
	+\eps_{\tilde{C}}  ) .
\end{align*}
We define $L(\tau, \mfn, \mfe) = (\id \otimes \mcQ_{> 0} \otimes \id)(\Delta_r^- \otimes \id)\tilde{\Delta}(\tau, \mfn, \mfe)$ and note that $(\ell^\Ban \circ \iota_y \otimes \Pi_y^{\Ban, \times} \circ \iota_y [\cdot](y) \otimes \chi_y^x) L(\tau, \mfn, \mfe) = 0$ so that omitting $L(\tau, \mfn, \mfe)$ will not affect the final result. We then have that
\begin{align*} 
	&(\Delta_r^-  \otimes \id)\tilde{\Delta} (\tau, \mfn,\mfe) - L(\tau, \mfn, \mfe)  \\ & = \sum_{C \ng \tilde{C}} \sum_{\substack{n_{\ng \tilde{C}} \\ n_{\ng C \cup \tilde{C}}}}\sum_{ \mck_{\ng \tilde{C} } ,\ell_{\ng \tilde{C}}}
	\sum_{\eps_C, \eps_{\tilde{C} } }   
	{\mfn \choose n_{\ge \tilde{C}} }  {n_{\ng \tilde{C}} + \pi (\mck_{\ng \tilde{C}} +{\eps_{\tilde{C}}})  \choose n_{\ng {C \cup \tilde{C}}} }
	{\mfe \choose \ell_{\ng \tilde{C}}}  	\frac{1}{\mck_{\ng \tilde{C}} !}  \frac{1}{\eps_C!} \frac{1}{\eps_{\tilde{C}} !} 
	\mathbf{1}_{|(\tau_{\ng \tilde{C}}, n_{\ng \tilde{C}} + \pi(\mck_{\ng \tilde{C}} + \eps_{\tilde{C}}, \mfe - \ell_{\ng \tilde{C}})|_\mfs < \gamma_{(\tau, \mfn, \mfe)}}
	\\
	& \qquad \qquad\qquad \qquad\qquad \qquad 
	\mfp_{\tilde{\mcT}_{-}}(\tau_{\ng C \cup \tilde{C}} , n_{\ng{C} \cup \tC} + \pi\eps_C
	, \mfe - \ell_{\ng \tilde{C}} , \mck_{\ng \tilde{C}} + \ell_{\ng \tilde{C}} ) \\
	& \qquad \qquad\qquad \qquad\qquad \qquad 
	\otimes
	\mcQ_{\le 0} \mathfrak{C} (\tau_{\ng \tilde{C}} ,\tau_{\ng C} , n_{\ng \tilde{C}} - n_{\ng C \cup \tilde{C}} + \pi (\mck_{\ng \tilde{C}} +{\eps_{\tilde{C}}}) 
	, \mfe - \ell_{\ng\tilde{C}} +\eps_C , \mck_{\ng \tilde{C}} + \ell_{\ng \tilde{C}} ) \\
	& \qquad \qquad\qquad \qquad\qquad \qquad 
	\otimes \mfp_{\mcT_{\sim}} \mathfrak{C} (\tau , \tau_{\ng \tilde{C}}, \mfn- n_{\ng \tilde{C}}, \mfe 
	+\eps_{\tilde{C}}  ) .
\end{align*}
We then note that since $\Delta_r^- \tau = \sum \tau^1 \otimes \tau^2$ implies that $|\tau^2|_\mfs \ge |\tau|_\mfs$, we have that 
\begin{align*}
	|	\mathfrak{C} (\tau_{\ng \tilde{C}} ,\tau_{\ng C} , n_{\ng \tilde{C}} - n_{\ng C \cup \tilde{C}} + \pi (\mck_{\ng \tilde{C}} +{\eps_{\tilde{C}}}) 
	, \mfe - \ell_{\ng\tilde{C}} +\eps_C , \mck_{\ng \tilde{C}} + \ell_{\ng \tilde{C}} ) |_\mfs \ge |(\tau_{\ng \tilde{C}}, n_{\ng \tilde{C}} + \pi(\mck_{\ng \tilde{C}} + \eps_{\tilde{C}}, \mfe - \ell_{\ng \tilde{C}})|_\mfs. 
\end{align*}
Since $\gamma_{(\tau, \mfn, \mfe)} > 0$, this implies that 
\begin{align} \label{MT_exp} 
	&(\Delta_r^-  \otimes \id)\tilde{\Delta} (\tau, \mfn,\mfe) - L(\tau, \mfn, \mfe)  \\ \nonumber & = \sum_{C \ng \tilde{C}} \sum_{\substack{n_{\ng \tilde{C}} \\ n_{\ng C \cup \tilde{C}}}}\sum_{ \mck_{\ng \tilde{C} } ,\ell_{\ng \tilde{C}}}
	\sum_{\eps_C, \eps_{\tilde{C} } }   
	{\mfn \choose n_{\ge \tilde{C}} }  {n_{\ng \tilde{C}} + \pi (\mck_{\ng \tilde{C}} +{\eps_{\tilde{C}}})  \choose n_{\ng {C \cup \tilde{C}}} }
	{\mfe \choose \ell_{\ng \tilde{C}}}  	\frac{1}{\mck_{\ng \tilde{C}} !}  \frac{1}{\eps_C!} \frac{1}{\eps_{\tilde{C}} !} 
	\\ \nonumber
	& \qquad \qquad\qquad \qquad\qquad \qquad 
	\mfp_{\tilde{\mcT}_{-}}(\tau_{\ng C \cup \tilde{C}} , n_{\ng{C} \cup \tC} + \pi\eps_C
	, \mfe - \ell_{\ng \tilde{C}} , \mck_{\ng \tilde{C}} + \ell_{\ng \tilde{C}} ) \\ \nonumber
	& \qquad \qquad\qquad \qquad\qquad \qquad 
	\otimes
	\mcQ_{\le 0} \mathfrak{C} (\tau_{\ng \tilde{C}} ,\tau_{\ng C} , n_{\ng \tilde{C}} - n_{\ng C \cup \tilde{C}} + \pi (\mck_{\ng \tilde{C}} +{\eps_{\tilde{C}}}) 
	, \mfe - \ell_{\ng\tilde{C}} +\eps_C , \mck_{\ng \tilde{C}} + \ell_{\ng \tilde{C}} ) \\ \nonumber
	& \qquad \qquad\qquad \qquad\qquad \qquad 
	\otimes \mfp_{\mcT_{\sim}} \mathfrak{C} (\tau , \tau_{\ng \tilde{C}}, \mfn- n_{\ng \tilde{C}}, \mfe 
	+\eps_{\tilde{C}}  ) .
\end{align}

We now turn to consider $(\id \otimes \tilde{\Delta})\Delta_r^-(\tau, \mfn, \mfe)$. By Proposition~\ref{prop: tDelta_form} and Definition~\ref{d:delta_r_-}, we have
\begin{align*}
	&(\id \otimes \tilde{\Delta})\Delta_r^-(\tau, \mfn, \mfe) = 
	\\ & \sum_{\tilde{C} \ge C} \sum_{n_{\ng C}, n_{\ng \tilde{C}}} \sum_{\eps_{C}, \eps_{\tilde{C}}} \sum_{\ell_{\ng \tilde{C}}^{\ge C}} \sum_{\mck_{\ng \tilde{C}}^{\ge C}} {\mfn \choose n_{\ng C}} {n - n_{\ng C} \choose n_{\ng \tC}} {\mfe+ \eps_C \choose \ell_{\ng \tC}^{\ge C}} \frac{1}{\eps_C!}\frac{1}{\eps_{\tC}!} \frac{1}{\mck_{\ng \tC}^{\ge C}!} \mathbf{1}_{| 	\mathfrak{C} (\tau_{\ng \tilde{C}} ,\tau_{\ng C} , n_{\ng \tilde{C}}  + \pi (\mck_{\ng \tilde{C}}^{\ge C} +{\eps_{\tilde{C}}}) |_\mfs < \gamma_{\mfC (\tau, \tau_{\ng C}, n - n_{\ng C}, \mfe + \eps_C)}}
		\\
	& \qquad \qquad\qquad \qquad\qquad \qquad 
	\mfp_{\mcT_{-}}(\tau_{\ng C} , n_{\ng{C}} + \pi\eps_C
	, \mfe) \\
	& \qquad \qquad\qquad \qquad\qquad \qquad 
	\otimes
	\mathfrak{C} (\tau_{\ng \tilde{C}} ,\tau_{\ng C} , n_{\ng \tilde{C}}  + \pi (\mck_{\ng \tilde{C}}^{\ge C} +{\eps_{\tilde{C}}}) 
	, \mfe - \ell_{\ng\tilde{C}}^{\ge C} +\eps_C , \mck_{\ng \tilde{C}}^{\ge C} + \ell_{\ng \tilde{C}}^{\ge C} ) \\
	& \qquad \qquad\qquad \qquad\qquad \qquad 
	\otimes \mfp_{\mcT_{\sim}} \mathfrak{C} (\tau , \tau_{\ng \tilde{C}}, \mfn- n_{\ng C} - n_{\ng \tC}, \mfe 
	+\eps_{{C}}  + \eps_{\tC}) .
\end{align*}
We again separate the right hand side into two parts depending on the sign of the degree of the middle tree in the tensor product. We note that since since $\gamma_\sigma$ was constructed recursively starting from $\gamma_0$ that is admissible for $\mcB$ and $\mcB$ is the generating set of a historic sector, we have that $\gamma_{\mfC (\tau, \tau_{\ng C}, n - n_{\ng C}, \mfe + \eps_C)} > 0$ .  Therefore, as before we can drop the indicator function in the part of the sum where the middle tree in the tensor is of negative degree. This leads us to
\begin{align*}
	&(\id \otimes \tilde{\Delta})\Delta_r^-(\tau, \mfn, \mfe) - R(\tau, \mfn, \mfe) = 
	\\ & \sum_{\tilde{C} \ge C} \sum_{n_{\ng C}, n_{\ng \tilde{C}}} \sum_{\eps_{C}, \eps_{\tilde{C}}} \sum_{\ell_{\ng \tilde{C}}^{\ge C}} \sum_{\mck_{\ng \tilde{C}}^{\ge C}} {n \choose n_{\ng C}} {n - n_{\ng C} \choose n_{\ng \tC}} {\mfe + \eps_C \choose \ell_{\ng \tC}^{\ge C}} \frac{1}{\eps_C!}\frac{1}{\eps_{\tC}!} \frac{1}{\mck_{\ng \tC}^{\ge C}!} 
	\\
	& \qquad \qquad\qquad \qquad\qquad \qquad 
	\mfp_{\mcT_{-}}(\tau_{\ng C} , n_{\ng{C}} + \pi\eps_C
	, \mfe) \\
	& \qquad \qquad\qquad \qquad\qquad \qquad 
	\otimes
	\mcQ_{\le 0} \mathfrak{C} (\tau_{\ng \tilde{C}} ,\tau_{\ng C} , n_{\ng \tilde{C}}  + \pi (\mck_{\ng \tilde{C}}^{\ge C} +{\eps_{\tilde{C}}}) 
	, \mfe - \ell_{\ng\tilde{C}}^{\ge C} +\eps_C , \mck_{\ng \tilde{C}}^{\ge C} + \ell_{\ng \tilde{C}}^{\ge C} ) \\
	& \qquad \qquad\qquad \qquad\qquad \qquad 
	\otimes \mfp_{\mcT_{\sim}} \mathfrak{C} (\tau , \tau_{\ng \tilde{C}}, \mfn- n_{\ng C} - n_{\ng \tC}, \mfe 
	+\eps_{C}  + \eps_{\tC}) 
\end{align*}
where $(\ell^\Ban \circ \iota_y \otimes \Pi_y^{\Ban, \times} \circ \iota_y [\cdot](y) \otimes \chi_y^x) R(\tau, \mfn, \mfe) = 0$. We now perform a series of substitutions into the above expression to bring $(\id \otimes \tilde{\Delta})\Delta_r^-(\tau, \mfn, \mfe) - R(\tau, \mfn, \mfe)$ closer to the form of $(\Delta_r^-  \otimes \id)\tilde{\Delta} (\tau, \mfn,\mfe) - L(\tau, \mfn, \mfe)$.

We begin by writing $C$ appearing in the sum above as $C^* \sqcup C^\dagger$ where $C^\dagger = C \cap \tC$. We note that this implies that $C^* \ng \tC$ and that given $C^*$ and $\tC$ we can uniquely recover $C^\dagger = \{e \in \tC: \not \exists e' \in C^*, e > e'\}$. We also decompose $\eps_C = \eps_{C^*} + \eps_{\dC}$. We can then write
\begin{align*}
	&(\id \otimes \tilde{\Delta})\Delta_r^-(\tau, \mfn, \mfe) - R(\tau, \mfn, \mfe) = 
	\\ & \sum_{C^* \ng \tC} \sum_{n_{\ng C^* \cup \tC}, n_{\ng \tilde{C}}} \sum_{\eps_{C^*}, \eps_{\dC}, \eps_{\tilde{C}}} \sum_{\ell_{\ng \tilde{C}}^{\ge C^*}} \sum_{\mck_{\ng \tilde{C}}^{\ge C^*}} {\mfn \choose n_{\ng C^* \cup \tC}} {\mfn - n_{\ng C^* \cup \tC} \choose n_{\ng \tC}} {\mfe + \eps_{C^*} \choose \ell_{\ng \tC}^{\ge C^*}} \frac{1}{\eps_{\dC}! \eps_{C^*}!}\frac{1}{\eps_{\tC}!} \frac{1}{\mck_{\ng \tC}^{\ge C^*}!} 
	\\
	& \qquad \qquad\qquad \qquad\qquad \qquad 
	\mfp_{\mcT_{-}}(\tau_{\ng C^* \cup \tC} , n_{\ng{C^* \cup \tC}} + \pi (\eps_{C^*} + \eps_{\dC})
	, \mfe) \\
	& \qquad \qquad\qquad \qquad\qquad \qquad 
	\otimes
	\mcQ_{\le 0} \mathfrak{C} (\tau_{\ng \tilde{C}} ,\tau_{\ng C^* \cup \tC} , n_{\ng \tilde{C}}  + \pi (\mck_{\ng \tilde{C}}^{\ge C^*} +{\eps_{\tilde{C}}}) 
	, \mfe - \ell_{\ng\tilde{C}}^{\ge C^*} +\eps_{C^*} , \mck_{\ng \tilde{C}}^{\ge C^*} + \ell_{\ng \tilde{C}}^{\ge C^*} ) \\
	& \qquad \qquad\qquad \qquad\qquad \qquad 
	\otimes \mfp_{\mcT_{\sim}} \mathfrak{C} (\tau , \tau_{\ng \tilde{C}}, \mfn- n_{\ng \tilde{C}} - n_{\ng \tC}, \mfe 
	+\eps_{\dC}  + \eps_{\tC}) 
\end{align*}
where we made use of the supports of $\eps_{C^*}$ and $\eps_{\dC}$ and the fact that the contraction operator is invariant under changing the edge decoration on the coloured component in order to simplify expressions.

We now make the substitutions $\bar{\eps}_{\tC} = \eps_{\tC} + \eps_{\dC}$, $\bar{n}_{\ng \tC} = n_{\ng C^* \cup \tC} + n_{\ng \tC}$ and $\bar{n}_{\ng C^* \cup \tC} = n_{\ng C^* \cup \tC} + \pi \eps_{\dC}$. We introduce the convention that $\frac{1}{k!} = 0$ if $k < 0$ so that the resulting constraints on the domains of summation are automatically encoded in the multinomial coefficients. This leads to
\begin{align*}
	&(\id \otimes \tilde{\Delta})\Delta_r^-(\tau, \mfn, \mfe) - R(\tau, \mfn, \mfe) = 
	\\ & \sum_{C^* \ng \tC} \sum_{\bar{n}_{\ng C^* \cup \tC},  \bar{n}_{\ng \tilde{C}}} \sum_{\eps_{C^*}, \eps_{\dC}, \bar{\eps}_{\tilde{C}}} \sum_{\ell_{\ng \tilde{C}}^{\ge C^*}} \sum_{\mck_{\ng \tilde{C}}^{\ge C^*}} {\mfn \choose \bar{n}_{\ng C^* \cup \tC} - \pi \eps_{\dC}} {\mfn - \bar{n}_{\ng C^* \cup \tC} + \pi \eps_{\dC} \choose \bar{n}_{\ng \tC} - \bar{n}_{\ng C^* \cup \tC} + \pi \eps_{\dC}} {\mfe + \eps_{C^*} \choose \ell_{\ng \tC}^{\ge C^*}} \frac{1}{\eps_{\dC}! \eps_{C^*}!}\frac{1}{(\bar{\eps}_{\tC} - \eps_{\dC})!} 
	\\
	& \qquad \qquad\qquad \qquad\qquad \qquad 
	\frac{1}{\mck_{\ng \tC}^{\ge C^*}!}  \mfp_{\mcT_{-}}(\tau_{\ng C^* \cup \tC} , \bar{n}_{\ng{C^* \cup \tC}} + \pi (\eps_{C^*})
	, \mfe) \\
	& \qquad \qquad\qquad \qquad\qquad \qquad 
	\otimes
	\mcQ_{\le 0} \mathfrak{C} (\tau_{\ng \tilde{C}} ,\tau_{\ng C^* \cup \tC} , \bar{n}_{\ng \tilde{C}} -  \bar{n}_{\ng{C^* \cup \tC}}  + \pi (\mck_{\ng \tilde{C}}^{\ge C^*} +{\bar{\eps}_{\tilde{C}}}) 
	, \mfe - \ell_{\ng\tilde{C}}^{\ge C^*} +\eps_{C^*} , \mck_{\ng \tilde{C}}^{\ge C^*} + \ell_{\ng \tilde{C}}^{\ge C^*} ) \\
	& \qquad \qquad\qquad \qquad\qquad \qquad 
	\otimes \mfp_{\mcT_{\sim}} \mathfrak{C} (\tau , \tau_{\ng  \tilde{C}}, \mfn- \bar{n}_{\ng \tilde{C}}, \mfe +
	\bar{\eps}_{\tC}) 
\end{align*}

We then note that
\begin{align*}
	{\mfn \choose \bar{n}_{\ng C^* \cup \tC} - \pi \eps_{\dC}}  {n - \bar{n}_{\ng C^* \cup \tC} + \pi \eps_{\dC} \choose \bar{n}_{\ng \tC} - \bar{n}_{\ng C^* \cup \tC} + \pi \eps_{\dC}} = {\mfn \choose \bar{n}_{\ng \tC}} {\bar{n}_{\ng \tC} \choose \bar{n}_{\ng C^* \cup \tC} - \pi \eps_{\dC}}, \quad \frac{1}{\eps_{\dC}!}\frac{1}{(\bar{\eps}_{\tC} -  \eps_{\dC})!} = {\bar{\eps}_{\tC} \choose \eps_{\dC}} \frac{1}{\bar{\eps}_{\tC}!}.
\end{align*}
Substituting in these identities and noting that the trees do not depend on $\eps_{\dC}$, we see that the sum over $\eps_{\dC}$ factors out in the form
\begin{align*}
	\sum_{\eps_{\dC}} {\bar{n}_{\ng \tC} \choose \bar{n}_{\ng C^* \cup \tC} - \pi \eps_{\dC}} {\bar{\eps}_{\tC} \choose \eps_{\dC}} = {\bar{n}_{\ng \tC} + \pi \bar{\eps}_{\tC} \choose \bar{n}_{\ng C^* \cup \tC}}
\end{align*}
where the equality follows from the Chu-Vandermonde identity (see also \cite[Lemma 2.1] {BHZ19}). Therefore, we obtain that
\begin{align*}
&(\id \otimes \tilde{\Delta})\Delta_r^-(\tau, \mfn, \mfe) - R(\tau, \mfn, \mfe) = 
\\ & \sum_{C^* \ng \tC} \sum_{\bar{n}_{\ng C^* \cup \tC},  \bar{n}_{\ng \tilde{C}}} \sum_{\eps_{C^*}, \bar{\eps}_{\tilde{C}}} \sum_{\ell_{\ng \tilde{C}}^{\ge C^*}} \sum_{\mck_{\ng \tilde{C}}^{\ge C^*}} {\mfn \choose \bar{n}_{\ng \tC}} {\bar{n}_{\ng \tC} + \pi \bar{\eps}_{\tC} \choose \bar{n}_{\ng C^* \cup \tC}} {\mfe + \eps_{C^*} \choose \ell_{\ng \tC}^{\ge C^*}} \frac{1}{ \eps_{C^*}!}\frac{1}{\bar{\eps}_{\tC} !} \frac{1}{\mck_{\ng \tC}^{\ge C^*}!}  
\\
& \qquad \qquad\qquad \qquad\qquad \qquad 
 \mfp_{\mcT_{-}}(\tau_{\ng C^* \cup \tC} , \bar{n}_{\ng{C^* \cup \tC}} + \pi (\eps_{C^*})
, \mfe) \\
& \qquad \qquad\qquad \qquad\qquad \qquad 
\otimes
\mcQ_{\le 0} \mathfrak{C} (\tau_{\ng \tilde{C}} ,\tau_{\ng C^* \cup \tC} , \bar{n}_{\ng \tilde{C}} -  \bar{n}_{\ng{C^* \cup \tC}}  + \pi (\mck_{\ng \tilde{C}}^{\ge C^*} +{\bar{\eps}_{\tilde{C}}}) 
, \mfe - \ell_{\ng\tilde{C}}^{\ge C^*} +\eps_{C^*} , \mck_{\ng \tilde{C}}^{\ge C^*} + \ell_{\ng \tilde{C}}^{\ge C^*} ) \\
& \qquad \qquad\qquad \qquad\qquad \qquad 
\otimes \mfp_{\mcT_{\sim}} \mathfrak{C} (\tau , \tau_{\ng\tilde{C}}, \mfn- \bar{n}_{\ng \tilde{C}}, \mfe +
\bar{\eps}_{\tC}). 
\end{align*}
We now observe that the left-most tree in the above expression and in \eqref{MT_exp} differ. This is accounted for by the fact that we apply $\ell^\Ban \circ \iota_y$ to the left-most tree in the latter expression but $\ell_y^\eq = \ell^\Ban \circ \iota_y \circ \mfD$ to the left-most slot in the former expression where we recall that $\mfD$ was defined in \eqref{eq:D_def}.
We have that that $\mfD \circ \mfp_{\mcT_-} = \mfD$ since $|(\tau, \mfn + \pi \mck, \mfe - \ell, \mck + \ell)|_\mfs \ge |(\tau, \mfn, \mfe)|_\mfs$. Therefore, 
\begin{align*}
	&(\mfD \otimes \id \otimes \id) \big ((\id \otimes \tilde{\Delta})\Delta_r^-(\tau, \mfn, \mfe) - R(\tau, \mfn, \mfe) \big )= 
	\\ & \sum_{C^* \ng \tC} \sum_{\bar{n}_{\ng C^* \cup \tC},  \bar{n}_{\ng \tilde{C}}} \sum_{\eps_{C^*}, \bar{\eps}_{\tilde{C}}} \sum_{\substack{\ell_{\ng \tilde{C}}^{\ge C^*} \\ \ell_{\ng C^* \cup \tC}}} \sum_{\substack{\mck_{\ng \tilde{C}}^{\ge C^*} \\ \mck_{\ng C^* \cup \tC}}}  {\mfn \choose \bar{n}_{\ng \tC}} {\bar{n}_{\ng \tC} + \pi \bar{\eps}_{\tC} \choose \bar{n}_{\ng C^* \cup \tC}} {\mfe + \eps_{C^*} \choose \ell_{\ng \tC}^{\ge C^*}} {\mfe \choose \ell_{\ng C^* \cup \tC}} \frac{1}{ \eps_{C^*}!}\frac{1}{\bar{\eps}_{\tC}!} \frac{1}{\mck_{\ng \tC}^{\ge C^*}!}   \frac{1}{\mck_{\ng C^* \cup \tC}!}
	\\
	& \qquad \qquad\qquad \qquad\qquad \qquad 
	\mfp_{\tilde{\mcT}_{-}}(\tau_{\ng C^* \cup \tC} , \bar{n}_{\ng{C^* \cup \tC}} + \pi (\eps_{C^*} + \mck_{\ng C^* \cup \tC})
	, \mfe - \ell_{\ng C^* \cup \tC}, \ell_{\ng C^* \cup \tC} + \mck_{\ng C^* \cup \tC}) \\
	& \qquad \qquad\qquad \qquad\qquad \qquad 
	\otimes
	\mcQ_{\le 0} \mathfrak{C} (\tau_{\ng \tilde{C}} ,\tau_{\ng C^* \cup \tC} , \bar{n}_{\ng \tilde{C}} -  \bar{n}_{\ng{C^* \cup \tC}}  + \pi (\mck_{\ng \tilde{C}}^{\ge C^*} +{\bar{\eps}_{\tilde{C}}}) 
	, \mfe - \ell_{\ng\tilde{C}}^{\ge C^*} +\eps_{C^*} , \mck_{\ng \tilde{C}}^{\ge C^*} + \ell_{\ng \tilde{C}}^{\ge C^*} ) \\
	& \qquad \qquad\qquad \qquad\qquad \qquad 
	\otimes \mfp_{\mcT_{\sim}} \mathfrak{C} (\tau , \tau_{\ng\tilde{C}}, \mfn- \bar{n}_{\ng \tilde{C}}, \mfe +
	\bar{\eps}_{\tC}). 
\end{align*}
We now write
\begin{align*}
	{\mfe + \eps_{C^*} \choose \ell_{\ng \tC}^{\ge C^*}} = {\mfe \choose \ell_{\ng \tC}^{> C^*}} {\mfe + \eps_{C^*} \choose \ell_{\ng \tC}^{=C^*}} = {\mfe \choose \ell_{\ng \tC}^{> C^*}} \sum_{ \ell_{\ng \tC}^{=C^*}  = \hat{\ell}_{\ng \tC}^{=C^*} + \tilde{\ell}_{\ng \tC}^{=C^*}} {\mfe \choose \hat{\ell}_{\ng \tC}^{=C^*} } {\eps_{C^*} \choose \tilde{\ell}_{\ng \tC}^{=C^*} }
\end{align*}
where the last equality follows by Chu-Vandermonde. Writing $\ell_{\ng \tC} = \hat{\ell}_{\ng \tC}^{= C^*} + \ell_{\ng \tC}^{> C^*} + \ell_{\ng C^* \cup \tC}$, this yields
\begin{align*}
	&(\mfD \otimes \id \otimes \id) \big ((\id \otimes \tilde{\Delta})\Delta_r^-(\tau, \mfn, \mfe) - R(\tau, \mfn, \mfe) \big )= 
	\\ & \sum_{C^* \ng \tC} \sum_{\bar{n}_{\ng C^* \cup \tC},  \bar{n}_{\ng \tilde{C}}} \sum_{\eps_{C^*}, \bar{\eps}_{\tilde{C}}} \sum_{\substack{\ell_{\ng \tilde{C}} \\ \tilde{\ell}_{\ng \tC}^{=C^*}}} \sum_{\substack{\mck_{\ng \tilde{C}}^{\ge C^*} \\ \mck_{\ng C^* \cup \tC}}} 
	{\mfn \choose \bar{n}_{\ng \tC}} {\bar{n}_{\ng \tC} + \pi \bar{\eps}_{\tC} \choose \bar{n}_{\ng C^* \cup \tC}} {\mfe  \choose \ell_{\ng \tC}} {\eps_{C^*} \choose \tilde{\ell}_{\ng \tC}^{=C^*}} \frac{1}{ \eps_{C^*}!}\frac{1}{\bar{\eps}_{\tC}!} \frac{1}{\mck_{\ng \tC}^{\ge C^*}!}   \frac{1}{\mck_{\ng C^* \cup \tC}!}
	\\
	&  \qquad\qquad \qquad 
	\mfp_{\tilde{\mcT}_{-}}(\tau_{\ng C^* \cup \tC} , \bar{n}_{\ng{C^* \cup \tC}} + \pi (\eps_{C^*} + \mck_{\ng C^* \cup \tC})
	, \mfe - \ell_{\ng C^* \cup \tC}, \ell_{\ng  \tC} + \mck_{\ng C^* \cup \tC}) \\
	&  \qquad\qquad \qquad 
	\otimes
	\mcQ_{\le 0} \mathfrak{C} (\tau_{\ng \tilde{C}} ,\tau_{\ng C^* \cup \tC} , \bar{n}_{\ng \tilde{C}} -  \bar{n}_{\ng{C^* \cup \tC}}  + \pi (\mck_{\ng \tilde{C}}^{\ge C^*} +{\bar{\eps}_{\tilde{C}}}) 
	, \mfe - \ell_{\ng\tilde{C}} - \tilde{\ell}_{\ng \tC}^{=C^*} +\eps_{C^*} , \mck_{\ng \tilde{C}}^{\ge C^*} + \ell_{\ng \tilde{C}} + \tilde{\ell}_{\ng \tC}^{=C^*}) \\
	&  \qquad\qquad \qquad 
	\otimes \mfp_{\mcT_{\sim}} \mathfrak{C} (\tau , \tau_{\ng\tilde{C}}, \mfn- \bar{n}_{\ng \tilde{C}}, \mfe +
	\bar{\eps}_{\tC}). 
\end{align*}
To prepare to perform the sum over $\tilde{\ell}_{\ng \tC}^{=C^*}$, we substitute $\bar{\eps}_{C^*} = \eps_{C^*} - \tilde{\ell}_{\ng \tC}^{=C^*}$, $\hat{n}_{\ng C^* \cup \tC} = \bar{n}_{\ng C^* \cup \tC} + \pi \tilde{\ell}_{\ng \tC}^{=C^*}$ and $\hat{\mck}_{\ng \tC}^{\ge C^*} = \mck_{\ng \tC}^{\ge C^*} + \tilde{\ell}_{\ng \tC}^{=C^*}$. This yields
\begin{align*}
	&(\mfD\otimes \id \otimes \id) \big ((\id \otimes \tilde{\Delta})\Delta_r^-(\tau, \mfn, \mfe) - R(\tau, \mfn, \mfe) \big )= 
	\\ & \sum_{C^* \ng \tC} \sum_{\hat{n}_{\ng C^* \cup \tC},  \bar{n}_{\ng \tilde{C}}} \sum_{\bar{\eps}_{C^*}, \bar{\eps}_{\tilde{C}}} \sum_{\substack{\ell_{\ng \tilde{C}} \\ \tilde{\ell}_{\ng \tC}^{=C^*}}} \sum_{\substack{\hat{\mck}_{\ng \tilde{C}}^{\ge C^*} \\ \mck_{\ng C^* \cup \tC}}} 
	{\mfn \choose \bar{n}_{\ng \tC}} {\bar{n}_{\ng \tC} + \pi \bar{\eps}_{\tC} \choose \hat{n}_{\ng C^* \cup \tC} - \pi \tilde{\ell}_{\ng \tC}^{=C^*}} {\mfe  \choose \ell_{\ng \tC}} {\bar{\eps}_{C^*} + \tilde{\ell}_{\ng \tC}^{=C^*} \choose \tilde{\ell}_{\ng \tC}^{=C^*}} \frac{1}{ (\bar{\eps}_{C^*} + \tilde{\ell}_{\ng \tC}^{=C^*})!}\frac{1}{\bar{\eps}_{\tC}!} \frac{1}{(\bar{\mck}_{\ng \tC}^{\ge C^*} - \tilde{\ell}_{\ng \tC}^{=C^*})!}  
	\\
	&  \qquad\qquad \qquad  \frac{1}{\mck_{\ng C^* \cup \tC}!}
	\mfp_{\tilde{\mcT}_{-}}(\tau_{\ng C^* \cup \tC} , \hat{n}_{\ng{C^* \cup \tC}} + \pi (\bar{\eps}_{C^*} + \mck_{\ng C^* \cup \tC})
	, \mfe - \ell_{\ng C^* \cup \tC}, \ell_{\ng \tC} + \mck_{\ng C^* \cup \tC}) \\
	&  \qquad\qquad \qquad 
	\otimes
	\mcQ_{\le 0} \mathfrak{C} (\tau_{\ng \tilde{C}} ,\tau_{\ng C^* \cup \tC} , \bar{n}_{\ng \tilde{C}} -  \hat{n}_{\ng{C^* \cup \tC}}  + \pi (\hat{\mck}_{\ng \tilde{C}}^{\ge C^*} +{\bar{\eps}_{\tilde{C}}}) 
	, \mfe - \ell_{\ng\tilde{C}} + \bar{\eps}_{C^*} , \hat{\mck}_{\ng \tilde{C}}^{\ge C^*} + \ell_{\ng \tilde{C}}) \\
	&  \qquad\qquad \qquad 
	\otimes \mfp_{\mcT_{\sim}} \mathfrak{C} (\tau , \tau_{\ng\tilde{C}}, \mfn- \bar{n}_{\ng \tilde{C}}, \mfe +
	\bar{\eps}_{\tC})
\end{align*}
where we note that the constraints on the domain of summation appearing due to our substitutions are encoded by our convention on the multinomial coefficients. We note that
\begin{align*}
	{\bar{\eps}_{C^*} + \tilde{\ell}_{\ng \tC}^{=C^*} \choose \tilde{\ell}_{\ng \tC}^{=C^*}} \frac{1}{ (\bar{\eps}_{C^*} + \tilde{\ell}_{\ng \tC}^{=C^*})!} \frac{1}{(\bar{\mck}_{\ng \tC}^{\ge C^*} - \tilde{\ell}_{\ng \tC}^{=C^*})!}  = {\hat{\mck}_{\ng \tC}^{\ge C^*} \choose \tilde{\ell}_{\ng \tC}^{=C^*}} \frac{1}{\bar{\eps_{C^*}}} \frac{1}{\hat{\mck}_{\ng \tC}^{\ge C^*}!}.
\end{align*}
Substituting this in, we see that the sum over $\tilde{\ell}_{\ng \tC}^{=C^*}$ factors out in the form
\begin{align*}
	\sum_{\tilde{\ell}_{\ng \tC}^{=C^*}} {\hat{\mck}_{\ng \tC}^{\ge C^*} \choose \tilde{\ell}_{\ng \tC}^{=C^*}} {\bar{n}_{\ng \tC} + \pi \bar{\eps}_{\tC} \choose \hat{n}_{\ng C^* \cup \tC} - \pi \tilde{\ell}_{\ng \tC}^{=C^*}}  = {\bar{n}_{\ng \tC} + \pi(\bar{\eps}_{\tC} + \hat{\mck}_{\ng \tC}^{\ge C^*}) \choose \hat{n}_{\ng C^* \cup \tC}} 
\end{align*}
where the equality follows by Chu-Vandermonde. In total, we have obtained that
\begin{align*}
	&(\mfD \otimes \id \otimes \id) \big ((\id \otimes \tilde{\Delta})\Delta_r^-(\tau, \mfn, \mfe) - R(\tau, \mfn, \mfe) \big )= 
	\\ & \sum_{C^* \ng \tC} \sum_{\hat{n}_{\ng C^* \cup \tC},  \bar{n}_{\ng \tilde{C}}} \sum_{\bar{\eps}_{C^*}, \bar{\eps}_{\tilde{C}}} \sum_{\ell_{\ng \tilde{C}} } \sum_{\substack{\hat{\mck}_{\ng \tilde{C}}^{\ge C^*} \\ \mck_{\ng C^* \cup \tC}}} 
	{\mfn \choose \bar{n}_{\ng \tC}} {\bar{n}_{\ng \tC} + \pi(\bar{\eps}_{\tC} + \hat{\mck}_{\ng \tC}^{\ge C^*}) \choose \hat{n}_{\ng C^* \cup \tC}}  {\mfe  \choose \ell_{\ng \tC}} 
	 \frac{1}{\bar{\eps}_{C^*}!}\frac{1}{\bar{\eps}_{\tC}!} \frac{1}{\bar{\mck}_{\ng \tC}^{\ge C^*}!}   \frac{1}{\mck_{\ng C^* \cup \tC}!}
	\\
	&  \qquad\qquad \qquad 
	\mfp_{\tilde{\mcT}_{-}}(\tau_{\ng C^* \cup \tC} , \hat{n}_{\ng{C^* \cup \tC}} + \pi (\bar{\eps}_{C^*} + \mck_{\ng C^* \cup \tC})
	, \mfe - \ell_{\ng C^* \cup \tC}, \ell_{\ng  \tC} + \mck_{\ng C^* \cup \tC}) \\
	&  \qquad\qquad \qquad 
	\otimes
	\mcQ_{\le 0} \mathfrak{C} (\tau_{\ng \tilde{C}} ,\tau_{\ng C^* \cup \tC} , \bar{n}_{\ng \tilde{C}} -  \hat{n}_{\ng{C^* \cup \tC}}  + \pi (\hat{\mck}_{\ng \tilde{C}}^{\ge C^*} +{\bar{\eps}_{\tilde{C}}}) 
	, \mfe - \ell_{\ng\tilde{C}} + \bar{\eps}_{C^*} , \hat{\mck}_{\ng \tilde{C}}^{\ge C^*} + \ell_{\ng \tilde{C}}) \\
	&  \qquad\qquad \qquad 
	\otimes \mfp_{\mcT_{\sim}} \mathfrak{C} (\tau , \tau_{\ng\tilde{C}}, \mfn- \bar{n}_{\ng \tilde{C}}, \mfe +
	\bar{\eps}_{\tC}).
\end{align*}
We now substitute $\check{n}_{\ng C^* \cup \tC} = \hat{n}_{\ng C^* \cup \tC} + \pi \mck_{\ng C^* \cup \tC}$ and set $\mck_{\ng \tC} = \hat{\mck}_{\ng \tC}^{\ge C^*} + \mck_{\ng C^* \cup \tC}$. We note that since $\mck_{\ng C^* \cup \tC}$ and $\hat{\mck}_{\ng \tC}^{\ge C^*} $ have disjoint supports, both can be uniquely recovered from $\mck_{\ng \tC}$. This gives us
\begin{align*}
	&(\mfD \otimes \id \otimes \id) \big ((\id \otimes \tilde{\Delta})\Delta_r^-(\tau, \mfn, \mfe) - R(\tau, \mfn, \mfe) \big )= 
	\\ & \sum_{C^* \ng \tC} \sum_{\substack{\bar{n}_{\ng \tilde{C}} \\ \check{n}_{\ng C^* \cup \tC}}}
	\sum_{\mck_{\ng \tC}, \ell_{\ng \tilde{C}} }  \sum_{\bar{\eps}_{C^*}, \bar{\eps}_{\tilde{C}}} 
	{\mfn \choose \bar{n}_{\ng \tC}} {\bar{n}_{\ng \tC} + \pi(\bar{\eps}_{\tC} + \mck_{\ng \tC} - \mck_{\ng C^* \cup \tC}) \choose \check{n}_{\ng C^* \cup \tC} - \pi \mck_{\ng C^* \cup \tC}}  {\mfe  \choose \ell_{\ng \tC}} 
	\frac{1}{{\mck}_{\ng \tC}!}    \frac{1}{\bar{\eps}_{C^*}!}\frac{1}{\bar{\eps}_{\tC}!} 
	\\
	&  \qquad\qquad \qquad 
	\mfp_{\tilde{\mcT}_{-}}(\tau_{\ng C^* \cup \tC} , \check{n}_{\ng{C^* \cup \tC}} + \pi \bar{\eps}_{C^*}
	, \mfe - \ell_{\ng C^* \cup \tC}, \ell_{\ng  \tC} + \mck_{\ng  \tC}) \\
	&  \qquad\qquad \qquad 
	\otimes
	\mcQ_{\le 0} \mathfrak{C} (\tau_{\ng \tilde{C}} ,\tau_{\ng C^* \cup \tC} , \bar{n}_{\ng \tilde{C}} -  \check{n}_{\ng{C^* \cup \tC}}  + \pi ({\mck}_{\ng \tilde{C}} +{\bar{\eps}_{\tilde{C}}}) 
	, \mfe - \ell_{\ng\tilde{C}} + \bar{\eps}_{C^*} , {\mck}_{\ng \tilde{C}} + \ell_{\ng \tilde{C}}) \\
	&  \qquad\qquad \qquad 
	\otimes \mfp_{\mcT_{\sim}} \mathfrak{C} (\tau , \tau_{\ng\tilde{C}}, \mfn- \bar{n}_{\ng \tilde{C}}, \mfe +
	\bar{\eps}_{\tC})
\end{align*}
where for the edge decorations in the middle tree, we made use of the fact that the contraction is invariant under changing either edge decoration on the coloured component. We now note that (up to relabelling) this matches the expression \eqref{MT_exp} except for in the second multinomial coefficient evaluated at vertices in $\tau_{\ng C^* \cup \tC}$. 

However, we can observe that
\begin{align*}
	\Pi_y^\Ban \circ \iota_y \mathfrak{C} (\tau_{\ng \tilde{C}} ,\tau_{\ng C^* \cup \tC} , \bar{n}_{\ng \tilde{C}} -  \check{n}_{\ng{C^* \cup \tC}}  + \pi ({\mck}_{\ng \tilde{C}} +{\bar{\eps}_{\tilde{C}}}) 
	, \mfe - \ell_{\ng\tilde{C}} + \bar{\eps}_{C^*} , {\mck}_{\ng \tilde{C}} + \ell_{\ng \tilde{C}}) = 0
\end{align*}
unless $\bar{n}_{\ng \tilde{C}} -  \check{n}_{\ng{C^* \cup \tC}}  + \pi ({\mck}_{\ng \tilde{C}} +{\bar{\eps}_{\tilde{C}}}) = 0$ on $\tau_{\ng C^* \cup \tC}$ since $\Pi_y^\Ban [ \iota_y \sigma ](y) = 0$ if $\sigma$ has a non-vanishing polynomial decoration at the root. This implies that the problematic part of the second multinomial coefficient is identically $1$ on the region of summation that contributes to $(\ell_y^\eq \otimes \Pi_y^\Ban \circ \iota_y[\cdot](y) \otimes \chi_y^x) (\id \otimes \tilde{\Delta})\Delta_r^- (\tau, \mfn, \mfe)$. 

By a similar argument, the contribution of vertices in $\tau_{\ng C}$ in the second multinomial coefficient of \eqref{MT_exp} is identically $1$ on the region of summation that contributes to $(\ell^\Ban \circ \iota_y \otimes \Pi_y^\Ban \circ \iota_y[\cdot](y) \otimes \chi_y^x) (\Delta_r^- \otimes \id)\tilde{\Delta}(\tau, \mfn, \mfe))$. Therefore we have established that
\begin{align*}
	(\ell_y^\eq \otimes \Pi_y^\Ban \circ \iota_y[\cdot](y) \otimes \chi_y^x) (\id \otimes \tilde{\Delta})\Delta_r^- (\tau, \mfn, \mfe) = (\ell^\Ban \circ \iota_y \otimes \Pi_y^\Ban \circ \iota_y[\cdot](y) \otimes \chi_y^x) (\Delta_r^- \otimes \id)\tilde{\Delta}(\tau, \mfn, \mfe)).
\end{align*}
\end{proof}
\section{Proof of Theorem~\ref{theo:main_result_1} - Convergence of Renormalised Models}\label{s:Proof_Main}

The goal of this subsection is to assemble the ingredients of the rest of the paper to obtain the proof of Theorem~\ref{theo:main_result_1}. Since the vast majority of the hard work has been performed in earlier sections, this mostly involves appropriately tracking the parameters and arranging the individual parts in the right way.

\begin{proof}
	We fix $\mcB$ as in the statement of the result and a value of $\mcO > 0$ which is assumed to satisfy a finite number of constraints that are collected throughout the proof. We fix also an arbitrary choice of $\gamma_0$ that is admissible for $\mcB$ (as defined in Definition~\ref{def: gamma_def}). We fix a regularity parameter 
$\alpha < - m_* - | \mfs | - \max (\mfs)$
	where $m_*$ is defined above Proposition~\ref{prop: PMD def} and we fix a weight $w$ that is to be specified later in the proof. We let $W_{\gamma_0}$ be the corresponding sector of $\mcT^\Ban(C_w^\alpha)$ as defined above Proposition~\ref{prop: PMD def}. We remind the reader that $\mcT^\Ban(C_w^\alpha)$ is defined with respect to the rule $\bar{R}$ arising as the derivative completion of $R$ as defined in Definition~\ref{def: derivative_completion}. For $K_1, K_2 \in \mcA_+^{\eq; \mcO}$, we let $A^1, A^2 \in \mcA_+^\mco$ denote the corresponding kernel assignments on $\mcT^\Ban(C_w^\alpha)$ where we require throughout the proof that 
$\mco + \max ( \mfs ) \leq \mcO$.
	
	 For fixed $\varepsilon > 0$, we define the (state-space dependent) preparation map $P^\varepsilon$ on $\langle \mcB \rangle$ via
	\begin{align*}
		P^{\varepsilon; i}(y, \btau) = (\ell_y^{\eq; i} \otimes \id) \Delta_r^-
	\end{align*}
	where $\ell^{\eq; i}$ is defined as in Section~\ref{ss: mod_eq} (see \eqref{eq:D_def} and the preceeding discussion) where $\ell^\Ban$ appearing there denotes the BPHZ functional on $W_{\gamma_0}$ with respect to the kernel assignment $A^i$ as  and the noise assignment $\xi^\varepsilon$ as defined in Lemma~\ref{lem: BPHZ exists}. We denote by $(\Pi^{\varepsilon, i}, \Gamma^{\varepsilon, i})$ the corresponding models on $\mcT(R)$. 
	
	We first obtain a uniform in $\varepsilon$ bound in $L^p(\Omega)$ for $\|Z^\varepsilon\|_{\langle \mcB \rangle; \mfK}$. For notational convenience, we temporarily suppress the decoration $i$. We note that the proof of Lemma~\ref{lem:Gam_from_Pi} did not use that the kernel assignments were translation invariant. As a consequence, it suffices to obtain bounds on $\|\Pi^\varepsilon\|_{\langle \mcB \rangle ; \mfK}$.   
	
	As a consequence of Lemma~\ref{lem:transfer_down} and Proposition~\ref{prop:f_algebraic_renormalisation}, if $\mco$ (and hence $\mcO$) are sufficiently large in a way that depends only on the parameter $\gamma_0$ and the rule $R$, we find that for every compact set $\mfK$ there exists a compact set $\bar{\mfK}$ and $k \in \mathbb{N}$ such that
	\begin{align*}
		\mathbb{E}^{\frac{1}{p}} [\|\Pi^\varepsilon \|_{\langle \mcB \rangle; \mfK}^p] \lesssim 1 + \|A\|_{\mcA_+^\mco}^k + \left ( \mathbb{E}^{\frac{1}{\bar{p}}} \Big [ \|\hat{Z}^\varepsilon\|_{W_{\gamma_0}; \bar{\mfK}}^{\bar{p}} \Big ] \right )^k 
	\end{align*}
	where we have set $\bar{p} = kp$ and where $\hat{Z}^{\varepsilon}$ denotes the BPHZ model on $W_{\gamma_0}$ with respect to the kernel assignment $A$ and noise assignment $\xi^\varepsilon$. Therefore, for our uniform estimate, it suffices to bound the expectation appearing on the right hand side above. By Lemma~\ref{l:quenched}, for this purpose it is sufficient to bound
	\begin{align*}
			\sup_{\boldsymbol{\tau} \in \mcB_{W_{\gamma_0}}}
		\sup_{x \in \tilde{\mfK} } 
		\sup_{\lambda \in (0, 1]}	
		\sup_{\phi \in \mcB^r}	
		\sup_{v \in (\mathcal{K}_{\mco}^{| \mfl |_\mfs})^{\proj \sttau}}
		\frac{\mathbb{E}^{\frac{1}{q_{\boldsymbol{\tau}, \bar{p}}}} \big[ \big| \hat{\Pi}_x^{\varepsilon; \mcK} v \big( \phi_x^{\lambda} \big) \big|^{q_{\boldsymbol{\tau}, \bar{p}}} \big]}{\| v \|_{(\mathcal{K}_{\mco}^{|\mfl|})^{\proj \sttau}} \, \lambda^{| \boldsymbol{\tau} |_{\mfs} + \kappa}} 
	\end{align*}
	uniformly in $\varepsilon$ where $q_{\btau, \bar{p}} \in [1, \infty]$ is a fixed parameter depending on $\btau$ and $\bar{p}$, $\tilde{\mfK}$ is a compact set depending on $\bar{\mfK}$ and $\mcB_{W_{\gamma_0}}$ and $\kappa > 0$ is a parameter that we are free to choose (so long as we do so uniformly in $\varepsilon$). Here $\hat{Z}^{\varepsilon, \mcK} = (\hat{\Pi}^{\varepsilon; \mcK}, \hat{\Gamma}^{\varepsilon; \mcK})$ denotes the BPHZ model associated to the noise assignment $\xi^\varepsilon$ and kernel assignment given by the identity map on the historic sector of $\mcT^\Ban(\mcK_\mco^{|\mfl|_\mfs})$ generated by $\mcB_{W_{\gamma_0}}$.
	
	The required bound on this quantity follows from Lemma~\ref{l:annealed_mb} so long as we can bound
	\begin{align*}
			\sup_{\boldsymbol{\tau} \in \mcB_{W_{\gamma_0}}}
		\sup_{x \in \tilde{\mfK} } 
		\sup_{\lambda \in (0, 1]}	
		\sup_{\phi \in \mcB^r}	\sup_{A^\tau} \frac{\mathbb{E}^{\frac{1}{q_{\btau, \bar{p}}}} \left [ \big| \hat{\Pi}_x^{\tau; \varepsilon} \tau^* \big( \phi_x^{\lambda} \big) \big|^{q_{\btau, \bar{p}}} \right ]}{ \lambda^{|\btau|_\mfs + \kappa} \prod_{e \in E_+(\tau)} \|A^\tau(\mfl(e))\|_{\mathcal{K}_{\mco}^{|\mfl(e)|}}}
	\end{align*}
	uniformly in $\varepsilon$ where for each $\btau$ we fixed an arbitrary choice of $\tau \in \btau$ and denoted by $\hat{Z}^{\tau ; \varepsilon} = (\hat{\Pi}^{\tau; \varepsilon}, \hat{\Gamma}^{\tau; \varepsilon})$ the BPHZ model associated to the kernel assignment $A^\tau$ and the noise assignment induced by $\xi^\varepsilon$ as defined in \eqref{eq:induced assigment} on the regularity structure $T^\tau$ defined above Lemma~\ref{lem:up_down} with respect to the rule $\bar{R}$. We recall that the supremum over $A^\tau$ is a supremum over all kernel assignments of order $\mco$ on $T^\tau$. 
	
	We note that since $T^\tau$ is the regularity structure with one-dimensional components corresponding to a complete subcritical rule and $A^\tau$ are translation invariant kernel assignments, we are in the setting of any of \cite{CH16, HS23, BH23} and therefore the required uniform bound will follow from any of these papers\footnote{Strictly speaking, those papers yield the bound without the extra $\kappa$ in the power of $\lambda$. Since $\mcB$ is a finite set, this can be absorbed into redefining the degree assignments on $T^\tau$. Since this kind of argument is standard in the literature, we omit the details. Furthermore, we note that \cite{HS23, BH23} do not track the requirements on kernels as finely as we require here. However, the proofs do yield the input required here in a straightforward manner so long as $\mco$ is sufficiently large. In the case of \cite{HS23} (which shares more analytic tools with this paper), the required modifications are very close to those that we perform in Sections~\ref{sec:mod_dist} and \ref{s:ren_models}.} as soon as we verify their hypotheses on each regularity structure $T^\tau$. It is then a straightforward exercise to check that for each of these papers, the required assumptions on $\mcT(\bar{R})$ imply the same assumptions on $T^\tau$. This is essentially because $T^\tau$ differs from $\mcT(\bar{R})$ only by introducing duplicate copies of trees allowing edges of the same type to correspond to different kernels. Since the important detail from the perspective of \cite{CH16, HS23, BH23} is the regularising effect of the kernel rather than the precise choice of kernel, this does not affect the assumptions. Since $\mcT(\bar{R})$ differs from $\mcT(R)$ only by removing derivatives from edges, it is in turn straightforward to see that any of the required sets of assumptions are stable under passing from $\mcT(R)$ to $\mcT(\bar{R})$. This completes our proof of the uniform in $\varepsilon$ bounds.
	
	To prove convergence of $Z^{\varepsilon}$ as $\varepsilon \to 0$, one follows the same line of argumentation for the quantity $\mathbb{E}^{1/p} [\|Z^\varepsilon, Z^\delta\|_{\langle \mcB \rangle; \mfK}^p]$. Indeed, by Lemma~\ref{lem:transfer_down} and Proposition~\ref{prop:f_algebraic_renormalisation} in combination with H\"older's inequality and the uniform bounds on moments of $\|\hat{Z}^{\varepsilon; i}\|_{W_{\gamma_0}; \bar{\mfK}}$ established in the previous part of the proof, we obtain the existence of $\bar{p} \in [1, \infty]$ and a compact set $\tilde{\mfK}$ such that
	\begin{align*}
		\mathbb{E}^{\frac{1}{p}} \Big [ \|\Pi^{\varepsilon; 1} - \Pi^{\delta; 2} \|_{\langle \mcB \rangle; \mfK}^{p} \Big ] \lesssim \|A^1 - A^2\|_{\mcA_+^\mco} + \mathbb{E}^{\frac{1}{\bar{p}}} \Big [ \|Z^{\varepsilon;1} ; Z^{\delta; 2}\|_{W_{\gamma_0}; \tilde{\mfK}}^{\bar{p}} \Big ]
	\end{align*}
	where the implict constant depends polynomially on $\|A^i\|_{\mcA_+^\mco}$. 
	
	As a consequence of Lemma~\ref{l:ka} and Lemma~\ref{l:quenched}, we therefore deduce that there exists a compact set $\tilde{\mfK}$ such that
	\begin{align}\label{eq:mod_lips_est}
		& \mathbb{E}^{\frac{1}{p}} \Big [ \|\Pi^{\varepsilon; 1} - \Pi^{\delta; 2} \|_{\langle \mcB \rangle; \mfK}^{p} \Big ] \nonumber \\
		& \lesssim \|K^1 - K^2\|_{\mcA_+^{\eq; \mcO}} + \sup_{\boldsymbol{\tau} \in \mcB_{W_{\gamma_0}}}
		\sup_{x \in \tilde{\mfK}}
		\sup_{\lambda \in (0, 1]}
		\sup_{\phi \in \mcB^r} 	
		\sup_{v \in (\mathcal{K}_{\mco}^{| \mfl |})^{\proj \sttau}}
		\frac{\mathbb{E}^{\frac{1}{q_{\boldsymbol{\tau}, \bar{p}}}} \big[ \big| \big( \hat{\Pi}_x^{\mcK; \varepsilon} - \hat{\Pi}_x^{\mcK; \delta} \big) v \big( \phi_x^{\lambda} \big) \big|^{q_{\boldsymbol{\tau}, \bar{p}}} \big]}
		{\| v \|_{(\mathcal{K}_{\mco}^{|\mfl|})^{\proj \sttau}} \, \lambda^{| \boldsymbol{\tau} |_{\mfs} + \kappa}}
	\end{align}
	where the implicit constant depends polynomially on $\|K^i\|_{\mcA_+^{\eq; \mcO}}$ and where we emphasise the fact that the models appearing on the right hand side are the BPHZ models associated to noise assignments $\xi^\varepsilon, \xi^\delta$ respectively and the kernel assignment given by the identity map. By taking $K^1 = K^2$ in the above and appealing to the last estimate of Lemma~\ref{l:annealed_mb} in combination with any one of \cite{CH16, HS23, BH23} as in the uniform bounds part, we see that for any compact set $\mfK$
	\begin{align}\label{eq:conv_eq}
		\mathbb{E}^{\frac{1}{p}} [ \|Z^{\eps}; Z^{\delta}\|_{\langle \mcB \rangle; \mfK}^p] \to 0, \qquad \text{ as } \eps, \delta \to 0.
	\end{align}
	By Markov's inequality, it follows that $Z^\eps$ is Cauchy in probability in $\mcM(\langle \mcB \rangle)$ and hence converges in probability to a limiting model $Z$. The fact that 
		$\mathbb{E}[\|Z^\varepsilon; Z\|_{\langle \mcB \rangle; \mfK}^p] \to 0$ for each $\mfK$ then follows by passing to an almost surely convergent subsequence in $\delta$ in \eqref{eq:conv_eq} and sending $\delta \to 0$ with the help of Fatou's lemma.

	 With the convergence result in hand, the locally Lipschitz dependence on the kernel assignment follows by sending $\varepsilon, \delta \to 0$ along almost surely convergent subsequences in \eqref{eq:mod_lips_est}, again with the help of Fatou's lemma. The final statement of Theorem~\ref{theo:main_result_1} pertaining to the form of the counterterm is immediate from the definition of $\ell^\eq$ given in Section~\ref{ss: mod_eq}.
	\end{proof}
\section{Decompositions of Green's Kernels for Parabolic Operators}\label{s:green_kernel}

	In this section, we provide the remaining ingredient to show that Corollary~\ref{theo:main_result_2} covers the case of suitably nice variable coefficient heat operators. 
More precisely,
we prove that Assumption~\ref{ass:locality} holds 
in the setting of 
heat kernels of 
second-order
differential operators on\footnote{in this section, contrary to the rest of this paper, we work in the ambient space $\mathbb{R}^{1 + d}$ instead of $\mathbb{R}^d$, in order to capture the different role of the time and space directions} 
$\mathbb{R}^{1+d}$ of the form
\begin{equation}\label{eq:main_diff_operator}
	L = \partial_t - \sum_{i,j = 1}^d a_{i j}(t,x) \partial_i \partial_j - \sum_{i = 1}^d b_{i}(t,x) \partial_i - c (t,x) .
\end{equation}
Due to the assumed form of the operator, we now assume that the scaling $\mfs$ is given by $\mfs = (2, 1, \dots, 1)$.
\subsection{Setting and Statement of Heat Kernel Decomposition}
\label{ss:statement_hk}

Throughout this section, we make the following assumptions on the coefficient fields $a$, $b$, $c$. 
We fix an ellipticity radius $\lambda > 0$.
\begin{assumption}\label{ass:elliptic_and_bounds}
	For some $\varrho \in \mathbb{N}$,
	the coefficient fields 
	$a \colon \mathbb{R}^{1 + d} \to \mathrm{Sym} ( d )$, $b \colon \mathbb{R}^{1 + d} \to \mathbb{R}^{1 + d}$ and $c \colon \mathbb{R}^{1+d}\to\mathbb{R}$, satisfy:
	\begin{enumerate}
		\item (Uniform parabolicity): 
		For all $z\in \mathbb{R}^{1 + d}$, and all $\zeta \in \mathbb{R}^d$,
		\begin{align*}
			\lambda | \zeta |^{2} 
			\leq \sum\limits_{i, j = 1}^d a_{i j} ( z ) \, \zeta_i \zeta_j
			\leq \lambda^{-1} | \zeta |^2 , 
			\qquad \text{ where } \qquad | \zeta |^2 = \sum_{i = 1}^d \zeta_i^2 .
		\end{align*}
		
		\item (Regularity): For $i, j \in \lbrace 1, \cdots, d\rbrace$, the functions $a_{i j}$, $b_i$, and $c$, are of class $C^{\varrho}$, with
		\begin{align*}
			\| a \|_{\varrho} 
			& = \max\limits_{i, j \in \lbrace 1, \cdots , d \rbrace}  \| a_{i j} \|_{C^{\varrho}}	< \infty , 
			\qquad
			\| b \|_{\varrho} 
			= \max\limits_{i \in \lbrace 1, \cdots , d \rbrace} \| b_{i} \|_{C^{\varrho}}
			< \infty ,
			\qquad
			\| c \|_{\varrho} 
			= \| c \|_{C^{\varrho}}
			< \infty 
		\end{align*}
		where $C^\varrho$ denotes the space of (globally) $\varrho$-H\"older functions with respect to the scaling $\mfs$.
	\end{enumerate} 
\end{assumption}

\begin{definition}
	We will denote by $\mathbf{A}_{\varrho}$ the vector space of triples $(a, b, c)$ satisfying Assumption~\ref{ass:elliptic_and_bounds}, which we endow with the corresponding norm $\| (a, b, c) \|_{\varrho} = \|a\|_{\varrho}+\|b\|_{\varrho}+\|c\|_{\varrho}$.
	We also define $\mathbf{A} = \bigcap_{\varrho \in \mathbb{N}} \mathbf{A}_{\varrho}$, which we endow with the family of semi-norms 
	$(\| \cdot \|_{\varrho})_{\varrho \in \mathbb{N}}$.
\end{definition}

It is well-known (see e.g.\ \cite{Fri08, Grieser}) that the heat kernel of $L$ can be expressed as the `Volterra series'
\begin{align}\label{eq:volterra_series}\Gamma = Z * \sum_{k = 0}^{\infty} ( - E)^{*k} ,
\end{align}
where:
\begin{enumerate}
	
	\item $*$ refers to the space-time convolution operator, defined by the expression
	\begin{align} \label{eq:a4}
		u * v ( z , \bar{z} ) = \int_{\mathbb{R}^{1 + d}} d \zeta \, u ( z, \zeta ) v ( \zeta, \bar{z} ) , \qquad \text{for } z, \bar{z} \in \mathbb{R}^{1 + d} 
	\end{align}
	whenever the integral converges.
	\item $Z \colon \mathbb{R}^{1 + d} \times \mathbb{R}^{1 + d} \to \mathbb{R}$ is explicitly expressed as 
	\begin{align} \label{eq:a18}
		Z ( z, \bar{z} ) = W^{\bar{z}} ( z - \bar{z} ) ,
	\end{align}
	where $W^{w}$ denotes the fundamental solution of the differential operator $L^{w} = \partial_t - \sum_{i,j} a_{i j}(w) \partial_i \partial_j$ whose coefficient is ``frozen'' at $w \in \mathbb{R}^{1 + d}$, 
	which is in turn given by
	\begin{equation}\label{eq:exlicit Z_0}
		W^{w}(t,x):= 
		\frac{1}{(4\pi)^{d/2} \det(a(w))^{1/2}} \, 
		\frac{\mathbf{1}_{ \{ t>0 \} }}{ t^{d/2}} 
		\exp \left(\frac{- \sum_{i,j=1}^d a^{i j}(w) x_i x_j}{4t}\right) .
	\end{equation}
	Note that here and throughout we denote by $a^{i j}$ the entries of the inverse matrix $a^{-1}$ of $a$.
	
	\item $E$ is equal to $LZ$ away from the diagonal where by convention the differential operator $L$ acts on the first $1 + d$ variables of $Z$.
	Explicitly, we have that
	\begin{align}\label{eq:a20} 
		E(z, \bar{z})&= \sum_{i,j = 1}^d (a_{i j}(\bar{z})-a_{i j}(z ) ) \, \partial_i \partial_j Z(z, \bar{z}) 
		- \sum_{i = 1}^d b_i ( z ) \, \partial_i Z ( z, \bar{z} ) 
		- c ( z ) Z ( z, \bar{z} ) . 
	\end{align}
\end{enumerate}

In order to construct a kernel assignment satisfying  
Assumption~\ref{ass:locality} 
we will need a more refined decomposition 
based on Taylor expansion in the coefficient field. For the first term in the Volterra series and to first order in the expansion, a similar Taylor expansion in the coefficient field was used in \cite[Lemma 2.7]{Sin23} in order to show that due to symmetry assumptions $Z \ast (-E)$ does not contribute a need for renormalisation in the situations considered there. At the level of generality considered here, appealing to qualitative symmetry assumptions is not possible and so we instead use Taylor expansion in the coefficient field for a different purpose. Namely, we show that all parts of the kernel that contribute a genuine need for renormalisation can be written in the precise form described in Definition~\ref{d:p} below. To the best of our knowledge, no decomposition to arbitrary order and obtaining a form such as in Definition~\ref{d:p} is available in the literature.
Therefore, the remainder of this subsection is dedicated to providing such a decomposition, see Proposition~\ref{prop:ker_dec} below.

We introduce a class $\Lambda$ of (families of) translation-invariant kernels, which encode a property of `locality in the coefficient field'.

\begin{definition}\label{d:p}
	Let $a$, $b$, $c$ satisfy Assumption~\ref{ass:elliptic_and_bounds} with respect to $\varrho \in \mathbb{N}$, and let $r \leq \varrho$.
	We say that a family $(K^w)_{w \in \mathbb{R}^{1 + d}}$ of translation-invariant kernels belongs to the space $\Lambda_{r}$ (we simply write $\Lambda$ when $r, a, b, c$ are clear from context), if it 
	can be expressed as a (finite) linear combination of kernels of the form
	\begin{align}
		K_{F, N, P}^{w} ( z - \bar{z} ) \label{eq:a1} 
		& = F(w) \,  
		\int \mathrm{d} y_1 \cdots \mathrm{d} y_{N} \, \prod_{i = 0}^{I} P^{[i]} ( y_{i} - y_{i+1} ) W^{w} ( y_{i} - y_{i+1} ) , 
	\end{align}
	for $I \in \mathbb{N}$,  with the convention $y_0 = z$, $y_{N + 1} = \bar{z}$.
	Furthermore, in \eqref{eq:a1} we require that
	\begin{enumerate}
		\item the $P^{[i]}$ are rational fractions of the form 
		\begin{align*}
			P^{[i]} ( t, x ) &= t^{-1/2} \, Q^{[i]} ( t^{1/2}, t^{-1/2} x ) ,
			\quad \text{for some polynomial $Q^{[i]}$,}
		\end{align*}
		\item the function $F$ is a polynomial in the entries of $a$, $a^{-1}$, $b$, $c$ and their derivatives 
		of degree up to $r$ i.e.\ 
		\begin{align} \label{eq:a1b}
			F ( w ) & = p \big( (\partial^l a_{i, j} ( w ), \partial^l a^{i, j} ( w ), \partial^l b_i ( w ), \partial^l c ( w ))_{| l |_\mfs \leq r; i, j \in \lbrace 1, \cdots, d \rbrace} \big) , 
			\quad 
			\text{for some polynomial $p$.}
		\end{align}
	\end{enumerate}
\end{definition}

\begin{remark}
	We note that the spaces $\Lambda$ thus defined depend on $a, b, c$ through \eqref{eq:a1b} and the appearance of $W^w$ in \eqref{eq:a1}.
	One may readily prove that the integral in \eqref{eq:a1} is (absolutely) convergent outside the diagonal $z = \bar{z}$.
\end{remark}

\begin{remark}
	We note that by the transpose-of-comatrix formula expressing the entries of $a^{-1}$ as rational fractions in the entries of $a$, the function $F$ in \eqref{eq:a1b} is also a rational fraction in the entries of $a$, $b$, $c$ and their derivatives, with nowhere vanishing denominator.
	In particular, it is thus a smooth function thereof.
\end{remark}

\begin{example}\label{ex:e1}
	As a simple example, note that the family $(W^w)_{w \in \mathbb{R}^{1 + d}}$ defined by \eqref{eq:exlicit Z_0}, belongs to $\Lambda_{r = 0}$.
	In this case, $N = 0$ and $P^{[0]} = F = 1$.
\end{example}

Recall that the spaces $\mcK_r^{\beta}$ of translation-invariant regularising kernels
from Definition~\ref{d:rk2n}
are not strictly speaking spaces of functions $F$, but rather spaces of dyadic decompositions $(F_n)_{n \in \mathbb{N}}$.
As already noted in 
Remark~\ref{rem:dec2ker}, 
one can pass from the decompositions to a function $F$ by summing the coefficients $F_n$.
In the converse direction there is no unique way of constructing the coefficients $F_n$ given a function $F$.
In this section, we will prescribe a way of performing such a decomposition by fixing once and for all a suitable dyadic partition of unity.
More precisely,
we now fix once and for all a smooth cut-off function
\begin{align} \label{eq:a33b}
	\chi \in C_c^{\infty} ( \mathbb{R}^{1 + d} ) , 
	\qquad \text{such that} 
	\qquad \mathbf{1}_{B_{\mathfrak{s}} (0, 1/2)} \leq \chi \leq \mathbf{1}_{B_{\mathfrak{s}} (0, 1)} ,
\end{align}
and define for $z = (t, x) \in \mathbb{R}^{1 + d}$, $\phi ( z ) = \chi ( z ) - \chi (4 t, 2 x)$, along with the dyadically rescaled $\phi_n (z) = \phi ( 2^{2n} t, 2^n x )$.
Thus, by telescoping, 
one has
$1 = (1 - \chi) + \sum\nolimits_{n \geq 0} \phi_n$ 
on $\mathbb{R}^{1 + d} \setminus \lbrace 0 \rbrace$,
whence the decomposition
$F = R + \sum_{n = 0}^{+ \infty} F_n$
with
$R = ( 1 - \chi ) F$ and $F_n = \phi_n F$.
Note that by construction, one automatically has
for all $n \in \mathbb{N}$
that 
$\operatorname{supp}( F_n ) \subset \lbrace (z, \bar{z}) \colon \left| \bar{z} - z \right|_{\mathfrak{s}} \leq 2^{-n} \rbrace$,
so that in the remainder of this section when we write that $F \in \mcK_r^{\beta}$ we
will mean the corresponding statement for the family $(F_n)_{n \in \mathbb{N}}$.

\begin{definition}
	We define $\chi \Lambda_r$ to be the space of
	kernels of the form $K^w ( z - \bar{z} ) = \chi ( z - \bar{z} ) \tilde{K}^w ( z - \bar{z} )$ with $\tilde{K} \in \Lambda_r$.
\end{definition}

Given $r \in \mathbb{N}$, 
we will denote by $C_{\mathrm{loc},t}^r$ the space of functions on $\mathbb{R}^{1 + d} \times \mathbb{R}^{1 + d}$ of class $C^{r}$ (for the scaling $\mfs$), endowed with the family of semi-norms
\begin{align*}
	\| F \|_{C_{\mathrm{loc},t}^r ( T )} 
	& = \sup_{| j_1 |_\mfs + |j_2|_\mfs \leq r} \sup_{x, y \in \mathbb{R}^d} \sup_{| t - s  | \leq T} \big| \partial_1^{j_1} \partial_2^{j_2} F ( z, \bar{z} ) \big| ,
	\qquad \text{for $T>0$}.
\end{align*}

We are now ready to state the main result of this section. 
\begin{prop}
	\label{prop:ker_dec}
	For all $r, M \in \mathbb{N}$, there exists $\varrho = \varrho ( r, M, d ) \in \mathbb{N}$ such that if the coefficient fields $a$, $b$ and $c$ satisfy Assumption~\ref{ass:elliptic_and_bounds} with respect to $\varrho$, then
	the kernel $\Gamma$ defined by the series \eqref{eq:volterra_series} can be decomposed
	in two ways,
	as
	\begin{align}
		\Gamma ( z, \bar{z} ) & = K^{z} ( z - \bar{z} ) + R ( z, \bar{z} ) , \label{eq:a2} \\
		& = \bar{K}^{\bar{z}} ( z - \bar{z} ) + \bar{R} ( z, \bar{z} ) , \label{eq:a2b}
	\end{align}
	where 
	$K, \bar{K}, R, \bar{R}$ depend only on $r$ (and $a, b, c$, but \emph{not} on $M$)
	and
	\begin{align*}
		K, \bar{K} \in C^M ( \mathcal{K}_M^2 ) \cap \chi \Lambda_{\varrho} , 
		\qquad R, \bar{R} \in C_{\mathrm{loc},t}^r .
	\end{align*}
	Furthermore, the decomposition can be performed in such a way that the following maps are continuous:
	\begin{align} 
		\begin{array}[t]{lrcl}
			& \mathbf{A}_{\varrho} & \longrightarrow & C^M ( \mathcal{K}_M^2 ) \oplus C_{\mathrm{loc},t}^r \\
			& (a, b, c) & \longmapsto & ( K , R ) ,
		\end{array}
		\qquad 
		\begin{array}[t]{lrcl}
			& \mathbf{A}_{\varrho} & \longrightarrow & C^M ( \mathcal{K}_M^2 ) \oplus C_{\mathrm{loc},t}^r \\
			& (a, b, c) & \longmapsto & ( \bar{K}, \bar{R} ) .
		\end{array}
		\label{eq:a55}
	\end{align}
\end{prop}	

We note that the output of Proposition~\ref{prop:ker_dec} implies that the locality property encoded in Assumption~\ref{ass:locality} holds for a value of $N$ which depends only on $r, M$ and $d$.
Indeed, $K \in \chi \Lambda_r$ implies that
$K^z = f ( ( \partial^l a ( z ) , \partial^l b ( z ), \partial^l c ( z ) )_{| l |_\mfs \leq r} )$ for some function $f$, 
and similarly for derivatives in the `upper slot' of $K$ since $\Lambda_r$ is stable under differentiation.

\begin{remark} \label{rem:adjoint}
	The decomposition \eqref{eq:a2b} is not an immediate consequence of \eqref{eq:a2}.
	Indeed, since our strategy is based on the Volterra series \eqref{eq:volterra_series}, each summand of which is \emph{not} symmetric in $z$ and $\bar{z}$ (see e.g.\ \eqref{eq:a18}, \eqref{eq:a20}), 
	inverting the roles of $z$ and $\bar{z}$ in our proof requires an argument.
	To that effect, we use the fact that $\Gamma ( z, \bar{z} ) = \Gamma^* ( \bar{z}, z )$, where $\Gamma^*$ is the heat kernel of the adjoint differential operator
	\begin{align*}
		L^* 
		& = 
		- \partial_t 
		- \sum_{i, j=1}^d a_{ij} ( z ) \partial_i \partial_j 
		- \sum_{i = 1}^d \Big( 2 \sum_{j= 1}^d \partial_j a_{ij} (z ) - b_i ( z ) \Big) \partial_i 
		- \Big( c ( z ) - \sum_{i=1}^d \partial_i b_i ( z ) + \sum_{i, j=1}^d \partial_{i} \partial_j a_{i j} ( z ) \Big) .
	\end{align*}
	See \cite[Chapter~1,Section~8]{Fri08} for a relevant discussion of the adjoint equation.
	In particular, one may construct $\Gamma^*$ as 
	the
	series $\Gamma^* = \sum_{k \geq 0} Z^* * ( - L^* Z^* )^{*k}$, 
	where $Z^{*} ( z, \bar{z} ) = W^{\bar{z}} ( \bar{z} - z )$, where $W$ is as in \eqref{eq:a18}.
	Up to interchanging the role of $t$ and $\bar{t}$ in the various
	scales of 
	spaces 
	$\Psi$, $\Phi$, $C ( \Psi )$, $C ( \Phi )$ 
	which we shall introduce in this section, 
	the decomposition \eqref{eq:a2}
	can be established in the same way as \eqref{eq:a2b}.
	Thus, in what follows we concentrate on establishing the latter decomposition.
\end{remark}

Let us also record the following 
(straightforward)
corollary
in the case where all derivatives of the coefficient field are controlled.
We define the space
$C^{\infty} (\mcK_{\infty}^{\beta} ) = \bigcap_{M\in \mathbb{R}} C^M ( \mathcal{K}_M^2 )$, 
which we endow with its natural Fr\'echet topology.
\begin{corollary}
	In the context of Proposition~\ref{prop:ker_dec}, 
	for any $r \in \mathbb{N}$,
	the following map is continuous
	\begin{align*} 
		\begin{array}[t]{lrcl}
			& \mathbf{A} & \longrightarrow & C^{\infty} (\mcK_{\infty}^{\beta} ) \oplus C_{\mathrm{loc},t}^r \\
			& (a, b, c) & \longmapsto & ( K , R ) .
		\end{array}
	\end{align*}
\end{corollary}

\begin{remark}
	The remainder $R$ does not play a crucial role in the rest of this paper since it is not included in the construction of the model. 
	However, if one wishes to apply the corresponding solution theory of \cite{Hai14,BCCH20} it must enjoy a number of properties, some of which we now briefly comment upon.
	
	For example, in the setting of spatially periodic problems, to directly apply the results of the papers mentioned above, $R$ must be itself periodic in space and of sufficiently fast decay at infinity in time.
	If the coefficient field is assumed to be periodic in space then both of these properties can be read off analagous properties of $\Gamma$ itself by writing $R ( z, \bar{z}) = \Gamma ( z , \bar{z}  ) - K^{z } ( z - \bar{z} )$.
	
	Furthermore, we note that \cite{Hai14,BCCH20} assume that $R$ is smooth, while Proposition~\ref{prop:ker_dec} only provides a remainder of class $C^r$.
	This is not problematic 
	since the results of those papers remain valid provided $r$ is large enough in function of the equation under consideration, 
	which can be guaranteed in our context provided the coefficient field is itself chosen regular enough.
\end{remark}

\subsection{Proof of Heat Kernel Decomposition} \label{ss:proof_hk}

This section is devoted to the proof of Proposition~\ref{prop:ker_dec} which proceeds by first showing suitable convergence of the Volterra series and then by post-processing it to obtain the locality statement.

For the convergence, we find it convenient to follow the approach of \cite{Grieser}.  We introduce in Definition~\ref{d:hc} a scale $\Psi = (\Psi_r^{\alpha})_{r \in \mathbb{N}, \alpha > 0}$ of spaces which are a minor adaptation of \cite[Definition 2.1]{Grieser} and are, sometimes called a `heat calculus'. These spaces satisfy various stability properties, in particular with respect to multiplication and convolution which allow us to control each term of the Volterra series \eqref{eq:volterra_series}. Since this part of the proof is only a minor adaptation of the construction given in \cite{Grieser}, we only briefly sketch the details.

The main point of this section is that the heat calculus mentioned above is insufficient to account for the locality property that we demand in the form $K \in \chi \Lambda_{\rho}$.
Indeed, on the one hand, while $Z$ satisfies the claimed decomposition by Example~\ref{ex:e1}, the same is \emph{not} true of $LZ$, due to the presence of the terms $a_{i j} ( z )$, $b_i ( z )$, $c ( z )$ in the right-hand-side of \eqref{eq:a20}.
Furthermore, the following property of a kernel $F$,
``$F ( z  , \bar{z}) = F^{\bar{z}} ( z - \bar{z} )$ for some $(F^{w})_{w} \in \Lambda$'', is \emph{not} stable by convolution, since $F * F ( z, \bar{z} ) = \int \mathrm{d} \zeta \, F^{\zeta} ( z - \zeta ) F^{\bar{z}} ( \zeta - \bar{z} )$
needs not be of the desired form, due to
the presence of $\zeta$ instead of $\bar{z}$ in the `upper slot' of
the first factor in the integral.

To resolve this difficulty, 
the idea is to perform Taylor expansions in the coefficient field within each term of the Volterra series, which leads to expressions of the desired form modulo terms coming from the Taylor remainders which we show to be much more regularising.

We now recall the definition of the `heat calculus' scale $\Psi = (\Psi_r^{\alpha})_{r \in \mathbb{N}, \alpha > 0}$.
We note that our definition slightly differs from the one in \cite{Grieser}, since here we allow for time-dependent coefficient-fields and work on $\mbR^d$ rather than a compact Riemannian manifold.
In comparison to \cite{Grieser}, we have also changed the convention for the exponent $\alpha$ so that here $\alpha > 0$ represents the regularising degree of the kernel.

\begin{definition}
	\label{d:hc}
	Let $\alpha \in \mbR$ and $r \in \mbN$. We define $\Psi_r^\alpha$ to be the space of functions $F: \mbR^{1+d} \times \mbR^{1+d} \to \mbR$ such that there exists a function $\tilde{F} : \mbR^{1+d} \times \mbR_+^{1+d} \to \mbR$ such that				
	\begin{align}\label{eq:decomp}
		F(z,\bar{z}) = \mathbf{1}_{t > \bar{t}} (t- \bar{t})^{\frac{\alpha - d - 2}{2}} \tilde{F} \left (\bar{z}, (t-\bar{t})^{1/2}, \frac{x - \bar{x}}{(t - \bar{t})^{1/2}} \right )
	\end{align}
	and such that for each $T \in \mbR_+$, $n \in \mbN$ we have that
	\begin{align}\label{eq:a3}
		\|F\|_{\alpha, r; T, n} = \sup_{z \in \mbR^{1+d}} \sup_{{t} \in [0, T^{1/2}]} \sup_{\bar{x} \in \mbR^d} \sup_{|j|_\mfs + |k|_\mfs \le r} (1 + |\bar{x}|)^n |\partial_1^j \partial_2^k \tilde{F}(\bar{z},z)| < \infty.
	\end{align}
	In particular, the family of semi-norms \eqref{eq:a3} turn $\Psi_r^{\alpha}$ into a Fr\'echet space.
\end{definition}

We will also consider the subspace of $\Psi_r^{\alpha}$ consisting of translation invariant kernels.
\begin{definition}
	We define $\Phi_r^{\alpha}$ 
	to be the subspace of $\Psi_r^{\alpha}$ consisting of translation invariant kernels (equivalently, $\tilde{F}$ defined in \eqref{eq:decomp} depends only on $(t - \bar{t})^{1/2}$ and $\frac{(x - \bar{x})}{(t- \bar{t})^{1/2}}$), 
	endowed with the induced family of semi-norms.
\end{definition}

We note that contrary to the spaces $\mcK_r^{\alpha}$, the spaces $\Phi_r^{\alpha}$ do not come with the assumption of having support in the unit ball.
Examples of functions belonging to $\Psi_r^{\alpha}$ with suitable $\alpha$
are provided in the following lemma, which can be established by elementary computations involving the chain rule and Taylor's theorem. The most subtle term treated below is the contribution of $E$ as defined in \eqref{eq:a20}.

The important point for this term is that whilst taking two derivatives typically loses two degrees of regularising effect, the term $a_{ij}(\bar{z}) - a_{ij}(z)$ above restores a degree of regularising effect by Taylor expansion at the cost of the loss of one derivative (which corresponds to up to two degrees of regularity due to the parabolic scaling). For brevity, we leave the remaining elementary computations required for the following lemma to the reader.
\begin{lemma} \label{l:ex_e2}
	Under Assumption~\ref{ass:elliptic_and_bounds}, 
	we have that 
	$Z \in \Psi_{\varrho}^{2}$,
	$\partial_i Z \in \Psi_{\varrho}^{1}$,
	$\partial_i \partial_j Z \in \Psi_{\varrho}^0$
	for $i, j \in \lbrace 1, \cdots, d \rbrace$, 
	and 
	$E \in \Psi_{\varrho - 2}^{1}$.
	Furthermore the following maps are continuous\footnote{ Here we slightly abuse notation by identifying $Z$ with the map from $(a, b, c)$ to $Z$ and similarly for the other maps.}
	\begin{align*}
		Z \colon \mathbf{A}_{\varrho} \to \Psi_{\varrho}^2 , 
		\qquad 
		\partial_i Z \colon \mathbf{A}_{\varrho} \to \Psi_{\varrho}^{1} ,
		\qquad 
		\partial_i \partial_j Z \colon \mathbf{A}_{\varrho} \to \Psi_{\varrho}^0 ,
		\qquad 
		E \colon \mathbf{A}_{\varrho} \to \Psi_{\varrho-2}^{1} .
	\end{align*}
\end{lemma}

In what follows, a key role is played by families $(F_{w})_{w \in \mathbb{R}^{1 + d}}$ of functions uniformly belonging to $\Psi_r^{\alpha}$, which we introduce now.
\begin{definition}\label{d:fhc}
	Let 
	$\alpha \in \mathbb{R}$, 
	$r \in \mathbb{N}$.
	We say that 
	a function
	$F = F_{w} ( z, \bar{z} )$ of the variables $w, z, \bar{z} \in \mathbb{R}^{1 + d}$,
	belongs to $C^r ( \Psi^{\alpha}_r )$, 
	if furthermore for all $m \in \mathbb{N}^{1+d}$ with $| m |_{\mfs} \leq r$ and all $w \in \mathbb{R}^{1+d}$
	one has $\partial_{w}^m F_{w} \in \Psi_r^{\alpha}$, 
	and if
	for all $n \in \mathbb{N}$ and $T > 0$,
	\begin{align}
		\| F \|_{\alpha, r; T, n}
		& = \sup_{w \in \mathbb{R}^{1+d}} \sup_{| m |_\mfs \leq r} \| \partial_{w}^{m} F_{w} \|_{\alpha, r; T, n}
		< \infty . \label{eq:a30}
	\end{align}
	The family of semi-norms \eqref{eq:a30} turns $C^r ( \Psi_r^{\alpha})$ into a Fr\'echet space. 
	
	Similarly, 
	the space $C^{r} ( \Phi_r^{\alpha} )$
	is the space of $C^r$ curves on $\mathbb{R}^{1+d}$ (with the scaling $\mfs$) valued into $\Phi_r^{\alpha}$, 
	which we endow
	with the family of semi-norms induced by \eqref{eq:a30}.
\end{definition}

\begin{example}
	One can prove 
	similarly to Lemma~\ref{l:ex_e2} 
	that the function $F_w ( z, \bar{z} ) = W^w ( z - \bar{z} )$ belongs to $C^{\varrho} ( \Phi_{\infty}^2 )$.
	Furthermore the corresponding map
	$\mathbf{A}_{\varrho} \ni (a, b, c) \mapsto W \in C^{\varrho} ( \Phi_{\infty}^2 )$
	is continuous.
\end{example}

The most important property of these spaces is their stability under convolution.
\begin{prop}
	\label{prop:con}
	Let $r \in \mathbb{N}$, $\alpha, \beta > 0$.
	Then 
	the convolution map 
	\eqref{eq:a4}
	defines a bilinear and continuous map from $\Psi_r^{\alpha} \times \Psi_r^{\beta} \to \Psi_r^{\alpha + \beta}$
	(and hence from $\Phi_r^{\alpha} \times \Phi_r^{\beta} \to \Phi_r^{\alpha + \beta}$),
	as well as from $C^r ( \Psi_r^{\alpha} ) \times C^r ( \Psi_r^{\beta} ) \to C^r ( \Psi_r^{\alpha + \beta} )$
	(and hence from $C^r ( \Phi_r^{\alpha} ) \times C^r ( \Phi_r^{\beta} ) \to C^r ( \Phi_r^{\alpha + \beta} )$).
	Furthermore,
	we have the following estimate,
	valid for both
	families of seminorms
	\eqref{eq:a3} and \eqref{eq:a30}:
	for all $\alpha > d$ and $n \in \mathbb{N}$,
	\begin{align} \label{eq:a5}
		\| F * G \|_{\alpha + \beta, 0; T, n}
		& \lesssim_{n, d} \mathrm{B} \Big( \frac{\alpha - d}{2} , \frac{\beta}{2} \Big) \, \| F \|_{\alpha, 0; T, n} \| G \|_{\beta, 0; T, n + d + 1} ,
	\end{align}
	where 
	$\mathrm{B} ( a, b ) = \int_{0}^1 \mathrm{d} s \, (1 - s)^{a - 1} \, s^{b - 1}$
	is the beta function.
\end{prop}

\begin{proof}
	Since this is similar to \cite[Proposition 2.6]{Grieser} we remain brief here.
	Taking $F_{w} ( z, \bar{z} ) = F ( z, \bar{z} )$, one sees that the result in $C^r ( \Psi_r^{\alpha} ) \times C^r ( \Psi_r^{\beta} ) \to C^r ( \Psi_r^{\alpha + \beta} )$ implies the one on $\Psi_r^{\alpha} \times \Psi_r^{\beta} \to \Psi_r^{\alpha + \beta}$.
	Thus, we consider the former case throughout.
	Writing explicitly the integral expression \eqref{eq:a4} of convolution,
	with integration variable $\zeta = (s, y)$,
	and performing the successive changes of variables
	$s\leftarrow(s - \bar{t})/(t - \bar{t})$,
	$y\leftarrow y - s x - (1 - s) \bar{x}$,
	and $y\leftarrow\big( s ( 1 - s ) ( t - \bar{t} ) \big)^{-1/2} \, y$,
	we obtain an expression of the form
	\begin{align*}
		F_{w} * G_{w} (z, \bar{z}) 
		& = \mathbf{1}_{\lbrace t > \bar{t} \rbrace} 
		\, (t - \bar{t})^{\frac{\alpha + \beta - (d + 2)}{2}} 
		\, \widetilde{F * G}_{w} \Big( \bar{t}, \bar{x} , \sqrt{t - \bar{t}} , \frac{x - \bar{x}}{\sqrt{t - \bar{t}}} \Big) ,
	\end{align*}
	for the function $\widetilde{F * G}_{w}$ of the variables $w \in \mathbb{R}^{1 +d}$, $(t, x) \in \mathbb{R}^{1 + d}$, $u \in \mathbb{R}_+$, and $v \in \mathbb{R}^d$,
	explicitly given by
	\begin{align}
		\widetilde{F * G}_{w} ( t, x , u, v ) \label{eq:a6}
		& 
		= \int_{0}^{1} \mathrm{d} s \, 
		( 1 - s )^{\frac{\alpha}{2} - 1} s^{\frac{\beta}{2} - 1} \nonumber \\
		& \quad \int_{\mathbb{R}^{d}} \mathrm{d} y \, 
		\tilde{G}_{w} \Big( t, x , u \sqrt{s} , y \sqrt{1 - s} + v \sqrt{s} \Big) \nonumber \\ 
		& \quad 
		\tilde{F}_{w} \Big( t + u^2 s, u y \sqrt{s ( 1 - s )} + x + u v s , u \sqrt{1 - s} , 
		v \sqrt{1 - s} - y \sqrt{s} \Big) .
	\end{align}
	From that expression one may readily insert the bounds \eqref{eq:a30} on $F$ and $G$ and deduce the claimed continuity after some elementary calculations.
	We rather prove \eqref{eq:a5} in detail, since it will be an important estimate later.
	Plugging the assumptions on $F$ and $G$ in the explicit expression \eqref{eq:a6}, 
	\begin{align*}
		\big| \widetilde{F * G}_{w} ( t, x , u, v ) \big|
		& 
		\leq 
		\| F \|_{\alpha, 0; T, n} \| G \|_{\beta, 0; T, n + d + 1}
		\int_{0}^{1} \mathrm{d} s \, 
		( 1 - s )^{\frac{\alpha}{2} - 1} s^{\frac{\beta}{2} - 1} \\
		& \quad
		\int_{\mathbb{R}^{d}} \mathrm{d} y \,	
		\frac{1}
		{( 1 + | v \sqrt{1 - s} - y \sqrt{s} |^{n} ) ( 1 + | y \sqrt{1 - s} + v \sqrt{s} |^{n + d + 1} )} .
	\end{align*}
	Noting that 
	$( 1 + | v \sqrt{1 - s } - y \sqrt{s} |^{2} ) ( 1 + | y \sqrt{1 - s} + v \sqrt{s} |^{2} ) \geq 1 + | v |^2$,
	the latter spatial integrand may be bounded by
	a constant multiple of 
	$(1 + | v |^n)^{-1} ( 1 + | y \sqrt{1 - s} + v \sqrt{s} |^{d + 1} )^{-1}$.		
	Plugging into the integral and performing the change of variable $y \leftarrow y \sqrt{1 - s} + v \sqrt{s}$ therein,
	we obtain \eqref{eq:a5}.
\end{proof}
We note that the convolution map defined above is associative in the obvious sense as a corollary of Fubini's Theorem since all integrals in the proof are absolutely convergent. Iterating the final statement in the above result then yields the following corollary.
\begin{lemma}\label{lem:it_con}
	Suppose that $\beta > 0$ and that $M$ is such that $M\beta > d$. Then for all $j \in \mbN$,
	\begin{align*}
		\|R^{\ast (M+j)}\|_{(M+j)\beta, 0;T, r} \lesssim \|R^{\ast M}\|_{M\beta, 0; T, r} \|R\|_{\beta, 0; T, r + d + 1}^j \prod_{i = 1}^j \mathrm{B} \left ( \frac{(M+ 
		i
		-1)\beta - d}{2}, \frac{\beta}{2} \right ) 
	\end{align*}
	with the convention that the product is $1$ if $j < 1$.
\end{lemma}
\begin{proof}
	The proof is by induction in $j$. The base case $j = 0$ is trivial whilst the induction step follows by applying Proposition~\ref{prop:con} and the induction hypothesis.
\end{proof}
With this result in hand, we provide the required convergence result for the Volterra series defining $\Gamma$.
\begin{lemma}\label{lem:convergence}
	Suppose that $\varrho \ge 2$. Then the series $$\Gamma = \sum_{k \ge 0} Z \ast (-E)^{\ast k}$$ converges in $C^{\varrho - 2}(\mbR^d \times \mbR^d \setminus \Delta)$ for $\Delta = \{(z, \bar{z}): t \neq \bar{t}\}$ to a heat kernel for $L$. Furthermore, there exists $K$ such that the map
	$$\mathbf{A}_\varrho \ni (a,b,c) \mapsto \sum_{k \ge K} Z \ast (-E)^{\ast k} \in C_{\mathrm{loc}, t}^{\varrho - 2}$$
	is continuous.
\end{lemma}
\begin{proof}
	Since all bar the final statement in this result are a direct analogue of \cite[Proposition 2.10]{Grieser}, we focus on proving the final statement (which also implies the convergence statement). We fix $j_1, j_2$ with $|j_1|_\mfs + |j_2|_\mfs \le r = \varrho - 2$. 
	
	We note that if $k$ is chosen so that $k> d + r$ then
	\begin{align*}
		\sup_{\substack{z, \bar{z} \\ |t - \bar{t}| <T}}|\partial_1^{j_1} \partial_2^{j_2} (Z \ast (-E)^{\ast k})(z, \bar{z})| \lesssim T^{\frac{k - |j_1|_\mfs - |j_2|_\mfs - d}{2}} \|\partial_1^{j_1} \partial_1^{j_2} Z \ast (-E)^{\ast k}\|_{2 + k - |j_1|_\mfs - |j_2|_\mfs, 0; T, 0} .
	\end{align*}
	
	We now estimate the right hand side by use of Proposition~\ref{prop:con} and Lemma~\ref{lem:it_con}.
	We have
	\begin{align*}
		\|\partial_1^{j_1} \partial_2^{j_2} Z& \ast (-E)^{\ast k}\|_{2 + k - |j_1|_\mfs - |j_2|_\mfs, 0; T, 0} \lesssim \mathrm{B} \left ( \frac{2 + M - |j_1|_\mfs - d}{2} , \frac{(k-M) - |j_2|_\mfs}{2} \right ) \\& \qquad \qquad \qquad  \|\partial_1^{j_1} Z\ast(-E)^{\ast M} \|_{2 + M - |j_1|_\mfs, 0; T, 0} \|\partial_2^{j_2} (-E)^{\ast (k-M)}\|_{(k-M) - |j_2|_\mfs, 0; T, d+1}
	\end{align*}
	under the assumption that first $M$ is chosen large enough so that $2 + M - r > d$ and under the restriction to $k$ large enough so that $k-M > r$.
	
	It follows straightforwardly by differentiating the expression \eqref{eq:decomp} that
	\begin{align*}
		\|\partial_1^{j_1} Z\ast(-E)^{\ast M} \|_{2 + M - |j_1|_\mfs, 0; T, 0} \lesssim (1+T)^{\frac{|j_1|_\mfs}{2}} \|Z \ast (-E)^{\ast M}\|_{2 + M, |j_1|_\mfs; T, 0}.
	\end{align*}
	Meanwhile, by Proposition~\ref{prop:con} and a similar derivative bound, we have
	\begin{align*}
		&\|\partial_2^{j_2} (-E)^{\ast (k-M)}\|_{k-M - |j_2|_\mfs, 0; T, d+1} \\ & \lesssim \mathrm{B} \left ( \frac{k-2M - d}{2}, \frac{M - |j_2|_\mfs}{2 } \right ) (1+T)^{\frac{|j_2|_\mfs}{2}} \|(-E)^{\ast (k-2M)}\|_{k-2M, 0; T, d+1} \|E^{\ast M}\|_{M, |j_2|_\mfs; T, 2d + 2}  
	\end{align*}
	under the assumption that $M > r$ and restricting to $k > 2M + d$. Then by Lemma~\ref{lem:it_con}, we have that

	\begin{align*}
		\|(-E)^{\ast (k-2M)}\|_{k-2M, 0; T, d+1} \lesssim \|E^{\ast M}\|_{M, 0; T, d+1} \|E\|_{1, 0; T, 2d + 2}^{k-3M} \prod_{i = 1}^{k-3M} \mathrm{B} \left ( \frac{M+i-1 - d}{2}, \frac{1}{2} \right ) ,
	\end{align*}
	under the assumption that $M > d$ and restricting to $k \geq 3M$.
	
	Combining these estimates and absorbing all terms that do not depend on $k$ in the implicit constant (which we can do since we are interested in summability in $k$), we obtain
		\begin{align*}
		\sup_{\substack{z, \bar{z} \\ |t - \bar{t}| <T}}|\partial_1^{j_1} \partial_2^{j_2} (Z \ast (-E)^{\ast k})(z, \bar{z})| & \lesssim T^{\frac{k}{\mfs_0}} \|E\|_{1, 0; T, 2d + 2}^{k-3M} \mathrm{B} \left ( \frac{2 + M- |j_1|_\mfs - d}{2}, \frac{k-M - |j_2|_\mfs}{2} \right ) \\ &  \mathrm{B} \left ( \frac{k-2M- d}{2}, \frac{M- |j_2|_\mfs}{2} \right )  \prod_{i = 1}^{k-3M} \mathrm{B} \left ( \frac{ M+i-1-d}{2}, \frac{1}{2} \right ) .
	\end{align*}
	The right hand side is summable in $k$ (starting from a large value of $k$) locally uniformly in $Z$ and $E$ due to the fact that $\mathrm{B}(a,b) = \mathrm{B}(b,a) \sim a^{-b}$ for $b$ fixed as $a \to \infty$. This establishes convergence of the series. 
	
	Since the convergence is locally uniform in $Z,E$ and the map $(a,b,c) \mapsto (Z,E)$ is continuous, the desired continuity claim follows by use of the continuity statement in Proposition~\ref{prop:con}.
\end{proof}

Since the above result shows that the tails of the Volterra series can be included into the remainder part of the decomposition $\Gamma = K + R$, we turn to providing suitable decompositions of the remaining finitely many terms of the form $Z \ast (-E)^{\ast k}$. Here the strategy will be to apply the anisotropic Taylor's formula presented in Lemma~\ref{Tay} and show that the resulting remainder terms can be included into $R$ in our decomposition of $\Gamma$.

\begin{prop}\label{prop:Z_decomp}
	Fix $r \le \varrho$. For any $w \in \mbR^{1+d}$, we can write
	\begin{align}\label{eq:Z_decomp}
		Z(z, \bar{z}) = \sum_{|k|_\mfs < r} (\bar{z} - w)^k Z_w^{ [k]}( z- \bar{z}) + \sum_{k \in \partial_\mfs r} (\bar{z} - w)^{k_\downarrow} Z_w^{\partial, [k]}(z, \bar{z})
	\end{align}
	where we have written $\partial_\mfs r = \partial \{ j : |j|_\mfs < r\}$. Furthermore, the above decomposition can be made such that $Z^{[k]} \in C^{\varrho - |k|_\mfs}(\Phi_\infty^2) \cap \Lambda_{|k|_\mfs}$ and $Z^{\partial, [k]} \in C^{\alpha}(\Psi_{\varrho - |k_\downarrow|_\mfs - \alpha }^2)$ for any $\alpha < \varrho - |k_\downarrow|_\mfs$. In addition, we have that $(t- \bar{t})^{-1/2} Z^{[k]} \in \Lambda_r$. Furthermore, the maps $(a,b,c) \mapsto Z^{[k]}, Z^{\partial,[k]}$ are continuous.
\end{prop}
\begin{proof}
	By applying Lemma~\ref{Tay} to the function $\tilde{Z}$ in its first argument, we obtain a decomposition of the form \eqref{eq:Z_decomp} with
	\begin{align*}
		Z_w^{[k]}(z) 
			&= \frac{1}{k!} \mathbf{1}_{t > 0} t^{-d/2} \partial_1^k \tilde{Z}\left (w ; t^{1/2}, \frac{x}{t^{1/2}} \right ), 
		\\ Z_w^{\partial, [k]}(z, \bar{z}) 
			&= \frac{1}{k_\downarrow!}\mathbf{1}_{t > \bar{t}} (t- \bar{t})^{-d/2} \int_{\mbR^{1+d}} \delta_k [\partial_1^{k_\downarrow} \tilde{Z}] \left (w + y(\bar{z}- w) ; ( t - \bar{t} )^{1/2} , \frac{(x - \bar{x})}{(t - \bar{t})^{1/2}} \right ) \mcQ^{k_\downarrow}(dy).
	\end{align*}
	It follows from the definition of $Z$ and an explicit computation of the derivatives $\partial_1^k \tilde{Z}$ that $Z^{[k]} \in C^{\varrho - |k|_\mfs}(\Phi_\infty^2) \cap \Lambda_{|k|_\mfs}$ and that $(t - \bar{t})^{-1/2} Z^{[k]} \in \Lambda_r$. The continuous dependence on $(a,b,c)$ can also be read off this computation.
	
	For the remainder, we distinguish the cases $k_{\mathfrak{m}(k)} = 1$ and $k_{\mathfrak{m}(k)} > 1$ where $\mfm(k)$ is as in Appendix~\ref{app:Tay}. In the former case, we have $\mathfrak{m}(k_\downarrow) > \mathfrak{m}(k)$ which implies that the integrand appearing is constant on the support of $\mcQ^{k_\downarrow}$. Writing out the definition, we therefore see that	
	\begin{align*}
		Z_w^{\partial, [k]}(z, \bar{z}) 
		& = \frac{1}{k_{\downarrow}!}\mathbf{1}_{t > \bar{t}} (t - \bar{t})^{-d/2} \Bigg [ \partial_1^{k_\downarrow} \tilde{Z} \Big ( \bar{z}_0, \dots, \bar{z}_{\mathfrak{m}(k)}, w_{\mathfrak{m} ( k ) + 1} \dots, w_d; (t - \bar{t})^{1/2}, \frac{(x - \bar{x})}{(t - \bar{t})^{1/2}} \Big ) \\ & \qquad  - \partial_1^{k_\downarrow}\tilde{Z} \Big ( \bar{z}_0, \dots, \bar{z}_{\mathfrak{m}(k)-1}, w_{\mathfrak{m}(k)} \dots, w_d; (t - \bar{t})^{1/2} , \frac{(x - \bar{x})}{(t - \bar{t})^{1/2}} \Big ) \Bigg ] .
	\end{align*}
	From this expression, it follows that in this case $Z^{\partial, [k]} \in  C_\mfs^{\alpha}(\Psi_{\varrho - |k_\downarrow|_\mfs - \alpha}^2)$ for any $\alpha < \varrho - |k_\downarrow|_\mfs$ and that the map sending $(a,b,c)$ to $Z^{\partial,[k]}$ is continuous.
	
	In the other case, we have
	\begin{align*}
		& Z_w^{\partial, [k]} (z, \bar{z}) 
		\\ & = \frac{1}{k_\downarrow!}\mathbf{1}_{t > \bar{t}} ( t- \bar{t})^{-d/2} \int_0^1 \Bigg[ \partial_1^{k_\downarrow} \tilde{Z}\Big (\bar{z}_0, \dots, \bar{z}_{\mathfrak{m}(k) - 1}, w_{\mathfrak{m}(k)} + y(\bar{z} - w)_{\mathfrak{m}(k)}, w_{\mathfrak{m}(k) + 1}, \dots, w_d; ( t - \bar{t} )^{1/2}, \frac{(x - \bar{x})}{(t - \bar{t})^{1/2}} \Big ) 
		\\ & \qquad \qquad - \partial_1^{k_\downarrow} \tilde{Z}\Big (\bar{z}_0, \dots, \bar{z}_{\mathfrak{m}(k) - 1}, w_{\mathfrak{m}(k)}, \dots, w_d; ( t - \bar{t} )^{1/2} , \frac{(x - \bar{x})}{(t - \bar{t})^{1/2}} \Big ) \Bigg ] n (1-y)^{n-1} dy
	\end{align*}
	where $n$ is 
	the first non-zero entry of $k_{\downarrow}$
	It again follows that $Z^{\partial, [k]} \in C_\mfs^\alpha(\Psi_{\varrho - |k_\downarrow|_\mfs - \alpha}^2)$ for any $\alpha < \varrho - |k_\downarrow|_\mfs$ and that the dependence on $(a,b,c)$ is continuous.
\end{proof}

Next, we obtain a similar expansion for the term $E$. The proof is similar but computationally much longer. 

\begin{prop}\label{prop:E_decomp}
	Fix $2 < r \le  \varrho$. Then there exists sets of multiindices $A_r^1, A_r^2$ such that for any $w \in \mathbb{R}^{1 + d}$ we can write 
	\begin{align*}
		&E(z, \bar{z}) = \sum_{k \in A_r^1} (\bar{z} - w)^k E_w^{[k]}(z - \bar{z}) + \sum_{k \in A_r^2} \sum_{l \le k_\downarrow} (\bar{z}- w)^l E_w^{\partial, [k]; l}(z, \bar{z}) 
	\end{align*}
	where $E^{[k]} \in C^{\varrho - r}(\Phi_{\infty}^1) \cap \Lambda_{|k|_\mfs}$  and $E^{\partial, [k], l} \in C^\alpha(\Psi_{\varrho -  r - \alpha}^{1 + |k_\downarrow - l|_\mfs})$ for any $\alpha \le \varrho - 
	r$. Furthermore $A_r^2$ may be chosen to have the property that $k \in A_r^2$ implies that $|k_\downarrow|_\mfs \ge \min\{|j_\downarrow|_\mfs: j \in \partial_\mfs r\}$. Furthermore, the maps $(a,b,c) \mapsto E^{[k]}, E^{\partial,[k],l}$ are continuous.
\end{prop}
\begin{proof}
	For brevity, we prove only the decomposition. The continuity statement follows from the explicit formulae provided for the terms in the decomposition. 
	
	We first assume for simplicity that $b = 0$ and $c= 0$ and focus on the leading order part.  In this case, we can compute that
	\begin{align*}
		E(z, \bar{z}) = (t-\bar{t})^{-1/2} \sum_{i,j} \frac{(a_{ij}(\bar{z}) - a_{ij}(z))}{(t- \bar{t})^{1/2}} \Bigg [  \sum_{i^\prime,j^\prime} a^{i i^\prime}(\bar{z}) a^{j j^\prime}(\bar{z}) \frac{(x - \bar{x})_{i^\prime}(x - \bar{x})_{j^\prime}}{t- \bar{t}}  - a^{ij}(\bar{z}) \Bigg ] Z(z, \bar{z}).
	\end{align*}
	We now expand each term separately. First
	by applying Lemma~\ref{Tay} to $a_{i j}$
	we write
	\begin{align*}
		& a_{ij}(\bar{z}) - a_{ij}(z) = \sum_{|k|_\mfs < r} \frac{(\bar{z} - w)^k - (z-w)^k}{k!} \partial^k a_{ij}(w) 
		\\ & 
		+ \sum_{k \in \partial_\mfs r} \underbrace{\left (\frac{(\bar{z} - w)^{k_\downarrow}}{k_{\downarrow}!} \int \delta_k[\partial^{k_\downarrow} a_{ij}](w + (\bar{z} - w)y) {\mcQ}^{k_\downarrow}(dy) - \frac{({z} - w)^{k_\downarrow}}{k_{\downarrow}!} \int \delta_k[\partial^{k_\downarrow} a_{ij}](w + ({z} - w)y) {\mcQ}^{k_\downarrow}(dy) \right )}_{R_{ij}^{1,k}(w, z, \bar{z})}.
	\end{align*}
	By writing $(z-w)^k = (z - \bar{z} + \bar{z} - w)^k$ and binomially expanding, we rewrite this as
	\begin{align*}
		a_{ij}(\bar{z}) - a_{ij}(z) = - \sum_{|k|_\mfs< r} \sum_{0 < l \le k} \frac{(\bar{z} - w)^{k-l} (z - \bar{z})^l}{(k-l)! l!} \partial^k a_{ij}(w) + \sum_{k \in \partial_\mfs r} R_{ij}^{1,k}(w, z, \bar{z}) .
	\end{align*}

	We next turn to the middle term. We first note that
	\begin{align*}
		& \sum_{i^\prime, j^\prime} a^{ii^\prime}(\bar{z}) a^{j j^{\prime}}(\bar{z}) \frac{(x - \bar{x})_{i^\prime}(x - \bar{x})_{j^\prime}}{t- \bar{t}} 
		= \sum_{i^\prime, j^\prime} \sum_{|k|_\mfs < r} \frac{(\bar{z}- w)^k}{k!} \partial^k (a^{ii^\prime} a^{jj^\prime})(w)  \frac{(x - \bar{x})_{i^\prime}(x - \bar{x})_{j^\prime}}{t- \bar{t}}
		\\ 
		& 
		\qquad \qquad + \sum_{i^\prime, j^\prime} \sum_{k \in \partial_\mfs r} \frac{(\bar{z} - w)^{k_\downarrow}}{k_\downarrow !} \underbrace{\int \delta_k [\partial^{k_\downarrow} (a^{ii^\prime} a^{jj^\prime})] (w + (\bar{z} - w)y) {\mcQ}^{k_\downarrow}(dy)  \frac{(x - \bar{x})_{i^\prime}(x - \bar{x})_{j^\prime}}{t- \bar{t} }}_{R_{2, 1}^{i,j,i^\prime, j^\prime, k}(w, z, \bar{z})}
	\end{align*}
	and
	\begin{align*}
		a^{ij}(\bar{z}) = \sum_{|k|_\mfs < r} \frac{(\bar{z} - w)^k}{k!} \partial^k a^{ij}(w) + \sum_{k \in \partial_\mfs r} \frac{(\bar{z} - w)^{k_\downarrow}}{k_\downarrow !} \underbrace{\int \delta_k [\partial^{k_\downarrow} a^{ij} ] (w + (\bar{z} - w)y) {\mcQ}^{k_\downarrow}(dy)}_{R_{2, 2}^{i, j, k}(w, z, \bar{z})}.
	\end{align*}
	In total, we obtain that
	\begin{align*}
		& \sum_{i^\prime,j^\prime} a^{i i^\prime}(\bar{z}) a^{j j^\prime}(\bar{z}) \frac{(x - \bar{x})_{i^\prime}(x - \bar{x})_{j^\prime}}{t- \bar{t}}  - a^{ij}(\bar{z}) \\ & = \sum_{|k|_\mfs < r} \frac{(\bar{z}- w)^k}{k!} \Bigg [ \sum_{i^\prime, j^\prime}   \partial^k (a^{ii^\prime} a^{jj^\prime})(w)  \frac{(x - \bar{x})_{i^\prime}(x - \bar{x})_{j^\prime}}{t- \bar{t}} - \partial^k a^{ij}(w) \Bigg ]
		\\ & \qquad + \sum_{k \in \partial_\mfs r}  \frac{(\bar{z} - w)^{k_\downarrow}}{k_\downarrow !} \Bigg [ \underbrace{\sum_{i^\prime, j^\prime} R_{2,1}^{i,j,i^\prime, j^\prime, k}(w, z, \bar{z}) + R_{2,2}^{i,j,k} (w, z, \bar{z})}_{R_{ij}^{2,k}(w, z, \bar{z})} \Bigg ]
	\end{align*}
	With these expansions and the decomposition of Proposition~\ref{eq:Z_decomp} for $Z$ in hand, we can write
	\begin{align*}
		E(z, \bar{z}) = &  (t- \bar{t})^{-1/2}\sum_{\substack{|k^\alpha|_\mfs < r \\ \alpha = 1, 2,3}} \sum_{0< l \le k^1} \sum_{i,j}  \frac{(\bar{z}-w)^{k^1 + k^2 + k^3 - l}}{(k^1-l)! k^2! k^3 ! l!}\frac{(z - \bar{z})^l}{(t- \bar{t})^{1/2}} \partial^{k^1} a_{ij}(w)
		\\ &
		\Bigg [ \sum_{i^\prime, j^\prime}   \partial^{k^2} (a^{ii^\prime} a^{jj^\prime})(w)  \frac{(x - \bar{x})_{i^\prime}(x - \bar{x})_{j^\prime}}{t- \bar{t}} - \partial^{k^2} a^{ij}(w) \Bigg ] Z_w^{[k^3]}(z- \bar{z})
		+ (t- \bar{t})^{-1/2} \sum_{\substack{A\subset \{1,2,3\} \\ A \neq \emptyset}} R_A(w, z, \bar{z})
	\end{align*}
	where the sum over $A$ includes all terms that include at least one of the terms $R_{ij}^{l,k}$ or $Z^{\partial,[k]}$ in place of the terms appearing above. Here the set $A$ indexes which of the three terms is of the remainder type so that the first sum on the right hand side above corresponds to the case $A = \emptyset$. For brevity, we don't write each $R_A$ explicitly here, but we will write $R_{\{1\}}$ below to demonstrate how these terms are dealt with. We will show that the first sum on the right hand side contributes the sum over $A_r^1$ in the statement whilst the second sum on the right hand side contributes the sum over $A_r^2$.
	
	For the part corresponding to $A_r^1$, we simply note that we can rewrite the first sum as
	\begin{align*}
		\sum_{|k|_\mfs < 3r} (\bar{z} - w)^k \Bigg [ \sum_{\substack{|k^\alpha|_\mfs < r \\ \alpha = 1, 2,3}} \sum_{0< l \le k^1} \sum_{i,j}  & \frac{\mathbf{1}_{k = k^1 + k^2 + k^3 - l}}{(k^1-l)! k^2! k^3 ! l!}\frac{(z - \bar{z})^l}{(t- \bar{t})^{1/2}} \partial^{k^1} a_{ij}(w)
		\\ &
		\Bigg ( \sum_{i^\prime, j^\prime}   \partial^{k^2} (a^{ii^\prime} a^{jj^\prime})(w)  \frac{(x - \bar{x})_{i^\prime}(x - \bar{x})_{j^\prime}}{t- \bar{t}} - \partial^{k^2} a^{ij}(w) \Bigg ) (t- \bar{t})^{-1/2}Z_w^{[k^3]}(z- \bar{z})\Bigg ]
	\end{align*}
	We then define $A_1^r = \{k: |k|_\mfs < 3r\}$ and define $E_w^{[k]}(z - \bar{z})$ to be the term appearing in the square bracket above (noting that it may be possible that this is $0$ for some values of $k$).
	
	With this definition, we do indeed have that $E^{[k]} \in C_\mfs^{\varrho - r}(\Phi_{\infty}^1) \cap \Lambda_r$. This follows from the fact that $(t- \bar{t})^{-1/2} Z^{[k_3]} \in C_\mfs^{\varrho - r}(\Phi_{\infty}^1) \cap \Lambda_r$ and that this space is stable under multiplication with polynomials in $\frac{(x-\bar{x})}{(t-\bar{t})^{1/2}}$, polynomials in ${(t- \bar{t})^{1/2}}$ and polynomials in  $\partial^k a_{ij}(w), \partial^k a^{ij}(w)$ for $|k|_\mfs \le r$. 
	
	We now demonstrate how to deal with the term $R_A$ when $A = \{1\}$ by way of example for all the remainder terms. This term already demonstrates the techniques for the remaining terms and the appearance of the sum over $l$ in the statement.
	We write
	\begin{align*}
		R_{\{1\}}(w, z, \bar{z}) = \sum_{k^1 \in \partial_\mfs r} \sum_{\substack{|k^\alpha|_\mfs < r \\ \alpha =  2,3}} \sum_{i,j}R_{ij}^{1,k}(w,z,\bar{z}) \frac{(\bar{z} - w)^{k^2}}{k^2!} & \Bigg [ \sum_{i^\prime, j^\prime} \partial^{k^2} (a^{ii^\prime} a^{jj^\prime})(w) \frac{(x- \bar{x})_{i^\prime}(x - \bar{x})_{j^\prime}}{t - \bar{t}} - \partial^{k^2}a^{ij}(w) \Bigg ]\\ &(\bar{z}- w)^{k^3}Z_w^{[k^3]}(z - \bar{z}).
	\end{align*}
	We now rewrite
	\begin{align*}
		R_{ij}^{1,k}(w,z,\bar{z}) &= \frac{(\bar{z}-w)^{k_\downarrow^1}}{k_\downarrow^1 !} \int \Bigg [ \delta_{k^1} [\partial^{k_\downarrow^1} a_{ij}]((\bar{z} - w) y) - \delta_{k^1} [\partial^{k_\downarrow^1} a_{ij}](w + ({z} - w) y) \Bigg ] {\mcQ}^{k_\downarrow^1}(dy)
		\\
		& \quad + \frac{(\bar{z}-w)^{k_\downarrow^1} -({z}-w)^{k_\downarrow^1} }{k_\downarrow^1 !} \int  \delta_{k^1} [\partial^{k_\downarrow^1} a_{ij}](w + ({z} - w) y)  {\mcQ}^{k_\downarrow^1}(dy)
	\end{align*}
	and treat the parts of the expression for $R_{\{1\}}$ corresponding to each term in the above expression separately. It is straightforward to check that the contributions from the first term can be included in the sum over $A_r^2$ in the statement with $l = k_\downarrow$. For the latter 
	term,
	the same result follows by binomially expanding $(z-w)^{k_\downarrow^1} = (z- \bar{z} + \bar{z}-w)^{k_\downarrow^1}$ and using that $(z - \bar{z})^\eta$ improves the regularising effect by $|\eta|_\mfs$.
	
	It remains to treat the case where $b,c$ may be non-zero. For brevity, we demonstrate only how to treat the term $c(z) Z(z,\bar{z})$. The other term is similar. Here, we write
	\begin{align*}
		c(z) Z(z, \bar{z}) = \sum_{\substack{|k^\alpha|_\mfs < r \\ \alpha = 1, 2}}\sum_{l \le k^1} \frac{(\bar{z}- w)^{k^1 + k^2 - l}(\bar{z}- z)^l}{l! (k^1 - l)! k^2!} \partial^{k^1} c(w) Z_w^{[k^2]}(z - \bar{z}) + \sum_{\substack{A \subset \{1,2\} \\ A \neq \emptyset}} R_A(w,z,\bar{z})
	\end{align*}
	by use of 
	Proposition~\ref{prop:Z_decomp} 
	for $Z$ and 
	Lemma~\ref{Tay}
	for $c$. Here we have binomially expanded $(z - w)^\eta = (z- \bar{z} + \bar{z}- w)^\eta$. The summands of the first sum are clearly of the form required to be included in the sum over $A_r^1$ so we consider only an illustrative example of the term $R_A$. We again demonstrate the computation in the case $A = \{1\}$. In this case we have
	\begin{align*}
		R_{\{1\}}(w,z, \bar{z}) = \sum_{k^1 \in \partial_\mfs r} \sum_{|k^2|_\mfs < r} \sum_{l \le k^1} \frac{(\bar{z} - w)^{k_\downarrow^1 + k^2 - l} (z - \bar{z})^l}{l! (k_\downarrow^1 - l)! k^2!} Z_w^{[k^2]}(z - \bar{z}) \int \delta_{k^1}[\partial^{k_\downarrow^1} c](w + (z-w)y) {\mcQ}^{k_\downarrow^1}(dy) 
	\end{align*}
	which we again see if of the correct form for the sum over $A_r^2$.
\end{proof}

With these two results in hand, we now show that $Z \ast (-E)^{\ast k}$ can be decomposed in a way that is compatible with our main assumption. 

\begin{lemma}\label{lem:conv_decomp}
	Fix $r > 2$ and suppose that $\varrho \ge 3r$. Then for each $n \in \mbN$, we can write
	$$Z \ast (-E)^{\ast n}(z, \bar{z}) = K_n^{\bar{z}}(z- \bar{z}) + R_n(z, \bar{z})$$
	where
	$K_n \in C^r(\Phi_r^2) \cap \Lambda_r$ and $R_n \in \Psi_r^{ r}$. Furthermore the maps $(a,b,c) \mapsto K_n, R_n$ are continuous.
\end{lemma}
\begin{proof}
	We write
	\begin{align*}
		A_{w, i}^{[k]}(z- \bar{z}) = \begin{cases}
			Z_w^{[k]}(z - \bar{z}) \qquad & \text{ if } i = 0
			\\
			E_w^{[k]}(z - \bar{z}) \qquad & \text{ otherwise}
		\end{cases}
	\end{align*}
	along with
	\begin{align*}
		A_{w, i}^{\partial, [k], l}(z, \bar{z}) = \begin{cases}
			Z_w^{\partial, [k]}(z, \bar{z}) \qquad & \text{ if } i = 0, l = k_\downarrow
			\\
			0 \qquad & \text{ if } i = 0, l \neq k_\downarrow
			\\ 
			E_w^{\partial, [k], l}(z, \bar{z}) \qquad & \text{ otherwise.}
		\end{cases}
	\end{align*}
	By Proposition~\ref{prop:Z_decomp} and Proposition~\ref{prop:E_decomp}, we then have
	for all $w \in \mathbb{R}^{1 + d}$
	\begin{align}\label{eq:conv_decomp}
		&Z \ast (-E)^{\ast n}(z, \bar{z})\\ & \nonumber = \sum_{|k^0|_\mfs < r} \sum_{\substack{k^\alpha \in A_r^1 \\ \alpha = 1, \dots, n}}  \int \prod_{i = 0}^n A_{w,i}^{[k^i]}(\zeta_i - \zeta_{i+1}) (\zeta_{i+1} - w)^{k^i} \zeta_1 \dots d\zeta_{n}
		\\ \nonumber
		& + \sum_{\substack{A\subset \{0, \dots, n\} \\ A \neq \emptyset}} \sum_{k^0} \sum_{\substack{k^\alpha \in A_r^{1 + \mathbf{1}_A(\alpha)} \\ \alpha = 1, \dots, n}} \sum_{\substack{l^\alpha \le k_\downarrow^\alpha \\ \alpha \in A}} \int \prod_{i \in A} A_{w, i}^{\partial, [k^i], l^i}(\zeta_i, \zeta_{i+1}) (\zeta_{i+1} - w)^{l^i} \prod_{i \not \in A} A_{w,i}^{[k^i]} (\zeta_i - \zeta_{i+1}) (\zeta_{i+1} - w)^{k^i} d\zeta_1 \dots d\zeta_n						\end{align}
	where we have adopted the convention that $\zeta_0 = z, \zeta_{n+1} = \bar{z}$ and where the sum over $k^0$ on the second line runs over $|k^0|_\mfs < r$ if $0 \in A$ and $k^0 \in \partial_\mfs r$ otherwise.
	
	We now write $(\zeta_{i+1} - w) = \sum_{j = 0}^{n - i - 1} (\zeta_{j+i + 1} - \zeta_{j+i+2}) + (\zeta_{n+1} - w)$ on the first line on the right hand side and binomially expand. This yields that this line is equal to
	\begin{align*}
		\sum_{|k^0|_\mfs < r} \sum_{\substack{k^\alpha \in A_r^1 \\ \alpha = 1, \dots, n}} \sum_{\substack{\beta_0^i + \dots + \beta_{n-i}^i = k^i \\ i = 0, \dots, n}} {k^i \choose \beta_0^i \dots \beta_{n-i}^i} \int &  \prod_{i = 0}^{n-1} A_{w,i}^{[k^i]} ( \zeta_i - \zeta_{i+1}) (\zeta_{i+1}-\zeta_{i+2})^{\sum_{i^\prime \le i} \beta_{i - i^\prime}^{i^\prime}}
		\\
		& A_{w, n}^{[k^n]}(\zeta_{n} - \bar{z}) (\bar{z}- w)^{\sum_{i^\prime \le n} \beta_{n - i^\prime}^{i^\prime}}			d\zeta_1 \dots d\zeta_n.	\end{align*}
	We then define $K^w(z - \bar{z})$ to consist of the part of this sum where $\beta_{n - i^\prime}^{i^\prime} = 0$ for each $i^\prime \le n$. We note that we can discard the remaining terms since they vanish when $w = \bar{z}$. We then note that $K_n^w$ is a linear combination of convolutions of one term in $C^{r}(\Phi_\infty^2) \cap \Lambda_{|k|_\mfs}$ with a number of terms in $C^r(\Phi_\infty^1) \cap \Lambda_r$. By Proposition~\ref{prop:con} and the stability of $\Lambda_r$ under convolution, we see that $K_n$ has the desired form. 
	
	It remains to show that the sum appearing in the second line of \eqref{eq:conv_decomp} can be written as $R_n(z, \bar{z}) + \text{Remainder}$ where the remainder vanishes upon setting $w = \bar{z}$ and where $R_n$ is of the form required in the statement. Since this argument is similar to the one for the first term but notationally more heavy, we simply sketch the  required changes. 
	
	By again binomially expanding $(\zeta_{i+1} - w)^{k^i}$ as above, we see that the typical summand on the second line of \eqref{eq:conv_decomp} can be written as a linear combination of terms that vanish when setting $w = \bar{z}$ and terms of the form
	\begin{align*}
		\int \prod_{i \in A} A_{w, i}^{\partial, [k^i], l^i}(\zeta_i, \zeta_{i+1}) (\zeta_{i+1} - \zeta_{i+2})^{\tilde{l}^i} \prod_{i \not \in A} A_{w,i}^{[k^i]} (\zeta_i - \zeta_{i+1}) (\zeta_{i+1} - \zeta_{i+2})^{\tilde{k}^i} d\zeta_1 \dots d\zeta_n		
	\end{align*}
	where ${\tilde{l}^i}, \tilde{k^i} = 0$ if $i = n$ and 
	$\sum_{i \in A} \tilde{l}^i + \sum_{i \not \in A} \tilde{k}^i = \sum_{i \in A} l^i + \sum_{i \not \in A} k^i$. 
	Since $A_w(z, \bar{z}) \in C^r(\Psi_r^\alpha)$ implies that $(z-\bar{z})^l \in C^r(\Psi_r^{\alpha + |l|_\mfs})$ and $|k^\alpha_\downarrow|_\mfs \ge r - 2$ for each $\alpha \in A$, by Proposition~\ref{prop:con} we see that this term defines an element of $C^r(\Psi_r^\beta)$ with 
	\begin{align*}
		\beta 
		&
		= 2 + \sum_{\substack{i \in A \\ i \neq 0}} (1 + |k_\downarrow^i|_\mfs - |l^i|_\mfs) + \sum_{\substack{i \not \in A \\ i \neq 0}} 1 + \sum_{i \in A} |\tilde{l}^i|_\mfs + \sum_{i \not \in A} |\tilde{k^i}|_\mfs
		\\ &
		\ge 2 + |l^0|_\mfs \mathbf{1}_{0 \in A} + \sum_{\substack{i \in A \\ i \neq 0}} (1 + |k_\downarrow^i|_\mfs) 
		\\ &
		\ge 2 + (r-2) \mathbf{1}_{0 \in A}  + \sum_{\substack{i \in A \\ i \neq 0}} (r-1) \ge r
	\end{align*}
	where to reach the second line we used the identity 	$\sum_{i \in A} \tilde{l}^i + \sum_{i \not \in A} \tilde{k}^i = \sum_{i \in A} l^i + \sum_{i \not \in A} k^i$. 
	The desired form of $R_n$ then follows from the continuous map $C^r(\Psi_r^r) \hookrightarrow \Psi_r^r$ given by setting $w = \bar{z}$.
	Finally, continuity of the maps $(a,b,c) \mapsto K_n, R_n$ follows from the continuity statements in Proposition~\ref{prop:Z_decomp}, Proposition~\ref{prop:E_decomp} and Proposition~\ref{prop:con}.
\end{proof}
We are now ready to provide a proof of Proposition~\ref{prop:ker_dec}.

\begin{proof}[Proof of Proposition~\ref{prop:ker_dec}]
	Without loss of generality, we may assume that $M \ge 2r$. Let $\chi$ be the smooth cut-off function fixed in \eqref{eq:a33b}. Choosing $\varrho$ large enough as a function of $r, d$ and $M$, by Lemma~\ref{lem:convergence} and Lemma~\ref{lem:conv_decomp}, we can write
	\begin{align*}
		\Gamma(z, \bar{z}) = \sum_{n \le N} \chi(z - \bar{z}) K_n^{\bar{z}}(z- \bar{z}) + \sum_{n \le N} (1- \chi(z- \bar{z})) K_n^{\bar{z}}(z- \bar{z}) + \sum_{n \le N} R_n(z, \bar{z}) + \sum_{n > N} Z \ast (-E)^{\ast n}(z, \bar{z})
	\end{align*}
	where $N$ is chosen to be large enough so that the final term defines an element of $C_{\mathrm{loc},t}^r$, $R_n \in \Psi_r^{r + d + 2}$ and $K_n \in C^M(\Phi_M^2)$.
	We define $K^w(z) = \sum_{n \le K} \chi(z) K^w(z)$ and $R(z,\bar{z}) = \Gamma(z,\bar{z}) - K^{\bar{z}}(z - \bar{z})$.
	The result then follows by noting that multiplication by $(1 - \chi)$ is a continuous operation from $C^r(\Phi_r^2)$ to $C_{\mathrm{loc},t}^r$ and multiplication by $\chi$ is a continuous operation from $C^{M}(\Phi_{M}^2)$ to $C^r(\mcK_r^2)$. 
	
	Continuity of the decomposition as a function of the coefficient field follows from the corresponding continuity statements in Lemma~\ref{lem:convergence} and Lemma~\ref{lem:conv_decomp}.
\end{proof}
\appendix
\section{Existence of Tensor Products for Symmetric Sets}\label{app: tensor}

	\begin{proof}[Proof of Proposition~\ref{prop:Symm_Char}]
	For uniqueness, we note that if $(\bar{V}^{\proj \mcs}, (\bar{\otimes}_\mcs^a)_{a \in \mcs})$ is a second $\mcs$-symmetric tensor product then by applying the universal properties, we obtain the existence of unique $f, \bar{f}$ such that the $a$-indexed family of diagrams
	\begin{center}
		\begin{tikzcd}
			V^{\times a} \arrow[dr, "\bar \otimes_\mcs^a"'] \arrow[r, "\otimes_\mcs^a"] & V^{\proj \mcs} \arrow[d, bend left, dotted, "f"] \\
			& \bar{V}^{\proj \mcs} \arrow[u, bend left, dotted, "\bar f"]
		\end{tikzcd}
	\end{center}
	commute.
	In turn, this yields the commutative diagrams
	\begin{center}
		\begin{tikzcd}
			V^{\times a} \arrow[dr, "\otimes_\mcs^a"'] \arrow[r, "\otimes_\mcs^a"] & V^{\proj \mcs} \arrow[d, bend left, "\bar f \circ f"]  \arrow[d, bend right, "\id"'] \\
			&  V^{\proj \mcs}.
		\end{tikzcd}
	\end{center}
	By uniqueness of the factorisation, we deduce that $\bar{f} \circ f = \id$. Similarly, we see that $f \circ \bar{f} = \id$. Since $\|f\| = \|\otimes_\mcs^a\| = 1$ and similarly for $\bar{f}$, we deduce that $f$ is the desired isometric isomorphism.
	
	For existence, we adapt the definition given in \cite[Equation 5.5]{CCHS22} to appropriately include topology
	and show that the resulting object satisfies the universal property described here. We first recall that given $a \in \Ob(\mcs)$, the non-symmetrised projective tensor product $V^{\proj a}$ is characterised by the fact that multilinear map $f: V^{\times a} \to Z$ there exists a unique bounded linear map $f^\otimes : V^{\proj a} \to Z$ with $\|f^\otimes\| = \|f\|$ such that the diagram
		\begin{center}
			\begin{tikzcd}
				V^{\times a} \arrow[dr, "f"'] \arrow[r, "\otimes_a"] & V^{\proj a} \arrow[d, dotted,  "f^{\otimes}"] \\
				&  Z
			\end{tikzcd}
		\end{center}
		commutes
	where $\otimes_a : V^{\times a} \to V^{\proj a}$ denotes the usual tensorisation map. 
	In particular, if $\gamma : a \to b$ is an isomorphism of typed sets then by considering the commutative diagrams
	\begin{center}
		\begin{tikzcd}
			V^{\times a} \arrow[dr, "\otimes_b \circ \gamma"'] \arrow[r, "\otimes_a"] & V^{\proj a} \arrow[d, dotted,  "\gamma"] \\
			&  V^{\proj b}
		\end{tikzcd}
	\end{center}
	we see that $\gamma$ can also be viewed as an isometric isomorphism between $V^{\proj a}$ and $V^{\proj b}$, which we denote by the same character since the meaning should always be clear from context. The fact that the map $\gamma$ constructed here is an isometric isomorphism follows similarly to the analogous fact for the maps in the uniqueness part of this proof. 
	
	We now define
	\begin{align*}
		V^{\proj \mcs} = \left \{ (v_a)_{a \in \Ob(\mcs)} : v_a \in V^{\proj a}, \gamma \in \hom_\mcs(a,b) \implies \gamma(v_a) = v_b \right \} \subset \prod_{a \in \Ob(\mcs)} V^{\proj a} \ .
	\end{align*}
	We equip $V^{\proj \mcs}$ with the norm
	$\|(v_a)_{a \in \Ob(\mcs)} \|_{V^{\proj \mcs}} = \|v_b\|_{V^{\proj b}}$
	for any choice\footnote{Since any symmetry in $\mcs$ induces an isometry and $\mcs$ is a groupoid (and in particular, is connected), this definition is independent of the choice of $b$.} of $b \in \Ob(\mcs)$.  
	A natural symmetrisation map $\pi_{\mcs, a} : V^{\proj a} \to V^{\proj \mcs}$ is given by setting for $b \in \Ob(\mcs)$ 
	\begin{align*}
		(\pi_{\mcs, a} (v_a))_b	= |\hom_\mcs(a,b)|^{-1} \sum_{\gamma \in \hom_\mcs(a,b)} \gamma(v_a).
	\end{align*}
	We then define $\otimes_\mcs^a = \pi_{\mcs, a} \circ \otimes_a$. 
	
	We now show that $(V^{\proj \mcs}, (\otimes_\mcs^a)_{a \in \Ob(\mcs)})$ satisfies the universal property. To this end, we fix an $\mcs$-symmetric family of multilinear maps $f^a : V^{\times a} \to Z$. By the universal property for the tensor product $V^{\proj a}$, each $f^a$ factors through a linear map $f^{a, \otimes} : V^{\proj a} \to Z$ which is of the same norm in the case where $Z$ is a Banach space. We claim that for any $\gamma \in \hom_\mcs(a,b)$, $f^{a, \otimes} = f^{b, \otimes} \circ \gamma$. 
	
	To see this, we note that $f^{b, \otimes} \circ \gamma \circ \otimes_a = f^{b, \otimes} \circ \otimes_b \circ \gamma$ since the action of $\gamma$ on $V^{\proj a}$ was defined via the universal property. Then by definition of $f^{b, \otimes}$, $$f^{b, \otimes} \circ \otimes_b \circ \gamma = f^b \circ \gamma = f^a$$ where the last equality follows since $(f^a)_{a \in \Ob(\mcs)}$ is $\mcs$-symmetric. This shows that $f^{b, \otimes} \circ \gamma$ factors $f^a$ so that by uniqueness of the factoring map for the usual projective tensor product, the two are equal thus proving the claim. In particular, it follows that $f: V^{\proj \mcs} \to Z$ defined by
		$f((v_a)_{a \in \Ob(\mcs)}) = f^{b, \otimes} (v_b)$
	is independent of the choice of $b \in \Ob(\mcs)$. 
	
	To see that $f$ defined in this way does factor the family $(f^a)_{a \in \Ob(\mcs)}$ via $(\otimes_\mcs^a)_{a \in \Ob(\mcs)}$, we note that 
	\begin{align*}
		f( \otimes_{\mcs}^a (v_x)_{x \in a})  & = f ( \pi_{\mcs, a}  \otimes_a (v_x)_{x \in a}) = |\hom_\mcs(a,a)|^{-1} \sum_{\gamma \in \hom_\mcs(a,a)} f^{a, \otimes} ( \gamma (\otimes_a (v_x)_{x \in a}))
		\\
		& = f^{a, \otimes} (\otimes_a (v_x)_{x \in a})
		 = f^a ((v_x)_{x \in a})
	\end{align*}
	where the second to last line follows by symmetry invariance of $f^{a, \otimes}$ and the last line follows by definition of $f^{a, \otimes}$. 
	
	In the case where $Z$ is a Banach space, we note that since $f((v_a)_{a \in \Ob(\mcs)}) = f^{b, \otimes}(v_b)$, $\|(v_a)_{a \in \Ob(\mcs)}\|_{V^{\proj \mcs}} = \|v_b\|_{V^{\proj b}}$ and $\|f^{b, \otimes} \| = \| f^b \|$, it is immediate that $\|f\| \le \|f^b\|$ for any $b \in \Ob(\mcs)$. For the converse inequality, we note that if $v_b \in V^{\proj b}$ satisfies
	\begin{align*}
		|f^{b, \otimes}(v_b)| \ge \|v_b\|_{V^{\proj b}} (\|f^b\| - \varepsilon)
	\end{align*}
	then by symmetry invariance of $f^{b, \otimes}$
	\begin{align*}
		|f(\pi_{\mcs, b} v_b)| &= |f^{b, \otimes}(v_b)| \ge \|v_b\|_{V^{\proj b}} (\|f^b\| - \varepsilon) \ge  \|\pi_{\mcs, b} v_b\|_{V^{\proj \mcs}} (\|f^b\| - \varepsilon) 
	\end{align*}	
	where the last inequality follows by the triangle inequality. Since $\varepsilon > 0$ was arbitrary, we conclude that $\|f\| = \|f^{b, \otimes}\| = \|f^b\|$ for any $b \in \Ob(\mcs)$.
	
	To complete the proof that $(V^{\proj \mcs}, (\otimes_\mcs^a)_{a \in \Ob(\mcs)})$ satisfies the universal property, it remains only to show that the factoring map $f$ is unique. This follows from the fact that the linear span of the range of $\otimes_\mcs^a$ is dense in $V^{\proj \mcs}$ which in turn follows from the fact that the linear span of the range of $\otimes_a$ is dense in $V^{\proj a}$ and the fact that $\pi_{\mcs, a}: V^{\proj a} \to V^{\proj \mcs}$ is a continuous linear surjection.

	To complete the proof of the proposition, it remains to establish the representation of its norm in terms of symmetric tensors. By the uniqueness statement, it suffices to do this for the explicit $(V^{\proj \mcs}, (\otimes_\mcs^a)_{a \in \Ob(\mcs)})$ constructed above. 
	Recall that by definition of the usual projective cross norm we can write 
	\begin{align*}
		\|(v_a)_{a \in \Ob(\mcs)} \|_{V^{\proj \mcs}} &= \|v_b\|_{V^{\proj b}}  = \inf \left \{ \sum_{i = 0}^\infty \prod_{x \in b} \|v_x^i\|_{V_{\mfl(x)}} : v_b = \sum_{i = 0}^\infty \otimes_b (v_x^i)_{x \in b} \right \}
	\end{align*}
	for any $b \in \Ob(\mcs)$. On the one hand, if $v_b = \sum_{i = 0}^\infty \otimes_b (v_x^i)_{x \in b}$ then since $v_b$ is symmetry invariant, 
	\begin{align*}
		(v_a)_{a \in \Ob(\mcs)} = \pi_{\mcs, b} v_b = \sum_{i = 0}^\infty \otimes_\mcs^b (v_x^i)_{x \in v_b}
	\end{align*}
	where we used the explicit representation $\otimes_{\mcs}^b = \pi_{\mcs, b} \circ \otimes_b$. This establishes that
	\begin{align*}
		\inf \left \{ \sum_{i = 1}^\infty \prod_{x \in a} \|v_x^i \|_{V_{\mfl(x)}} : v = \sum_{i = 1}^\infty \otimes_\mcs^a (v_x^i)_{x \in a} \right \} \le \|v_b\|_{V^{\proj b}} = \|(v_a)_{a \in \Ob (\mcs)} \|_{V^{\proj \mcs}}.
	\end{align*}
	For the converse inequality, we note that if $(v_a)_{a \in \Ob(\mcs)} = \sum_{i = 0}^\infty \otimes_\mcs^b (v_x^i)_{x \in b}$ then
	\begin{align*}
		v_b = \sum_{i = 0}^\infty |\hom_\mcs(b,b)|^{-1} \sum_{\gamma \in \hom_\mcs(b,b)} \gamma ( \otimes_b (v_x^i)_{x \in b}).
	\end{align*}
	Therefore
	\begin{align*}
		\|v_b\|_{V^{\proj b}} \le \sum_{i= 0}^\infty |\hom_\mcs(b,b)|^{-1} \sum_{\gamma \in \hom_\mcs(b,b)}  \prod_{x \in b} \|v_{\gamma^{-1}(x)}^i \|_{V_{\mfl(x)}}= \sum_{i = 0}^\infty \prod_{x \in b} \|v_x^i\|_{V_{\mfl(x)}}.
	\end{align*}
	Taking an infimum over such representation of $(v_a)_{a \in \Ob(\mcs)}$ in this inequality establishes the bound $\|v_b\|_{V^{\proj b}} \le \|(v_a)_{a \in \Ob(\mcs)}\|_{V^{\proj \mcs}}$ which completes the proof.
\end{proof}
\section{Some Facts about Topological Hopf Algebras}\label{Hopf}
In this paper, the construction of regularity structures from Hopf algebras that is well-understood in the literature was applied in an infinite dimensional setting in which issues of continuity of maps are less automatic. This requires some minor adaptations to the standard construction to appropriately introduce topology. These modifications are fairly routine and as such we simply gather the necessary modifications to definitions and results in the algebraic setting that are important to us for the reader's convenience.
\begin{definition}\label{def: top_hopf}
	We call a tuple $(\mathcal{H}, \mathcal{M}, \mathbf{1}, \Delta, \varepsilon)$ a topological Hopf Algebra\footnote{In the wider literature, it seems that the object we define here would be called a `Hopf monoid in the category of complete locally convex topological vector spaces with monoidal structure induced by the projective tensor product'.  Whilst this way of thinking unites the introduction of topology to the various objects under consideration, for brevity we choose a simpler naming convention here.} if
	\begin{enumerate}
	\item $\mcH$ is a complete locally convex topological vector space,
	\item  $(\mathcal{H}, \mathcal{M}, \mathbf{1}, \Delta, \varepsilon)$ satisfies all of the defining properties of an (algebraic) Hopf algebra but with all tensor products replaced by projective tensor products, 
	\item the maps $\mcM, \Delta, \varepsilon, \mathbf{1}$ are continuous.
	\end{enumerate}
	We define topological algebras, coalgebras, bialgebras and comodules by adapting the definition of their algebraic counterparts in the obvious analagous way. We will often write $x \ast y = \mcM(x,y)$. We say that a topological (bi-/co-/Hopf) algebra is graded if $\mcH = \bigoplus_{n \in \mathbb{N}} \mcH_n$ is a graded topological vector space (where the direct sum is given the inductive limit topology) and all the operations respect the grading in the same way as in the purely algebraic case.
\end{definition}

For our purposes, the important fact will be that the structure group can be realised as an appropriate group of characters over a Hopf algebra and that its action on a given regularity structure will be continuous. The lemma below adapts the standard properties in the purely algebraic setting to appropriately include topology.

\begin{lemma}
	Suppose that $\mcH$ is a topological Hopf algebra. We define $\mcG = \{g \in \mcH^* : g(x \ast y) = g(x) g(y) \}$ where we emphasise that $\mcH^*$ denotes the topological dual of $\mcH$. Then $\mcG$ is a group for the product
    $g \ast h = (g\otimes h) \Delta$ and inverse $g^{-1} = g \circ \mcA$. Furthermore, if $(\mcV, \rho: \mcV \to \mcV \proj \mcH)$ is a topological comodule over $\mcH$ then 
    $$g \cdot v = (\id \otimes g) \rho$$ 
    defines a continuous action of $\mcG$ on $\mcV$.
\end{lemma}
\begin{proof}
	Since $\mcG$ is a subset of the algebraic character group of $\mcH$ with the induced product and inverse, to see that $\mcG$ is a group we need only check that it is closed under the operations. This is immediate from continuity of $\Delta$ and $\mcA$. Similarly, to see that $\mcG$ acts continuously on $\mcV$, we need only check continuity since the fact that the map defined is a group action is proved in the same way as in the algebraic case. However, continuity is immediate from continuity of $\rho$ and $g$.
\end{proof}

The last ingredient on Hopf algebras required in this paper is a method to show that suitable topological bialgebras admit a continuous antipode. We expect that the algebraic analogue of our construction is well-known but we are only aware of references covering the case of a connected graded bialgebra which corresponds to the case $d = 0$ in Lemma~\ref{bi_to_hopf} below. We thus provide a construction appropriate for our purposes for completeness. In preparation for Lemma~\ref{bi_to_hopf}, we need one preparatory result. 
\begin{lemma}\label{bi_1}
	Suppose that $\mcH = \bigoplus_{n \in \mbN} \mcH_n$ is a graded topological Hopf algebra such that $\mcH_0$ is the free commutative topological algebra over\footnote{We insist on a finite generating set only to make the topological structure unambiguous and continuity properties simpler to obtain.} $X = \langle g_i: i = 1, \dots, d\rangle$ where for each $i$, $g_i$ is primitive. Then $\varepsilon |_{\mcH_n} = 0$ for $n \ge 1$ and $\varepsilon|_{\mcH_0}$ is the natural projection onto the span of the unit in $\mcH_0$ induced by viewing $\mcH_0$ as the (topological) symmetric tensor algebra over $X$. 
\end{lemma}
\begin{proof}
	The former property is a standard fact about graded bialgebras which follows the fact that $\Delta$ respects the grading and the counital property. For the latter fact, we note that $g_i = \varepsilon(g_i) \mathbf{1} + \varepsilon(\mathbf{1}) g_i$ by the counital property and the fact that $g_i$ is primitive. Since $\varepsilon$ is an algebra morphism, we deduce that $\varepsilon(g_i) = 0$. Again using that $\varepsilon$ is an algebra morphism and that $\mcH_0$ is the free commutative topological algebra generated by $X$, the result follows.
\end{proof}
\begin{lemma}\label{bi_to_hopf}
    Suppose that $(\mcH = \bigoplus_{n \in \mbN} \mcH_n, \mcM, 1, \Delta, \eps)$ is a graded topological bialgebra such that $\mcH_0$ is the free commutative topological algebra over $\langle g_i: i = 1, \dots, d \rangle$ where for each $i$, $g_i$ is primitive. Then $\mcH$ is a graded topological Hopf algebra.
\end{lemma}
\begin{proof}
    It remains to construct a continuous antipode $\mcA$ which we do by induction in the degree. In the base case $n=0$, we first note that on $\langle g_i : i = 1, \dots, d \rangle$ we must set $\mcA g_i = - g_i$ since $\Delta g_i = g_i \otimes \mathbf{1} + \mathbf{1} \otimes g_i$ so that the defining property of an antipode (in combination with the requirement that it is an algebra morphism on $\mcH_0$) imply that $\mcA g_i + g_i = \varepsilon(g_i) \mathbf{1} = 0$ where we made use of Lemma~\ref{bi_1}. We note that $\mcA$ is automatically continuous on $\langle g_i: i = 1, \dots, d\rangle$ by finite dimensionality. Then by the universal property of the free commutative topological algebra over a topological vector space $X$, $\mcA$ has a unique extension to a continuous algebra morphism on $\mcH_0$. Since $\Delta$ and $\varepsilon$ are algebra morphisms, this extension satisfies the required property that 
    $ \mcM (\mcA \otimes \id) \Delta = \eps(\cdot) 1 = \mcM(\id \otimes \mcA) \Delta.$
    
    For $n > 0$, we proceed inductively and construct $\mcA$ according to the first identity above. We will then show a posteriori that it satisfies the second automatically. 

    We assume that $\tau \in \mcH_n$ for $n > 0$. 
	We note that by Lemma~\ref{bi_1} and the counital property, we have that 
    $$\Delta \tau = \tau \otimes 1 + 1 \otimes \tau + \mathring{\Delta} \tau$$
    where $\mathring{\Delta} : \mcH_n \to \mcH_{< n} \otimes \mcH_{< n}$. Therefore the first identity holds on $\mcH_{n}$ if and only if define $\mcA \tau = - \tau - \mcM(\mcA \otimes \id) \mathring{\Delta} \tau$
    which is well-defined by the induction hypothesis. Since this is a linear combination of compositions of continuous linear maps, it is continuous. 
    It then remains only to show that with this definition $\mcM(\id \otimes \mcA) \Delta \tau = \mcM(\mcA \otimes \id) \Delta \tau$ which can be equivalently written as
    $\mcM(\id \otimes \mcA) \mathring{\Delta} \tau = \mcM(\mcA \otimes \id) \mathring{\Delta} \tau$ by definition of $\mathring{\Delta}$.
    We again proceed induction in $n$ where the base case $n = 0$ was handled already above.

    We insert  the definition of $\mcA$ into the expression on the left hand side to obtain
    \begin{equs}
        \mcM(\id \otimes \mcA) \mathring{\Delta} \tau &= \mcM [ \id \otimes ( - \id - \mcM(\mcA \otimes \id) \mathring{\Delta})] \mathring{\Delta} \tau
        \\ & = - \mcM \mathring{\Delta} \tau - \mcM(\id \otimes \mcM)(\id \otimes \mcA \otimes \id) (\id \otimes \mathring{\Delta}) \mathring{\Delta}\tau
        \\
        & =  - \mcM \mathring{\Delta} \tau - \mcM(\mcM \otimes \id)(\id \otimes \mcA \otimes \id) (\mathring{\Delta} \otimes \id) \mathring{\Delta}\tau
    \end{equs}
    where we applied associativity of $\mcM$ and coassociativity of $\mathring{\Delta}$ (which follows from coassociativity of $\Delta$) to obtain the final line. Since $\mathring{\Delta} : \mcH_n \to \mcH_{< n} \otimes \mcH_{< n}$, we may apply the induction hypothesis to rewrite the right hand side as
    \begin{equs}
    - \mcM \mathring{\Delta} \tau - \mcM(\mcM \otimes \id)(\mcA \otimes \id \otimes \id) (\mathring{\Delta} \otimes \id) \mathring{\Delta}\tau = - \mcM \mathring{\Delta} \tau - \mcM (\mcM \otimes \id) [(- \mcA \otimes 1 - 1 \otimes \id) \otimes \id]\mathring{\Delta} \tau 
    \end{equs}
    where we applied the definition of $\mcA$. The right hand side is nothing but $\mcM(\mcA \otimes \id) \mathring{\Delta} \tau$, as was required.
\end{proof}
\begin{remark}\label{rem:pol_bi}
In this work, the main application of Lemma~\ref{bi_to_hopf} is in the proof of Lemma~\ref{lem:Hopf}. To apply the result in this setting, we note that when $V_{\mcT_+(R)}$ is graded by the number of edges, its degree $0$ component is $V_{\mcT_{\mathrm{poly}}}$. Since $X^k$ is a tree with no edges, for any Banach space assignment $V$, $V_{\mcT_{\mathrm{poly}}}$ is nothing other than the space of polynomials with real coefficients in commuting variables $X_1, \dots, X_d$ equipped with the inductive limit topology arising from viewing it as the limit of the finite dimensional spaces of polynomials of degree at most $N$. To see that this is the free commutative topological algebra over $\langle X_1, \dots, X_d \rangle$, we recall that the space of polynomials in commuting variables is the free commutative algebra over that set. Therefore, for every linear map $T: \langle X_1, \dots, X_d \rangle \to A$ into a commutative algebra $A$, there exists a unique algebra morphism $\bar{T} : V_{\mcT_{\mathrm{poly}}} \to A$ extending $T$. It remains to show that if $T$ is a continuous linear map into a topological algebra then $\bar{T}$ is continuous since this would imply that $V_{\mcT_{\mathrm{poly}}}$ satisfies the universal property of the free commutative topological algebra. By definition of the inductive limit topology, this is true if and only if the restriction of $\bar{T}$ to the space of polynomials of degree at most $N$ is continuous for each $N$. This is in turn automatic by finite dimensionality.
\end{remark}
\section{An Anisotropic Taylor Formula}\label{app:Tay}

We state here an anisotropic Taylor's formula with a remainder that is in increment form.  The statement and proof are a minor modification of \cite[Proposition A.1]{Hai14} essentially following by being more restrictive in applying the Fundamental Theorem of Calculus in the inductive proof given there and by making a change of variables to rewrite the measures. In comparison to \cite{Hai14} the statement here is convenient both because it often allows us to save a derivative in favour of a H\"older bound and also because we have rewritten the measures to more easily access the polynomial contributions. In particular, we implicitly made use of this property in Section~\ref{s:green_kernel}.

We define measures on $\mbR^{d}$ via
\begin{align*}
	\mcQ^k(dy) = \left (\prod_{i=1}^{\mathfrak{m}(k) - 1} \delta_1(dy_i) \right ) l \mathbf{1}_{[0,1]} (1-y_{\mathfrak{m}(k)})^{l-1} dy_{\mathfrak{m}(k)} \left ( \prod_{i = \mathfrak{m}(k) +1}^d \delta_0(dy_i) \right )
\end{align*}
where we have written $\mathfrak{m}(k) = \min\{j : k_j \neq 0\}$. We additionally define for $k \neq 0$, $k_\downarrow = k - e_{\mathfrak{m}(k)}$. We then have the following statement.

\begin{lemma}\label{Tay}
	Let $A \subset \mbN^{d}$ be such that $k \in A$ and $l \le k$ implies that $l \in A$. Define $\partial A = \{ k \not \in A: k_\downarrow \in A\}$. Then if $f \in C^r$ and for all $k \in A$, $|k|_\mfs < r$ we have that
	\begin{align*}
		f(x) = \sum_{k \in A} \frac{D^k f(0)}{k!} x^k + \sum_{k \in \partial A} \frac{x^{k_\downarrow}}{k_\downarrow !} \int_{\mbR^d} \delta_k [ \partial^{k_\downarrow} f]((x_i y_i)_{i = 1}^d) \mcQ^{k_\downarrow}(dy)
	\end{align*}
	where $\delta_k[f](y) = f(y_1, \dots, y_{e_{\mathfrak{m}(k)}}, 0, \dots, 0) -  f(y_1, \dots, y_{e_{\mathfrak{m}(k) - 1}}, 0, \dots, 0)$.
\end{lemma}

%
%

\section{Index of Notations} \label{s:index_notations}

\begin{longtable}
	[H]
	{l l l}
	\hline
	\textbf{Symbol} & \textbf{Meaning} & \textbf{Page} \\ \hline \endfirsthead 
	\hline       
	\textbf{Symbol} & \textbf{Meaning} & \textbf{Page} \\ \hline \endhead 
	\vspace{1em}
	$\mcA_+^{\mathrm{eq}; \mcO}$ & Space of kernel assignments of order $\mcO$ on $\mcT^{\mathrm{eq}}$ & \pageref{d:kernel_ass_teq} \\
	$\mcA_-^{\mathrm{eq}}$ & Space of noise assignments on $\mcT^{\mathrm{eq}}$ & \pageref{d:noise_ass} \\
	$\mcA_-^{\infty, \mathrm{eq}}$ & Space of smooth noise assignments on $\mcT^{\mathrm{eq}}$ & \pageref{d:noise_ass} \\
	$\mcA_+^{\mco}$ & Space of kernel assignments of order $\mco$ on $\mcT^{\Ban}$ & \pageref{def:ka} \\
	$\mcA_-$ & Space of noise assigments on $\mcT^{\Ban}$ & \pageref{d:noise_ass_TBan}\\
	$\mcA_-^{\infty}$ & Space of smooth noise assigments on $\mcT^{\Ban}$ & \pageref{d:noise_ass_TBan} \\
	$\mathrm{Age} ( \boldsymbol{\tau} )$ & Age of the tree $\boldsymbol{\tau}$ & \pageref{d:age} \\
	$\mathcal{B}^r$ & Ball of test functions & \pageref{ss:notations} \\
	$C^r ( \mcK_r^{\beta} )$ & Space of $C^r$ curves valued in $\mathcal{K}_r^{\beta}$ & \pageref{d:rk3} \\
	$C_w^{\alpha}$ & H\"older space of regularity $\alpha$ with weight $w$ & \pageref{def:wHoel} \\
	$\mathbb{C} [ \tau ]$ & Set of cuts of combinatorial tree $\tau$ & \pageref{def:cut} \\
	$\mathcal{D}^{\gamma}$ & Space of modelled distributions & \pageref{sec:mod_dist} \\
	$\mathcal{D}^{\gamma, \nu; x}$ & Space of (H\"older) pointed modelled distributions & \pageref{d:pointed_md} \\
	$\underline{\Delta}$ & Coaction on combinatorial trees & \pageref{eq:underline_delta} \\
	$\Delta$ & Coaction on $V_{\mathcal{T}}$ & \pageref{def:triangle+} \\
	$\underline{\Delta}^+$ & Positive coproduct on combinatorial trees & \pageref{def:triangle+} \\
	$\Delta^+$ & Positive coproduct on $V_{\mathcal{T}}$ & \pageref{def:triangle+} \\
	$\underline{\Delta}_r^-$ & Negative rooted `coproduct' on combinatorial trees & \pageref{d:delta_r_-} \\
	$\Delta_r^-$ & Negative rooted `coproduct' on $V_{\mathcal{T}}$ & \pageref{d:delta_r_-} \\
	$\tilde{\Delta}$ & `Coproduct' on combinatorial trees related to $f_x^\btau$ & \pageref{def:delta_tilde} \\
	$\mathrm{Hist} ( \mcS)$ & History of set of trees $\mcS$ & \pageref{d:history} \\
	$\mcI^\zeta$ & Abstract integration map on $\mcT^{\Ban}$ & \pageref{d:int_tban} \\
	$J^\zeta$ & Component of abstract convolution map on $\mcT^{\Ban}$ & \pageref{theo: multi-level Schauder} \\
	$\mathfrak{K}$ & Generic compact subset of $\mathbb{R}^d$ & \pageref{ss:notations} \\
	$\bar{\mathfrak{K}}$ & 1-fattening of compact set $\mathfrak{K}$ & \pageref{ss:notations} \\
	$| k |_{\mfs}$ & Rescaled degree of multi-index $k \in \mathbb{N}^d$ & \pageref{ss:notations} \\
	$\mathcal{K}_r^{\beta}$ & Space of translation-invariant $\beta$-regularising kernels & \pageref{d:rk2n} \\
	$\mathcal{K}_{\gamma}^{\zeta}$ & Abstract convolution & \pageref{theo: multi-level Schauder} \\
	$\mathcal{K}_{\gamma, m}$ & Variable coefficient abstract convolution & \pageref{theo: VML Schauder} \\
	$\mathcal{K}_{\gamma, \nu}^{\zeta, x}$ & Pointed abstract convolution & \pageref{thm:pointed_mls} \\
	$\mathcal{K}_{\gamma, \nu, m}^{x}$ & Pointed variable coefficient abstract convolution & \pageref{theo: VML Pointed Schauder 2} \\
	$K^{\zeta}$ & Translation-invariant kernel obtained by testing $K$ against $\zeta$ in the `upper slot' & \pageref{eq:a34b} \\
	$| \mfl |_{\mfs}$ & Degree of the type $\mfl$ or of edges of type $\mfl$ & \pageref{def:deg_assignment} \\
	$L ( E, F )$ & Bounded linear maps between normed spaces $E$ and $F$ & \pageref{ss:notations} \\
	$\mfL$ & Generic set of types & \pageref{eq:SPDEsystem} \\
	$\mfL_+$ & Set of kernel types & \pageref{eq:SPDEsystem} \\
	$\mfL_-$ & Set of noise types & \pageref{eq:SPDEsystem} \\
	$\mathcal{L}_{\tau}^{(k)}$ & Lifting map & \pageref{def:identified_tree} \\
	$\mcM(\mcT)$ & Space of admissible models on regularity structure $\mcT$ & \pageref{def:admissible_model_Teq} \\
	$\mcM_{0}(\mcT)$ & Space of smooth admissible models on regularity structure $\mcT$ & \pageref{def:admissible_model_Teq} \\
	$\mcN_\gamma^\zeta$ & Component of abstract convolution map on $\mcT^{\Ban}$ & \pageref{theo: multi-level Schauder} \\
	$\mathrm{ord} ( W )$ & Order of a good sector $W$ & \pageref{def:goodasdf} \\
	$\mcQ_{\gamma}$ & Projection onto components of homogeneity $\gamma$ & \pageref{def:prep} \\
	$\mcQ_{< \gamma}$ & Projection onto components of homogeneity $< \gamma$ & \pageref{def:prep} \\
	$\mcQ_{\boldsymbol{\tau}}$ & Projection onto $V^{\proj \sttau}$ & \pageref{def:prep} \\
	$R$ & Complete and subcritical rule & \pageref{def:BHZreduced} \\
	$\mathcal{R}$ & Reconstruction map & \pageref{sec:mod_dist} \\
	$\langle \mcS \rangle$ & Free $\mathbb{R}$-vector space over a set $\mcS$ & \pageref{ss:notations} \\
	$\mfs$ & Generic scaling & \pageref{ss:notations} \\
	$\mcs$ & Generic symmetric set & \pageref{def: SymSet} \\
	$\tau$ & Concrete combinatorial tree & \pageref{def:decorated_tree} \\
	$\boldsymbol{\tau}$ & Isomorphism class of combinatorial trees $\tau$ & \pageref{def:Tree_SSet} \\
	$\sttau$ & Symmetric set associated to the isomorphism class $\boldsymbol{\tau}$ & \pageref{def:Tree_SSet} \\
	$| \boldsymbol{\tau} |_{\mfs}$ & Degree of $\boldsymbol{\tau}$ & \pageref{def:sets_of_trees} \\
	$| \boldsymbol{\tau} |_{\#}$ & Number of noise edges in $\boldsymbol{\tau}$ & \pageref{def:prep} \\
	$\mathcal{T}$ & Set of (isomorphism classes of) decorated trees & \pageref{def:sets_of_trees} \\
	$\mathcal{T} ( R )$ & Set of trees that strongly conform to the rule $R$ & \pageref{def:sets_of_trees} \\
	$\mathcal{T}_+ ( R )$ & Monoid of positive degree trees & \pageref{def:sets_of_trees} \\
	$\mathcal{T}_- ( R )$ & Set of negative degree trees & \pageref{def:sets_of_trees} \\
	$\mathcal{T}_{\mathrm{plant}}$ & Set of trees of the form $X^j \mcI_\mfl^k \btau$ & \pageref{def:der_of_planted} \\
	$\mathcal{T}_{\mathrm{poly}}$ & Set of monomials & \pageref{lem:Gamma_Poly} \\
	$\mathcal{T} ( R )$ & BHZ reduced regularity structure for the rule $R$ & \pageref{def:BHZreduced} \\
	$\mathcal{T}^{\eq}$ & BHZ reduced regularity structure for the given equation & \pageref{def:BHZreduced} \\
	$\mathcal{T}^{\mathrm{Ban}} (V )$ & \multicolumn{1}{p{.7\textwidth}}{Infinite-dimensional regularity structure over Banach space assigment $V$} & \pageref{d:tban} \\
	$T_{\gamma, \nu}^{\zeta, x}$ & Component of abstract (pointed) convolution map on $\mcT^{\Ban}$ & \pageref{thm:pointed_mls} \\
	$V$ & Generic Banach space assigment & \pageref{def:Ban_Ass} \\
	$V_{\mathcal{W}}$ & Structure space over set of trees $\mathcal{W}$ associated to Banach space assignment $V$ & \pageref{eq:banach_space over sets of trees} \\
	$W$ & Generic sector & \pageref{sec:mod_dist} \\
	$\prec$ & Inductive ordering on trees & \pageref{def:prec} \\
\end{longtable}

\endappendix 

\bibliographystyle{Martin}
\bibliography{VarBPHZ.bib}

\newcommand{\etalchar}[1]{$^{#1}$}
\begin{thebibliography}{BDFT23b}
\def\myhref#1#2{\href{#2}{\nolinkurl{#1}}}

\bibitem[BB16]{BB16}
\textsc{I.~Bailleul} and \textsc{F.~Bernicot}.
\newblock Heat semigroup and singular {PDE}s.
\newblock \emph{Journal of Functional Analysis} \textbf{270}, no.~9, (2016),
  3344--3452.
\newblock
  \myhref{doi:10.1016/j.jfa.2016.02.012}{https://dx.doi.org/10.1016/j.jfa.2016.02.012}.

\bibitem[BB19]{BB19}
\textsc{I.~Bailleul} and \textsc{F.~Bernicot}.
\newblock High order paracontrolled calculus.
\newblock In \emph{Forum of Mathematics, Sigma}, vol.~7,  e44. Cambridge
  University Press, 2019.
\newblock
  \myhref{doi:10.1017/fms.2019.44}{https://dx.doi.org/10.1017/fms.2019.44}.

\bibitem[BB21]{BB21}
\textsc{I.~Bailleul} and \textsc{Y.~Bruned}.
\newblock {Locality for Singular Stochastic {{PDEs}}}.
\newblock \emph{arXiv preprint} (2021).
\newblock \myhref{arXiv:2109.00399}{https://arxiv.org/abs/2109.00399}.

\bibitem[BCCH20]{BCCH20}
\textsc{Y.~Bruned}, \textsc{A.~Chandra}, \textsc{I.~Chevyrev}, and
  \textsc{M.~Hairer}.
\newblock Renormalising {SPDE}s in regularity structures.
\newblock \emph{Journal of the European Mathematical Society} \textbf{23},
  no.~3, (2020), 869--947.
\newblock \myhref{doi:10.4171/jems/1025}{https://dx.doi.org/10.4171/jems/1025}.

\bibitem[BCH{\etalchar{+}}25]{BAUD25}
\textsc{F.~Baudoin}, \textsc{L.~Chen}, \textsc{C.-H. Huang},
  \textsc{C.~Ouyang}, \textsc{S.~Tindel}, and \textsc{J.~Wang}.
\newblock {Weighted Besov spaces on Heisenberg groups and applications to the
  Parabolic Anderson model}.
\newblock \emph{arXiv preprint} (2025).
\newblock
  \myhref{arXiv:arXiv:2501.04593}{https://arxiv.org/abs/arXiv:2501.04593}.

\bibitem[BCZ24]{BCZ24}
\textsc{L.~Broux}, \textsc{F.~Caravenna}, and \textsc{L.~Zambotti}.
\newblock {Hairer’s multilevel Schauder estimates without regularity
  structures}.
\newblock \emph{Transactions of the American Mathematical Society} (2024).
\newblock \myhref{doi:10.1090/tran/9245}{https://dx.doi.org/10.1090/tran/9245}.

\bibitem[BDFT23a]{BDFT23a}
\textsc{I.~Bailleul}, \textsc{N.~Dang}, \textsc{L.~Ferdinand}, and
  \textsc{T.~T{\^o}}.
\newblock {$\Phi^4_3$ measures on compact {R}iemannian 3-manifolds}.
\newblock \emph{arXiv preprint} (2023).
\newblock \myhref{arXiv:2304.10185}{https://arxiv.org/abs/2304.10185}.

\bibitem[BDFT23b]{BDFT23b}
\textsc{I.~Bailleul}, \textsc{N.~Dang}, \textsc{L.~Ferdinand}, and
  \textsc{T.~D. T{\^o}}.
\newblock {Global harmonic analysis for $\Phi^4_3 $ on closed {R}iemannian
  manifolds}.
\newblock \emph{arXiv preprint} (2023).
\newblock \myhref{arXiv:2306.07757}{https://arxiv.org/abs/2306.07757}.

\bibitem[BH23]{BH23}
\textsc{I.~Bailleul} and \textsc{M.~Hoshino}.
\newblock {Random models on regularity-integrability structures}.
\newblock \emph{arXiv preprint} (2023).
\newblock \myhref{arXiv:2310.10202}{https://arxiv.org/abs/2310.10202}.

\bibitem[BHK24]{BHK24}
\textsc{I.~Bailleul}, \textsc{M.~Hoshino}, and \textsc{S.~Kusuoka}.
\newblock {Regularity structures for quasilinear singular SPDEs}.
\newblock \emph{Archive for Rational Mechanics and Analysis} \textbf{248},
  no.~6, (2024), 127.
\newblock
  \myhref{doi:10.1007/s00205-024-02069-6}{https://dx.doi.org/10.1007/s00205-024-02069-6}.

\bibitem[BHZ19]{BHZ19}
\textsc{Y.~Bruned}, \textsc{M.~Hairer}, and \textsc{L.~Zambotti}.
\newblock {Algebraic renormalisation of regularity structures}.
\newblock \emph{Inventiones mathematicae} \textbf{215}, (2019), 1039--1156.
\newblock
  \myhref{doi:10.1007/s00222-018-0841-x}{https://dx.doi.org/10.1007/s00222-018-0841-x}.

\bibitem[BL24]{BL24}
\textsc{Y.~Bruned} and \textsc{P.~Linares}.
\newblock {A top-down approach to algebraic renormalization in regularity
  structures based on multi-indices}.
\newblock \emph{Archive for Rational Mechanics and Analysis} \textbf{248},
  no.~6, (2024), 111.
\newblock
  \myhref{doi:10.1007/s00205-024-02041-4}{https://dx.doi.org/10.1007/s00205-024-02041-4}.

\bibitem[BOS25]{BOS25}
\textsc{L.~Broux}, \textsc{F.~Otto}, and \textsc{R.~Steele}.
\newblock {Multi-Index {{Based Solution Theory}} to the $\Phi^4$ {{Equation}}
  in the {{Full Subcritical Regime}}}.
\newblock \emph{arXiv preprint} (2025).
\newblock \myhref{arXiv:2503.01621}{https://arxiv.org/abs/2503.01621}.

\bibitem[BOTW23]{BOTW23}
\textsc{F.~Baudoin}, \textsc{C.~Ouyang}, \textsc{S.~Tindel}, and
  \textsc{J.~Wang}.
\newblock {Parabolic Anderson model on Heisenberg groups: The Itô setting}.
\newblock \emph{Journal of Functional Analysis} \textbf{285}, no.~1, (2023),
  109920.
\newblock
  \myhref{doi:10.1016/j.jfa.2023.109920}{https://dx.doi.org/10.1016/j.jfa.2023.109920}.

\bibitem[Bru18]{Bru18}
\textsc{Y.~Bruned}.
\newblock Recursive formulae in regularity structures.
\newblock \emph{Stochastics and Partial Differential Equations: Analysis and
  Computations} \textbf{6}, (2018), 525--564.
\newblock
  \myhref{doi:10.1007/s40072-018-0115-z}{https://dx.doi.org/10.1007/s40072-018-0115-z}.

\bibitem[CCHS22]{CCHS22}
\textsc{A.~Chandra}, \textsc{I.~Chevyrev}, \textsc{M.~Hairer}, and
  \textsc{H.~Shen}.
\newblock Langevin dynamic for the {2D} {Y}ang--{M}ills measure.
\newblock \emph{Publications math{\'e}matiques de l'IH{\'E}S} (2022), 1--147.
\newblock
  \myhref{doi:10.1007/s10240-022-00132-0}{https://dx.doi.org/10.1007/s10240-022-00132-0}.

\bibitem[CF24a]{CF24a}
\textsc{A.~Chandra} and \textsc{L.~Ferdinand}.
\newblock A flow approach to the generalized {KPZ} equation.
\newblock \emph{arXiv preprint} (2024).
\newblock \myhref{arXiv:2402.03101}{https://arxiv.org/abs/2402.03101}.

\bibitem[CF24b]{CF24b}
\textsc{A.~Chandra} and \textsc{L.~Ferdinand}.
\newblock Rough differential equations in the flow approach.
\newblock \emph{arXiv preprint} (2024).
\newblock \myhref{arXiv:2411.07157}{https://arxiv.org/abs/2411.07157}.

\bibitem[CFW24]{CFW24}
\textsc{A.~Chandra}, \textsc{G.~d.~L. Feltes}, and \textsc{H.~Weber}.
\newblock A priori bounds for 2-d generalised parabolic anderson model.
\newblock \emph{arXiv preprint} (2024).
\newblock \myhref{arXiv:2402.05544}{https://arxiv.org/abs/2402.05544}.

\bibitem[CFX24]{CFX23}
\textsc{Y.~Chen}, \textsc{B.~Fehrman}, and \textsc{W.~Xu}.
\newblock Periodic homogenisation for two dimensional generalised parabolic
  {A}nderson model.
\newblock \emph{arXiv preprint} (2024).
\newblock \myhref{arXiv:2401.05718}{https://arxiv.org/abs/2401.05718}.

\bibitem[CH16]{CH16}
\textsc{A.~Chandra} and \textsc{M.~Hairer}.
\newblock An analytic {BPHZ} theorem for regularity structures.
\newblock \emph{arXiv preprint} (2016).
\newblock \myhref{arXiv:1612.08138}{https://arxiv.org/abs/1612.08138}.

\bibitem[Cha12]{charlap2012bieberbach}
\textsc{L.~S. Charlap}.
\newblock \emph{Bieberbach groups and flat manifolds}.
\newblock Springer Science \& Business Media, 2012.
\newblock
  \myhref{doi:10.1007/978-1-4613-8687-2}{https://dx.doi.org/10.1007/978-1-4613-8687-2}.

\bibitem[CMW23]{CMW23}
\textsc{A.~Chandra}, \textsc{A.~Moinat}, and \textsc{H.~Weber}.
\newblock A {P}riori {B}ounds for the $\phi^4$ {E}quation in the {F}ull
  {S}ub-critical {R}egime.
\newblock \emph{Archive for Rational Mechanics and Analysis} \textbf{247},
  no.~3, (2023), 48.
\newblock
  \myhref{doi:10.1007/s00205-023-01876-7}{https://dx.doi.org/10.1007/s00205-023-01876-7}.

\bibitem[CO25]{CO25}
\textsc{H.~Chen} and \textsc{C.~Ouyang}.
\newblock {Global Geometry within an SPDE Well-Posedness Problem}.
\newblock \emph{arXiv preprint} (2025).
\newblock \myhref{arXiv:2502.04572}{https://arxiv.org/abs/2502.04572}.

\bibitem[CX23]{CX23}
\textsc{Y.~Chen} and \textsc{W.~Xu}.
\newblock Periodic homogenisation for $p(\phi)_2$.
\newblock \emph{arXiv preprint} (2023).
\newblock \myhref{arXiv:2311.16398}{https://arxiv.org/abs/2311.16398}.

\bibitem[DDD19]{DDD19}
\textsc{A.~Dahlqvist}, \textsc{J.~Diehl}, and \textsc{B.~K. Driver}.
\newblock The parabolic {A}nderson model on {R}iemann surfaces.
\newblock \emph{Probability Theory and Related Fields} \textbf{174}, no. 1-2,
  (2019), 369--444.
\newblock
  \myhref{doi:10.1007/s00440-018-0857-6}{https://dx.doi.org/10.1007/s00440-018-0857-6}.

\bibitem[DGR23]{DGR23}
\textsc{P.~Duch}, \textsc{M.~Gubinelli}, and \textsc{P.~Rinaldi}.
\newblock Parabolic stochastic quantisation of the fractional {$\Phi^4_3$}
  model in the full subcritical regime.
\newblock \emph{arXiv preprint} (2023).
\newblock \myhref{arXiv:2303.18112}{https://arxiv.org/abs/2303.18112}.

\bibitem[Duc22]{Duc22}
\textsc{P.~Duch}.
\newblock Renormalization of singular elliptic stochastic {PDEs} using flow
  equation.
\newblock \emph{arXiv preprint} (2022).
\newblock \myhref{arXiv:2201.05031}{https://arxiv.org/abs/2201.05031}.

\bibitem[Duc25]{Duc21}
\textsc{P.~Duch}.
\newblock Flow equation approach to singular stochastic {PDE}s.
\newblock \emph{Probability and Mathematical Physics} \textbf{6}, no.~2,
  (2025), 327--437.
\newblock
  \myhref{doi:10.2140/pmp.2025.6.327}{https://dx.doi.org/10.2140/pmp.2025.6.327}.

\bibitem[EH19]{EH16}
\textsc{D.~Erhard} and \textsc{M.~Hairer}.
\newblock Discretisation of regularity structures.
\newblock \emph{Annales de l'Institut Henri Poincar{\'e}, Probabilit{\'e}s et
  Statistiques} \textbf{55}, no.~4(2019).
\newblock
  \myhref{doi:10.1214/18-AIHP947}{https://dx.doi.org/10.1214/18-AIHP947}.

\bibitem[EMR24]{EMR24}
\textsc{H.~Eulry}, \textsc{A.~Mouzard}, and \textsc{T.~Robert}.
\newblock Anderson stochastic quantization equation.
\newblock \emph{arXiv preprint} (2024).
\newblock \myhref{arXiv:2401.12742}{https://arxiv.org/abs/2401.12742}.

\bibitem[EW24]{EW24}
\textsc{S.~Esquivel} and \textsc{H.~Weber}.
\newblock A priori bounds for the dynamic fractional {{$\Phi^4$}} model on
  {{$\mathbb{T}^3$}} in the full subcritical regime.
\newblock \emph{arXiv preprint} (2024).
\newblock \myhref{arXiv:2411.16536}{https://arxiv.org/abs/2411.16536}.

\bibitem[Fri64]{Fri08}
\textsc{A.~Friedman}.
\newblock \emph{Partial differential equations of parabolic type}.
\newblock Prentice-Hall, Inc., Englewood Cliffs, NJ, 1964,  xiv+347.

\bibitem[GH19]{GH19}
\textsc{M.~Gerencs{\'e}r} and \textsc{M.~Hairer}.
\newblock A solution theory for quasilinear singular spdes.
\newblock \emph{Communications on Pure and Applied Mathematics} \textbf{72},
  no.~9, (2019), 1983--2005.
\newblock \myhref{doi:10.1002/cpa.21816}{https://dx.doi.org/10.1002/cpa.21816}.

\bibitem[GH22]{GH21}
\textsc{M.~Gerencs{\'e}r} and \textsc{M.~{Hairer}}.
\newblock Boundary renormalisation of {{SPDEs}}.
\newblock \emph{Communications in Partial Differential Equations} \textbf{47},
  no.~10, (2022), 2070--2123.
\newblock
  \myhref{doi:10.1080/03605302.2022.2109173}{https://dx.doi.org/10.1080/03605302.2022.2109173}.

\bibitem[GHM24]{GHM24}
\textsc{A.~Gerasimovi{\v{c}}s}, \textsc{M.~Hairer}, and \textsc{K.~Matetski}.
\newblock Directed mean curvature flow in noisy environment.
\newblock \emph{Communications on Pure and Applied Mathematics} \textbf{77},
  no.~3, (2024), 1850--1939.
\newblock \myhref{doi:10.1002/cpa.22158}{https://dx.doi.org/10.1002/cpa.22158}.

\bibitem[GIP15]{GIP15}
\textsc{M.~Gubinelli}, \textsc{P.~Imkeller}, and \textsc{N.~Perkowski}.
\newblock Paracontrolled distributions and singular pdes.
\newblock In \emph{Forum of Mathematics, Pi}, vol.~3, ~e6. Cambridge University
  Press, 2015.
\newblock
  \myhref{doi:10.1017/fmp.2015.2}{https://dx.doi.org/10.1017/fmp.2015.2}.

\bibitem[Gri04]{Grieser}
\textsc{D.~Grieser}.
\newblock {Notes On Heat Kernel Asymptotics}, 2004.

\bibitem[Hai14]{Hai14}
\textsc{M.~Hairer}.
\newblock A theory of regularity structures.
\newblock \emph{Inventiones mathematicae} \textbf{198}, no.~2, (2014),
  269--504.
\newblock
  \myhref{doi:10.1007/s00222-014-0505-4}{https://dx.doi.org/10.1007/s00222-014-0505-4}.

\bibitem[HL15]{HL15}
\textsc{M.~Hairer} and \textsc{C.~Labb\'{e}}.
\newblock A simple construction of the continuum parabolic {A}nderson model on
  {${\bf R}^2$}.
\newblock \emph{Electron. Commun. Probab.} \textbf{20}, (2015), no. 43, 11.
\newblock
  \myhref{doi:10.1214/ECP.v20-4038}{https://dx.doi.org/10.1214/ECP.v20-4038}.

\bibitem[HL17]{HL17}
\textsc{M.~Hairer} and \textsc{C.~Labb{\'e}}.
\newblock The reconstruction theorem in {{Besov}} spaces.
\newblock \emph{Journal of Functional Analysis} \textbf{273}, no.~8, (2017),
  2578--2618.
\newblock
  \myhref{doi:10.1016/j.jfa.2017.07.002}{https://dx.doi.org/10.1016/j.jfa.2017.07.002}.

\bibitem[HS23]{HS23m}
\textsc{M.~Hairer} and \textsc{H.~Singh}.
\newblock Regularity structures on manifolds and vector bundles.
\newblock \emph{arXiv preprint} (2023).
\newblock \myhref{arXiv:2308.05049}{https://arxiv.org/abs/2308.05049}.

\bibitem[HS24]{HS23}
\textsc{M.~Hairer} and \textsc{R.~Steele}.
\newblock The bphz theorem for regularity structures via the spectral gap
  inequality.
\newblock \emph{Archive for Rational Mechanics and Analysis} \textbf{248},
  no.~1, (2024), 9.
\newblock
  \myhref{doi:10.1007/s00205-023-01946-w}{https://dx.doi.org/10.1007/s00205-023-01946-w}.

\bibitem[HS25]{HS23per}
\textsc{M.~Hairer} and \textsc{H.~Singh}.
\newblock Periodic space-time homogenisation of the $\phi^4_2$ equation.
\newblock \emph{Journal of Functional Analysis} \textbf{288}, no.~5, (2025),
  110762.
\newblock
  \myhref{doi:10.1016/j.jfa.2024.110762}{https://dx.doi.org/10.1016/j.jfa.2024.110762}.

\bibitem[HZZZ24]{HZZZ24}
\textsc{Z.~Hao}, \textsc{X.~Zhang}, \textsc{R.~Zhu}, and \textsc{X.~Zhu}.
\newblock Singular kinetic equations and applications.
\newblock \emph{The Annals of Probability} \textbf{52}, no.~2, (2024),
  576--657.
\newblock
  \myhref{doi:10.1214/23-AOP1666}{https://dx.doi.org/10.1214/23-AOP1666}.

\bibitem[Kup16]{Kup16}
\textsc{A.~Kupiainen}.
\newblock Renormalization {{Group}} and {{Stochastic PDEs}}.
\newblock \emph{Annales Henri Poincar{\'e}} \textbf{17}, no.~3, (2016),
  497--535.
\newblock
  \myhref{doi:10.1007/s00023-015-0408-y}{https://dx.doi.org/10.1007/s00023-015-0408-y}.

\bibitem[LOT23]{LOT}
\textsc{P.~Linares}, \textsc{F.~Otto}, and \textsc{M.~Tempelmayr}.
\newblock The structure group for quasi-linear equations via universal
  enveloping algebras.
\newblock \emph{Comm. Amer. Math. Soc.} \textbf{3}, (2023), 1--64.
\newblock \myhref{doi:10.1090/cams/16}{https://dx.doi.org/10.1090/cams/16}.

\bibitem[LOTT24]{LOTT}
\textsc{P.~Linares}, \textsc{F.~Otto}, \textsc{M.~Tempelmayr}, and
  \textsc{P.~Tsatsoulis}.
\newblock A diagram-free approach to the stochastic estimates in regularity
  structures.
\newblock \emph{Invent. Math.} \textbf{237}, no.~3, (2024), 1469--1565.
\newblock
  \myhref{doi:10.1007/s00222-024-01275-z}{https://dx.doi.org/10.1007/s00222-024-01275-z}.

\bibitem[Mou22]{Mou22}
\textsc{A.~Mouzard}.
\newblock Weyl law for the {A}nderson {H}amiltonian on a two-dimensional
  manifold.
\newblock \emph{Annales de l'Institut Henri Poincare (B) Probabilites et
  statistiques} \textbf{58}, no.~3, (2022), 1385--1425.
\newblock
  \myhref{doi:10.1214/21-AIHP1216}{https://dx.doi.org/10.1214/21-AIHP1216}.

\bibitem[MS25]{MS23}
\textsc{A.~Mayorcas} and \textsc{H.~Singh}.
\newblock Singular spdes on homogeneous lie groups.
\newblock \emph{Proceedings of the Royal Society of Edinburgh: Section A
  Mathematics} (2025), 1–72.
\newblock
  \myhref{doi:10.1017/prm.2025.1}{https://dx.doi.org/10.1017/prm.2025.1}.

\bibitem[MW20]{MW20}
\textsc{A.~Moinat} and \textsc{H.~Weber}.
\newblock Space-{T}ime {L}ocalisation for the {D}ynamic $\phi^4$ {M}odel.
\newblock \emph{Communications on Pure and Applied Mathematics} \textbf{73},
  no.~12, (2020), 2519--2555.
\newblock \myhref{doi:10.1002/cpa.21925}{https://dx.doi.org/10.1002/cpa.21925}.

\bibitem[OSSW24]{OSSW}
\textsc{F.~Otto}, \textsc{J.~Sauer}, \textsc{S.~A. Smith}, and
  \textsc{H.~Weber}.
\newblock A priori bounds for quasi-linear {{SPDEs}} in the full subcritical
  regime.
\newblock \emph{Journal of the European Mathematical Society} (2024).
\newblock \myhref{doi:10.4171/jems/1574}{https://dx.doi.org/10.4171/jems/1574}.

\bibitem[Sin25]{Sin23}
\textsc{H.~Singh}.
\newblock Canonical solutions to non-translation invariant singular spdes.
\newblock \emph{Electronic Journal of Probability} \textbf{30}, (2025), 1--24.
\newblock
  \myhref{doi:10.1214/25-EJP1315}{https://dx.doi.org/10.1214/25-EJP1315}.

\bibitem[Tem24]{Tem24}
\textsc{M.~Tempelmayr}.
\newblock Characterizing models in regularity structures: a quasi-linear case.
\newblock \emph{Probab. Theory Relat. Fields} (2024).
\newblock
  \myhref{doi:10.1007/s00440-024-01292-2}{https://dx.doi.org/10.1007/s00440-024-01292-2}.

\end{thebibliography}

\end{document}